\documentclass[10pt, a5paper, twoside, headsepline, DIV=calc]{scrbook}

\usepackage{amsmath, amsfonts, amsthm, amssymb, amsbsy}
\usepackage{mathrsfs} 
\usepackage{graphicx}
\usepackage{subfig}
\usepackage{caption} 
\usepackage{latexsym}
\usepackage[comma, authoryear]{natbib}
\usepackage{pgf} 
\usepackage{tikz} 
\usepackage{upgreek} 
\usepackage[ngerman,english]{babel}

\setkomafont{section}{\bfseries\boldmath\Large}
\setkomafont{subsection}{\boldmath}
\setcapindent{0em} 
\newtheorem{thm}{Theorem}[section]
\newtheorem{prop}[thm]{Proposition} 
\newtheorem{cor}[thm]{Corollary}
\newtheorem{lemma}[thm]{Lemma}
\theoremstyle{definition}
\newtheorem{exa}[thm]{Example}

\newtheorem{remark}[thm]{Remark}

\def\iid{i.i.d.\ }
\def\bl{\boldsymbol{\lambda}}
\def\Rstar{R^\star_{k^\star, l^\star}}
\def\lsnew{\lambda_{\gamma,\delta}^\star}
\def\blsnew{\boldsymbol{\lambda}_{\gamma, \delta}^\star}
\def\gade{1 + \gamma + \delta}

\def\bpis{\boldsymbol{\pi}^\star}
\def\ls{\lambda^\star}
\def\bls{\boldsymbol{\lambda}^\star}
\def\bz{\mathbf{z}}
\def\As{A^\star}

\begin{document}
\frenchspacing 
\begin{titlepage}
\vfill
\begin{center}
\setlength{\baselineskip}{20pt}
\begin{LARGE}{Stein's Method for Multivariate Extremes}\end{LARGE}
\\\vspace{+2cm}
\small{ {\bf Dissertation}\vspace{+0.5cm}
\\zur
\\Erlangung der naturwissenschaftlichen Doktorw\"urde
\\(Dr.\ sc.\ nat.)\vspace{+0.5cm}\\
vorgelegt der
\\Mathematisch-naturwissenschaftlichen Fakult\"at
\\der
\\Universit\"at Z\"urich\vspace{+0.5cm}
\\von\vspace{+0.5cm}
\\{\bf Anne Feidt}\vspace{+0.5cm}
\\aus
\\Luxemburg\vspace{1cm}
\\{\bf Promotionskomitee}\vspace{+0.3cm}
\\Prof.\ Dr.\ Erwin Bolthausen (Vorsitz)
\\Prof.\ Dr.\ Andrew Barbour \vspace{+0.5cm}
\\ Z\"urich, 2013}
\end{center}
\vfill
\end{titlepage}
\chapter*{\normalsize{Abstract}}
We apply the Stein-Chen method to problems from extreme value theory. On the one hand, the Stein-Chen method 
for Poisson approximation allows us to obtain bounds on the Kolmogorov distance between the law of the maximum 
of \iid random variables, following certain well known distributions, and an extreme value distribution. On 
the other hand, we introduce marked point processes of exceedances (MPPE's) whose \iid marks can be either 
univariate or multivariate. We use the Stein-Chen method for Poisson process approximation to determine 
bounds on the error of the approximation, in some appropriate probability metric, of the law of the MPPE 
by that of a Poisson process. The Poisson process that we approximate by has intensity measure equal to 
that of the MPPE. In some cases, this intensity measure is difficult to work with, or varies with the 
sample size; we then approximate by a further easier-to-use Poisson process and estimate the error of this additional approximation. 
\vspace{0.3cm}
\section*{\normalsize{Zusammenfassung}}
\selectlanguage{ngerman}
Wir wenden die Stein-Chen Methode auf Fragestellungen der Extremwerttheorie an. Einerseits 
erlaubt uns die Stein-Chen Methode f\"ur Poisson-Approxima\-tion die Kol\-mo\-go\-row-Dis\-tanz zwischen 
der Verteilung des Maximums von \iid Zufallsvariablen und einer Extremwertverteilung nach oben abzusch\"atzen. 
Andererseits f\"uhren wir markierte Punktprozesse von Grenwert\"uberschreitungen (MPPE genannt) ein, deren 
Markierungen so\-wohl univariat als auch multivariat sein k\"onnen. Wir verwenden die Stein-Chen Methode zur 
Approximation durch Punktprozesse, um den Fehler der Approximation der Verteilung eines MPPE durch die 
Verteilung eines Poissonprozesses in einer geeigneten Metrik abzusch\"atzen. Das Intensit\"atsmass des 
approximierenden Poissonprozesses ist dasselbe wie das des MPPEs. In manchen F\"allen ist dieses 
Intensit\"atsmass schwierig anwendbar; wir approximieren dann durch einen weiteren, einfacheren 
Poissonprozess und bestimmen den Fehler dieser zus\"atzlichen Approximation.
\vspace{0.5cm}
\phantom{\cite{Barbour/Hall:1984}} 
\selectlanguage{english}
\chapter*{\normalsize{Acknowledgements}}
Special thanks go to Prof.~A.~D.~Barbour for taking me on as PhD student and introducing me to Stein's method. 
I am most grateful for the time, patience and care he devoted to my benefit. 
His advice and continued support were essential toward the completion of this thesis. 
I would also like to thank Prof.~E.~Bolthausen for taking over my project and Prof.~J.~H\"usler for acting as external referee. 
Further thanks go to Prof.~J.~Ne{\v{s}}lehov\'a and Prof.~P.~
Embrechts for first introducing me to the study of extremes.  \\

I am grateful to my officemates, Philipp Thomann, Dominik Tasnady, Rajat Hazra and Martin Gallauer
for their support, both mathematical and non-mathematical, throughout the years. I likewise thank Anna P\'osfai, as well as my fellow assistants at the 
University of Zurich, in particular Alessandra Cipriani, Kim Dang, Noemi Kurt, Erich Baur, Felix Fontein, Dominik Heinzmann and Felix Rubin. \\ 

Moreover, I would like to thank Caroline, Fran\c{c}oise, Lynn, Martha, Brian, Claude, Georges and Josy for being my Luxembourgish family 
in Zurich. Finally, special thanks go to my parents for their constant support, and to Sigi, for everything.
\tableofcontents
\chapter{Introduction}\label{Chap: Introduction}
\setcounter{thm}{0}
Rare events of unexpected magnitude and momentous impact are playing an ever larger role all over the globe. 
In addition to ever-present threats of natural disasters such as earthquakes, volcanic eruptions or tsunamis, the last few
decades have seen more and more natural catastrophes that may have been induced by a changing climate due to global warming. Prime examples are
devastating cyclones like North Atlantic Hurricanes Andrew in 1992, Katrina in 2005, Sandy in 2012, or 
like North Indian Ocean Cyclone Aila in 2009. Other natural disasters are extreme rainfalls and floods
such as the 2005 European or 2010 Pakistan floods, as well as heat waves as experienced in Europe in 2003 or in North America in 2012. 
Further examples of extreme events are given by environmental disasters such as the 2010 Gulf of Mexico oil spill or the 2011 Fukushima Daiichi nuclear 
disaster.
Such events translate directly into extreme financial losses as damages need to be repaired. Extremely high financial losses can of course also
arise from other causes, most prominently from stock market crashes such as the Wall Street Crash of 1929 or the 2008 financial crisis.  
It is clear that it is of paramount importance to be able to model and predict the size and frequency of such events in order to both
establish adequate emergency measures to warrant the safety of people in danger from looming disasters, and to correctly determine 
minimum capital requirements of financial institutions such as insurance and reinsurance companies to prevent bankruptcy and ensure coverage of damages. 

The study of extremes has become a well-established field in probability theory and statistics. 
The classical approach in Extreme Value Theory (EVT)
is to model extremes as the maximum or minimum values of a sample of $n$ independent and identically distributed ({i.}{i.}d.) random variables
$X_1, \ldots, X_n$. Since $\min_{1 \le i \le n} X_i = -\max_{1 \le i \le n} (-X_i)$, it suffices to study maxima. A first major contribution by 
\cite{Fisher/Tippett:1928} assures that if there exists an affine transformation under which the maximum of \iid random variables converges in distribution to a 
non-degenerate distribution function $H$, then $H$ is an extreme value distribution, that is, 
either a Fr\'echet, Weibull or Gumbel distribution. Another approach in EVT, among others, is to consider point processes, for instance,
point processes of exceedances that indicate the points exceeding a threshold, or marked point processes that are random configurations of points in space. 
Such point processes have been treated in detail in, for instance \cite{Leadbetter_et_al:1983} and \cite{Resnick:1987}, respectively. 
The main result is that under certain conditions they asymptotically, for $n \to \infty$, behave like Poisson processes. 

Extreme value distributions and Poisson processes are generally used as approximations for the distributions of maxima and of point processes of exceedances, 
respectively, for any finite sample size $n$. However, such approximations only make sense if we have some estimate of the errors involved in terms of $n$. 
Underlying both kinds of approximations is Poisson approximation, since the number of those among the $n$ random variables of the sample that exceed a threshold
is binomial and may thus be approximated by a Poisson distribution, provided that the probability of a threshold exceedance is small. 
It thus makes sense to attempt to determine the accuracy of the above approximations in some probability metric by way of Stein's method for Poisson approximation.

Stein's method is a way to determine explicit bounds on the error involved in approximating one probability distribution by another.  
It was first introduced by \cite{Stein:1972} for approximation by the normal distribution. 
The method is however applicable to approximations by other distributions. Stein's method for Poisson approximation was developed 
by \cite{Chen:1975a,Chen:1975b} and is therefore often called the Stein-Chen method. The method compares expectations of test functions under the two distributions, 
with the choice of test functions determining the probability metric in which the error will be expressed. For example, with the choice of indicator functions of measurable
subsets of the state space, the difference between the expectations gives rise to the total variation distance. It is linked to an identity 
which characterises the approximating distribution via an equation, called the \textit{Stein equation}, whose solution needs to be determined in terms of the test functions. 
Instead of bounding the difference of the expectations, it turns out to be easier to bound the other side of the equation by using smoothness estimates of the solution.  
 
More concretely, the Stein-Chen method can be briefly sketched as follows. An integer-valued random variable $Z$ is Poisson distributed with mean $\lambda >0$ if and 
only if 
\[
E[\lambda g(Z+1) - Zg(Z)]=0 
\]
for any bounded function $g:\, \mathbb{Z}_+ \to \mathbb{R}$. Thus $E[\lambda g(Z+1) - Zg(Z)]=0$ is an identity 
(called \textit{Stein identity}) 
characterising the Poisson distribution. Let $Z \sim \mathrm{Poi}(\lambda)$ and let $W$ be another integer-valued random variable that we suppose
to be almost Poisson distributed. Then we expect that $\mathbb{E}h(W) \approx \mathbb{E}h(Z)$ for a class of test functions $h$. 
By the Stein equation whose solution $g$ needs to be determined in terms of $h$, 
we have 
\[
\mathbb{E}h(W) - \mathbb{E}h(Z)= E[\lambda g(Z+1) - Zg(Z)],
\]
and it is clear that instead of trying to find a uniform bound on $\mathbb{E}h(W) - \mathbb{E}h(Z)$,
we may instead attempt to do so for $E[\lambda g(Z+1) - Zg(Z)]$ by using properties of $g$. 
Note that with the choice $h:= I_A$, where $A \subseteq\mathbb{Z}_+$, 
a uniform upper bound on the difference of the expectations 
$\mathbb{E}h(W) - \mathbb{E}h(Z)$ results in a bound on the total variation distance between the laws of $W$ and $Z$, 
i.e.\ on $d_{TV}(\mathcal{L}(W),\mathrm{Poi}(\lambda))$ $=$ $ \sup_{A} |P(W \in A) - P(Z \in A)|$.

In standard examples it is assumed that $W= \sum_{i=1}^n I_i$, where the $I_i$ are independent Bernoulli random variables with 
success probabilities $p_i$, $i=1, \ldots, n$, or, more simply, \iid Bernoulli random variables with success probability $p$. 
Chapter 1 in \cite{Barbour_et_al.:1992} reviews earlier results for bounds on $d_{TV}(\mathcal{L}(W), \mathrm{Poi}(\lambda))$ which were established without
the use of Stein's method. One of these is the bound $\sum_{i=1}^n p_i^2$ by \cite{LeCam:1960}. It was also obtained by \cite{Serfling:1975} by using the 
simple coupling inequality $d_{TV}(\mathcal{L}(W),\mathcal{L}(Z)) \le P(W \neq Z)$. 
Another bound by \cite{LeCam:1960} is $8\lambda^{-1}\sum_{i=1}^n p_i^2$ which was sharpened by \cite{Kerstan:1964} 
to $1.05\lambda^{-1}\sum_{i=1}^n p_i^2$. Compared to these, the Stein-Chen method gives the sharpest bound:
\[
(1-e^{-\lambda})\lambda^{-1}\sum_{i=1}^n p_i^2 \le \min(1, \lambda^{-1})\sum_{i=1}^n p_i^2.
\]
The main strength of the Stein-Chen method however lies in that it is easily 
adapted to the case where the Bernoulli variables are no longer independent; see, e.g. \cite{Chen:1975b} or \cite{Barbour_et_al.:1992}. 

By writing the solution $g$ of the Stein equation as the first backward difference of a function $\upgamma$, i.e.\ $g(k)= \upgamma(k)-\upgamma(k-1)$ for all $k \in \mathbb{Z}_+$,
\cite{Barbour:1988} noticed that the Stein identity for Poisson approximation could be rephrased in terms of the infinitesimal generator
of a Markov process with Poisson equilibrium distribution. More precisely, 
\[
\lambda g(k+1)- kg(k) = \lambda \upgamma(k+1)+ k\upgamma(k-1) -(\lambda + k)\upgamma(k) = (\mathcal{A}\upgamma)(k), 
\]
for all $k \in \mathbb{Z}_+$, where $\mathcal{A}$ is the infinitesimal generator of an immigration-death process on $\mathbb{Z}_+$ with constant immigration rate $\lambda$ and unit per-capita death rate. 
The solution $\upgamma$ to the reformulated Stein equation can then be interpreted in a probabilistic way and coupling 
arguments may be used to determine smoothness estimates. 
\cite{Barbour/Brown:1992} used this approach to develop the Stein-Chen method for Poisson process
approximation. For approximation of the law of a point process $\Xi$ with finite mean measure $\bl$ on a locally compact separable metric space $E$, the Markov process used is a 
spatial immigration-death process with immigration intensity $\bl$ and unit per-capita death rate,
whose equilibrium distribution is a Poisson process with mean measure $\bl$ (that we denote by $\mathrm{PRM}(\bl)$).
\cite{Barbour/Brown:1992} established bounds on $d_{TV}(\mathcal{L}(\Xi), \mathrm{PRM}(\bl))$ that resemble the results obtained for 
Poisson approximation, but lack the multiplicative factors depending on $\lambda=\bl(E)$ which decrease, and thereby reduce the 
error, as $\lambda$ increases. The reason for this is that if there are small shifts in the positions of the points of the two processes,
the total variation distance takes its maximum value, i.e.\ $1$. 
An example is given by the following: let
$I_1, \ldots, I_n$ be independent Bernoulli random variables with success probability $p \in (0,1)$ and let $W = \sum_{i=1}^n I_i$. Let
$\Xi_n = \sum_{i=1}^n I_i \delta_{i/n}$ be a point process
on $[0,1]$ with $\mathbb{E}\Xi_n([0,1]) = np$. It would be natural to want to compare this process by a Poisson process $\Xi$ with rate $np$ on $[0,1]$,
which can be realised by $\sum_{j=1}^{W^\star} \delta_j$, where $W^\star \sim \mathrm{Poi}(np)$. However, the Poisson process then takes no points in $n^{-1}\mathbb{Z}_+$ with probability $1$,
whereas the Bernoulli process $\Xi_n$ takes no points in $n^{-1}\mathbb{Z}_+$ with probability 
$P(I_1 =0, \ldots, I_n=0) = (1-p)^n = \exp\{{n\log (1-p)}\} \sim e^{-\lambda}$. Thus, 
for $R:= \{\sum_{j \in n^{-1}\mathbb{Z}_+} \delta_j \}$,
\[
d_{TV}(\mathcal{L}(\Xi_n),\mathcal{L}(\Xi) ) \ge \left| P(\Xi_n \notin R) - P(\Xi \notin R)\right| \sim |1-e^{-\lambda}| \to 1, 
\]
as $\lambda$ (and $n$) $\to \infty$. The total variation distance is thus not suited to approximate a point process with points
on a lattice, for instance $\mathbb{Z}^d$, by a Poisson process with continuous intensity on $\mathbb{R}^d$, where $d \ge 1$. 
\cite{Barbour/Brown:1992} therefore construct the weaker $d_2$-distance which 
recovers factors in $\lambda$ that reduce the error. 
For the special case of marked point processes $\Xi= \sum_{i=1}^n I_i \delta_{X_i}$, smaller errors may yet be obtained in the total variation distance.
Indeed, if the marks $X_1, \ldots, X_n$ are \iid (and independent of the $I_i$'s), 
$\Xi$ can be realised as $\sum_{j=1}^{W} \delta_{Z_j}$, with $Z_j$ i.i.d., distributed as $\mathcal{L}(X_1)$
and independent of $W= \sum_{i=1}^n I_i$, whereas a Poisson process with mean measure $\bl$ can be realised as $\sum_{j=1}^{W^\star} \delta_{Z_j}$, with 
$W^\star \sim \mathrm{Poi}(\lambda)$ independent of the $Z_j$'s.  
An argument made by \cite{Michel:1988} shows that the bounds obtained for Poisson approximation may then be reused:
\begin{equation}\label{intro: Michel}
d_{TV}(\mathcal{L}(\Xi), \mathrm{PRM(\bl)}) \le d_{TV}(\mathcal{L}(W), \mathrm{Poi}(\lambda)).
\end{equation}

The aim of this thesis is to apply the Stein-Chen method for Poisson and for Poisson process approximation to problems from EVT, with a focus on multivariate extremes. 
In particular, we are interested in studying random configurations of ``extreme points'' in space. By ``extreme point''
we always mean an atypical realisation of a random variable $X$ or a $d$-dimensional random vector $\mathbf{X}= (X_1, \ldots, X_d)$. Throughout the thesis, we suppose that
we have a sample of $n$ \iid random vectors $\mathbf{X}_1, \ldots, \mathbf{X}_n$ with state space $E \subseteq \mathbb{R}^d$, where $d \ge 1$, and we denote 
by $A=A_n$ a measurable subset of $E$ containing the extreme points. For $d=1$, $A$ is typically of the form $(u_n, \infty)$, i.e.\ it contains exceedances of a threshold
$u_n$. For $d>1$, there is more flexibility as to the choice of ``extreme region'' $A$. 
We might, for instance, set $A:= [u_1, \infty) \times \ldots \times [u_d, \infty)$ 
which implies that points in $A$ are extreme in \textit{all} 
components. We might also define $A$ as the complement of $(-\infty, u_1) \times \ldots \times (-\infty, u_d) $;
then $A$ not only contains jointly extreme points but also points that might have only one extreme component. 
We denote the random number of points in $A$ by $W_A = \sum_{i=1}^n I_{\{\mathbf{X}_i \in A\}}$, where $I_{\{\mathbf{X}_i \in A\}}$, $i=1, \ldots, n$, are Bernoulli random variables with probability of success $P(\mathbf{X} \in A)$.
Furthermore, we define \textit{marked point processes of exceedances} (MPPE's) by
\[
\Xi_A = \sum_{i=1}^n I_{\{\mathbf{X}_i \in A\}} \delta_{\mathbf{X}_i}, 
\]
where $\delta_{\mathbf{x}}$ denotes Dirac measure on $\mathcal{B}(E)$, and we call $\mathbf{X}_1, \ldots, \mathbf{X}_n$ the \textit{marks} of the MPPE.
Using (\ref{intro: Michel}) and results from the Stein-Chen method for Poisson approximation, we establish the following general error estimate:
\begin{equation}\label{intro: Michel_MPPE}
d_{TV}(\mathcal{L}(\Xi_A), \mathrm{PRM}(\mathbb{E}\Xi_A))\le d_{TV}(\mathcal{L}(W_A), \mathrm{Poi}(\mathbb{E}W_A))
\le P(\mathbf{X} \in A). 
\end{equation}
This result serves as the basis for a wide range of applications.

For a first application, let $X_{(n)}$ be the maximum of a sample of $n$ random variables whose distribution function and survival function 
we denote by $F$ and $\overline{F}$, respectively. Then (\ref{intro: Michel_MPPE}) in particular gives
\[
\left| P\left(W_A = 0\right) - \mathrm{Poi}(\mathbb{E}W_A)\{0\}\right| = \left| P\left(X_{(n)} \le u_n \right) - e^{-n\overline{F}(u_n)}\right|  \le \overline{F}(u_n),
\]
which we use to determine the error of the approximation in the Kolmogorov distance of the law of $X_{(n)}$ by that of an extreme value distribution. We achieve this
for random variables whose distribution is one of the following: exponential, Pareto, uniform, standard normal, standard Cauchy, geometric. 
Sometimes the approximation to an extreme value distribution has to be done in more than one step; we then give error bounds for each step.
The geometric distribution is a special case here, as it is well known that there is no non-degenerate limit distribution for the distribution of 
the maximum unless $\lim_{x \uparrow x_F}\bar{F}(x)/\bar{F}(x-) =1$, where 
$x_F$ denotes the right endpoint of $F$. This condition is not satisfied 
for some of the most well known discrete distributions, in particular, the 
geometric distribution. However, by letting the parameter(s) of the distribution 
vary with the sample size $n$ at a suitable rate, it is possible to find a limiting 
distribution for the law of the maximum. For geometric random variables, we show that 
we do not actually need the success probability to be varying with $n$ in order to 
determine a limit law for their maximum, if we allow this limit law to be discrete. 
More precisely, we approximate the law of the maximum by a discretised Gumbel distribution 
and determine a bound on the error of this approximation. We also establish an estimate of the error that arises from the approximation by the 
continuous Gumbel distribution, which will clarify in what way 
the success probability needs to vary with $n$ for a small error. Furthermore, we determine 
a better choice of normalising constants than the ones used by \cite{Nadarajah/Mitov:2002} 
for their asymptotic result.

We further apply (\ref{intro: Michel_MPPE}) to MPPE's whose marks are distributed according to any of the univariate distributions mentioned above.
The error will be given in terms of the chosen threshold, and we discuss the relation between the expected number of threshold exceedances 
and the size of the error of the approximation in the total variation distance. 
The bigger the threshold, the smaller the error will be and the fewer exceedances will be expected. We can thus regulate the size of the threshold according to how many of
the biggest order statistics we want the MPPE to capture. 
Moreover, as the aim should be to approximate $\mathcal{L}(\Xi_A)$
by a Poisson process that is easy to use, we determine in each case a practicable intensity function for the approximating Poisson process. 
Sometimes we would like to approximate $\mathcal{L}(\Xi_A)$ by a Poisson process with a different 
intensity measure than $\mathbb{E}\Xi_A$, say, $\tilde{\bl}_A$, which we suppose 
has a ``nicer'' intensity function. We then add an error estimate for 
$d_{TV}(\mathrm{PRM}(\mathbb{E}\Xi_A), \mathrm{PRM}(\tilde{\bl}_A))$ to the previous error. Again, the geometric distribution is a special case,
as the intensity measure of an MPPE with geometric marks is only defined on a lattice. 
We would prefer to approximate the law of the MPPE with geometric marks by a Poisson process 
with continuous intensity function equal to that obtained for the MPPE with exponential marks. As the total variation distance is too strong to
achieve this, we instead use the weaker $d_2$-distance. For a sharp error estimate, we again need the condition that the success probability 
of the geometric distribution varies with $n$ at a suitable rate. 

We next suppose that $d \ge 2$, i.e.\ that the marks of the MPPE's are multivariate. We distinguish between sets $A$ that contain only points 
which exceed thresholds in all $d$ components (which we call ``\textit{joint threshold exceedances}''), 
and sets $A$ that also contain points for which there might be a threshold exceedance in only one
component (which we dub ``\textit{single-component threshold exceedances}'').
It only makes sense to study joint threshold exceedances if the common 
multivariate distribution of the marks is likely to give rise to joint extremes, though it is 
not always clear how to find out whether this is the case. One possibility, for $d=2$, is to 
compute the coefficient of upper tail dependence. 
An issue that arises for MPPE's with multivariate marks is thus the choice of the set $A$. 
Another issue is that there are infinitely many choices for the dependence structure between the marginal distributions of the marks. 
Though (\ref{intro: Michel_MPPE}) gives a very general result,  
it might be difficult to understand in what way the error estimate 
$P(\mathbf{X} \in A)$ behaves and how many threshold exceedances to expect if the joint distribution function has a complicated structure.
For both kinds of sets $A$, we therefore determine easier error bounds for 
$d_{TV}(\mathcal{L}(\Xi_A), \mathrm{PRM}(\mathbb{E}\Xi_A))$, which are valid for any joint 
distribution function of the marks. Again, the aim should be to approximate an MPPE by a Poisson process with 
a practicable intensity function. 
Also, the Poisson process should, if possible, be independent of the sample size $n$. In cases where the intensity function of the Poisson process does not
meet these requirements, the goal is to approximate by a further Poisson process with a nicer intensity and to determine an error bound on this approximation. 
Ad hoc considerations are needed to determine whether this is needed, 
and we therefore have to restrict ourselves to examples.
We treat two main examples of MPPE's $\Xi_A$ for which the intensity function of the approximating $\mathrm{PRM}(\mathbb{E}\Xi_A)$ is both too complicated
to handle with ease, and varies with $n$. 

First, we study the example of MPPE's whose marks are distributed according to any one out of a subclass of 
Archimedean copulas that exhibit upper tail dependence.
More precisely, we suppose that the generator function $\phi$ of the Archimedean copulas that we use satisfies 
$\lim_{r\downarrow 0} r\phi'(1-r)/\phi(1-r) = \theta \in (1, \infty)$. 
\cite{Charpentier/Segers:2009} showed that these copulas have a certain kind of asymptotic tail behaviour, that we make use of to determine 
a more suitable intensity function. We proceed by establishing a bound on the error in the total variation distance of the approximation
of the Poisson process with mean measure $\mathbb{E}\Xi_A$ by the Poisson process with the new intensity function. 
Secondly, we consider an MPPE whose marks follow the bivariate Marshall-Olkin geometric distribution. As for univariate geometric marks, the 
approximating Poisson process lives on a lattice and the aim is to further approximate by a Poisson process with a continuous intensity. We construct a suitable
continuous intensity function by spreading the point probabilities of the Marshall-Olkin geometric distribution uniformly over the coordinate squares of the lattice.
This intensity function still has to depend on the sample size $n$ if we want the corresponding error to be sharp. We therefore determine its pointwise limit
as $n \to \infty$, under certain conditions on the parameters of the Marshall-Olkin geometric distribution, and use this limit as the new intensity function. 
We prove error bounds for each step in the $d_2$-distance, and, whenever possible, in the total variation distance. 

Many more examples could be studied and we hope that this thesis serves as a starting point for more research on the application of the Stein-Chen method for
Poisson and Poisson process approximation to topics in EVT. A noticeable limitation of our work is the focus on \iid samples, as one of the strengths of Stein's method is that it is applicable also to random variables that display some dependence on each other; see Section \ref{s: Remarks_on_choice} for a brief discussion.

The structure of the thesis is as follows: 
\textbf{Chapter \ref{Chap: Stein-Chen}} introduces the Stein-Chen method in detail, both for Poisson approximation and 
for Poisson process approximation. \textbf{Chapter \ref{Chap: Univariate_extremes}} then treats univariate extremes, 
that is, the approximation of the maximum law of univariate random variables by an extreme value distribution in the 
Kolmogorov distance, as well as the approximation of MPPE's by suitable Poisson processes in the total variation distance, or, 
if necessary, the $d_2$-distance. \textbf{Chapter \ref{Chap: Multivariate_extremes}} studies Poisson process approximation for MPPE's with multivariate marks that either have independent components, or components following a certain dependence structure which we describe by way of copulas.



\chapter{The Stein-Chen method}\label{Chap: Stein-Chen}
\setcounter{thm}{0}

Stein's method provides a way to determine bounds on the errors that arise when 
approximating one probability distribution by another. This chapter gives an introduction to Stein's method for approximation by a Poisson distribution, as developed by \cite{Chen:1975a}, as well as for approximation by a Poisson process, as studied by \cite{Barbour/Brown:1992}.  
Section \ref{Sec: Poi_approx_Bin} first discusses the error involved in the law of small numbers, that is, in the approximation of the binomial distribution 
by the Poisson distribution. 
Section \ref{Sec: distances} lists some distances between probability measures: 
the total variation, Kolmogorov, and Wasserstein distances. We later always express the errors of the approximations that we study in one of these distances. 
Section \ref{Sec: LeCam} gives an error bound for the approximation, in the total variation distance, of a sum of independent indicator variables by a Poisson distribution. 
This result by \cite{LeCam:1960} only uses a simple coupling argument. 
Section \ref{Sec: PoiAppr_dTV} develops the Stein-Chen method for Poisson approximation in the total variation distance and gives an improvement on the 
result obtained in Section \ref{Sec: LeCam}. It also treats a result by \cite{Chen:1975b} for sums of dependent indicator variables, and outlines two more general procedures proposed by \cite{Stein:1986} and \cite{Barbour:1988}, respectively. 
Section \ref{Sec: Poi_proc_approx} introduces point processes and Poisson processes, and proceeds to develop the Stein-Chen method for approximation by
a Poisson process, again in the total variation distance, as achieved by \cite{Barbour/Brown:1992}. Since errors are worse than for Poisson approximation, 
Section \ref{Sec: Improved_rates} introduces the $d_2$-distance by \cite{Barbour/Brown:1992}, which is weaker
than the total variation distance and yields sharper estimates.

\section{Poisson approximation of the binomial distribution}\label{Sec: Poi_approx_Bin}
The Poisson distribution describes the probability of a given number of independent events occurring within a fixed interval of time (or space) 
when the average rate of occurrence of such events is known from previous observation. 
It is named after Sim\'eon Denis Poisson who introduced it in his 1837 treatise
\textit{Recherches sur la probabilit\'e des jugements en mati\`ere criminelle et en mati\`ere civile}, where he showed that the Poisson distribution arises
as the limit of a binomial distribution whose probability of success $p=p_n$ varies with the sample size $n$ in such a way that $p_n \to 0$ and
$np_n \to \lambda>0$ as $n \to \infty$. This asymptotic result is called the ``law of small numbers'' or
``law of rare events''. 
We show that the binomial converges pointwise to the Poisson by considering the probability mass function of a $\mathrm{Bin}(n,p_n)$-random variable $W$ 
for any fixed $k \in \mathbb{Z}_+$, and for any integer $n \ge k$:
\begin{align}\label{d: Binomial_point_prob}
P(W=k) &= \binom{n}{k} p_n^k (1-p_n)^{n-k} \nonumber  \\
&=  \prod_{j=0}^{k-1} (n-j)  \,\frac{p_n^k}{k!} \, e^{(n-k) \log(1-p_n)} \nonumber \\
& = \prod_{j=0}^{k-1} \left( 1- \frac{j}{n}\right) \, \frac{(np_n)^k}{k!} \, e^{n\log(1-p_n)} \cdot e^{-k\log(1-p_n)}\\
&\sim \frac{(np_n)^k}{k!} \,  e^{-np_n}  \to \frac{\lambda^k}{k!}\,  e^{-\lambda}, \quad \textnormal{ as } n \to \infty, \nonumber
\end{align}
since $\log(1-p_n) \sim -p_n$, as $p_n \to 0$.
In view of this result, it is natural to think of approximating a binomial distribution by a Poisson distribution even for a fixed sample size $n$,
as long as the sample size is large and the success probability is small. But what is the error resulting from such an approximation of a $\mathrm{Bin}(n,p_n)$ distribution 
by a $\mathrm{Poi}(np_n)$? 
We may assume that $np_n^2 \le 1/2$, i.e.\ $p_n \le (2n)^{-1/2}$, since else $np_n > (2p_n)^{-1} \to \infty$ as $p_n \to 0$. 
On the one hand, note that
\[
\prod_{j=0}^{k-1}\left(1-\frac{j}{n}\right) \le e^{-\frac{k(k-1)}{2n}},
\]
since $1-j/n \le e^{-j/n}$ for all $j\ge 0$, and $\sum_{j=0}^{k-1} j/n = k(k-1)/2n$.
Moreover, 
\[
(1-p_n)^{n-k} \le e^{-(n-k)\left( p_n + \frac{p_n^2}{2}\right)}, 
\]
since $\log(1-p_n) \le -p_n - p_n^2/2$.
It follows that $P(W=k)$ from (\ref{d: Binomial_point_prob}) is smaller than
\[
\frac{(np_n)^k}{k!}\, e^{-np_n}  \cdot \exp\left\{ -\frac{1}{2} \left( \frac{k^2}{n} + np_n^2\right) + kp_n + \frac{1}{2} \left( \frac{k}{n} + kp_n^2\right)\right\}. 
\]
By the inequality of the arithmetic and geometric means, we have 
\[
\frac{1}{2} \left( \frac{k^2}{n} + np_n^2\right) \ge \sqrt{\frac{k^2}{n} \cdot np_n^2} = kp_n, 
\]
and thus
\[
P(W=k) -  \mathrm{Poi}(np_n)\{k\} \le \frac{(np_n)^k}{k!} \, e^{-np_n} \left[ \exp\left\{\frac{1}{2} \left( \frac{k}{n} + kp_n^2\right) \right\} - 1 \right].
\]
Since, for $n \ge k$ and $np_n^2 \le \frac{1}{2}$, we have $\exp\left\{\frac{1}{2} \left( \frac{k}{n} + kp_n^2\right) \right\} \le 1+e^{3/4}\left(\frac{k}{n} + np_n^2\right)$ 
(using $e^{z/2} \le 1+ ze^{z/2}$ for $z \ge 0$), it holds that
\begin{equation}\label{p: upp_bound_bin_minus_Poi}
P(W=k) -  \mathrm{Poi}(np_n)\{k\}  \le  \frac{(np_n)^k}{k!}\, e^{-np_n}\cdot e^{3/4} \left(\frac{k}{n} + np_n^2\right).
\end{equation}
On the other hand, since $np_n^2 \le 1/2$, we have $p_n \le 1/2$ for all $n \ge 2$, and therefore
$(1-p_n)^{-1} \le 1+2p_n$. It follows that 
\[
e^{n \log(1-p_n)}  \ge e^{-\frac{np_n}{1-p_n}} \ge e^{-np_n - 2np_n^2} \ge e^{-np_n} \left(1-2np_n^2\right).
\]
Moreover, an induction proof readily shows that $\prod_{j=0}^{k-1}(1-j/n)\ge 1 - k^2/n$. Thus,
\begin{align}\label{p: upp_bound_Poi_minus_bin}
\mathrm{Poi}(np_n)\{k\} - P(W=k) 
&\le \frac{(np_n)^k}{k!}\, e^{-np_n} \left\{ 1-\left(1-2np_n^2 \right)\left(1-\frac{k^2}{n}\right)\right\} \nonumber\\
&\le \frac{(np_n)^k}{k!}\, e^{-np_n} \left( \frac{k^2}{n} + 2np_n^2\right).
\end{align}
It follows from (\ref{p: upp_bound_bin_minus_Poi}) and (\ref{p: upp_bound_Poi_minus_bin}) that, for $k \le n$ and $np_n^2 \le 1/2$,
\begin{equation}\label{t: Bound_Bin_Poi}
\left| P(W=k) - \mathrm{Poi}(np_n)\{k\}\right| \le  e^{3/4} \, \frac{(np_n)^k}{k!} \,e^{-np_n} \left(\frac{k^2}{n}  + np_n^2\right).
\end{equation}
Note that we have not used $np_n \to \lambda$, as $n \to \infty$, in order to establish (\ref{t: Bound_Bin_Poi}), and that (\ref{t: Bound_Bin_Poi})
is stronger than just a limit result: it gives an explicit estimate of the error of the approximation of a binomial by a Poisson point probability at $k$ for any 
sample size $n \ge k$. By (\ref{t: Bound_Bin_Poi}), the approximation is good so long as the terms $np_n^2$ and $k^2n^{-1}$ are small.
We can aim for an even stronger statement than (\ref{t: Bound_Bin_Poi}) by investigating the accuracy of the approximation 
of one probability distribution by another in a probability metric, such as the total variation distance that we will define below.
For instance, for any subset $A \subseteq \{0, 1, \ldots, n\}$, (\ref{t: Bound_Bin_Poi}) gives
\begin{align}
&\left| P(W \in A) - \mathrm{Poi}(np_n)\{A\}\right| \nonumber \\
&\le e^{3/4} e^{-np_n} \sum_{k \in A} \frac{(np_n)^k}{k!} \left(\frac{k^2}{n}  + np_n^2\right)\nonumber \\
& \le  e^{3/4}e^{-np_n} \left\{np_n^2\sum_{k=0}^\infty \frac{(np_n)^k}{k!} + p_n \sum_{k=1}^\infty \frac{(np_n)^{k-1}}{(k-1)!} \, k \right\}\nonumber \\
& = e^{3/4}\left\{2np_n^2 +p_n\right\},\label{p: bin_Poi_dTV_firsthalf}
\end{align}
where we used 
\[
\sum_{k=1}^\infty \frac{(np_n)^{k-1}}{(k-1)!} \, k = \sum_{l=0}^\infty \frac{(np_n)^{l}}{l!} \, (l+1) 
= np_n \sum_{l=1}^\infty \frac{(np_n)^{l-1}}{(l-1)!} + \sum_{l=0}^\infty \frac{(np_n)^l}{l!} 
\]
and $\sum_{k=0}^\infty (np_n)^k/k! = e^{np_n}$. Furthermore, note that $P(W \in [n+1,\infty))=0$, and that for all $n \ge 2$, $\mathrm{Poi}(np_n)\{[n+1,\infty)\}$ 
is of smaller order than the bound in (\ref{p: bin_Poi_dTV_firsthalf}). More precisely, Proposition A.2.3 in
\cite{Barbour_et_al.:1992}, gives, for all $n \ge 2$, 
\begin{align*}
\mathrm{Poi}(np_n)\{[n+1,\infty)\} &\le \frac{(n+2)e^{-\frac{(n+1-np_n)^2}{2(n+1+np_n)}}}{(n+2-np_n)\sqrt{2\pi (n+1)}} 
 \le \frac{e^{-\frac{n}{2}\left(\frac{1}{2} + \frac{1}{n}\right)^2/ \left(\frac{3}{2} + \frac{1}{n}\right)}}{\left( 1-\frac{np_n}{n+2}\right)\sqrt{2\pi n}}  \\
& \le \left( 1 + \frac{n}{n+2}\right) \frac{e^{-\frac{n}{4} \left( \frac{1}{4} + \frac{1}{n} + \frac{1}{n^2} \right)}}{\sqrt{2\pi n}}
\le \frac{e^{-\frac{n}{16}}}{\sqrt{ n}}\,,
\end{align*}
where we also used that $p_n \le 1/2$ for $n \ge 2$, and $(1-z)^{-1} \le 1+2z$ for $z= np_n/(n+2)$. 
For any $A \subseteq \mathbb{Z}_+$, the error $|P(W \in A) - \mathrm{Poi}(np_n)\{A\}|$ (and thereby the total variation distance between
$\mathcal{L}(W)$ and $\mathrm{Poi}(np_n)$; see Section \ref{Sec: distances} below) is then at most of order $\max(np_n^2, p_n)$.  
The approximation between the $\mathrm{Bin}(n, p_n)$ and the $\mathrm{Poi}(np_n)$ is therefore sharp if 
$p_n = o\left(n^{-1/2}\right)$.
\section{Distances between probability measures}\label{Sec: distances}
In general the aim is to find an upper bound for the difference between the expectations of a test function 
from a predetermined family of test functions under the two distributions. Each family of test functions determines an associated metric. We list
three important examples of such distances between probability measures: the total variation, the Kolmogorov and the Wasserstein distances. 
For each of these, suppose that $\mu$ and $\nu$
 are two probability measures on a measurable space $(E, \mathcal{E})$.
\subsection*{Total variation distance} The \textit{total variation distance} between $\mu$ and $\nu$ is defined as follows:
\[
d_{TV}(\mu,\nu) = \sup_{h \in H} \left| \int_E hd\mu - \int_E h d\nu  \right| = \sup_{B \in \mathcal{E}} |\mu(B)-\nu(B)|,
\]
where the test functions $h$ are indicators of measurable subsets of $E$, i.e.\ $H :=  \{I_B;\, B \in \mathcal{E}\}$, where, for any $x \in E$,
$I_B(x)=1$ if $x \in B$ and $I_B(x)=0$ if $x \notin B$. 
Note that for any set $B \in \mathcal{E}$, we have $\mu(B) - \nu(B) = \nu(B^C) - \mu(B^C)$, with $B^C \in \mathcal{E}$ the complement of $B$. Therefore,
\begin{equation}\label{d: dTV_without_mod}
d_{TV}(\mu,\nu) =  \sup_{B \in \mathcal{E}} \left\{\mu(B)-\nu(B)\right\}.
\end{equation}
An equivalent definition of the total variation distance is given by
\begin{equation}\label{d: equiv_def_dTV}
d_{TV}(\mu,\nu) = \sup_{\tilde{h} \in \widetilde{\mathcal{H}}}  \left| \int_E \tilde{h}d\mu - \int_E \tilde{h} d\nu  \right|,
\end{equation}
where $\widetilde{\mathcal{H}}:= \{\tilde{h}: E \to \mathbb{R};\, 0 \le h(x) \le 1, \text{ for all }x \in E \}$. 
In order to see that the two definitions are equivalent, note first that any $\tilde{h} \in \widetilde{\mathcal{H}}$ may be defined as $\tilde{h}(x):= h(x)- q$, 
for any $x \in E$, $h \in \mathcal{H}$, and for any choice of $q \in [0,1]$.
Then, 
\begin{align*}
\sup_{\tilde{h} \in \widetilde{\mathcal{H}}} \left| \int_E \tilde{h}d\mu - \int_E \tilde{h} d\nu  \right|
&= \sup_{h \in H} \left| \int_E (h - q)d\mu - \int_E (h - q) d\nu  \right|\\
&= \sup_{h \in H}  \left| \int_E hd\mu - q \mu(E) - \int_E h  d\nu + q\nu(E) \right|\\
&= \sup_{h \in H} \left| \int_E hd\mu - \int_E h d\nu  \right|,
\end{align*}
since $\mu(E)=\nu(E)=1$. Further equivalent definitions can be found in \cite{Barbour_et_al.:1992}, pp. 253-254.
Note also that if $E$ is a separable metric space 
and if $X$ and $Y$ are two $E$-valued random variables defined on the same probability space with distributions $\mu$ and $\nu$, respectively,
then, for $B \in \mathcal{E}$,
\begin{align*}
\mu(B)-\nu(B) &= P(X \in B) - P(Y \in B)
 = \mathbb{E}\left[I_B(X)-I_B(Y)\right]\\
 &\le \mathbb{E}I_{\{X \neq Y\}}
 = P(X \neq Y), 
\end{align*}
and therefore
\begin{equation}\label{t: coupling_dTV}
d_{TV}(\mu,\nu) = \sup_{B \in \mathcal{E}} \left|P(X \in B) - P(Y \in B)\right| \le P(X \neq Y).
\end{equation}
\begin{remark}
We used a \textit{coupling} of the two distributions $\mu$ and $\nu$ in order to establish (\ref{t: coupling_dTV}). Throughout this work, we will time and again
use couplings. In general, coupling means the joint construction of two random variables (or processes) $X$ and $Y$ which (marginally) 
follow two distributions $\mu$ and $\nu$
of interest. Its purpose is to relate the two previously unrelated distributions $\mu$ and $\nu$ in some way so as to be able to compare them. For details on the coupling
method, consult \cite{Lindvall:2002} or \cite{Thorisson:2000}. 
\end{remark}
\begin{remark}
Suppose that $(X_n)_{n \in \mathbb{N}}$ is a sequence of random variables with a discrete state space $E$, for instance $E=\mathbb{Z}_+$. This sequence then converges
in total variation to a random variable $X$, i.e.\ $d_{TV}(\mathcal{L}(X_n),\mathcal{L}(X)) \to 0$ as $n \to \infty$, if and only if $P(X_n=k) \to P(X=k)$ for all 
$k \in E$ as $n \to \infty$, that is, if and only if it converges in distribution (or weakly) to $X$. 
We can see this by noting that, on the one hand, convergence in distribution follows from convergence in total variation, since, for any $k \in E$,
\begin{align*}
0 &\le |P(X_n =k) - P(X=k)| \\
&\le \sup_{B \subseteq E} |P(X_n \in B) - P(X \in B)|\\
&= d_{TV}(\mathcal{L}(X_n),\mathcal{L}(X)).
\end{align*}
On the other hand, suppose that $P(X_n=k) \to P(X=k)$ for all $k \in E$, as $n \to \infty$, define $\tilde{B}:= \{k \in E:\, P(X_n = k) \ge P(X=k)\}$,
and note that $d_{TV}(\mathcal{L}(X_n),\mathcal{L}(X))$ equals
\begin{align*}
&P(X_n \in \tilde{B}) - P(X \in \tilde{B}) \\
&= \frac{1}{2} \left\{ P(X_n \in \tilde{B}) - P(X \in \tilde{B}) + P(X \notin \tilde{B}) - P(X_n \notin \tilde{B}) \right\}\\
&= \frac{1}{2} \left\{ \sum_{k \in \tilde{B}} \left[P(X_n=k) - P(X=k)\right] + \sum_{k \in \tilde{B}^C} \left[P(X=k) - P(X_n=k)\right] \right\}\\
& = \frac{1}{2} \sum_{k \in E} \left|P(X_n = k) - P(X=k)\right|. 
\end{align*}
For each $\epsilon > 0$, there exists a finite set $K \subseteq E$ such that $\sum_{k \in K^C} P(X=k) \le \epsilon/4$, and there exists $n_0 \in \mathbb{N}$ such that
for all $n \ge n_0$, $\sum_{k \in K} |P(X_n = k) - P(X=k)| \le \epsilon/4$. Convergence in total variation then follows from convergence in distribution, since, for all
$n \ge n_0$, 
\begin{align*}
&\sum_{k \in E} |P(X_n = k) - P(X = k)|\\ 
& = \sum_{k \in K} |P(X_n = k) - P(X = k)| + \sum_{k \in K^C} |P(X_n = k) - P(X = k)|\\
& \le \frac{\epsilon}{4} + \sum_{k \in K^C} P(X_n = k) + \sum_{k \in K^C}P(X=k)
 = \frac{\epsilon}{4} + 1 - \sum_{k \in K} P(X_n = k) +  \frac{\epsilon}{4}\\
& \le \frac{\epsilon}{4} +  1- \sum_{k \in K}P(X=k) + \sum_{k \in K} |P(X = k) - P(X_n=k)| +  \frac{\epsilon}{4}\\
& \le  \frac{\epsilon}{4} +  \frac{\epsilon}{4} +  \frac{\epsilon}{4} +  \frac{\epsilon}{4} = \epsilon.
\end{align*}
If $E$ is not discrete, for example, if $E=\mathbb{R}$, then 
convergence in total variation is stronger than convergence in distribution, and it might occasionally be even too strong to be of use. 
\end{remark}
\subsection*{Kolmogorov distance} Suppose that $E=\mathbb{R}$ and that the test functions are the indicators of half-lines in $\mathbb{R}$, i.e.\ 
$\mathcal{H}=\{I_{(-\infty,x]};\, x \in \mathbb{R}\}$. The \textit{Kolmogorov distance} is then defined as follows:
\[
d_K(\mu,\nu) = \sup_{h \in \mathcal{H}} \left| \int_E hd\mu - \int_E h d\nu\right| = \sup_{x \in \mathbb{R}} |\mu\{(-\infty,x]\} - \nu\{(-\infty,x]\}|.
\]
For two random variables $X \sim \mu$ and $Y \sim \nu$, the Kolmogorov distance is thus given by
$\sup_{x \in \mathbb{R}} |P(X \le x) - P(Y \le x)|$.
\subsection*{Wasserstein distance induced by $\boldsymbol{d}$} 
Suppose that $E$ is a separable metric space
with associated metric $d$ and equipped with its Borel $\sigma$-field. 
For the \textit{Wasserstein distance}, which is also known as the \textit{Dudley}, \textit{Fortet-Mourier} or \textit{Kantorovich distance}, 
we only consider probability measures $\mu$ such that for some, and then for
any, $x_0 \in E$, $\mathbb{E}d(X,x_0) = \int_E d(x,x_0) d\mu(x) < \infty$, where $X \sim \mu$. 
For probability measures $\mu$ and $\nu$ satisfying this condition, the Wasserstein distance induced by $d$ is given by
\[
d_W(\mu, \nu) = \sup_{h \in \mathcal{H}} \left| \int_E hd\mu - \int_E hd\nu \right|,  
\]
where the test functions $h$ are uniformly Lipschitz with constant $1$, i.e.\ $\mathcal{H}=\{h:E \to \mathbb{R}, \, |h(x)-h(y)| \le d(x,y)\}$.
It can be shown that
\begin{equation}\label{p: Wasserstein_coupling}
d_W(\mu, \nu) = \inf \mathbb{E}d(X,Y), 
\end{equation}
where the infimum is taken over all couplings $(X,Y)$ of $\mu$ and $\nu$ (see, for instance, Section 7.1 in \cite{Ambrosio_et_al:2005}).

Chapter \ref{Chap: Stein-Chen} mainly uses the total variation distance. The Wasserstein distance will make its first appearance only in Section \ref{Sec: Improved_rates} 
where we will use it to construct a distance between probability measures over a set of point measures, in the context of approximation by Poisson processes. The Kolmogorov
distance will be widely used in Section \ref{Sec: univ_max} where we approximate the law of the maximum of \iid random variables by an extreme value distribution. 
\section{Le Cam's result for Poisson approximation 
}\label{Sec: LeCam}
\noindent 
\cite{LeCam:1960} determined an upper bound for the accuracy in total variation of the approximation of a sum of independent Bernoulli random variables 
by a Poisson distribution with the same mean. We give the argument by \cite{Serfling:1975} that uses a simple coupling inequality.
\begin{thm}\label{t: LeCam}(Le Cam, 1960) Let $I_1, \ldots, I_n$ be independent Bernoulli random variables with $P(I_i=1)=p_i$ and
$P(I_i=0) =1-p_i$, where $0 < p_i <1$, for all $i=1, \ldots, n$. Let $W := \sum_{i=1}^n I_i$ and $\lambda := \mathbb{E}W = \sum_{i=1}^n p_i$. Then
\[
d_{TV}(\mathcal{L}(W), \mathrm{Poi}(\lambda)) \le \sum_{i=1}^{n} p_i^2.
\]
\end{thm}
\begin{proof}
We perform a coupling of the two distributions $\mathcal{L}(W)$ and $ \mathrm{Poi}(\lambda)$ by defining random variables 
$J_i$ and $Y_i$ on probability spaces $(\Omega_i, P_i)$, for each $i=1, \ldots, n$, such that $\sum_{i=1}^n J_i$ and $\sum_{i=1}^n Y_i$ follow the distributions
$\mathcal{L}(W)$ and $\mathrm{Poi}(\lambda)$, respectively. To achieve this, define 
\begin{align*}
\Omega_i &:= \{-1, 0, 1, 2, \ldots\}, \\
P_i(0)   &:= 1-p_i,\\
P_i(k)   &:= e^{-p_i} p_i^k / k!, \textnormal{ for all } k\ge 1,\\
P_i(-1)  &:= 1- P_i(0) - \sum_{k\ge 1}P_i(k) = e^{-p_i} -(1-p_i), 
\end{align*}
for all $i=1, \ldots, n$. By construction, $(\Omega_i, P_i)$ are probability spaces. Let $(\Omega, P)$ be the product space of the probability spaces $(\Omega_i, P_i)$, i.e.\ let
$\Omega := \Omega_1 \times \ldots \times \Omega_n$ and, for any $\omega = (\omega_1, \ldots, \omega_n) \in \Omega$, define 
$P(\omega) := P_1(\omega_1) \cdot \ldots \cdot P_n(\omega_n)$. Then $\sum_{\omega \in \Omega} P(\omega) = 1$. Now define, for any $\omega \in \Omega$,
\[
J_i(\omega) :=  
\left\{ 
\begin{array}{ll} 0, & \textnormal{ if } \omega_i = 0, \\ 1, & \textnormal{ else}, \end{array} 
\right.
\]
and 
\[
Y_i(\omega) :=  
\left\{ 
\begin{array}{ll} k, & \textnormal{ if } \omega_i = k, \, k \ge 1, \\ 0, & \textnormal{ else}.\end{array} 
\right.
\]
By definition, the random variables $J_i$ have the same distribution as the Ber\-noulli random variables $I_i$, i.e.\ $P(J_i=1)=p_i=1-P(J_i=0)$, for all $i=1, \ldots, n$,
and, by definition of the product space, they are independent. Moreover, the random variables $Y_i$ are independent and follow the $\mathrm{Poi}(p_i)$-distribution,
since $P(Y_i = k) = P_i(k)$ for all $k \ge 1$, and $P(Y_i=0)= P_i(-1)+P_i(0) = e^{-p_i}$. It follows that the random variable $Y:= Y_1 + \ldots+ Y_n$ is 
$\mathrm{Poi}(\lambda)$-distributed. The random variables $J_i$ and $Y_i$ now take the same values if $\omega_i \in \{0,1\}$ and thus,
$P(I_i = Y_i) = P_i(0) + P_i(1) = (1-p_i) + p_i e^{-p_i}$, and 
\[
P(I_i \neq Y_i) = p_i\left(1-e^{-p_i}\right) \le p_i^2, 
\]
for all $i=1, \ldots, n$, since $1-e^{-z} \le z$ for $z > 0$. From (\ref{t: coupling_dTV}), it now follows that
\[
d_{TV}(\mathcal{L}(W), \mathrm{Poi}(\lambda)) \le P(W \neq Y ) \le \sum_{i=1}^n P(I_i \neq Y_i) \le \sum_{i=1}^n p_i^2.
\]
\end{proof}
\noindent For the accuracy in total variation between the binomial distribution and a Poisson distribution with the same mean, it then follows immediately:
\begin{cor}\label{t: Cor_Barbour_Hall}
\[
d_{TV}(\mathrm{Bin}(n,p), \mathrm{Poi}(np)) \le np^2. 
\]
\end{cor}
\noindent 
For a good approximation of a binomial distribution by a Poisson, we thus need
the probability of success $p$ to vary with the sample size $n$ such that $np^2 \to 0$ as $n \to \infty$, i.e.\ we need $p=p_n=o(n^{-1/2})$ as $n \to \infty$. 
We already noticed this in Section \ref{Sec: Poi_approx_Bin}.
The error $np^2$ can, however, still be improved. 
In order to get an inkling of why this is the case, note first that for large $np$, most realisations of the $\mathrm{Bin}(n,p)$ and $\mathrm{Poi}(np)$ distributions
lie in the vicinity of $np$. With the following inequalities that can be established for $n \ge 2$ and $p=p_n \le 1/2$, and that are (in part) 
more precise than the ones used in Section \ref{Sec: Poi_approx_Bin},
\begin{align*}
e^{-\frac{k^2}{2n} + \frac{k}{2n} - \frac{2k^3}{3n^2}} 
\le  &\,\,\prod_{j=0}^{k-1} \left( 1- \frac{j}{n}\right) \le e^{-\frac{k^2}{2n} + \frac{k}{2n}},\\
e^{-np}\cdot e^{-\frac{np^2}{2} - 8np^3} \le &\,\, (1-p)^n \le e^{-np}\cdot e^{-\frac{np^2}{2}},\\
e^{kp} \le &\,\,(1-p)^{-k} \le e^{kp + \frac{kp^2}{2}},
\end{align*}
it turns out that for $k=np$, 
\[
P(W=k)= \left\{ 1+ O\left(p, np^3\right) \right\}\mathrm{Poi}(np)\{k\}, 
\]
which suggests that the error estimate from Corollary \ref{t: Cor_Barbour_Hall} may be reduced.
With Stein's method we indeed find an improved result; see Corollary \ref{t: Barbour_Hall_1984} in the next section.

\section{The Stein-Chen method for Poisson approximation 
}\label{Sec: PoiAppr_dTV}
Stein's method for Poisson approximation was first worked out by \cite{Chen:1975a} and is therefore usually named the Stein-Chen method. We demonstrate the method 
for the example of sums of independent, but non-identically distributed indicator variables, which will provide us with an improvement on the result by Le Cam in Theorem
\ref{t: LeCam}. First note the following two observations, summarised in Theorems \ref{t: Chen:1975} and \ref{t: Char_of_Poi}:

\begin{thm}\label{t: Chen:1975}(Chen, 1975a) Let $f:\, \mathbb{Z}_+ \to \mathbb{R}$ be a bounded function. The following are equivalent:\\
(i) There exists a bounded solution $g = g_{f,\lambda}:\, \mathbb{Z}_+ \to \mathbb{R}$ of  
\begin{equation}\label{t: steineq}
\lambda g(k+1) - kg(k) = f(k),\quad \textnormal{for all }k \in \mathbb{Z}_+.
\end{equation}
(ii) $\mathbb{E}f(Z) = 0$, for $Z \sim \mathrm{Poi}(\lambda)$.\\
Furthermore, if either is satisfied, then the solution $g=g_{f,\lambda}$ to equation (\ref{t: steineq}) is 
\begin{equation}\label{gsolution}
g(k+1) = \frac{k!}{\lambda^{k+1}}\sum_{j=0}^k p_{j,\lambda}e^{\lambda} f(j) = - \frac{k!}{\lambda^{k+1}}\sum_{j=k+1}^{\infty} p_{j,\lambda}e^{\lambda} f(j),
\end{equation}
for all $k\in \mathbb{Z}_+$, where $\displaystyle p_{j,\lambda}= \frac{\lambda^j}{j!}e^{-\lambda}$.
\end{thm}
\begin{proof} 
$(i) \Rightarrow (ii)$:  Let $g: \,\mathbb{Z}_+ \to \mathbb{R}$ be a bounded function and let $Z\sim \mathrm{Poi}(\lambda)$, with $\lambda >0$. Then
\[
\mathbb{E}  \lambda g(Z+1) = \sum_{k\ge0} \lambda g(k+1) \frac{\lambda^k}{k!} e^{-\lambda}, 
\]
and 
\begin{align*}
\mathbb{E}Zg(Z) &= \sum_{l \ge 1} lg(l) \frac{\lambda^l}{l!} e^{-\lambda} \\
&= \sum_{k \ge 0} \lambda (k+1)g(k+1)\frac{\lambda^k}{(k+1)!} e^{-\lambda} \\
&= \sum_{k\ge0} \lambda g(k+1) \frac{\lambda^k}{k!} e^{-\lambda}.
\end{align*}
It follows from (\ref{t: steineq}) that
\[
\mathbb{E}f(Z)=\mathbb{E}[\lambda g(Z+1) - Zg(Z)] =0.
\]
Note that we have not needed $g$ to be of the form (\ref{gsolution}) here.
$(ii) \Rightarrow (i)$:  Suppose $f: \, \mathbb{Z}_+ \to \mathbb{R}$ is a bounded function such that $\mathbb{E}f(Z) =0$ for $Z \sim \mathrm{Poi}(\lambda)$. We use an induction proof to show that $g=g_{f,\lambda}$ given by (\ref{gsolution}) solves 
equation (\ref{t: steineq}). Without loss of generality we set $g(0):=0$. Then, for $k=0$, (\ref{gsolution}) gives
$g(1)=f(0) p_{0,\lambda}e^\lambda/\lambda=f(0)/\lambda$, which solves (\ref{t: steineq}). Assume that (\ref{gsolution}) solves (\ref{t: steineq})
for an integer $k \ge 0$. Then the induction step
\begin{align*}
g(k+2) &= \frac{(k+1)!}{\lambda^{k+2}} \sum_{j=0}^{k+1} p_{j,\lambda} e^{\lambda}f(j) \\
&= \frac{k+1}{\lambda} \cdot \frac{k!}{\lambda^{k+1}} \left\{ \sum_{j=0}^k p_{j,\lambda} e^\lambda f(j) + p_{k+1,\lambda} e^\lambda f(k+1)\right\}\\
&= \lambda^{-1} \left\{ (k+1)g(k+1) +f(k+1) \right\}
\end{align*}
completes the argument. Note that we have not yet needed the condition $\mathbb{E}f(Z)=0$, which implies that the solution $g$ 
we found so far always exists. The condition is needed, however, for the alternative representation of $g(k+1)$ as will be made clear by the following:
\begin{align*}
0 &= \mathbb{E}f(Z) = \sum_{j=0}^k \frac{\lambda^j}{j!}e^{-\lambda}f(j)
+ \sum_{j=k+1}^\infty \frac{\lambda^j}{j!}e^{-\lambda}f(j).
\end{align*}
Multiplication of both sides by $k!e^\lambda /\lambda^{k+1}$ now yields
\[
\frac{k!}{\lambda^{k+1}} \sum_{j=0}^k p_{j,\lambda}e^\lambda f(j)
= -\frac{k!}{\lambda^{k+1}} \sum_{j=k+1}^\infty p_{j,\lambda}e^\lambda f(j).
\]
Furthermore, it follows from (\ref{gsolution}) that 
\begin{align*}
|g(k+1)| &\le ||f|| \sum_{j=k+1}^\infty \frac{k! \lambda^{j-(k+1)}}{j!} \\
&= ||f||\sum_{m=1}^\infty \frac{\lambda^{m-1}}{(k+m)\cdot \ldots \cdot (k+1)} \\
&\le ||f|| \sum_{m=1}^\infty \frac{\lambda^{m-1}}{m!},
\end{align*}
since, obviously, $k+m \ge m$ for all $m \ge 1$. Moreover, 
\[
\sum_{m=1}^\infty \frac{\lambda^{m-1}}{m!} \le \sum_{m=1}^\infty \frac{\lambda^{m-1}}{(m-1)!} = e^\lambda,
\]
and thus $||g|| \le e^\lambda ||f||$, i.e.\ $g$ is bounded.
\end{proof}
\noindent We can in fact give a characterisation of the Poisson distribution:
\begin{thm}\label{t: Char_of_Poi}\textit{(Characterisation of the Poisson distribution)}
Let $Z$ be a random variable taking values in $\mathbb{Z}_+$. The following are equivalent:\\
(i) $Z \sim \mathrm{Poi}(\lambda)$.\\
(ii) For every bounded function $g:\, \mathbb{Z}_+ \to \mathbb{R}$, we have 
\begin{equation}\label{t: ii}
\mathbb{E}[\lambda g(Z+1) - Zg(Z)] =0.
\end{equation}
\end{thm}
\begin{proof} $(i) \Rightarrow (ii)$: 
Suppose that $Z \sim \mathrm{Poi}(\lambda)$. Let $h:\, \mathbb{Z}_+ \to \mathbb{R}$ be any bounded function and define 
\[
f(k):= h(k) - e^{-\lambda}\sum_{k=0}^\infty h(k) \,\frac{\lambda^k}{k!}, \quad \text{for all } k \in \mathbb{Z}_+.
\]
Then $f$ is bounded and $\mathbb{E}f(Z) = 0$. By Theorem \ref{t: Chen:1975} there exists a bounded function $g=g_{f, \lambda}: \mathbb{Z}_+ \to \mathbb{R}$ satisfying
(\ref{t: steineq}), and thus,
\[
\mathbb{E}\left[ \lambda g(Z + 1) - Z g(Z)\right] = \mathbb{E}f(Z) = 0.  
\]
$(ii) \Rightarrow (i)$: 
Suppose that $Z$ is a random variable with state space  $\mathbb{Z}_+$ satisfying (\ref{t: ii}). 
Let $X \sim \mathrm{Poi}(\lambda)$, choose any subset $A \subseteq \mathbb{Z}_+$, and define the bounded function $h(k):=h_{A}(k):= I_{\{k\in A\}}$
for all $k \in \mathbb{Z}_+$. Define
\[
f(k):= f_A(k):= h_A(k) - \mathbb{E}h_A(X), \quad \text{for all }k \in \mathbb{Z}_+.
\]
Then $f$ is bounded, $\mathbb{E}f(X)=0$, and by 
Theorem \ref{t: Chen:1975} there exists a bounded function $g:=g_{f,\lambda}:\, \mathbb{Z}_+ \to \mathbb{R}$ satisfying (\ref{t: steineq}) for all
$k \in \mathbb{Z}_+$. By (\ref{t: ii}), we thus obtain
\begin{align*}
0 &= \mathbb{E}[\lambda g(Z+1) - Zg(Z)] = \mathbb{E}f(Z) \\ 
&= \mathbb{E}h_A(Z)- \mathbb{E}h_A(X) = P(Z \in A)-P(X \in A),
\end{align*}
and therefore $Z \stackrel{d}{=} X$, i.e.\ $Z \sim \mathrm{Poi}(\lambda)$.
\end{proof}
\noindent In Section \ref{s: Stein_eq_Construction} below, we construct the Stein equation for Poisson approximation, by using Theorem \ref{t: Chen:1975}, and 
give smoothness estimates of its solution. Section \ref{s: Independent_indicators} then applies the results to the problem of determining 
a bound on the error that arises with the approximation of the law of a sum of independent Bernoulli random variables
by a Poisson distribution. 
Section \ref{s: Dependent_indicators} does the same for sums of dependent Bernoulli variables. 
Section \ref{s: antisymmetric_function} introduces a general procedure that uses an exchangeable pair of random variables 
and an antisymmetric function for establishing a Stein equation.  
Section \ref{s: generator_interpretation} relates the Stein equation for Poisson approximation to generators of immigration-death processes,
which allows for a probabilistic interpretation of its solution. 
\subsection{Construction of the Stein equation and smoothness estimates}\label{s: Stein_eq_Construction} Let $Z \sim \mathrm{Poi}(\lambda)$ and let $A \subseteq \mathbb{Z}_+$. Define
\[
f(k) := f_A(k):= I_{\{k \in A\}} - \mathrm{Poi}(\lambda)\{A\}, \quad k \in \mathbb{Z}_+, 
\]
where $\mathrm{Poi}(\lambda)\{A\} = P(Z \in A)$. The function $f_A$ is obviously
bounded by $1$ and we have
\[
\mathbb{E}f_A(Z) = P(Z \in A) - \mathrm{Poi}(\lambda)\{A\} =0. 
\]
Thus, by Theorem \ref{t: Chen:1975}, there exists a bounded solution $g:= g_{f_A, \lambda} := g_{A, \lambda}: \mathbb{Z}_+ \to \mathbb{R}$ to the \textit{Stein equation}
\begin{equation}\label{univ_Stein_equation}
\lambda g(k+1) - kg(k) = I_{\{k \in A\}} - \mathrm{Poi}(\lambda)\{A\}, \quad k \in \mathbb{Z}_+,
\end{equation}
and the \textit{Stein solution} is given by plugging $f_A(k)$ into (\ref{gsolution}):
\begin{align*}
g(k+1) &= \frac{k!}{\lambda^{k+1}}e^{\lambda}   \{P(Z \in A, Z \le k) - P(Z \in A)P(Z \le k)\}\\
&= \frac{k!}{\lambda^{k+1}}e^{\lambda}  \{\mathrm{Poi}(\lambda)\{A \cap U_k\} - \mathrm{Poi}(\lambda)\{A\}\mathrm{Poi}(\lambda)\{U_k\}\},
\end{align*}
where $U_k := \{0,1, \ldots, k\}$, for all $k \in \mathbb{Z}_+$. Let $W$ be a random variable taking values in $\mathbb{Z}_+$. 
By the Stein equation (\ref{univ_Stein_equation}), taking expectations, we have
\begin{equation}\label{univ_Stein_equation_E}
\mathbb{E}[\lambda g(W+1) - Wg(W)] = P(W \in A) - \mathrm{Poi}(\lambda)\{A\}. 
\end{equation}
In order to find an upper bound for the error in total variation of the approximation of the law of $W$ by that of a Poisson distribution with mean $\lambda >0$, 
it suffices, by (\ref{univ_Stein_equation_E}), to bound $\mathbb{E}[\lambda g(W+1) - Wg(W)]$ uniformly in $A \subseteq \mathbb{Z}_+$. To achieve this, we first need smoothness estimates of the function 
$g=g_{A, \lambda}$ as given in Lemma \ref{t: B_Eagleson}. For the proofs of (i) and (ii) of Lemma \ref{t: B_Eagleson}, we refer to \cite{Barbour_et_al.:1992} (Remark 10.2.4)
and \cite{Barbour/Eagleson:1983}, respectively. 
\begin{lemma}\label{t: B_Eagleson} For the solution $g:= g_{A, \lambda}$ of the Stein equation, it holds that
\begin{enumerate}
\item[(i)] $\quad  \displaystyle ||g|| := \sup_{k \ge 0} |g(k)| \le \min\left(1, \sqrt{\frac{2}{e\lambda}}\right),$
\item[(ii)] $\quad \displaystyle \Delta g := \sup_{k \ge 0} |g(k+1)-g(k)| \le \frac{1-e^{-\lambda}}{\lambda} \le \min\left(1, \frac{1}{\lambda}\right).$
\end{enumerate}
\end{lemma}
\begin{flushright}
$\qed$ 
\end{flushright}


\subsection{
Independent indicator variables}\label{s: Independent_indicators}
We use (\ref{univ_Stein_equation_E}) to find a sharper bound than the one given by Le Cam in Theorem \ref{t: LeCam}. We assume the setting of Theorem \ref{t: LeCam}, 
i.e.\ let $I_1, \ldots, I_n$ be independent Bernoulli random variables with $P(I_i=1)=p_i$ and
$P(I_i=0) =1-p_i$, where $0 < p_i <1$, for all $i=1, \ldots, n$. Let $W = \sum_{i=1}^n I_i$ and $\lambda = \mathbb{E}W = \sum_{i=1}^n p_i$.
Then
\begin{align}
&\label{p: g}|P(W \in A) - \mathrm{Poi}(\lambda)\{A\}|  \le 2 ||g||\sum_{i=1}^n p_i^2,\\
&\label{p: Deltag}|P(W \in A) - \mathrm{Poi}(\lambda)\{A\}|  \le  \Delta g \sum_{i=1}^n p_i^2,
\end{align}
where $||g|| := \sup_{k \ge 0} |g(k)|$ and $\Delta g := \sup_{k \ge 0} |g(k+1)-g(k)|$.
We may show (\ref{p: g}) and (\ref{p: Deltag}) as follows:
for each $i=1, \ldots,n$, define $W_i := \sum_{j=1, j\neq i}^n I_j$.
We have 
\begin{equation}\label{p: use_of_independence}
\mathbb{E}[I_ig(W)] = \mathbb{E}[I_i g(W_i+1)] = p_i \mathbb{E}[g(W_i+1)], 
\end{equation}
since $I_i$ and $W_i$ are independent. We may thus write the left hand side of (\ref{univ_Stein_equation_E}) as
\begin{align*}
\mathbb{E}[\lambda g(W+1) - Wg(W)] 
&= \sum_{i=1}^n p_i \left\{ \mathbb{E}[g(W+1)] - \mathbb{E}[g(W_i+1)]\right\}\\
&= \sum_{i=1}^n p_i \sum_{k \ge 0} g(k+1)\left\{ P(W=k) - P(W_i=k)\right\},
\end{align*}
where, using independence between $W_i$ and $I_i$, we find that
\begin{align*}
P(W=k) &= P(W_i+I_i=k, I_i=0) + P(W_i+I_i=k, I_i=1) \\
& = (1-p_i)P(W_i=k) + p_iP(W_i+1 =k). 
\end{align*}
We thus obtain
\begin{equation}\label{p: intermediate}
\mathbb{E}[\lambda g(W+1) - Wg(W)] = \sum_{i=1}^n p_i^2 \sum_{k \ge 0} g(k+1) \left\{ P(W_i+1=k) - P(W_i=k) \right\}.
\end{equation}
On the one hand, for (\ref{p: g}), this may be bounded from above by
\begin{align*} 
&||g|| \sum_{i=1}^n p_i^2 \left\{ \sum_{k\ge 0} P(W_i+1=k) + \sum_{k\ge 0} P(W_i=k) \right\} \le 2||g|| \sum_{i=1}^n p_i^2. 
\end{align*}
On the other hand, for (\ref{p: Deltag}), note that 
\begin{align*}
\sum_{k \ge 0}g(k+1)P(W_i=k-1) &= \sum_{k \ge 1}g(k+1)P(W_i=k-1) \\
&= \sum_{k' \ge 0}g(k'+2) P\left(W_i=k'\right). 
\end{align*}
Then, (\ref{p: intermediate}) gives
\begin{align*}
\left| \sum_{i=1}^n p_i^2 \sum_{k \ge 0} P(W_i=k)\left\{ g(k+2) - g(k+1)\right\} \right| 
&\le \Delta g \sum_{i=1}^n p_i^2 \sum_{k \ge 0} P(W_i=k) \\
&\le \Delta g \sum_{i=1}^n p_i^2.
\end{align*}
The results now follow using (\ref{univ_Stein_equation_E}).
With (\ref{univ_Stein_equation_E}), (\ref{p: Deltag}) and Lemma \ref{t: B_Eagleson}, we get an improvement on Le Cam's result and thereby also better rates for the approximation of a binomial by a Poisson distribution:
\begin{thm}\textit{(Barbour and Hall, 1984)}\label{t: Barbour_Hall_1984}
 Let $I_1, \ldots, I_n$ be independent Ber\-noulli random variables with $P(I_i=1)=p_i$ and
$P(I_i=0) =1-p_i$, where $0 < p_i <1$, for all $i=1, \ldots, n$. Let $W = \sum_{i=1}^n I_i$ and $\lambda = \mathbb{E}W = \sum_{i=1}^n p_i$. Then
\begin{equation}\label{t: bound_Barbour_Hall_1984}
d_{TV}(\mathcal{L}(W), \mathrm{Poi}(\lambda)) \le \frac{1-e^{-\lambda}}{\lambda}\sum_{i=1}^{n} p_i^2 \le \min\left(1, \frac{1}{\lambda}\right)\sum_{i=1}^{n} p_i^2.
\end{equation}
\qed
\end{thm}
\begin{cor}
$ \displaystyle d_{TV}(\mathrm{Bin}(n,p), \mathrm{Poi}(np)) \le np^2 \min\left( 1, \frac{1}{np}\right) = O(p).$
\end{cor}
\subsection{
Dependent indicator variables -- The local approach}\label{s: Dependent_indicators}
Theorems \ref{t: LeCam} and \ref{t: Barbour_Hall_1984} generalise the problem of the approximation of a binomial distribution by a Poisson in the sense that the indicator variables 
need no longer be identically distributed. 
One of the strengths of the Stein-Chen method is that we can further relax the conditions put on the indicator variables by dropping the assumption of independence.
Indeed, independence is used only once, in Equation (\ref{p: use_of_independence}) and thus it is only (\ref{p: use_of_independence}) that needs to be modified 
in a way to allow for some kind of dependence. One of the ways to do this was 
suggested by \cite{Chen:1975b}. For each of the indicator variables $I_i$, $i=1, \ldots, n$, the idea is to classify the 
$n-1$ remaining indicator variables into two different categories, those ``strongly'' dependent on $I_i$ and those ``weakly'' dependent on $I_i$.
\begin{thm}\label{t: dTV_bounds_Poi_localapproach}  Let $I_1, \ldots, I_n$ be Bernoulli random variables with $P(I_i=1)=\mathbb{E}I_i=p_i$ and
$P(I_i=0) =1-p_i$, where $0 < p_i <1$, for all $i=1, \ldots, n$. Let $W = \sum_{i=1}^n I_i$ and $\lambda = \mathbb{E}W = \sum_{i=1}^n p_i$.
For any choice of index $i \in \{1, \ldots, n\}$, let
$\Gamma_i^s \subseteq \{1, \ldots, n\} \smallsetminus \{i\}$ be the set of indices comprising all  $j \neq i$ for which $I_j$ is strongly dependent on 
$I_i$, and let $\Gamma_i^w$ similarly be the set of indices $j$ for which $I_j$ is weakly dependent on $I_i$. Furthermore,
let 
\[
W_i= \sum_{j=1, j\neq i}^n I_j, \quad Z_i = \sum_{j \in \Gamma_i^s} I_j, \quad \textnormal{and } Y_i= \sum_{j \in \Gamma_i^w} I_j = W-I_i-Z_i = W_i-Z_i.
\]
Then
\begin{align*}
&d_{TV}(\mathcal{L}(W),\mathrm{Poi}(\lambda)) \\
&\le \sum_{i=1}^n \left[ \left( p_i^2 + p_i \mathbb{E}Z_i + \mathbb{E}(I_i Z_i) \right)\right] \min\left(1, \frac{1}{\lambda}\right)
+ \sum_{i=1}^n \eta_i \min\left(1, \sqrt{\frac{2}{e\lambda}}\right),
\end{align*} 
where $\eta_i$ is chosen such that
\begin{equation}\label{t: eta_local_approach}
\left| \mathbb{E}[I_i g(Y_i+1) -p_i \mathbb{E}g(Y_i+1)]\right| \le \eta_i ||g||. 
\end{equation}
For instance, $\eta_i$ may be chosen as follows:
\[
\eta_i = \mathbb{E} \left|\mathbb{E}\{I_i|(I_j, j \in \Gamma_i^w)\} - p_i\right|. 
\]
\end{thm}
\begin{remark}
In the case that $Y_i$ is precisely independent of $I_i$ the second error term $\sum_{i=1}^n \eta_i $ $\min\left\{1, (2/e\lambda)^{1/2}\right\}$ disappears.
\end{remark}
\begin{proof} From (\ref{univ_Stein_equation_E}), we have that
\[
d_{TV}(\mathcal{L}(W), \mathrm{Poi}(\lambda)) = \sup_{A \subseteq \mathbb{Z}_+} \left| \mathbb{E}\left[ \lambda g(W+1) -Wg(W) \right] \right|,
\]
where $g=g_{\lambda,A}$. We replace Equation (\ref{p: use_of_independence}) by the following:
\begin{align*}
\mathbb{E}\left[I_ig(W) \right] &= \mathbb{E}\left[ I_i g(W_i +1)\right] \\
&= \mathbb{E}\left[ I_i g(Y_i +1)\right] + \mathbb{E}\left[ I_i(g(Z_i + Y_i+1) - g(Y_i+1))\right].
\end{align*}
Then,
\begin{align*}
\sum_{i=1}^n \mathbb{E}[p_ig(W+1) - I_i g(W)] 
&= \sum_{i=1}^n \left\{ \mathbb{E}[p_ig(I_i+Z_i+Y_i+1)] -\mathbb{E}[I_ig(Y_i+1)]\right. \\
& \left. \qquad \quad  - \mathbb{E}[I_i(g(Z_i + Y_i +1) - g(Y_i+1))]\right\}, 
\end{align*}
and therefore, by adding and subtracting $p_i\mathbb{E}[g(Y_i+1)]$, 
\begin{align}
\begin{split}\label{p: E_mess_local_approach}
&|\mathbb{E}[\lambda g(W+1) - Wg(W)]| \\
&\le \sum_{i=1}^n \left|p_i \mathbb{E}[g(I_i+Z_i+Y_i+1) - g(Y_i+1)]\right|
 + \left|p_iE[g(Y_i+1)] - \mathbb{E}[I_i g(Y_i+1)]\right|\\
&\qquad \quad + \left| \mathbb{E}\left[ I_i(g(Z_i + Y_i+1) - g(Y_i+1))\right]\right|.
\end{split}
\end{align}
With Lemma \ref{t: B_Eagleson} (ii), we have $\left|g(j+k) - g(j)\right| \le k\Delta g \le k \min\left(1, \lambda^{-1}\right)$.
We thus find for the first and third terms in (\ref{p: E_mess_local_approach}),
\begin{align*}
\left|p_i \mathbb{E}[g(I_i+Z_i+Y_i+1) - g(Y_i+1)]\right| &\le p_i(\mathbb{E}I_i + \mathbb{E}Z_i) \min\left(1, \lambda^{-1}\right),\\
\left| \mathbb{E}\left[ I_i(g(Z_i + Y_i+1) - g(Y_i+1))\right]\right|&\le \mathbb{E}[I_iZ_i] \min\left(1, \lambda^{-1}\right),
\end{align*}
respectively. For the second term in (\ref{p: E_mess_local_approach}) we choose $\eta_i$ such that (\ref{t: eta_local_approach}) is satisfied and use Lemma \ref{t: B_Eagleson} (i). 
Finally, we show that (\ref{t: eta_local_approach}) holds for the choice $\eta_i $ $= \mathbb{E} \left|\mathbb{E} \right.$ $\left.\{I_i|(I_j, j \in \Gamma_i^w)\} - p_i\right|$:
\begin{align*}
\left| \mathbb{E}\left[ g(Y_i+1)(p_i-I_i)\right]\right|
 &\le ||g|| \cdot\left| \mathbb{E}\left[p_i-I_i  \right]\right| \\
&= ||g|| \cdot \left| \mathbb{E} \left[ \mathbb{E} \left[ (p_i-I_i)| (I_j, j\in \Gamma_i^w)\right]\right]\right|
\\ & \le \eta_i||g||.
\end{align*}
\end{proof}

\subsection{The antisymmetric function approach}\label{s: antisymmetric_function}
Stein's method for normal and for Poisson approximation can be put into a broader framework. A general approach using an exchangeable pair of random variables 
and an antisymmetric function was first proposed by \cite{Stein:1986} and later also discussed by 
\cite{Chen:1998} and \cite{Erhardsson:2005}. 
We give a brief summary that closely follows \cite{Erhardsson:2005}. 
Let $(S, \mathcal{S}, \mu)$ be a probability space, denote by $\mathcal{H}$ the set of measurable functions 
$h:\,S \to \mathbb{R}$ and by $\mathcal{H}_0 \subset \mathcal{H}$ a set of $\mu$-integrable functions. The goal is to compute $\int_S h d\mu$ for 
all $h \in \mathcal{H}_0$, but the structure of $\mu$ might be too complicated to do this. It could, for instance, be the distribution of a sum of a large number of
dependent random variables. An idea to circumvent the problem of evaluating $\int_S h d\mu$ precisely is to instead replace $\mu$ by a probability measure 
$\mu_0$ that is close to $\mu$, with the advantage of being better known and easier to handle, classical examples for $\mu_0$ being the normal and the Poisson distributions. 
So the new probability measure $\mu_0$ should be chosen on $(S, \mathcal{S})$ such that the $\mu$-integrable functions $h \in \mathcal{H}_0$ are 
also $\mu_0$-integrable and $\int_S hd\mu_0$ is easily computed for any $h \in \mathcal{H}_0$.
It then remains to estimate (preferably uniformly over all $h \in \mathcal{H}_0$) the error of the approximation of $\int_S h d\mu$ by $\int_S h d\mu_0$. 
To that end, we have to find a set of functions $\mathcal{G}_0$ and a mapping 
$T_0:\, \mathcal{G}_0 \to \mathcal{H}$ such that for all $h \in \mathcal{H}_0$, the 
equation 
\[
T_0 g = h - \int_S hd\mu_0 ,
\]
that we call the \textit{Stein equation}, has a solution $g \in \mathcal{G}_0$. We call $T_0$ a \textit{Stein operator} for the distribution $\mu_0$. 
If the above equation holds, then
\[
\int_S (T_0 g ) d\mu = \int_S h d\mu - \int_S h d\mu_0.
\]
The hope is then that it is easier to estimate $|\int_S (T_0 g) d\mu|$ than the actual approximation error $|\int_S h d\mu - \int_S h d\mu_0|$. 
But how to find a suitable operator $T_0$? Note that by the above Stein equation a necessary property for the Stein operator is that
\[
\int_S (T_0g) d\mu_0 = \int_S h d\mu_0 - \int_S hd\mu_0 = 0.
\]
\cite{Stein:1986} proposed the following way to construct $T_0$:
\begin{enumerate}
\item[(a)] Let $(\Omega, \mathcal{F}, P)$ be a probability space with associated expectation operator $\mathbb{E}$. Let $(X,Y)$ be an exchangeable pair of 
mappings of $(\Omega, \mathcal{F},P)$ into the probability space $(S, \mathcal{S}, \mu_0)$ in the sense that $P(X \in A)=P(Y \in A)=\mu_0(A)$ for all 
$A \in \mathcal{S}$ and $P(X \in A, Y \in A') = P(X \in A', Y \in A)$ for all $A, A' \in \mathcal{S}$. 
\item[(b)] Choose a mapping $\alpha:\, \mathcal{G}_0 \to \mathcal{G}$, where $\mathcal{G}$ is the space of antisymmetric functions $G:\, S^2 \to \mathbb{R}$ 
(i.e.\ $G(s,s')=-G(s',s)$ for all $s,s' \in S$) such that
$\mathbb{E}|G(X,Y)| 
< \infty$.
\item[(c)] Take $T_0 = T \circ \alpha$, where the linear mapping $T:\,\mathcal{G} \to \mathcal{H}$ is defined by
$TG:= $ $\mathbb{E}^X G(X,Y)$, 
with $\mathbb{E}^X$ denoting conditional expectation given $X$. 
\end{enumerate}
Then
\[
\int_S (T_0g)d\mu_0 = \int_S (TG) d\mu_0 = \mathbb{E}\mathbb{E}^X G(X,Y)=\mathbb{E}G(X,Y),  
\]
for all $g\in \mathcal{G}_0$, where $G := \alpha g$. 
Moreover,
\[
\mathbb{E}G(X,Y) = \mathbb{E}G(Y,X) = \mathbb{E}[-G(X,Y)] = - \mathbb{E}G(X,Y), 
\]
where we use exchangeability for the first equality and the antisymmetry of $G$ for the second. It follows that
\begin{equation}\label{d: Stein_identity}
\mathbb{E}G(X,Y) = \int_S (T_0g)d\mu_0 =0, \quad \textnormal{for all } g \in \mathcal{G}_0, 
\end{equation}
and thus the necessary property for the Stein operator, called \textit{Stein identity} for the target distribution $\mu_0$, is satisfied by the above choice of $T_0$. 
The following subsection will go into further detail on how to apply this procedure for the example of approximation by a Poisson distribution. In general, there is 
unfortunately no guarantee that the procedure will yield a Stein operator $T_0$ giving sharp estimates for the approximation error. 
Additional considerations have to be made for each case.

\subsection{Immigration-death processes and the generator interpretation}\label{s: generator_interpretation}
\cite{Barbour:1988} discovered a way to relate the Stein equation by \cite{Chen:1975a} to the generator of a Markov process whose equilibrium distribution 
is $\mu_0=\mathrm{Poi}(\lambda)$. 
This section gives a brief outline, while details will be discussed in a more general setting in Section \ref{s: Generator_process} below.
The Markov process in question is a stationary immigration-death process $Z:=\{Z_t, \, t \in \mathbb{R}_+\}$ on $\mathbb{Z}_+$ 
with constant immigration rate $\lambda>0$ and unit per-capita death rate, where
$Z_t$ describes the number of particles in a population at time $t$. 
For this process, immigrations of particles into the population and deaths of particles already in the population occur independently of one another. 
Also, each of the particles in the population dies after an $\mathrm{Exp}(1)$ lifetime,
independently of the others. As illustrated in Figure \ref{f: Immigration_death_process}, when the population has size $k \in \mathbb{Z}_+$, 
i.e.\ when the process $Z$ is in state $k$, transitions can only be to one of the adjoining states $k+1$ (immigration of one particle with constant 
rate $\lambda$) or $k-1$ (death of one particle with rate $k$, the current population size). 
\begin{figure}[!ht]
\begin{center}{\footnotesize \input{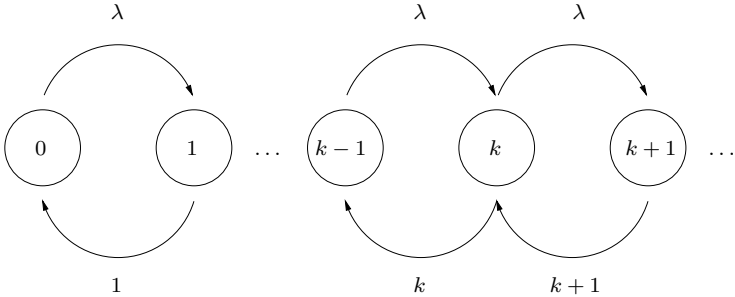}}\end{center}
\caption{Transitions of the immigration-death process can only be to one of the neighbouring states. Immigration of one particle occurs with constant rate $\lambda$, whereas 
the death rate depends on the current population size as each particle in the population has unit per-capita death rate.}
\label{f: Immigration_death_process}
\end{figure}
Now remember from (\ref{univ_Stein_equation}) that the Stein equation by \cite{Chen:1975a} is given by
\[
\lambda g(k+1) - kg(k) = I_{\{k \in A\}} - \mathrm{Poi}(\lambda)\{A\} \quad (=: f_A(k) =: f(k)),\quad \textnormal{for all }k \in \mathbb{Z}_+.
\] 
This is a first-order equation. \cite{Barbour:1988} noted that if the solution $g$ is written as the first backward difference 
$\nabla \upgamma$ of a function $\upgamma$, i.e.\ if $g(k)=\upgamma(k)-\upgamma(k-1)$ for all $k \in \mathbb{Z}_+$
(where $\upgamma(-1):=0$),
the left-hand side of the Stein equation may be written as 
\[
\lambda g(k+1) - kg(k) = \lambda \upgamma(k+1) + k\upgamma(k-1) - (\lambda + k)\upgamma(k) =: (\mathcal{A}\upgamma)(k), \,\, \textnormal{for all }k \in \mathbb{Z}_+. 
\]
Here, $\mathcal{A}$ is the infinitesimal generator of the process $Z$, and the Stein equation may now be reformulated by way of the second-order equation 
\begin{equation}\label{d: Stein_eq_generator}
\mathcal{A}\upgamma = f.
\end{equation}

In order to determine $\upgamma$, note that a solution $x$ to an equation of the form $\mathcal{A}x=f$ is typically given by $x(z)$ $=$  $-\int_0^\infty \mathbb{E}^zf(Z_t)dt$, for all bounded $f$ 
such that $\int_S fd\mu_0=0$, where $Z$ is an immigration-death process on $S$ with infinitesimal generator $\mathcal{A}$ and equilibrium distribution $\mu_0$, and where 
$\mathbb{E}^z$ denotes the distribution of the process conditioned on $Z_0 = z \in S$. In the case of the Stein equation (\ref{d: Stein_eq_generator}), 
where the equilibrium distribution of the immigration-death process $Z$ is $\mu_0 = \mathrm{Poi}(\lambda)$ and the function 
$f(j)= I_{\{j\in A\}} - \mathrm{Poi}(\lambda)\{A\}$ is obviously bounded for all $j \in S:=\mathbb{Z}_+$, we indeed have
\[
\int_S fd\mu_0 = \sum_{j \in \mathbb{Z}_+}\left[ I_{\{j\in A\}} - \mathrm{Poi}(\lambda)\{A\}\right] \mathrm{Poi}(\lambda)\{j\} = \mathrm{Poi}(\lambda)\{A\} - \mathrm{Poi}(\lambda)\{A\}=0, 
\]
and the solution $\upgamma$ of the Stein equation is thus given by
\begin{align*}
\gamma(k) &= -\int_0^\infty \mathbb{E}^k f(Z_t)dt \\
&= -\int_0^\infty \sum_{j \in \mathbb{Z}_+} \left[ I_{\{j \in A\}} - \mathrm{Poi}(\lambda)\{A\}\right]P\left(Z_t = j| Z_0 = k\right) dt\\
&= -\int_0^\infty \left[ P(Z_t\in A |Z_0=k) - \mathrm{Poi}(\lambda)\{A\}\right]dt,
\end{align*}
for all $k \in \mathbb{Z}_+$ (see also, for instance, Theorem 2.4 in \cite{Erhardsson:2005}). 

One of the advantages of the above approach is that it provides a probabilistic interpretation of the solution of the Stein equation, 
thus enabling the use of probabilistic arguments to determine smoothness estimates of the 
solution. 
Another advantage is that it
is applicable to a wide range of approximation problems; most importantly for us, to the problem of
approximating a point process by a Poisson process with the same mean measure.
We refer to Section \ref{s: Generator_process} for more details. 

We now delineate the connection between the antisymmetric function approach from Section \ref{s: antisymmetric_function}, applied to Poisson approximation, 
and the above generator interpretation.  
Let $(S, \mathcal{S}, \mu) = (\mathbb{Z}_+, \mathcal{P}(\mathbb{Z}_+), \mu)$, 
where $\mathcal{P}(\mathbb{Z}_+)$ is the power $\sigma$-algebra of $\mathbb{Z}_+$.
The aim is to approximate 
$\mu(A)=\int hd\mu$ by
$\mathrm{Poi}(\lambda)\{A\}$ $=$ $\int h d\mu_0$, where $h=I_A$, for any $A \in \mathcal{P}(\mathbb{Z}_+)$.
Since the immigration-death process $Z$ that we introduced above is reversible and has stationary distribution $\mu_0$, $(Z_0,Z_t)$ is an exchangeable pair 
with marginal distribution $\mu_0$.
Let $\mathcal{G}_0= \mathcal{H}$ be the set of real-valued functions on $\mathbb{Z}_+$ and let $\mathcal{G}$ be the set of antisymmetric functions
$\mathbb{Z}_+^2 \to \mathbb{R}$. Define $\alpha: \, \mathcal{H} \to \mathcal{G}$ by $(\alpha \upgamma)(k,l) = \upgamma(l)-\upgamma(k)$ 
for functions $\upgamma$ that do not grow too fast and note that
$\alpha$ is antisymmetric, since $\alpha \upgamma(k,l) = -\alpha \upgamma(l,k)$.
For all $t \ge 0$, take $T_0^t = T_t \circ \alpha$, where
we define $T_t:\, \mathcal{G} \to \mathcal{H}$ by
$ T_tG := \mathbb{E}^{Z_0}G(Z_0,Z_t)$. 
With $G := \alpha \upgamma$, we then have 
\[
\mathbb{E}[T_0^t\upgamma] = \mathbb{E}[T_tG] = \mathbb{E}G(Z_0,Z_t) = \mathbb{E}[\upgamma(Z_t) - \upgamma(Z_0)], 
\]
and, following the arguments from Section \ref{s: antisymmetric_function}, we obtain $\mathbb{E}[T_0^t \upgamma] = 0$ for all $t \ge 0$. 
In order to see the connection to the generator $\mathcal{A}$ of the immigration-death process $Z$, note that for all $k \in \mathbb{Z}_+$,
\begin{align*}
(T_0^t \upgamma)(k)
&=\mathbb{E}[\alpha \upgamma (Z_0,Z_t)|Z_0=k]
=  \mathbb{E}[\upgamma(Z_t)|Z_0=k]-\upgamma(k)\\
&=  \sum_{l =k-1}^{k+1}\upgamma(l)P(Z_t=l|Z_0=k)-\upgamma(k).
\end{align*}
As it can easily be shown that $P(Z_t=k+1|Z_0=k) = \lambda t + o(t)$, $P(Z_t = k-1 | Z_0=k) = kt + o(t)$, and 
$P(Z_t = k | Z_0 = k) = 1 - (\lambda + k)t + o(t)$, we have that 
\[
\lim_{t \downarrow 0 } \frac{1}{t}\,(T_0^t \upgamma)(k) =  \lambda \upgamma(k+1) + k\upgamma(k-1) -(\lambda + k)\upgamma(k) =  (\mathcal{A}\upgamma)(k).
\]
Informally (supposing that the limit and the expectation may be interchanged), we then observe the following connection between the Stein operators $T_0^t$ and
the generator $\mathcal{A}$:
\[
\lim_{t \downarrow 0} \frac{1}{t}\, \mathbb{E}[T_0^t \upgamma] = 0 = \mathbb{E}[(\mathcal{A}\upgamma)(Z_0)].  
\]
\section{Poisson process approximation in the total variation distance}\label{Sec: Poi_proc_approx}
So far, we have been concerned with determining the error in total variation of the approximation of the law of a random variable $W = \sum_{i=1}^n I_i$
by a Poisson distribution with the same mean, the $I_i$'s being possibly non-identically distributed and/or dependent indicator variables.
We may generalise $W$ by instead considering point processes. Loosely speaking, these not only give a random number of ``points'' but 
also the random configuration of such points in space. As $\mathcal{L}(W)$ may be approximated by $\mathrm{Poi}(\mathbb{E}W)$, 
we will show that the law $\mathcal{L}(\Xi)$ may similarly be approximated by the law of a Poisson point process with mean measure 
$\mathbb{E}\Xi$. \cite{Barbour/Brown:1992} and \cite{Barbour_et_al.:1992} extended the generator approach from the previous section to this problem. 
We formally define point processes in Section \ref{s: Point_process} and introduce the particular example of Poisson processes in Section \ref{s: Poisson_process}.
Section \ref{s: Palm_Janossy} gives short introductions to some tools from point process theory, namely Palm processes and Janossy densities, that we will need in Section 
\ref{s: Generator_process} to study approximation errors in the total variation distance using the generator approach. Finally, Section \ref{s: Michel_MPP} gives an easier way to determine 
approximation errors for the special case of marked point processes.  

Throughout, let $E$ be a locally compact separable metric space. 
In later applications, that is, starting from Chapter \ref{Chap: Univariate_extremes}, we simply use $E\subseteq \mathbb{R}^d$, $d \ge 1$. 
Let $E$ be equipped with its Borel $\sigma$-algebra $\mathcal{E}:= \mathcal{B}(E)$, i.e.\ the $\sigma$-algebra generated by the open sets.
The \textit{Dirac measure} $\delta_z$ on $\mathcal{E}$ for a point $z \in E$ is given by
\[
\delta_z(B) = \left\{ \begin{array}{ll}
1  \quad \textnormal{if } z \in B,\\
0  \quad \textnormal{if } z \notin B,
\end{array}\right.
\]
for any  $B \in \mathcal{E}$. For a countable collection $\{z_i\}$, $i\ge 1$, of not necessarily distinct points in $E$, consider the counting measure 
$\xi:= \sum_{i=1}^\infty \delta_{z_i}$ on $\mathcal{E}$, which assigns values in $\{0, 1, \dots\} \cup \{\infty\}$ to the sets that it measures. 
Suppose that $\xi$ is a Radon measure, i.e.\ suppose that $\xi(K) < \infty$ 
for compact sets $K \in \mathcal{E}$. Integer-valued Radon measures such as $\xi$ are called \textit{point measures}, and sometimes also \textit{point configurations}
on $E$. Denote by $\overline{M}_p(E)$ the space of all point 
measures on $E$ and equip $\overline{M}_p(E)$ with the $\sigma$-algebra $\overline{\mathcal{M}}_p(E)$ that is the smallest $\sigma$-algebra containing all sets of the form 
$\{ \xi \in \overline{M}_p(E): \, \xi(B) \in M\}$ for any $B \in \mathcal{E}$ and for any Borel set $M \subset [0, \infty]$. In other words, $\overline{\mathcal{M}}_p(E)$ is 
the smallest $\sigma$-algebra making the evaluation maps $\xi \to \xi(B)$ from $\overline{M}_p(E)$ to $[0, \infty]$ measurable for any set $B \in \mathcal{E}$. 
Furthermore, denote by $M_p(E) \subset \overline{M}_p(E)$ the space of all \textit{finite} point measures, i.e. of all point measures that assign values 
in $\{0,1,\ldots\}$ to the sets that they measure. Equip $M_p(E)$ with the $\sigma$-algebra $\mathcal{M}_p(E)$ that is the 
smallest $\sigma$-algebra making the evaluation maps $ \xi \to \xi(B)$ from $M_p(E)$ to $[0,\infty)$ measurable for any set $B \in \mathcal{E}$.  

\subsection{Point process}\label{s: Point_process}

Let $(\Omega, \mathcal{F}, P)$ be a probability space. A \textit{point process} $\Xi$ on $E$ is a measurable map from a probability space to the space of 
point measures, 
\begin{align*}
\Xi: \, (\Omega, \mathcal{F}, P) &\to (\overline{M}_p(E), \overline{\mathcal{M}}_p(E))\\
\omega &\mapsto \Xi(\omega) = \xi.
\end{align*}
A point process is thus a random element of $\overline{M}_p(E)$, i.e., for fixed $\omega \in \Omega$, a realisation 
$\Xi(\omega) := \Xi(\omega,\,. \,)$ is a point measure $\xi(\,.\,) \in \overline{M}_p(E)$. 
For fixed $B \in \mathcal{E}$, $\Xi(B):= \Xi(\,.\,,B)$ is a random variable taking values in  $\{0, 1, \dots\} \cup \{\infty\}$.
The space $E$ that the point process lives on is called \textit{state space}.

\begin{exa}
There are numerous ways to represent a point process by way of the Dirac measure and random variables. Suppose, for instance, that $\Gamma$ is a finite or
countable index set, and that $I_\alpha$, $\alpha \in \Gamma$, are \iid Bernoulli random variables with probability of success $P(I_\alpha=1)=p_\alpha \in (0,1)$. Also, let
$X_\alpha$, $\alpha \in \Gamma$, be \iid $E$-valued random variables, defined on the same probability space as the $I_\alpha$'s, but independent of these. We give some examples
of point processes:
\begin{enumerate}
\item[(a)] $\sum_{\alpha \in \Gamma } I_\alpha \delta_\alpha$ is a point process with state space $\Gamma$. 
\item[(b)] For any integer $n \ge 1$ and for $\Gamma = \{1, \ldots, n\}$, $\sum_{i=1}^n I_i \delta_{i/n}$ is a point process with state space $[0,1]$.
\item[(c)] $\sum_{\alpha \in \Gamma} I_\alpha \delta_{X_\alpha}$ is a point process with state space $E$. It is called \textit{marked point process} and the $X_\alpha$'s are 
called \textit{marks}. 
\item[(d)] With $P(I_\alpha=1)=1$ in (c), for all $\alpha \in \Gamma$, the marked point process is $\Xi:=\sum_{\alpha \in \Gamma} \delta_{X_\alpha}$.
Here, $\Xi(\omega) =  \sum_{\alpha \in \Gamma} \delta_{x_\alpha}$ 
gives the configuration of the points $x_\alpha = X_\alpha(\omega)$ in $E$, whereas 
$\Xi(B)$ gives the random number of points in the subset $B$ of $E$. $\Xi(\omega, B)$ gives the number of 
points $x_\alpha$ lying in the set $B$; see Figure \ref{f: Point_process} for an illustration.
\begin{figure}[!ht]
\begin{center}{\footnotesize \input{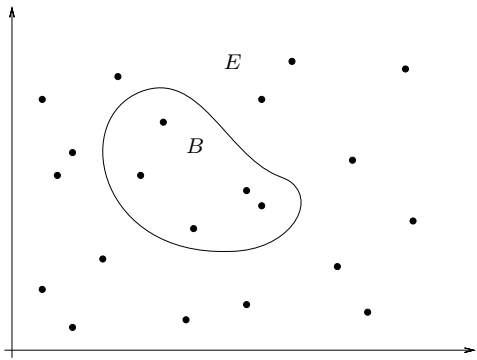}}\end{center}
\caption{A configuration of points $x_\alpha = X_\alpha(\omega)$ in $E=\mathbb{R}^2_+$. The number of points falling into the set $B$ is given by $\Xi(\omega,B)$.}
\label{f: Point_process}
\end{figure}
\end{enumerate}
\end{exa}

The probability law $P_\Xi$ of the point process $\Xi$ is the measure $P \circ \Xi^{-1}(\,.\,) = P(\Xi \in \,.\,)$ on $\overline{\mathcal{M}}_p(E)$.
It is uniquely determined by the set of \textit{finite-dimensional distributions}, i.e. the distributions of random vectors $(\Xi(B_1), \ldots, \Xi(B_m))$ for any choice of
$m \ge 1$ and $B_1, \ldots, B_m \in \mathcal{E}$. The \textit{intensity measure} or \textit{mean measure} of $\Xi$ is the measure $\bl$ on $\mathcal{E}$ defined,
for any $B \in \mathcal{E}$, by 
\[
\bl(B) = \mathbb{E}\Xi(B) = \int_{\Omega} \Xi(\omega,B) P(d\omega) = \int_{\overline{M}_p(E)} \xi(B)P_{\Xi}(d\xi).
\]
(Note that $\bl$ need not be Radon.) A point process $\Xi$ is called \textit{simple} if $\Xi(\omega, \{z\}) \le 1$ for all $z \in E$ and for all $\omega \in \Omega$, 
i.e. if an element $z \in E$ can be hit by at most one point of the process. We call the point process $\Xi$ \textit{finite}, or more precisely, 
\textit{almost surely finite}, if $P(|\Xi| < \infty) = 1$, where $|\Xi|$ denotes the random total number of points of $\Xi$.
\subsection{Poisson process}\label{s: Poisson_process}
Let $\bl$ be a $\sigma$-finite measure on $(E,\mathcal{E})$. By this we mean that $E$ may be written as a countable union of sets, each of which has a finite measure.
It also implies that $\bl$ is locally finite, i.e. every point $z \in E$ has a neighbourhood with finite measure.
An example of a $\sigma$-finite measure is the Lebesgue measure on Euclidean spaces.

A point process $\Xi$ on $E$ is called a \textit{Poisson process} or \textit{Poisson random measure (PRM)}
with mean measure $\bl$ if $\Xi$ satisfies:
\begin{enumerate}
\item[(a)] For any $B \in \mathcal{E}$, we have that $\Xi(B) \sim \mathrm{Poi}(\bl(B))$, i.e.
\[
P(\Xi(B) = k) = \left\{ \begin{array}{ll} \frac{(\bl(B))^k}{k!} \,e^{-\bl(B)}, & \bl(B)  < \infty, \\
0, & \bl(B) = \infty,
\end{array}\right. \textnormal{ for any } k\in \mathbb{Z}_+.
\]
\item[(b)] For any $m \ge 1$, if $B_1,\ldots, B_m$ are mutually disjoint sets in $\mathcal{E}$, then
$\Xi(B_1),  \ldots, $ $\Xi(B_m)$ are independent random variables.
\end{enumerate}
It follows from (a) that $\bl(B) = \infty$ implies $P(\Xi(B) = \infty)=1$.  By Proposition 3.6(i) in \cite{Resnick:1987} we know that, given a $\sigma$-finite measure $\bl$, a Poisson process 
with mean measure $\bl$ exists and its law is uniquely determined by (a) and (b).
We denote the law of a Poisson process $\Xi$ with mean measure $\bl$ by $\mathrm{PRM}(\bl)$,
i.e. $\Xi \sim \mathrm{PRM}(\bl)$. 
\begin{exa}
Suppose that $E \subseteq \mathbb{R}^d$, $d \ge 1$.
\begin{enumerate} 
\item[(a)] Let $\bl = \lambda |\,.\,|$, where $\lambda >0$ and $|\,.\,|$ denotes Lebesgue measure on $E$. Then
$\Xi \sim \mathrm{PRM}(\lambda |\,.\,|)$ is called a \textit{homogeneous Poisson process with intensity} $\lambda$. 
\item[(b)] Alternatively,
suppose that the mean measure $\bl$ of a Poisson process $\Xi$ is absolutely continuous with respect to Lebesgue measure, i.e. that there exists a non-negative function $\lambda$
such that for any $B \in \mathcal{E}$,
\[
\bl(B) = \int_B \lambda(x)dx.
\]
Then $\Xi$ is called \textit{inhomogeneous Poisson process} with \textit{rate} or \textit{intensity function} $\lambda(\,.\,)$.
\end{enumerate} 
\end{exa}
\subsection{Palm processes and Janossy densities}\label{s: Palm_Janossy}
Suppose $\Xi$ is a point process on $E$ with $\sigma$-finite mean measure $\bl$.
For any $z \in E$, a point process $\Xi_z$ is called \textit{Palm process
associated with $\Xi$ at $z$} if, for any measurable function $f: E \times \overline{M}_p(E) \to \mathbb{R}_+$,
\begin{equation}\label{d: Palm_equality}
\mathbb{E} \left[ \int_E f(z, \Xi) \Xi(dz)\right] = \mathbb{E} \left[ \int_E f(z,\Xi_z) \bl(dz)\right]. 
\end{equation}
We may define probability measures $\{P_z, z \in E\}$, called \textit{Palm distributions}, by setting
\[
P_z(R) := P\left( \Xi_z \in R \right) := \frac{\mathbb{E}\left[ I_{\{\Xi \in R\}} \Xi(dz)\right]}{\bl(dz)}, 
\]
for all $R \in \overline{\mathcal{M}}_p(E)$. A point process $\Xi_z$ on $E$ is then called a Palm process associated with $\Xi$ at $z$
if it has the Palm distribution $P_z$ of $\Xi$ at $z$. Palm processes can be used to give a characterisation of Poisson processes: 
\begin{equation}\label{d: Char_PP_Palm}
\textnormal{A process } \Xi \textnormal{ is a Poisson process if and only if }
\mathcal{L}(\Xi_z) = \mathcal{L}(\Xi + \delta_z) \textnormal{ $\bl$-a.s.}
\end{equation}
For more details on Palm theory, see Chapter 10 in \cite{Kallenberg:1983} or Chapter 13 in \cite{Daley/VereJones:2008}.

Another important tool in point process theory is given by the so-called Janossy measures. These are used to express the probability of a point process having a certain number
of points and these points being located in a certain region. Suppose that $\Xi$ is a finite point process on $E$. Then there exist measures $\{J_m\}$, $m \ge 0$,
called \textit{Janossy measures}, such that for measurable functions $f: {M}_p(E) \to \mathbb{R}_+$,
\[
\mathbb{E}f(\Xi) = \sum_{m \ge 0} \int_{E^m}  f\left(\sum_{i=1}^m \delta_{z_i}\right) \frac{1}{m! }\, J_m(dz_1, \ldots, dz_m), 
\]
where $J_m(dz_1, \ldots, dz_m)/m!$ describes the probability of the process having $m$ points lying close to $z_1, \ldots, z_m$. Suppose there exists a fixed $\sigma$-finite measure $\boldsymbol{\nu}$ on $E$. For instance, let $\boldsymbol{\nu}$ be the counting measure in case $E$ is a finite set, 
or let it be Lebesgue measure for $E$ a compact subset of Euclidean space. Suppose furthermore that for each $m \ge 0$, $J_m$ is absolutely continuous with respect
to $\boldsymbol{\nu}$. The Radon-Nikodym theorem then ensures the existence of derivatives $j_m: E^m \to [0, \infty)$ of $J_m$ with respect to $\boldsymbol{\nu}^m$, so that
\[
\mathbb{E}f(\Xi) =  \sum_{m \ge 0} \int_{E^m}  f\left(\sum_{i=1}^m \delta_{z_i}\right) \frac{1}{m! }\, j_m(z_1, \ldots, z_m) \boldsymbol{\nu}^m(dz_1, \ldots, dz_m).
\]
The derivatives $\{j_m\}$, $m \ge 0$, are called \textit{Janossy densities}. In the above expression for $\mathbb{E}f(\Xi)$, the term with $m=0$ is
interpreted as $j_0 f(\emptyset)$. By Lemma 5.4.III in 
\cite{Daley/VereJones:2003}, 
the density $\mu$ of the first moment measure of $\Xi$, i.e. of the intensity measure
$\bl$ of $\Xi$, may then be expressed in terms of the Janossy densities:
\begin{equation}\label{d: density_mean_measure}
\mu(z)=  \sum_{m \ge 0} \int_{E^m} \frac{1}{m!} \, j_{m+1}(z, z_1, \ldots, z_m)\boldsymbol{\nu}^m(dz_1, \ldots, dz_m),
\end{equation}
where the term with $m=0$ is interpreted as $j_1(z)$. We then have $\bl(dz)= \mu(z)\boldsymbol{\nu}(dz)$. 
Janossy densities may furthermore be used to express the conditional probability density of a point 
being located at $z$ given the configuration $\Xi^z$ of $\Xi$ outside a neighbourhood $N_z$ of $z$. More precisely, suppose that the point process 
$\Xi$ is simple and that for each $z \in E$,
$N_z \in \mathcal{E}$ is a neighbourhood of $z$, with $z \in N_z$, such that the following two mappings
are product measurable:
\begin{align}\label{d: neighbourhood_structure}
\begin{split}
& \mathcal{M}_p(E) \times E \to [0, \infty): (\xi, z) \mapsto \xi(N_z), \phantom{blaaaaaaaaaaaaaaaaaaaaaaaaaaaaaaa}\\
& \mathcal{M}_p(E) \times E \to \mathcal{M}_p(E): (\xi, z) \mapsto \xi \textnormal{ restricted to } N_z^c. 
\end{split}
\end{align}
For any $z \in E$ and for some fixed integer $m$, fix $m$ points $x_1, \ldots, x_m \in N_z^c$ and let $\boldsymbol{x}=(x_1, \ldots, x_m)$.
Define
\begin{equation}\label{d: conditional_density_Janossy}
g(z, \boldsymbol{x}) = \frac{\sum_{k \ge 0} \int_{N_z^k} j_{m+k+1}(z, \boldsymbol{x},\boldsymbol{y}) (k!)^{-1} \boldsymbol{\nu}^k(d\boldsymbol{y})}
{\sum_{l \ge 0}\int_{N_z^l} j_{m+l}(\boldsymbol{x},\boldsymbol{w}) (l!)^{-1} \boldsymbol{\nu}^l(d\boldsymbol{w})}.
\end{equation}
$g(z, \boldsymbol{x})$ is the conditional density of a point at $z$ given that $\Xi^z$ is $\sum_{i=1}^m \delta_{x_i}$ 
(the term with $k=0$ is interpreted as $j_{m+1}(z, \boldsymbol{x})$ and the term with $l=0$ similarly; moreover, if the denominator is zero, we interpret
$g(z, \boldsymbol{x})$ as zero.)
 \cite{Barbour/Brown:1992} (see (2.7) on p. 16) show
that, for a bounded measurable function $f: M_p(E) \to \mathbb{R}_+$, we then have
\begin{equation}\label{d: equality_cond_dens_Janossy}
\mathbb{E} \left[ \int_E f(\Xi^z) \Xi(dz)\right] = \mathbb{E}\left[ \int_E f(\Xi^z) g(z, \Xi^z) \boldsymbol{\nu}(dz)\right]. 
\end{equation}
Finally, like Palm processes, Janossy densities may also be used to give a characterisation of Poisson processes (see Theorem 2.11 in \cite{Xia:2005}):
\begin{thm}
A point process $\Xi$ on $E$ with mean measure $\bl$ and $\lambda = \bl(E) < \infty$ 
is a Poisson process if and only if, with respect to $\bl$, its Janossy densities $j_m$
are constant and equal to $e^{-\lambda}$, for all $m \in \mathbb{Z}_+$.
\end{thm}
\noindent For more details on Janossy measures and densities, consult Sections 5.3 and 5.4 in 
\cite{Daley/VereJones:2003}.

\subsection{Approximation of point processes -- the generator interpretation}\label{s: Generator_process}
Let $\Xi$ be a finite point process on $E$ with finite intensity measure $\bl$, where $\lambda := \bl(E) < \infty$. 
Let $Z:= \{Z_t, t \in \mathbb{R}_+\}$ be an immigration-death process on $E$ with immigration intensity $\bl$ and unit per-capita death rate.
This process is called \textit{spatial immigration-death process} by \cite{Preston:1975} and \cite{Xia:2005}. 
$Z_t$ takes values in $M_p(E)$ and describes the point configuration of particles of a population on $E$ at time $t$. 
Given that the process takes a configuration $\xi \in M_p(E)$, the process stays in state $\xi$ for an $\mathrm{Exp}(1/(|\xi| + \lambda))$-distributed
period of time. Then, with probability $\lambda/(|\xi| + \lambda)$, a new particle immigrates to the population and puts itself on $z \in E$, which is chosen
from the distribution $\bl/\lambda$, independently of the existing configuration. The new configuration is then $\xi + \delta_z$. Or, with probability 
$|\xi|/(|\xi| + \lambda)$, one particle from the population dies, that is, a point, say, $\delta_w$, is chosen uniformly at random from the existing configuration
and is erased. The new configuration is then $\xi - \delta_w$. 
See Figure \ref{f: Process_Imm_death} for an illustration.
\begin{figure}[!ht]
\begin{center}{\footnotesize \input{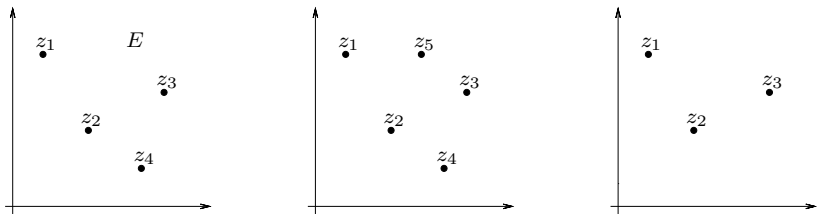}}\end{center}
\caption{(Left) A configuration $\xi$ of four points $z_1, z_2, z_3, z_4$ in $E = \mathbb{R}^2_+$. (Middle) An additional point $z_5$ is added to $\xi$ so that the 
new configuration is $\xi + \delta_{z_5}$. (Right) Point $z_4$ disappears from $\xi$. The new configuration is $\xi-\delta_{z_4}$.}
\label{f: Process_Imm_death}
\end{figure}

\noindent The generator of $Z$ is given by
\begin{align}\label{d: generator}
&(\mathcal{A}\upgamma)(\xi) \nonumber \\
&=  \lim_{t \downarrow 0} \frac{1}{t}\left\{ \mathbb{E}\left[ \upgamma(Z_t) | Z_0 = \xi\right] - \upgamma(\xi)\right\} \nonumber \\
&= \lim_{t \downarrow 0}\frac{1}{t} \left\{\int_E \sum_{\zeta \in \{\xi + \delta_z, \xi-\delta_z, \xi\}} \upgamma(\zeta) 
P(Z_t = \zeta|Z_0 = \xi)dz - \upgamma(\xi) \right\} \nonumber \\
&= \lim_{t \downarrow 0}\frac{1}{t} \left\{ 
\int_E \upgamma(\xi + \delta_z)t\bl(dz) + \int_E \upgamma(\xi - \delta_z)t\xi(dz)  \right. \nonumber\\
& \phantom{\lim_{t \to 0} \frac{1}{t} blaaaaaaaaaaaaaa}\left.+ \upgamma(\xi) \left[1 - \int_E t(\bl+\xi)(dz)\right] - \upgamma(\xi)\right\}\nonumber\\
& = \int_E [\upgamma(\xi + \delta_z) - \upgamma(\xi)]\bl(dz) + \int_E [\upgamma(\xi - \delta_z) - \upgamma(\xi)]\xi(dz),
\end{align}
for all $\xi \in M_p(E)$  and for $\upgamma$ a suitable function $M_p(E) \to \mathbb{R}$; see also (3.6) in \cite{Xia:2005}.
Let $\Xi_{\bl}$ be a Poisson process with intensity measure $\bl$, i.e. $\Xi_{\bl} \sim \mathrm{PRM}(\bl)$. 
By (\ref{d: generator}),
\begin{align*}
&\mathbb{E}\left[(\mathcal{A}\upgamma)(\Xi_{\bl})\right] \\
&= \mathbb{E} \left\{ \int_E \left[ \upgamma(\Xi_{\bl} + \delta_z) - \upgamma(\Xi_{\bl})\right] \bl(dz)
+ \int_E \left[ \upgamma(\Xi_{\bl} - \delta_z) - \upgamma(\Xi_{\bl})\right]\Xi_{\bl} (dz) \right\}.  
\end{align*}
For any $z \in E$, let $\Xi_{\bl,z}$ be the Palm process associated with $\Xi_{\bl}$ at $z$. Then, by (\ref{d: Palm_equality}) and (\ref{d: Char_PP_Palm}),
\begin{align*}
\mathbb{E}\int_E [\upgamma(\Xi_{\bl} - \delta_z) - \upgamma(\Xi_{\bl})] \Xi_{\bl}(dz) 
&= \mathbb{E}\int_E [\upgamma(\Xi_{\bl,z} - \delta_z) - \upgamma(\Xi_{\bl,z})] \bl(dz) \\
&= \mathbb{E} \int_E [\upgamma(\Xi_{\bl}) - \upgamma(\Xi_{\bl} + \delta_z)]\bl(dz),
\end{align*}
and it follows that 
\begin{align}
\begin{split}\label{p: Stein_id_generator_proc}
\mathbb{E}\left[(\mathcal{A}\upgamma)(\Xi_{\bl})\right] 
= &\,\,\mathbb{E} \int_E [\upgamma(\Xi_{\bl} + \delta_z) - \upgamma(\Xi_{\bl})]\bl(dz)\\ 
&+ \mathbb{E} \int_E [\upgamma(\Xi_{\bl}) - \upgamma(\Xi_{\bl} + \delta_z)] \bl(dz)
= 0.
\end{split}
\end{align}
Together with Theorem 7.1 in \cite{Preston:1975}, this implies that $\mathrm{PRM}(\bl)$ is the unique equilibrium distribution of $Z$ 
(see also Proposition 3.4 in \cite{Xia:2005}). 
We next fix some notation: 
\begin{itemize}
\item 
$ \displaystyle \mathrm{PRM}(\bl)\{R\} := P_{\Xi_{\bl}}(R) = P\left(\Xi_{\bl} \in R\right),$ for any $R \in \mathcal{M}_p(E)$. 
\item 
$ \displaystyle
\mathrm{PRM}(\bl)(h) := \int_{M_p(E)} h(\xi) \mathrm{PRM}(\bl)\{d\xi\}$, for a function $h: M_p(E) \to \mathbb{R}$.
\item For a fixed set $R \in \mathcal{M}_p(E)$ and the particular choice $h(\xi) := I_{\{\xi \in R\}}$, we then have
\[
\mathrm{PRM}(\bl)(h) = \mathrm{PRM}(\bl)\{R\}.
\]
\item Let $\mathbb{P}^\xi$ denote the distribution of the immigration-death 
process $Z$ given that it has the initial configuration
$\xi \in M_p(E)$, i.e. 
\[
\mathbb{P}^\xi(Z_t=\zeta) = P(Z_t = \zeta| Z_0=\xi),
\]
for any $\zeta \in M_p(E)$ and $t \ge 0$. Likewise,
$\mathbb{E}^{\xi}f(Z_t) = \int_{M_p(E)} f(\zeta) P(Z_t = \zeta| Z_0=\xi) d\zeta$.
\end{itemize}
The process analogue of the Stein equation is given by
\begin{equation}\label{d: Stein_equation_process}
(\mathcal{A}\upgamma)(\xi) = h(\xi) - \mathrm{PRM}(\bl)(h),
\end{equation}
for any $\xi \in M_p(E)$.  We construct a solution $\upgamma$ to this equation: by Proposition \ref{t: upgamma_solves_Stein} below, the function $\upgamma$, 
given in Proposition \ref{t: upgamma_welldefined}, 
solves the Stein equation. 
\begin{prop}\label{t: upgamma_welldefined}
Let $Z= \{Z_t, t \in \mathbb{R}_+\}$ be an immigration-death process on $E$ with immigration intensity $\bl$ and unit per-capita death rate. For any bounded $h: M_p(E) \to \mathbb{R}$, the function $\upgamma: M_p(E) \to \mathbb{R}$ given by
\[
\upgamma(\xi) = - \int_0^\infty \left\{ \mathbb{E}^\xi h(Z_t) - \mathrm{PRM}(\bl)(h) \right\} dt
\]
is well defined, and $\sup_{\xi:\, \xi(E)=k} |\upgamma(\xi)| < \infty$ for each $k \in \mathbb{Z}_+$.
\end{prop}
\begin{proof}
We consider a coupling of the immigration-death process $Z$ under $\mathbb{P}^\xi$ with another immigration-death process $\tilde{Z}$ 
under $\mathbb{P}^{\mathrm{PRM}(\bl)}$, setting $Z=Z^0+D$, $\tilde{Z}=Z^0+\tilde{D}$, where $Z^0$, $D$ and $\tilde{D}$ are independent, 
$Z^0$ is an immigration-death process under $\mathbb{P}^0$, i.e. having no initial particles, and $D$ and $\tilde{D}$ are both pure death processes 
with unit per-capita death rate, such that $D_0=\xi$, and $\tilde{D}_0 \sim \mathrm{PRM}(\bl)$. 
Let $\tau$ denote the earliest time at which both pure death processes
have lost all of their particles:
\[
\tau = \inf\{ u \ge 0:\, D_u=\tilde{D}_u=0\}. 
\]
After time $\tau$, the two processes $Z$ and $\tilde{Z}$ behave identically, i.e. $Z_t=\tilde{Z}_t$ for all $t \ge \tau$. Then,
\begin{align}\label{p: boundupgamma}
\begin{split}
|\upgamma(\xi)| &\le \int_0^\infty \left| \mathbb{E}^\xi h(Z_t) - \mathrm{PRM}(\bl)(h)\right|dt 
=  \int_0^\infty \left| \mathbb{E}^\xi h(Z_t) - \mathbb{E}h(\tilde{Z}_t)\right|dt\\
& \le \int_0^\infty \mathbb{E}\left(| h(Z_t) - h(\tilde{Z}_t) | \cdot I_{\{\tau > t\}} \,|\,  D_0=\xi\right)\\
& \le 2 ||h|| \int_0^\infty P\left( \tau > t \,|\, D_0=\xi\right)dt = 2||h||\mathbb{E}\left( \tau \,|\, D_0=\xi\right).
\end{split}
\end{align}
To determine $\mathbb{E}\left( \tau \,|\, D_0=\xi\right)$, note that the total number of points that the two processes $D$ and $\tilde{D}$ have to lose until time
$\tau$ is given by the random integer $|D_0| + |\tilde{D}_0|$, and that
\[
\tau = \sum_{i=1}^{|D_0| + |\tilde{D}_0|} \tau_i,
\]
where $\tau_i$ denotes the time between the $(i-1)$th and $i$th death. Since the two pure death processes have unit per-capita
death rates, the time $\tau_i$, for a fixed realisation $k = |\xi| + |\tilde{D}_0|(\omega)$, 
is exponentially distributed with rate $k-i +1$, for each
$i=1, \ldots, k$.
Therefore, 
\[
|\upgamma(\xi)| \le  2||h||\mathbb{E} \psi(|\xi| + |\tilde{D}_0|) < \infty,
\]
where $\psi(k)= \sum_{i=1}^k 1/(k-i+1) = \sum_{i=1}^k 1/i$, and the last inequality
follows because $|\tilde{D}_0| \sim \mathrm{Poi}(\lambda)$.
\end{proof}
\begin{prop}\label{t: upgamma_solves_Stein}
The function $\upgamma$ defined in Proposition \ref{t: upgamma_welldefined} satisfies the Stein equation
\[
(\mathcal{A}\upgamma)(\xi) = h(\xi)-\mathrm{PRM}(\bl)(h),
\]
for all $\xi \in M_p(E)$.
\end{prop}
\begin{proof}
Let $\upgamma_t(\xi) = -\int_0^t \left\{ \mathbb{E}^\xi h(Z_u) - \mathrm{PRM}(\bl)(h)\right\}du$, and let $T$ denote an exponential random variable modelling the first time 
that a particle is born or dies, with rate $q_\xi = \lambda + |\xi|$. Then, we may rewrite $\upgamma_t(\xi)$ as
\begin{align*}
& -\int_0^t \left\{ \mathbb{E}^\xi \left[ h(Z_u) - \mathrm{PRM}(\bl)(h)\right] 
\cdot I_{\{T > t\}} \right.\\
&\left. \phantom{blaaaa}+ \mathbb{E}^\xi \left[ h(Z_u) - \mathrm{PRM}(\bl)(h)\right] \cdot I_{\{T \le t\}}
\right\} du\\
= &-[h(\xi) - \mathrm{PRM}(\bl)(h)] e^{-q_\xi t} + \int_0^t e^{-q_\xi u} \left\{ -q_{\xi} u[h(\xi)-\mathrm{PRM}(\bl)(h)] \phantom{\int_E} \right.\\
& 
\left. 
+ \int_E \upgamma_{t-u}(\xi+\delta_z)\boldsymbol{\lambda}(dz) + \int_E \upgamma_{t-u}(\xi - \delta_z)\xi(dz)\right\}du\\
= & -\frac{1}{q_\xi} [h(\xi)-\mathrm{PRM}(\bl)(h)]\left( 1-e^{-q_\xi t}\right)\\
& + \int_0^t e^{-q_\xi u} \left\{ \int_E \upgamma_{t-u}(\xi + \delta_z) \boldsymbol{\lambda}(dz) + \int_E \upgamma_{t-u}(\xi -\delta_z) \xi(dz)\right\}du,
\end{align*}
where we obtained the last equality by evaluating the integral $-\int_0^t e^{-q_\xi u} q_{\xi} $ $u[h(\xi)-\mathrm{PRM}(\bl)(h)]du$.
From the proof of Proposition \ref{t: upgamma_welldefined}, the functions $\upgamma_t(\xi)$ are uniformly bounded in $t$ for each $\xi$. Letting $t \to \infty$ and using 
dominated convergence, it follows that 
\begin{align}
\upgamma(\xi) = & -\frac{1}{q_\xi}[h(\xi) - \mathrm{PRM}(\bl)(h)] \nonumber \\
 & + \int_0^\infty e^{-q_\xi u}\left\{ \int_E \upgamma(\xi 
 + \delta_z) \boldsymbol{\lambda}(dz) + \int_E \upgamma(\xi -\delta_z) \xi(dz)
 \right\}du  \nonumber\\
=& \frac{1}{q_\xi} \left\{ -[h(\xi) - \mathrm{PRM}(\bl)(h)] 
+ \int_E \upgamma(\xi + \delta_z) \boldsymbol{\lambda}(dz) 
+ \int_E \upgamma(\xi -\delta_z) \xi(dz)\right\} 
\label{p: upgamma_rewritten}
\end{align}
By rearranging (\ref{p: upgamma_rewritten}) and noting that $-q_\xi \upgamma(\xi) = -\int_E \upgamma(\xi)(\boldsymbol{\lambda}(dz) + \xi(dz))$, we find
\begin{align*}
&h(\xi)-\mathrm{PRM}(\bl)(h) \\
&= \int_E \left[ \upgamma(\xi + \delta z) - \upgamma(\xi) \right] \boldsymbol{\lambda}(dz) + 
\int_E \left[ \upgamma(\xi - \delta_z) -\upgamma(\xi) \right]\xi(dz)\\
&= (\mathcal{A}\upgamma)(\xi).
\end{align*}
\end{proof}
\noindent The solution $\gamma$ from Proposition \ref{t: upgamma_welldefined} solves the Stein equation for any bounded function $h: M_p(E) \to \mathbb{R}$, 
and by (\ref{p: Stein_id_generator_proc}),
\[
\mathbb{E}\left[ (\mathcal{A}\upgamma)(\Xi_{\bl})\right] = 0, 
\]
as there is no expected variation of the immigration-death process $Z$ when it is in equilibrium. In order to
determine approximation errors in the total variation distance we now choose $h(\xi) = I_{\{\xi \in R\}}$ for some set $R \in \mathcal{M}_p(E)$. 
Then, taking expectations on both sides of the Stein equation (\ref{d: Stein_equation_process}) and replacing $\xi$ by the process $\Xi$, we obtain
\begin{equation}\label{d: Exp_Stein_process_generator}
\left| \mathbb{E} \left(\mathcal{A}\upgamma\right)(\Xi)\right| 
= \left|\mathbb{E} I_{\{\Xi \in R\}} - \mathrm{PRM}(\bl)\{R\}\right| 
 = \left| P\left( \Xi \in R\right) - P\left( \Xi_{\bl} \in R\right)\right|,
\end{equation}
which is obviously $0$ for $\Xi=\Xi_{\bl}$, showing that the generator $\mathcal{A}$ satisfies the required Stein identity 
(mentioned previously in Sections \ref{s: antisymmetric_function} and \ref{s: generator_interpretation}). 
Then the problem of determining an upper bound on the total variation distance between the laws of the two processes $\Xi$ and $\Xi_{\bl} \sim \mathrm{PRM}(\bl)$,
that is,
\[
d_{TV}\left( \mathcal{L}(\Xi), \mathrm{PRM}(\bl)\right) = \sup_{R \in \mathcal{M}_p(E)} \left| P\left( \Xi \in R\right) - P\left( \Xi_{\bl} \in R\right)\right|,  
\]
is equivalent to determining a uniform bound on $\left| \mathbb{E} \left(\mathcal{A}\upgamma\right)(\Xi)\right|$. To achieve the latter, we require smoothness
estimates of the function $\upgamma$:
\begin{lemma}\label{t: dtv_bounds_upgamma}
If $\upgamma$ is defined as in Proposition \ref{t: upgamma_welldefined} and if $h(\xi)= I_{\{\xi \in R\}}$ for some set $R \in \mathcal{M}_p(E)$, then
\begin{align*}
(i)\quad  \Delta_1 \upgamma  
&= \sup_{\xi \in M_p(E), z \in E} \left| \upgamma(\xi + \delta_z) 
- \upgamma(\xi)\right| \le 1, \\
(ii) \quad \Delta_2 \upgamma &= 
\sup_{\xi \in M_p(E); z,w \in E} \left| \upgamma(\xi + \delta_z + \delta_w)
- \upgamma(\xi + \delta_z) -\upgamma(\xi + \delta_w) +\upgamma(\xi)\right| \\
&\le 1.
\end{align*}
\end{lemma}
\begin{proof}
(i) From the definition of $\upgamma$ in Proposition \ref{t: upgamma_welldefined} we have that for any $\xi \in M_p(E)$ and for any $z \in E$,
\[
\upgamma(\xi + \delta_z) - \upgamma(\xi) = \int_0^\infty \left\{ \mathbb{E}^\xi h(Z_t) - \mathbb{E}^{\xi + \delta_z}h(Z_t) \right\}dt, 
\]
where the immigration-death process $Z= \{Z_t, t \in \mathbb{R}_+\}$ on $E$ with immigration intensity $\bl$ and unit per-capita death rate is realised under 
$\mathbb{P}^\xi$. Let $T$ be an exponential random variable with rate $1$ that is independent of $Z$. It follows that the process $Z'_t = Z_t + \delta_z I_{\{T > t\}}$
has distribution $\mathbb{P}^{\xi + \delta_z}$. Then
\begin{align*}
\upgamma(\xi + \delta_z) - \upgamma(\xi) 
&= \int_0^\infty \mathbb{E}^\xi \left[ \left\{h(Z_t) - h(Z'_t)\right\} I_{\{T > t\}}\right] dt\\
&= \int_0^\infty \mathbb{E}^\xi \left[h(Z_t) - h(Z_t + \delta_z)\right]e^{-t}dt
\end{align*}
Since $\int_0^\infty e^{-t}dt = 1$ and 
\[
|h(\xi)-h(\xi +  \delta_z)| = \left| I_{\{\xi \in R\}} - I_{\{\xi +  \delta_z \in R\}}\right| \le 1,
\]
we have $\Delta_1 \upgamma \le 1$. (ii) By the definition of $\upgamma$, we have
\begin{align*}
&\upgamma(\xi + \delta_z + \delta_w) - \upgamma(\xi + \delta_z) -\upgamma(\xi + \delta_w) +\upgamma(\xi)\\
& = -\int_0^\infty \left\{ \mathbb{E}^{\xi + \delta_z + \delta_w}h(Z_t) - \mathbb{E}^{\xi + \delta_z}h(Z_t) - \mathbb{E}^{\xi + \delta_w}h(Z_t)
+ \mathbb{E}^\xi h(Z_t)\right\}dt.
\end{align*}
Let $Z$ be realised under $\mathbb{P}^\xi$ as in (i) and let $T^z$ and $T^w$ be two independent exponential random variables with rate $1$. The processes
\begin{align*}
Z_t^z &=  Z_t + \delta_z I_{\{T^z > t\}}, \\
Z_t^w &= Z_t + \delta_w I_{\{T^w > t\}}, \\
Z_t^{zw}  &= Z_t + \delta_zI_{\{T^z > t\}} + \delta_w I_{\{T^w > t\}}
\end{align*}
then have distributions $\mathbb{P}^{\xi + \delta_z}$, $\mathbb{P}^{\xi + \delta_w}$ and $\mathbb{P}^{\xi + \delta_z + \delta_w}$, respectively. 
Therefore, 
\begin{align*}
&\upgamma(\xi + \delta_z + \delta_w) - \upgamma(\xi + \delta_z) -\upgamma(\xi + \delta_w) +\upgamma(\xi)\\
& = - \int_0^\infty \mathbb{E}^\xi \left[ \left\{ h(Z_t^{zw}) - h(Z_t^z) - h(Z_t^w) + h(Z_t)\right\}I_{\{T^z > t\}}I_{\{T^w > t\}}\right] dt \\
& = -\int_0^\infty \mathbb{E}^\xi \left[ h(Z_t + \delta_z + \delta_w) - h(Z_t + \delta_z) - h(Z_t + \delta_w) + h(Z_t)\right] e^{-2t}dt,
\end{align*}
which gives $\Delta_2 \upgamma \le 1$.
\end{proof}
We are now in shape to prove a process analogue of Theorem \ref{t: dTV_bounds_Poi_localapproach}:
\begin{thm}\label{t: dTV_PP_generator}
Suppose there exists a fixed measure $\boldsymbol{\nu}$ on $E$ and suppose that $\Xi$ is a finite simple point process on $E$ with finite mean measure $\bl$
and Janossy densities $\{j_m\}_{m \ge 0}$. Let the density $\mu$ of $\bl$ 
with respect to $\boldsymbol{\nu}$ be given by (\ref{d: density_mean_measure}) and let $\{N_z\}_{z \in E}$
be a neighbourhood structure satisfying (\ref{d: neighbourhood_structure}). 
Then,
\begin{multline*}
d_{TV}\left(\mathcal{L}(\Xi), \mathrm{PRM}(\bl)\right) 
\le \int_E \mathbb{E} \Xi(N_z) \mu(z)\boldsymbol{\nu}(dz) 
+\mathbb{E} \left[ \int_E \Xi(N_z\setminus \{z\}) \Xi(dz)\right] \\
+ \int_E \mathbb{E} \left| g(z, \Xi^z) - \mu(z)\right| \boldsymbol{\nu}(dz),
\end{multline*}
where the conditional density $g(z, \Xi^z)$  at $z$ given the configuration $\Xi^z$ of $\Xi$ outside $N_z$ is defined in (\ref{d: conditional_density_Janossy}).
\end{thm}
\begin{remark}
Since $E$ is a metric space, an example for a neighbourhood $N_z$ would be a closed ball with a certain radius centred at $z$. 
\cite{Barbour/Brown:1992} show that this choice satisfies (\ref{d: neighbourhood_structure}).  
\end{remark}
\begin{remark}
Intuitively, the first error term in Theorem \ref{t: dTV_PP_generator} measures the size of the neighbourhoods, the second measures the extent of local dependence, i.e. 
inside a neighbourhood, whereas the third term measures the size of the difference between what happens at z and what happens outside its neighbourhood $N_z$. 
There is clearly a trade-off between the sizes of the first and third error terms -- the smaller $N_z$, the bigger the dependence between $z$ and $N_z^c$, and vice versa.
\end{remark}
\begin{proof}
From (\ref{d: Exp_Stein_process_generator}) we know that it is sufficient to find a uniform bound for the modulus of 
$\mathbb{E}(\mathcal{A}\upgamma)(\Xi)$,
where $\mathcal{A}$ is defined as in (\ref{d: generator}), and $\upgamma$ is the solution of the Stein equation from Proposition \ref{t: upgamma_welldefined}
with $h(\xi) = I_{\{\xi \in R\}}$, for any $R \in \mathcal{M}_p(E)$. We have 
\begin{align*}
&|\mathbb{E}(\mathcal{A}\upgamma)(\Xi)| \\
= &\,\,\left| \mathbb{E} \left\{ 
\int_E \left[ \upgamma(\Xi + \delta_z) - \upgamma(\Xi)\right]\bl(dz) + \int_E \left[ \upgamma(\Xi- \delta_z) - \upgamma(\Xi)\right]\Xi(dz)\right\}\right|.
\end{align*}
Let $\Xi^z$ denote the configuration of $\Xi$ outside $N_z$. We add and subtract $\upgamma(\Xi^z + \delta_z) - \upgamma(\Xi^z)$ to both 
of the integrands above. Then, 
\begin{align}\label{p: Exp_generator_sum}
\begin{split}
|\mathbb{E}(\mathcal{A}\upgamma)(\Xi)| 
\le & \left| \mathbb{E} \int_E \left[ \upgamma(\Xi + \delta_z) - \upgamma(\Xi) - \upgamma(\Xi^z + \delta_z) + \upgamma(\Xi^z)\right]\bl(dz)\right|\\
& + \left|\mathbb{E}\int_E  \left[ \upgamma(\Xi) - \upgamma(\Xi^z + \delta_z) - \upgamma(\Xi - \delta_z) + \upgamma(\Xi^z) \right] \Xi(dz) \right|\\
& + \left| \mathbb{E} \int_E \left[ \upgamma(\Xi^z + \delta_z) - \upgamma(\Xi^z)\right] \left\{ \Xi(dz) - \bl(dz)\right\} \right|.
\end{split}
\end{align}
Denote realisations of $\Xi$ and $\Xi^z$ by $\xi$ and $\xi^z$, respectively, and note that
\[
\xi = \xi^z + \sum_{w \in N_z}\delta_w =\xi^z + \sum_{w \in N_z\setminus\{z\}}\delta_w + \delta_z,
\]
where the last equality holds only for $z \in E$ such that $\Xi(\{z\})=1$. 
Then the modulus of the integrand of the first summand in (\ref{p: Exp_generator_sum}) corresponds to
\begin{multline*}
\left| \upgamma\left(\xi^z + \sum_{w \in N_z}\delta_w + \delta_z\right) - \upgamma\left(\xi^z + \sum_{w \in N_z}\delta_w\right) - \upgamma\left(\xi^z + \delta_z\right) 
+ \upgamma\left(\xi^z\right) \right|\\
\le  \xi(N_z)\Delta_2\upgamma,
\end{multline*}
whereas that of the second summand, for $z \in E$ such that $\Xi(\{z\})=1$, corresponds to
\begin{align*}
\left| \upgamma\left(\xi^z + \sum_{w \in N_z \setminus \{z\}} \delta_w + \delta_z\right) - \upgamma\left(\xi^z + \delta_z\right) - 
\upgamma\left(\xi^z + \sum_{w \in N_z \setminus \{z\}} \delta_w\right) + \upgamma\left(\xi^z\right)\right|,
\end{align*}
which may be bounded by $\xi(N_z \setminus \{z\}) \Delta_2 \upgamma$. Lemma \ref{t: dtv_bounds_upgamma} (ii) gives $\Delta_2 \upgamma \le 1$.
Upper bounds for the first and second summand are then given by
\[
\int_E \mathbb{E} \Xi(N_z) \bl(dz) \quad \textnormal{and} \quad \mathbb{E}\left[ \int_E \Xi(N_z \setminus \{z\})\Xi(dz) \right], 
\]
respectively, where we additionally used the Fubini-Tonelli theorem for the first summand. For the third summand,
(\ref{d: equality_cond_dens_Janossy}) gives
\begin{align*}
& \left| \mathbb{E} \int_E \left[ \upgamma(\Xi^z + \delta_z) - \upgamma(\Xi^z)\right] \left\{ \Xi(dz) - \bl(dz)\right\} \right|\\
& = \left| \mathbb{E} \int_E \left[\upgamma(\Xi^z + \delta_z) - \upgamma(\Xi^z) \right] \left\{ g(z, \Xi^z) - \mu(z) \right\}\boldsymbol{\nu}(dz) \right| \\
& \le \Delta_1 \upgamma \int_E \mathbb{E} \left| g(z, \Xi^z) - \mu(z)\right|\boldsymbol{\nu}(dz),
\end{align*}
where $\Delta_1 \upgamma \le 1$ due to Lemma \ref{t: dtv_bounds_upgamma} (i).
\end{proof}

Suppose we want to 
approximate the law of a point process $\Xi$ with mean measure $\bl$ by that of a
Poisson process with, say, mean measure $\tilde{\bl}$, different (but not too different) from $\bl$. 
We then simply do the approximation in two steps and use the triangle inequality:
\[
d_{TV}\left( \mathcal{L}(\Xi), \mathrm{PRM}(\tilde{\bl})\right) \le d_{TV}\left(\mathcal{L}(\Xi), \mathrm{PRM}(\bl)\right) + 
d_{TV}\left(\mathrm{PRM}(\bl), \mathrm{PRM}(\tilde{\bl})\right). 
\]
The following proposition gives an estimate of the additional error term.
\begin{prop}\label{t: dTV_two_PRM}
Let $\bl$ and $\tilde{\bl}$ be two finite measures on $E$. Then
\[
d_{TV}\left(\mathrm{PRM}(\bl), \mathrm{PRM}(\tilde{\bl})\right) \le \int_E | \bl - \tilde{\bl}|(dz)
\]
\end{prop}
\begin{proof}
Let $\Xi:= \Xi_{\tilde{\bl}} \sim \mathrm{PRM}(\tilde{\bl})$ and let $\Xi_{\bl} \sim \mathrm{PRM}(\bl)$. 
Let $\mathcal{A}$ be the generator of an immigration-death process with immigration intensity $\bl$, unit per-capita death rate, and equilibrium distribution
$\mathcal{L}(\Xi_{\bl})$. By (\ref{d: Exp_Stein_process_generator}), 
$| P\left(\Xi \in R\right)$  $ -$  $P\left(\Xi_{\bl} \in R\right)|$ 
equals
\[
\left|\mathbb{E}(\mathcal{A}\upgamma)\left(\Xi\right)\right|
= \left| \mathbb{E}\int_E \left[\upgamma(\Xi + \delta_z) - \upgamma(\Xi)\right]\bl(dz) + 
\mathbb{E}\int_E \left[\upgamma(\Xi- \delta_z) - \upgamma(\Xi)\right]\Xi(dz) \right|
\]
and it is sufficient to determine a uniform bound for $\left|\mathbb{E}(\mathcal{A}\upgamma)\left(\Xi\right)\right|$.
By (\ref{d: Palm_equality}), we have that for a bounded measurable function $\upgamma$, 
\[
\mathbb{E} \int_E \left[ \upgamma(\Xi - \delta_z) - \upgamma(\Xi) \right] \Xi(dz)
= \mathbb{E} \int_E \left[ \upgamma(\Xi_z - \delta_z) - \upgamma(\Xi_z)\right]\tilde{\bl}(dz),
\]
where $\Xi_z$ is the Palm process for $\Xi$ at $z$. It follows from (\ref{d: Char_PP_Palm}) that $\Xi_z$ is a Poisson process with mean measure $\tilde{\bl}$
with the addition of a deterministic atom at $z$. Likewise, $\mathcal{L}(\Xi_z - \delta_z) = \mathcal{L}(\Xi)$. The integral on the right-hand side then equals 
\[
\mathbb{E} \int_E \left[ \upgamma(\Xi) - \upgamma(\Xi + \delta_z)\right] \tilde{\bl}(dz) 
= -\int_E \mathbb{E}  \left[\upgamma(\Xi + \delta_z) - \upgamma(\Xi) \right] \tilde{\bl}(dz).
\]
Hence,
\begin{align*}
\left|\mathbb{E}(\mathcal{A}\upgamma)\left(\Xi\right)\right| 
&= \left| \int_E \mathbb{E} \left[ \upgamma(\Xi + \delta_z) - \upgamma(\Xi)\right](\bl - \tilde{\bl})(dz)\right| \\
& \le \int_E \Delta_1 \upgamma  | \bl - \tilde{\bl}| (dz)
\le  \int_E  | \bl - \tilde{\bl}| (dz),
\end{align*}
where we used Lemma \ref{t: dtv_bounds_upgamma} (i) for the last inequality. 
\end{proof}

The following two corollaries exemplify the use of Theorem \ref{t: dTV_PP_generator}. For both corollaries, we suppose that the state space $E$ is a finite index set,
called $\Gamma$, and that the point process $\Xi$ on $\Gamma$ is of the form $\sum_{\alpha \in \Gamma} I_\alpha \delta_\alpha$, with the $I_\alpha$'s being 
indicator variables defined on the probability space $(\Omega, \mathcal{F}, P)$. 
Note that $\Xi$ is a finite point process since $\Gamma$ is finite. It is also simple, since, for all $\omega \in \Omega$ and for all
$\beta \in \Gamma$, 
\[
\Xi(\omega, \{\beta\}) = \sum_{\alpha \in \Gamma} I_\alpha(\omega) \delta_\alpha(\{\beta\}) \le \delta_\beta(\{\beta\}) = 1. 
\] 
Corollary \ref{t: proc_version_dTV_indpt} treats the case of independent indicator variables so the point process has no 
dependence whatsoever between ``regions'' of $\Gamma$. It gives a
process analogue to Theorem \ref{t: Barbour_Hall_1984} for Poisson approximation of sums of independent indicator variables. 
Likewise, Corollary \ref{t: proc_version_dTV_local_dep} 
gives a process version of Theorem
\ref{t: dTV_bounds_Poi_localapproach}, where we have local dependence between the indicator variables. 
\begin{cor}\label{t: proc_version_dTV_indpt}
Let $\Gamma$ be a finite index set. Let $I_\alpha$, $\alpha \in \Gamma$, be independent Bernoulli random variables with success probability 
$P(I_\alpha = 1) = p_\alpha \in (0,1)$, for all $\alpha \in \Gamma$. Let $\Xi = \sum_{\alpha \in \Gamma} I_\alpha \delta_\alpha$ be a point process on $\Gamma$ with 
intensity measure $\bl =\sum_{\alpha \in \Gamma} p_\alpha \delta_\alpha$. Then
\[
d_{TV}\left( \mathcal{L}(\Xi), \mathrm{PRM}(\bl)\right) \le \sum_{\alpha \in \Gamma} p_\alpha^2. 
\]
\end{cor}
\begin{proof} As the $I_\alpha$'s are independent, we choose neighbourhoods $N_\alpha = \{\alpha\}$, for all $\alpha \in \Gamma$. Clearly, the second and third error
terms from Theorem \ref{t: dTV_PP_generator} vanish, and we obtain
\[
d_{TV}(\mathcal{L}(\Xi), \mathrm{PRM}(\bl)) \le \sum_{\alpha \in \Gamma} \mathbb{E}\Xi(\{\alpha\}) \bl(\{\alpha\}) 
= \sum_{\alpha \in \Gamma} \mathbb{E}\Xi(\{\alpha\})^2 = \sum_{\alpha \in \Gamma}p_\alpha^2.
\]
\end{proof}
\begin{cor}\label{t: proc_version_dTV_local_dep}
Let $\Gamma$ be a finite index set. Let $I_\alpha$, $\alpha \in \Gamma$, be Bernoulli random variables with success probability 
$P(I_\alpha = 1) = p_\alpha \in (0,1)$, for all $\alpha \in \Gamma$. Let $\Xi = \sum_{\alpha \in \Gamma} I_\alpha \delta_\alpha$ be a point process on $\Gamma$ with 
intensity measure $\bl =\sum_{\alpha \in \Gamma} p_\alpha \delta_\alpha$. For any choice of index $\alpha \in \Gamma$, define 
$\Gamma_\alpha^s \subseteq \Gamma \setminus\{\alpha\}$ to be the set of indices containing all $\beta \neq \alpha$ for which $I_\beta$ is strongly dependent 
on $I_\alpha$, and define $\Gamma_\alpha^w$ similarly as the set of indices $\beta$ for which $I_\beta$ is weakly dependent on $I_\alpha$. Furthermore, let
$Z_\alpha = \sum_{\beta \in \Gamma_\alpha^s} I_\beta$. Then, 
\[
d_{TV}\left( \mathcal{L}(\Xi), \mathrm{PRM}(\bl)\right) 
\le \sum_{\alpha \in \Gamma}\left\{ p_\alpha^2  + p_\alpha \mathbb{E}Z_\alpha + \mathbb{E}I_\alpha Z_\alpha + \eta_\alpha \right\}, 
\]
where
\[
\eta_\alpha = \mathbb{E}\left| \mathbb{E} \left\{ (I_\beta, \beta \in \Gamma_\alpha^w)\right\} - p_\alpha\right|.
\]
\end{cor}
\begin{proof}
For each $\alpha \in \Gamma$, we choose the neighbourhood $N_\alpha = \Gamma_\alpha^s \cup \{\alpha\}$. 
Let $\Xi^\alpha = \sum_{\beta \in \Gamma_\alpha^w}I_\beta \delta_\beta$ be the configuration of $\Xi$ outside $N_\alpha$. From the proof of
Theorem \ref{t: dTV_PP_generator},
\begin{align*} 
 d_{TV}\left(\mathcal{L}(\Xi), \mathrm{PRM}(\bl)\right) 
 \le &\sum_{\alpha \in \Gamma} \mathbb{E}\Xi(N_\alpha)\bl(\{\alpha\}) 
+  \sum_{\alpha \in \Gamma}\mathbb{E}\left[ \Xi(N_\alpha \setminus \{\alpha\}) \Xi(\{\alpha\})\right]\\
&+ \left| \sum_{\alpha \in \Gamma} \mathbb{E} \left[ \upgamma(\Xi^\alpha+ \delta_\alpha) - \upgamma(\Xi^\alpha)\right]
\left\{ \Xi(\{\alpha\}) - \bl(\{\alpha\}) \right\}\right|.
\end{align*}
The first of these three error terms equals 
\begin{align*}
\sum_{\alpha \in \Gamma} \mathbb{E}\Xi\left(\Gamma_\alpha^s \cup \{\alpha\}\right)p_\alpha 
&= \sum_{\alpha \in \Gamma}p_\alpha  \left( \mathbb{E}\Xi(\{\alpha\} + \sum_{\beta_\in \Gamma_\alpha^s} \mathbb{E}\Xi(\{\beta\}))\right)\\
&= \sum_{\alpha \in \Gamma} \left\{ p_\alpha^2  + p_\alpha \mathbb{E}Z_\alpha\right\},
\end{align*}
and the second error term equals
\[
\sum_{\alpha \in \Gamma} \mathbb{E}\left[ \Xi(\Gamma_\alpha^s) \Xi(\{\alpha\})\right] 
= \sum_{\alpha \in \Gamma} \mathbb{E} \left[ \left( \sum_{\beta_\in \Gamma_\alpha^s} I_\beta\right)I_\alpha\right] 
=\sum_{\alpha \in \Gamma} \mathbb{E}[I_\alpha Z_\alpha].
\]
It remains to observe that the third error term is bounded by
\[
\sum_{\alpha \in \Gamma} 
\left| \mathbb{E} (I_\alpha - p_\alpha)[\upgamma(\Xi^\alpha+ \delta_\alpha) - \upgamma(\Xi^\alpha)]\right|
\le \sum_{\alpha \in \Gamma}  \eta_\alpha \Delta_1 \upgamma, 
\]
and $\Delta_1 \upgamma \le 1$ by Lemma \ref{t: dtv_bounds_upgamma} (i).
\end{proof}

When comparing Theorems  \ref{t: Barbour_Hall_1984} and \ref{t: dTV_bounds_Poi_localapproach} (which treated the 
concrete example of $\Gamma= \{1, \ldots, n\}$) with their process analogues, Corollaries \ref{t: proc_version_dTV_indpt} and \ref{t: proc_version_dTV_local_dep},
we see that the respective error bounds are of the same form except for the multiplicative factors in $\lambda$ that are absent in the process results.
As these factors decrease towards zero with increasing $\lambda$, the lack of them in the process results shows that Poisson process approximation gives bigger errors than 
Poisson approximation. The reason for this is that the total variation distance is so strong that it does not allow for even the smallest shifts in the positions of points on the
carrier space. That is, if the sets of placements of the points of two point processes in a carrier space are disjoint, then, even if the points of 
the two processes are placed close to each other with respect to some metric on the carrier space, the total variation
distance takes value $1$, the maximum value it can take for a pair of probability distributions. As a consequence, the total variation
distance is not at all suited for approximating a process on a lattice in $\mathbb{R}^d$ by a process with a continuous intensity over $\mathbb{R}^d$. 
An example was given in Chapter \ref{Chap: Introduction}.

The hope is to find a way to recover multiplicative factors that decrease with increasing $\lambda$ when approximating a point process by a 
Poisson process. One way to do this is to use a metric that is weaker than the total variation metric and able to exploit the closeness in the positions of the points of the two 
processes. We would thus compare $\mathbb{E}h(\Xi_n)$ by $\mathbb{E}h(\Xi)$ for a set of functionals $h$ that is smaller than the one used
for approximation in total variation, and whose elements are not too sensitive to small differences in the positions
of points. \cite{Barbour/Brown:1992} and \cite{Barbour_et_al.:1992} constructed a suitable weaker metric, the $d_2$-metric. 
We will give their results in Section \ref{Sec: Improved_rates}. Another way to improve Poisson process approximation
is to consider marked point processes (in situations where the use of marks makes sense).
As we see in the following section, it is sometimes possible to actually recover the sharper results from Poisson approximation. 
Consider, for instance, the basic case where we associate a point $z$ from a carrier space $E$ to each indicator variable $I_i$, when we
know that the law of $W=\sum_{i=1}^n I_i$ is close to the Poisson distribution with parameter $\mathbb{E}W$. Then, fixing any $z$,
the process $\sum_{i=1}^n I_i \delta_z = W\delta_z$ gives $W$ points at position $z$ and it is clear that Poisson process approximation is the 
same as Poisson approximation for $W$. 
\subsection{Approximation of point processes with \iid marks -- Michel's argument}\label{s: Michel_MPP}
Suppose $\Xi$ is a marked point process of the form $\sum_{i=1}^n I_i \delta_{X_i}$, 
where the Ber\-noulli random variables $I_1, \ldots, I_n$ are independent of the \iid $E$-valued marks $X_1, \ldots, X_n$. 
We may then use an argument made by \cite{Michel:1988} to show 
that the total variation distance between $\mathcal{L}(\Xi)$ and the law of a Poisson process with mean measure 
$\mathbb{E}\Xi$ is smaller than or equal to the total variation distance between the law of $W$ and that of a Poisson random variable with
mean $\mathbb{E}W$. We may thus use Theorem \ref{t: Barbour_Hall_1984} to estimate the approximation error between $\mathcal{L}(\Xi)$ and $\mathrm{PRM}(\mathbb{E}\Xi)$.
\begin{thm} \label{t: Michel}
For each integer $n \ge 1$, let $I_1, \ldots, I_n$ be Bernoulli random variables with probability of success $P(I_i=1)=p_i \in (0,1)$. 
Let $E$ be a locally compact separable metric space and let
$X_1, \ldots, X_n$ be \iid $E$-valued random variables, independent of the $I_i$'s.
Moreover, let $\Xi = \sum_{i=1}^n I_i\delta_{X_i}$ and let $W = \sum_{i=1}^n  I_i$. Then,
\[
d_{TV}(\mathcal{L}(\Xi), \mathrm{PRM}(\mathbb{E}\Xi)) \le d_{TV}(\mathcal{L}(W), \mathrm{Poi}(\mathbb{E}W)). 
\]
\end{thm}
\begin{proof} 
Let $Z_1, \ldots, Z_n$ be \iid random variables with distribution $\mathcal{L}(X_1)$, and let them be independent of $W$. 
Then the process $\sum_{j=1}^W \delta_{Z_j}$ has the same distribution as the process of interest $\Xi$. 
Furthermore, note that a $\mathrm{PRM}(\mathbb{E}\Xi)$ can be realised as $\sum_{j=1}^{W^\star} \delta_{Z_j}$, where $W^\star \sim \mathrm{Poi}(\mathbb{E}W)$ 
is independent of the $Z_j$'s. Then, using (\ref{d: dTV_without_mod}) for the total variation distance,
\begin{align*}
& d_{TV}(\mathcal{L}(\Xi), \mathrm{PRM}(\mathbb{E}\Xi)) \\
& = \sup_{R} \left\{ P(\Xi \in R) - P(\mathrm{PRM}(\mathbb{E}\Xi) \in R) \right\}\\
& =\sup_{R} \left\{ P\left(\sum_{j=1}^{W} \delta_{Z_j} \in R\right) - P\left(\sum_{j=1}^{W^\star} \delta_{Z_j} \in R\right) \right\}\\
& =\sup_{R} \left\{ \sum_{l=0}^n P\left( \sum_{j=1}^{W} \delta_{Z_j} \in R \, , \, W=l \right) 
- \sum_{l=0}^\infty P\left( \sum_{j=1}^{W^\star} \delta_{Z_j} \in R \, , \, W^\star=l \right)\right\}\\
& = \sup_{R} \left\{ 
\sum_{l=0}^n P\left( \sum_{j=1}^{l} \delta_{Z_j} \in R \right)P\left( W=l \right) 
\right. \\
& \left. \phantom{blaaaaaaaaaaaaaaaaaaaaaaaaaaaa}- \sum_{l=0}^\infty P\left( \sum_{j=1}^{l} \delta_{Z_j} \in R \right)P\left( W^\star=l \right) \right\}\\
& \le \sup_{R}\sum_{l=0}^n P\left( \sum_{j=1}^{l} \delta_{Z_j} \in R \right)\left\{P(W=l)-  P(W^\star=l) \right\}_+\\
& \le \sum_{l=0}^n \left\{P(W=l)-  P(W^\star=l) \right\}_+,
\end{align*}
where $\{\, .\,\}_+ = \max(\, .\,,0)$. Now define 
\[
B_0 = \{l \in \{1, \ldots,n\}:\, P(W = l)>P(W^\star = l)\}. 
\]
Then 
\begin{align*}
\sum_{l=0}^n \left\{P(W=l)-  P(W^\star=l) \right\}_+
& = \sum_{l \in B_0} \left\{P(W=l)-  P(W^\star=l) \right\}\\
& = P(W \in B_0)-  P(W^\star \in B_0)\\
& \le \sup_{B \subseteq \mathbb{Z}_+} |P(W \in B)-  P(W^\star \in B)|\\
& =d_{TV}(\mathcal{L}(W), \mathrm{Poi}(\mathbb{E}W)).
\end{align*}
\end{proof}
\section{Improved rates for Poisson process approximation using the $d_2$-distance}\label{Sec: Improved_rates}
This section gives the results by \cite{Barbour/Brown:1992} and \cite{Barbour_et_al.:1992} for Poisson process approximation in a 
metric that is weaker than the total variation metric.
As before in Section \ref{Sec: Poi_proc_approx}, we assume that $E$ is a locally compact separable metric space. Let $d_0$ be a metric on $E$ that is bounded by $1$.
We now define metrics on both the space $M_p(E)$ of finite point measures over $E$ and on the set of probability measures over $M_p(E)$. 
Let $\mathcal{K}$ denote the set of functions $\kappa: E \to \mathbb{R}$ such that 
\[
s_1(\kappa) = \sup_{z_1 \neq z_2 \in E} \frac{|\kappa(z_1) - \kappa(z_2)|}{d_0(z_1, z_2)} < \infty,  
\]
which implies that for all $z_1 \neq z_2 \in E$, $|\kappa(z_1) - \kappa(z_2)| \le s_1(\kappa)d_0(z_1, z_2)$. Thus each function $\kappa \in \mathcal{K}$
is Lipschitz continuous with constant $s_1(\kappa)$.
Define a distance $d_1$ between two finite measures $\boldsymbol{\rho}$ and $\boldsymbol{\sigma}$ over $E$ by
\begin{equation}\label{d: def_d1}
d_1(\boldsymbol{\rho},\boldsymbol{\sigma}) = \left\{ 
\begin{array}{ll} 
1, & \textnormal{if }  \boldsymbol{\rho}(E) \neq \boldsymbol{\sigma}(E),\\
\displaystyle \frac{1}{\boldsymbol{\rho}(E)}\, \sup_{\kappa \in \mathcal{K}} \frac{\left| \int_E \kappa d\boldsymbol{\rho} - \int_E \kappa d\boldsymbol{\sigma}\right|}{s_1(\kappa)},
& \textnormal{if } \boldsymbol{\rho}(E) = \boldsymbol{\sigma}(E).
\end{array}
\right.
\end{equation}
Note that $d_1$ is bounded by $1$. We can use $d_1$ as distance between point measures in $M_p(E)$. The $d_1$-distance is then a Wasserstein metric induced by $d_0$
over point measures on $E$. 
Suppose that we have two point configurations $\xi_1, \xi_2 \in M_p(E)$ with the same number of points $|\xi_1|=|\xi_2|=m$. 
It then follows from (\ref{p: Wasserstein_coupling}) that $d_1(\xi_1, \xi_2)$ can be interpreted as the average distance between the points 
$(z_{11}, \ldots, z_{1m})$ and $(z_{21}, \ldots, z_{2m})$ of $\xi_1$
and $\xi_2$ under their closest matching, i.e.
\begin{equation}\label{d: closest_matching}
d_1(\xi_1, \xi_2) = \min_{\pi \in S_m} \frac{1}{m} \sum_{j=1}^m d_0({z}_{1j},{z}_{2\pi(j)}).
\end{equation}
See Figure \ref{f: d2metric} for an illustration.
\begin{figure}[!ht]
\begin{center}{\footnotesize \input{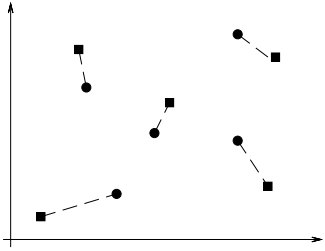}}\end{center}
\caption{Two point configurations $\xi_1$ (bullets) and $\xi_2$ (squares) with five points each on $E=\mathbb{R}^2_+$. 
The dashed lines represent the $d_0$-distances (here the Euclidean distances bounded by $1$) between the closest matchings.}
\label{f: d2metric}
\end{figure}

We establish a useful result for the $d_1$-distance between two point configurations differing only in one point, i.e. for $d_1(\xi + \delta_z, \xi + \delta_w)$, 
where $\xi \in M_p(E)$ and $z \neq w \in E$. To achieve this, note that 
\begin{align*}
\int_E \kappa(v)(\xi + \delta_z)(dv) - \int_E \kappa(v)(\xi + \delta_w)(dv) 
&= \int_E \kappa(v)\delta_z(dv) - \int_E \kappa(v) \delta_w(dv)\\
&= \kappa(z) - \kappa(w),
\end{align*}
and we therefore obtain, using $s_1(\kappa)^{-1} \le d_0(z,w)/|\kappa(z)-\kappa(w)|$ from Lipschitz continuity of $\kappa \in \mathcal{K}$, that
\begin{align}
\begin{split}\label{p: helpful_d1_bound}
&d_1(\xi + \delta_z, \xi + \delta_w) \\
&= \frac{1}{|\xi|+1} \, \sup_{\kappa \in \mathcal{K}} 
\frac{ d_0(z,w) \left|\int_E \kappa(v)(\xi + \delta_z)(dv) - \int_E \kappa(v)(\xi + \delta_w)(dv)\right|}{|\kappa(z)-\kappa(w)|}\\
& = \frac{1}{|\xi|+1} \, d_0(z,w).
\end{split}
\end{align}

We next construct a metric $d_2$ that is a Wasserstein metric induced by $d_1$ over probability measures on $M_p(E)$. 
Let $\mathcal{H}$ denote the set of functions $h : M_p(E) \to \mathbb{R}$ such that 
\begin{equation}\label{d: s_2h}
s_2(h) = \sup_{\xi_1 \neq \xi_2 \in M_p(E)} \frac{|h(\xi_1)-h(\xi_2)|}{d_1(\xi_1, \xi_2)} < \infty,
\end{equation}
i.e. each function $h \in \mathcal{H}$ is Lipschitz continuous with constant $s_2(h)$. We define a distance $d_2$ between probability measures 
$\mu$ and $\nu$ over $M_p(E)$ by
\begin{equation}\label{d: def_d2}
d_2(\mu,\nu)  = \frac{1}{s_2(h)}\, \sup_{h \in \mathcal{H}} \left| \int_{M_p(E)}h d\mu - \int_{M_p(E)} h d\nu \right|.
\end{equation}
Note that $d_2$ is bounded by $1$. By setting $\tilde{h}:= h/s_2(h)$ for each $h \in \mathcal{H}$, we may equivalently write
\[
d_2(\mu, \nu) = \sup_{\tilde{h} \in \widetilde{\mathcal{H}}}  \left| \int_{M_p(E)}\tilde{h} d\mu - \int_{M_p(E)} \tilde{h} d\nu \right|,
\]
where $\widetilde{\mathcal{H}} = \{ \tilde{h}: M_p(E) \to \mathbb{R}; \, |\tilde{h}(\xi_1) - \tilde{h}(\xi_2)| \le d_1(\xi_1, \xi_2) \le 1, \forall
\xi_1, \xi_2 \in M_p(E)\}$. By (\ref{d: equiv_def_dTV}), the test functions $\tilde{h}$ used to define the total variation distance between two probability measures 
$\mu$ and $\nu$ on $M_p(E)$ satisfy $|\tilde{h}(\xi_1) - \tilde{h}(\xi_2)|\le 1$, for any $\xi_1, \xi_2 \in M_p(E)$. The set of test functions used to define $d_2$
is thus contained in the set of test functions used to define the total variation distance. It follows that for any probability measures $\mu$ and $\nu$ on $M_p(E)$,
\begin{equation}\label{t: d2_smaller_dTV}
d_2(\mu,\nu) \le d_{TV}(\mu,\nu). 
\end{equation}

For a point process $\Xi$ on $E$ with intensity measure $\bl$, let $\mathcal{A}$ be the generator of an immigration-death process
with immigration intensity $\bl$, unit per-capita death rate, and equilibrium distribution $\mathrm{PRM}(\bl)$, i.e. 
let $\mathcal{A}$ be as defined in (\ref{d: generator}). Moreover, let $\upgamma$ be as defined in Proposition
\ref{t: upgamma_welldefined} for $h \in \mathcal{H}$. Then, by Proposition \ref{t: upgamma_solves_Stein}, 
\begin{align}\label{d: Stein_eq_d2}
\begin{split}
\left| \mathbb{E}(\mathcal{A}\upgamma)(\Xi)\right| &= \left| \mathbb{E}h(\Xi) - \mathrm{PRM}(\bl)(h)\right| \\
& = \left| \int_{M_p(E)} hd\mathcal{L}(\Xi) - \int_{M_p(E)} h d\mathrm{PRM}(\bl)\right|,
\end{split}
\end{align}
By (\ref{d: def_d2}), 
finding an upper bound on $d_2(\mathcal{L}(\Xi), \mathrm{PRM}(\bl))$ is thus equivalent to finding a uniform upper bound on 
$| \mathbb{E}(\mathcal{A}\upgamma)(\Xi)|/s_2(h)$. For the latter we need smoothness estimates of the solution $\upgamma$ 
of the Stein equation when $h \in \mathcal{H}$ (instead of $h$ being an indicator function as in Section \ref{Sec: Poi_proc_approx}). We determine
such estimates in Lemmas \ref{t: Delta_1_upgamma_d2} and \ref{t: Delta_2_upgamma_d2} below. In order to prove these, we first need the following 
two lemmas. Note that our proof of Lemma \ref{t: helping_bound_immdeath} below corrects a slight mistake in \cite{Barbour/Brown:1992} and \cite{Barbour_et_al.:1992}.
\begin{lemma}\label{l: Xia_sum}
Let $Z$ and $Z^0$ be immigration-death processes on $E$ with immigration intensity $\bl$ and unit per-capita death rate, where $\lambda = \bl(E) < \infty$,  
$Z$ has point configuration $\xi = \sum_{j=1}^{|\xi|} \delta_{z_j} \in M_p(E)$ at time $t=0$, and $Z^0$ has no initial particles. 
Let $T_1, T_2, \ldots, T_{|\xi|}$ be independent $\mathrm{Exp}(1)$-random variables, independent of $Z^0$. Then 
\[
Z_t \stackrel{d}{=} Z_t^0 + D_t, \quad \text{for all } t \in \mathbb{R}_+,
\]
where $D_t = \sum_{j=1}^{|\xi|} \delta_{z_j}I_{\{T_j>t\}}$ is a pure death process and $Z_t^0 \sim \mathrm{PRM}((1-e^{-t})\bl)$.
\end{lemma}
\begin{proof}
See Proposition 3.5 in \cite{Xia:2005}.
\end{proof}
\begin{lemma}\label{t: helping_bound_immdeath}
Let $Z = \{Z_t, t \in \mathbb{R}_+\}$ be an immigration-death process on $\mathbb{Z}_+$ with constant immigration rate $\lambda >0$ and 
unit per-capita death rate, with $k$ initial particles, i.e. $P(Z_0=k)=1$. Then
\[
\int_0^\infty e^{-t} \mathbb{E} \left[ (Z_t + 1)^{-1}\right]dt \le \left( \frac{1}{\lambda} + \frac{1}{k+1} \right)(1-e^{-\lambda}).
\]
\end{lemma}
\begin{proof}
The particles that are alive in the population at time $t$ can be grouped into two categories: those 
among the $k$ particles that were in the population from time $0$, and those that arrived later. 
By Lemma \ref{l: Xia_sum}, we may thus express the number $Z_t$ of particles in the 
population at time $t$ as 
the sum of independent random variables $X_t \sim \mathrm{Bin}(k, e^{-t})$ and $Y_t \sim \mathrm{Poi}(\lambda(1-e^{-t}))$.
We then have
\[
\mathbb{E} \left[ (Z_t + 1)^{-1}\right] \le \mathbb{E} \left[ (X_t + 1)^{-1}\right] \quad \textnormal{and } \quad 
\mathbb{E} \left[ (Z_t + 1)^{-1}\right] \le \mathbb{E} \left[ (Y_t + 1)^{-1}\right].
\]
On the one hand, setting $m=l+1$,
\begin{align*}
\mathbb{E} \left[ (X_t + 1)^{-1}\right] 
&=  \sum_{l=0}^k \frac{1}{l+1} \, \binom{k}{l} (e^{-t})^l (1-e^{-t})^{k-l}\\
&=  (1-e^{-t})^k \left(\frac{e^{-t}}{1-e^{-t}}\right)^{-1} \sum_{m=1}^{k+1} \frac{k!}{m! (k-m+1)!} \, \left( \frac{e^{-t}}{1-e^{-t}}\right)^m\\
&=(1-e^{-t})^{k+1} \, \frac{e^t}{k+1} \left\{  \sum_{m=0}^{k+1} \frac{(k+1)!}{m!(k+1-m)!} \, \left( \frac{e^{-t}}{1-e^{-t}}\right)^m - 1\right\}\\
&=(1-e^{-t})^{k+1} \, \frac{e^t}{k+1} \left\{  \left( 1+ \frac{e^{-t}}{1-e^{-t}}\right)^{k+1} - 1\right\}\\
& = \frac{e^t}{k+1} \left\{ 1 - (1-e^{-t})^{k+1}\right\}.
\end{align*}
On the other hand, setting $m=l+1$, 
\begin{align*}
\mathbb{E} \left[ (Y_t + 1)^{-1}\right]  = e^{-\lambda_t} \sum_{l=0}^\infty \frac{\lambda_t^l}{(l+1)!}
= \frac{e^{-\lambda_t}}{\lambda_t} \left( \sum_{m=0}^\infty \frac{\lambda_t^m}{m!} -1\right)
= \frac{1- e^{-\lambda (1-e^{-t})}}{\lambda (1-e^{-t})}.
\end{align*}
With $\tau$ such that $e^{-\tau} = \lambda/(\lambda + k+ 1)$, we now have 
\begin{align}\label{p: techn_lemma_step0}
\begin{split} 
&\int_0^\infty e^{-t} \mathbb{E} \left[ (Z_t + 1)^{-1} \right]dt \\
&\le \int_0^{\tau} e^{-t} \mathbb{E} \left[ (X_t + 1)^{-1} \right]dt  + \int_{\tau}^\infty e^{-t}\mathbb{E} \left[ (Y_t + 1)^{-1} \right]dt\\
& \le \int_0^{\tau}\frac{1}{k+1} \left\{ 1 - (1-e^{-t})^{k+1}\right\}dt + \frac{1-e^{-\lambda}}{\lambda}
\int_{\tau}^\infty \frac{e^{-t}}{1-e^{-t}}\, dt.
\end{split}
\end{align}
The second of these integrals equals 
\begin{equation}\label{p: techn_lemma_2nd_integral}
- \frac{1-e^{-\lambda}}{\lambda}\, \log\left(1-e^{-\tau}\right) = \frac{1-e^{-\lambda}}{\lambda} \, \log\left( 1 + \frac{\lambda}{k+1}\right) 
\le \frac{1-e^{-\lambda}}{k+1},
\end{equation}
since $-\log(1-z) \le z/(1-z)$ for $z = \lambda/(\lambda+k+1) <1$. Furthermore, due to 
\[
 1-(1-e^{-t})^{k+1} =e^{-t} \cdot \frac{1-(1-e^{-t})^{k+1}}{1-(1-e^{-t})} = e^{-t }\sum_{j=0}^k \left( 1-e^{-t}\right)^j,
\]
the first integral equals
\begin{equation}\label{p: techn_lemma_step}
\frac{1}{k+1} \, \sum_{j=0}^k \int_0^{\tau} e^{-t}\left( 1-e^{-t}\right)^j dt 
 = \frac{1}{k+1} \, \sum_{j=1}^{k+1} \frac{(1-e^{-\tau})^j}{j}. 
\end{equation}
By setting $x:= \lambda/(k+1)$, we may rewrite the expression 
in (\ref{p: techn_lemma_step}) as
$(x/\lambda)\cdot$ $\sum_{j=1}^{k+1}$ $[j(1+x)^j]^{-1}$,
which is smaller than
\begin{align}\label{p: techn_lemma_1st_integral}
\frac{x}{\lambda}\, \sum_{j=1}^{k+1}  \frac{1}{(1+x)^j} 
& = \frac{x}{\lambda} \, \left[ \sum_{j=0}^\infty \frac{1}{(1+x)^j} - 1 - \sum_{j=k+2}^\infty \frac{1}{(1+x)^j} \right]
\nonumber \\
& = \frac{x}{\lambda} \, \left[ \frac{1}{x} - \frac{(1+x)^{-k-2}}{1-(1+x)^{-1}} \right] \nonumber
= \frac{1}{\lambda}\, \left[ 1- \frac{1}{(1+x)^{k+1}} \right]\\
& = \frac{1}{\lambda} \, \left( 1- e^{-\frac{\lambda \log(1+x)}{x}} \right)
 \le \frac{1-e^{-\lambda}}{\lambda},
\end{align}
where we used $\log (1+x) \le x$ for the last inequality. 
By combining (\ref{p: techn_lemma_step0}), (\ref{p: techn_lemma_2nd_integral}) and (\ref{p: techn_lemma_1st_integral}), we obtain the lemma.
\end{proof}

\noindent The following two lemmas give smoothness estimates of the solution $\upgamma$ of the Stein equation when $h \in \mathcal{H}$. They are the counterparts of
(i) and (ii) of Lemma \ref{t: dtv_bounds_upgamma}, respectively, for the smaller class of Lipschitz continuous functions $\mathcal{H}$. 
\begin{lemma}\label{t: Delta_1_upgamma_d2}
Let $\bl$ be a  finite measure over $E$ with $\bl(E)= \lambda$. Let $\upgamma: M_p(E) \to \mathbb{R}$ be defined as in Proposition
\ref{t: upgamma_welldefined}, where $h: M_p(E) \to \mathbb{R}$ is any function in $\mathcal{H}$. Then, for any $\xi \in M_p(E)$,
\[
\Delta_1 \upgamma  \le s_2(h) \left(1 \wedge \frac{1.65}{\sqrt{\lambda}} \right). 
\] 
\end{lemma}
\begin{proof} From Proposition \ref{t: upgamma_welldefined}, we have that for any $\xi = \sum_{j \in J} \delta_{w_j}\in M_p(E)$, where $J \subseteq \mathbb{N}$, 
and for any $z \in E$,
\[
\upgamma(\xi + \delta_z) - \upgamma(\xi) = \int_0^\infty \left\{ \mathbb{E}^\xi h(Z_t) - \mathbb{E}^{\xi + \delta_z} h(Z_t) \right\}dt, 
\]
where $Z=\{Z_t, t \in \mathbb{R}_+\}$ is the immigration-death process on $E$ with immigration intensity $\bl$ and unit per-capita death rate. Let $Z$ be realised 
under $\mathbb{P}^\xi$ and let $T$ be an exponential random variable with parameter $1$, independent of $Z$. Then the process $Z'$ defined by 
$Z_t' = Z_t + \delta_z I_{\{T > t\}}$ has distribution $\mathbb{P}^{\xi + \delta_z}$. Moreover, let $Z^0$ be realised under $\mathbb{P}^0$ and let $D$ be a pure death
process with unit per-capita death rate starting with $D_0=\xi.$ Then, $Z_t = Z_t^0 + D_t$ by Lemma \ref{l: Xia_sum}, and 
\begin{align}
\begin{split}\label{p: diff_upgamma_eq}
&\upgamma(\xi + \delta_z) - \upgamma(\xi) \\
&= \int_0^\infty e^{-t} \,\mathbb{E}^\xi \left[ h\left(Z_t^0 + D_t\right) - h\left(Z_t^0 + D_t + \delta_z\right) \right]dt \\
&= \int_0^\infty e^{-t} \sum_{\eta \in \mathcal{N}} \mathbb{E}\left[ h\left(Z_t^0 + \eta\right) - h\left(Z_t^0 + \eta + \delta_z\right) \right]P\left( D_t = \eta \right)dt,
\end{split}
\end{align}
where $\mathcal{N} = \{\sum_{j \in J'} \delta_{w_j}\, ;\, J' \subseteq J\}$. We first show that $\Delta_1 \upgamma \le s_2(h)$. 
Using Lipschitz continuity of $h$, as well as the fact that the
$d_1$-distance between point configurations of different sizes is $1$ (see (\ref{d: s_2h}) and (\ref{d: def_d1}), respectively), we find that
\begin{align*}
\left| \upgamma(\xi + \delta_z) - \upgamma(\xi) \right| 
 \le \int_0^\infty e^{-t}\, \sum_{\eta \in \mathcal{N}} P(D_t = \eta) dt
\le s_2(h) \int_0^\infty e^{-t} dt = s_2(h). 
\end{align*}
In order to show that $\Delta_1 \upgamma \le s_2(h)(1.65/\sqrt{\lambda})$, note first that
\begin{align*}
&\mathbb{E}\left[ h\left(Z_t^0 + \eta\right) - h\left( Z_t^0 + \eta + \delta_z \right) \right]\\
& =  \sum_{k \ge 0} P\left( \left| Z_t^0\right| = k \right) 
\mathbb{E} \left[ h\left( Z_t^0 +\eta \right) - h\left( Z_t^0 + \eta + \delta_z \right) \left|\right.  \left|Z_t^0\right|=k\right] \\
& = P\left( \left| Z_t^0\right| = 0 \right)h(\eta) 
+ \sum_{k\ge0} \left\{ P\left( \left| Z_t^0\right| = k+1 \right)\mathbb{E} \left[ h\left( Z_t^0 +\eta \right) \left|\right.  \left|Z_t^0\right|=k+1\right] \right.\\
&\left.\phantom{ P\left( \left| Z_t^0\right| \right) P\left( \left| Z_t^0\right| = 0 \right)h(\eta) 
+} - P\left( \left| Z_t^0\right| = k \right)\mathbb{E} \left[ h\left( Z_t^0 +\eta+ \delta_z\right) \left|\right.  \left|Z_t^0\right|=k\right]\right\}.
\end{align*}
For the part in curly brackets we use $|a_1a_2 - b_1 b_2| \le a_2|a_1-b_1| + b_1|a_2 - b_2|$,
where 
\begin{align*}
a_1 &:=  P\left( \left| Z_t^0\right| = k+1 \right),\quad  &a_2 :=\,\,& \mathbb{E} \left[ h\left( Z_t^0 +\eta \right) \left|\right.  \left|Z_t^0\right|=k+1\right],\\
b_1 &:=  P\left( \left| Z_t^0\right| = k \right),\quad    &b_2 :=\,\,& \mathbb{E} \left[ h\left( Z_t^0 +\eta+ \delta_z\right) \left|\right.  \left|Z_t^0\right|=k\right].
\end{align*}
We have
\begin{align}
\begin{split}\label{p: Delta_2_d2_bound_number_of_part}
|a_2-b_2| &= \frac{1}{\lambda} \left|\int_E \mathbb{E}[h(Z_t^0 + \eta + \delta_w)- h\left( Z_t^0 +\eta+ \delta_z\right)  \left|\right. \left|Z_t^0\right|=k]\bl(dw) \right|\\
&\le \frac{1}{\lambda} \int_E \mathbb{E}\left[s_2(h) d_1(Z_t^0 + \eta + \delta_w , Z_t^0 +\eta+ \delta_z)\left|\right. \left|Z_t^0\right|=k\right]\bl(dw)\\
& \le  \frac{s_2(h)}{\lambda}   \int_E \mathbb{E}\left[(|Z_t^0| + |\eta| + 1)^{-1}d_0(w,z)\left|\right. \left|Z_t^0\right|=k\right]\bl(dw)\\
& \le  s_2(h) (k + |\eta| + 1)^{-1},
\end{split}
\end{align}
where we used used Lipschitz continuity of $h$, (\ref{p: helpful_d1_bound}), and boundedness of $d_0$ by $1$ for the first, second and third inequalities, respectively.
It follows that
\[
\sum_{k\ge0} b_1|a_2-b_2| \le s_2(h) \sum_{k \ge 0 }\frac{P(|Z_t^0|=k)}{k+|\eta| + 1}  
\le s_2(h) \mathbb{E}\left[ (|Z_t^0| + 1)^{-1}\right].
\]
Furthermore, note that the use of the function $ h-(\inf_\xi h + \sup_\xi h)/2$ instead of $h$ leaves (\ref{p: diff_upgamma_eq}) unchanged. Therefore, we may use 
$\sup_\xi |h|=s_2(h)/2$, which entails the following two bounds:
\begin{align}\label{p: Delta_2_d2_E}
\begin{split}
& a_2 \le \left| \mathbb{E}\left[ h(Z_t^0 + \eta) \left|\right. \left| Z_t^0 \right|=k\right]\right| \le \frac{s_2(h)}{2}\, ,\\
& P\left( \left| Z_t^0\right|=0\right)h(\eta) \le \frac{s_2(h)}{2}\,P\left( \left| Z_t^0\right|=0\right). 
\end{split}
\end{align}
We obtain 
\begin{align}\label{p: Delta_1_d2_bd1}
\begin{split}
&\left| \mathbb{E}\left[ h\left(Z_t^0 + \eta\right) - h\left( Z_t^0 + \eta + \delta_z \right) \right] \right|\\
\le\, \,& s_2(h) \left\{\frac{1}{2}\,P\left( \left|Z_t^0\right|=0\right) 
+ \frac{1}{2}\sum_{k \ge 0 } \left| P(|Z_t^0| = k+1) - P(|Z_t^0| = k) \right| \right\}\\
&+ s_2(h)\mathbb{E}\left[ (|Z_t^0|  + 1)^{-1}\right].\\
\end{split}
\end{align}
By Lemma \ref{l: Xia_sum}, we have $|Z_t^0| \sim \mathrm{Poi}(\lambda_t)$ with $\lambda_t := \lambda(1-e^{-t})$. Thus, 
\begin{equation}\label{p: Delta_1_d2_bd1b}
 \mathbb{E}\left[ (|Z_t^0| + 1)^{-1}\right]  = \frac{e^{-\lambda_t}}{\lambda_t} \sum_{k \ge 0} \frac{\lambda_t^{k+1}}{(k+1)!} 
= \frac{1-e^{-\lambda_t}}{\lambda_t}\,.  
\end{equation}
Furthermore, note that 
\[
P(|Z_t^0|=k+1) - P(|Z_t^0|=k) =  \frac{e^{-\lambda_t}\lambda_t^k}{k!} \left(\frac{\lambda_t}{k+1} -1\right), 
\]
and that, if $k < \lambda_t - 1$, then $P(|Z_t^0|=k+1) > P(|Z_t^0|=k)$, and else, if $k > \lambda_t - 1$, then $P(|Z_t^0|=k+1) < P(|Z_t^0| =k)$. Thus,
\begin{multline*}
 \frac{1}{2}\, P(|Z_t^0| =0) + \frac{1}{2}\,\sum_{k \ge 0 } \left| P(|Z_t^0| = k+1) - P(|Z_t^0| = k) \right| \\
=\,\, \frac{1}{2} P(|Z_t^0| =0) + \frac{1}{2} \sum_{k=0}^{\lfloor \lambda_t - 1 \rfloor} \left[  P(|Z_t^0| = k+1) - P(|Z_t^0| = k)\right]\\
 + \frac{1}{2} \sum_{k = \lceil \lambda_t - 1 \rceil}^\infty \left[ P(|Z_t^0| = k) - P(|Z_t^0| = k+1)\right]
\end{multline*}
equals 
\begin{align}\label{p: Delta_1_d2_bd1a}
\begin{split}
\frac{1}{2} \left\{ P(|Z_t^0| = \lfloor \lambda_t - 1\rfloor + 1) + P(|Z_t^0| = \lceil \lambda_t - 1 \rceil) \right\} 
&\le \max_{k \ge 0} P(|Z_t^0| = k) \\  
&\le \frac{1}{\sqrt{2 e \lambda_t}}\,,
\end{split}
\end{align}
where the last inequality is due to Proposition A.2.7 in \cite{Barbour_et_al.:1992}. 
In addition to the estimate (\ref{p: Delta_1_d2_bd1}) with (\ref{p: Delta_1_d2_bd1b}) and (\ref{p: Delta_1_d2_bd1a}), 
we get the following more direct estimate from Lipschitz continuity of $h$:
\begin{equation}\label{p: Delta_1_d2_bd2}
\left| \mathbb{E}\left[ h\left(Z_t^0 + \eta\right) - h\left( Z_t^0 + \eta + \delta_z \right) \right] \right| \le s_2(h). 
\end{equation} 
Choose $\tau$ such that $e^{-\tau}= 1-\lambda^{-1}$.
With (\ref{p: Delta_1_d2_bd1})-(\ref{p: Delta_1_d2_bd2}), 
the following then holds for any $\eta \in \mathcal{N}$:
\begin{align*}
& \frac{1}{s_2(h)} \int_0^\infty e^{-t} \left| \mathbb{E}\left[ h\left(Z_t^0 + \eta\right) - h\left( Z_t^0 + \eta + \delta_z \right) \right] \right|dt \\
& \le \int_0^\tau e^{-t} dt + \int_\tau^\infty e^{-t}\left( \frac{1}{\sqrt{2e\lambda_t}} +\frac{1-e^{-\lambda_t}}{\lambda_t} \right)dt\\
& \le \int_0^\tau e^{-t} dt + \frac{1}{\sqrt{2e\lambda}}\int_\tau^\infty \frac{e^{-t}}{\sqrt{1-e^{-t}}}\,dt + \frac{1}{\lambda}\int_\tau^\infty \frac{e^{-t}}{1-e^{-t}}\,dt \\
& = 1-e^{-\tau} +\frac{2-2\sqrt{1-e^{-\tau}}}{\sqrt{2e\lambda}} - \frac{\log(1-e^{-\tau})}{\lambda}
= \frac{1}{\lambda}+ \sqrt{\frac{2}{e\lambda}} -\frac{1}{\lambda} \sqrt{\frac{2}{e}} + \frac{\log \lambda}{\lambda}
\\
& \le \frac{1}{\sqrt{\lambda}}\left( \frac{0.14223 +\log \lambda}{\sqrt{\lambda}} + 0.86  \right)\le \frac{1}{\sqrt{\lambda}}\, (0.79+0.86) = \frac{1.65}{\sqrt{\lambda}}\,.
\end{align*}
Thus, 
\begin{align*}
\Delta_1 \upgamma
& \le\sum_{\eta \in \mathcal{N}} P(D_t=\eta) \int_0^\infty e^{-t}\left| \mathbb{E}\left[ h\left(Z_t^0 + \eta\right) - h\left( Z_t^0 + \eta + \delta_z \right) \right] \right|dt\\
& \le \frac{1.65 s_2(h)}{\sqrt{\lambda}} \sum_{\eta \in \mathcal{N}} P(D_t=\eta) \le \frac{1.65 s_2(h)}{\sqrt{\lambda}}\, .
\end{align*}
\end{proof}
\begin{lemma}\label{t: Delta_2_upgamma_d2}
Under the conditions of Lemma \ref{t: Delta_1_upgamma_d2},
\[ 
\Delta_2 \upgamma \le  s_2(h)\left\{ 1 \wedge \frac{2}{\lambda} \left( 1 + 2 \log_+ \left( \frac{\lambda}{2}\right)\right) \right\}.  
\]
\end{lemma}
\begin{proof} As in the proof of Lemma \ref{t: dtv_bounds_upgamma} (ii), we may write, for any $\xi$ $ =$ $ \sum_{j \in J} \delta_{w_j}$ $\in$ $ M_p(E)$, where $J \subseteq \mathbb{N}$, 
and for any $z,w \in E$,
\begin{align*}
&\upgamma(\xi + \delta_z + \delta_w) - \upgamma(\xi+ \delta_z) - \upgamma(\xi + \delta_w) + \upgamma(\xi) \\
&= -\int_0^\infty \mathbb{E}^\xi \left[ h(Z_t + \delta_z + \delta_w) - h(Z_t + \delta_z) - h(Z_t + \delta_w) + h(Z_t)\right] e^{-2t}dt, 
\end{align*}
where $Z= \{Z_t, t \in \mathbb{R}_+\}$ is an immigration-death process on $E$ with immigration intensity $\bl$ and unit per-capita death rate realised under 
$\mathbb{P}^\xi$. Let $Z^0$ be realised under $\mathbb{P}^0$ and let $D$ be a pure death process with unit per-capita death rate starting with $D_0 = \xi$. Then
$Z_t = Z_t^0 + D_t$ by Lemma \ref{l: Xia_sum}, and 
\begin{align}
\begin{split} \label{p: Delta_2_d2_expression}
&\upgamma(\xi + \delta_z + \delta_w) - \upgamma(\xi+ \delta_z) - \upgamma(\xi + \delta_w) + \upgamma(\xi) \\ 
& = - \int_0^\infty e^{-2t}\sum_{\eta \in \mathcal{N}} 
\mathbb{E} \left[h(Z_t^0 + \eta + \delta_z + \delta_w)  - h(Z_t^0 + \eta + \delta_z) \right. \\
&\left.\phantom{blaaaaaaaaaaaaaaaaaaaaaa} 
 - h(Z_t^0 + \eta + \delta_w) + h(Z_t^0 + \eta)\right]P(D_t = \eta)dt,
\end{split}
\end{align}
where $\mathcal{N} = \{\sum_{j \in J'} \delta_{w_j}\, ;\, J' \subseteq J\}$. We first show that $\Delta_2 \upgamma \le s_2(h)$. For any $\eta \in \mathcal{N}$, 
it follows from Lipschitz continuity of $h$ that
\begin{align*}
&\left|\mathbb{E}\left[ h(Z_t^0 + \eta + \delta_z + \delta_w)  - h(Z_t^0 + \eta + \delta_z)- h(Z_t^0 + \eta + \delta_w) + h(Z_t^0 + \eta)\right] \right|\\
& \le \mathbb{E} \left|h(Z_t^0 + \eta + \delta_z + \delta_w)  - h(Z_t^0 + \eta + \delta_z)\right| + \mathbb{E} \left|h(Z_t^0 + \eta + \delta_w) - h(Z_t^0 + \eta) \right|\\
&\le 2s_2(h).
\end{align*}
Then,
\begin{align*}
\Delta_2 \upgamma \le 2s_2(h) \int_0^\infty e^{-2t} \sum_{\eta \in \mathcal{N}} P(D_t = \eta)dt 
\le 2s_2(h) \int_0^\infty e^{-2t}dt = s_2(h). 
\end{align*}
For the $\lambda$-dependent bound, note that the expectation in (\ref{p: Delta_2_d2_expression}) may be rewritten as
\begin{align*}
&\sum_{k \ge -1} \left\{ \right.  P(|Z_t^0|=k) \mathbb{E}\left[ h(Z_t^0 + \eta + \delta_z + \delta_w) \left|\right. \left| Z_t^0 \right|=k\right] \\
& \phantom{blaaaa} \left. - P(|Z_t^0|=k+1) \mathbb{E}\left[ h(Z_t^0 + \eta + \delta_z) + h(Z_t^0 + \eta + \delta_w) \left|\right. \left| Z_t^0 \right|=k+1\right] \right.\\
& \phantom{blaaaa} \left. + P(|Z_t^0|=k+2) \mathbb{E}\left[ h(Z_t^0 + \eta ) \left|\right. \left| Z_t^0 \right|=k+2\right] \right\} \\
&+ P(|Z_t^0|=0)h(\eta). 
\end{align*}
We add and subtract both $P(|Z_t^0|=k) \mathbb{E}[h(Z_t^0 + \eta + \delta_z) \left| \right. |Z_t^0|=k+1]$ and
$P(|Z_t^0|=k+2)\mathbb{E}[h(Z_t^0 + \eta + \delta_w) \left| \right. |Z_t^0|= k+1]$ to the part in curly brackets. 
Then, using \ref{p: Delta_2_d2_bound_number_of_part}) and (\ref{p: Delta_2_d2_E}), we find that
\begin{align*}
&\left| \mathbb{E} \left[h(Z_t^0 + \eta + \delta_z + \delta_w)  - h(Z_t^0 + \eta + \delta_z)- h(Z_t^0 + \eta + \delta_w) + h(Z_t^0 + \eta)\right]\right| \\
\le &\sum_{k \ge -1} s_2(h)\left\{ P(|Z_t^0|=k) + P(|Z_t^0|=k+2)\right\}/(k+2 + |\eta|) \\
&+ \left| \sum_{k \ge -1} \mathbb{E}\left[\frac{1}{2}\, h(Z_t^0 + \eta + \delta_z) + \frac{1}{2}\, h(Z_t^0 + \eta + \delta_w) \left|\right. |Z_t^0|= k+1 \right] \right.\\
& \left. \phantom{bl \left[\frac{s_2(h)}{2}\right]}  \cdot \left\{ P(|Z_t^0|=k) - 2 P(|Z_t^0|=k+1) + P(|Z_t^0|= k+2)\right\} \right| 
\\ &+ P(|Z_t^0|=0)|h(\eta)| \\
\le &\,\,s_2(h)\sum_{k \ge -1}\left\{ P(|Z_t^0|=k) + P(|Z_t^0|=k+2)\right\}/(k+2 + |\eta|) \\
&+ \frac{s_2(h)}{2}\left\{\sum_{k \ge -1} \left| P(|Z_t^0|=k) - 2 P(|Z_t^0|=k+1) 
+ P(|Z_t^0|= k+2)\right| \right.\\
&\left.\phantom{\frac{s_2(h)}{2}blaaaa}+P(|Z_t^0|=0)\right\}.
\end{align*}
Note that 
\begin{align}
&\sum_{k \ge -1} \frac{P(|Z_t^0|=k) + P(|Z_t^0|=k+2) }{k+ |\eta|+2} \nonumber \\
& \le \sum_{k \ge 0} \frac{P(|Z_t^0|=k)}{k+2} + \sum_{k \ge 1} \frac{P(|Z_t^0|=k)}{k} \nonumber\\
& = \mathbb{E}\left[\left(|Z_t^0| + 2\right)^{-1}\right] + \mathbb{E}\left[\left(|Z_t^0|\right)^{-1}I_{\{|Z_t^0|\ge 1\}} \right] \nonumber \\
&\le 3\mathbb{E}\left[\left(|Z_t^0| + 1\right)^{-1}\right].\label{p: Delta_2_d2_3E}
\end{align}
Moreover, for $X \sim \mathrm{Poi}(\nu)$, 
\[
P(X=k)-2P(X=k-1)+P(X=k-2) = P(X=k)\left\{ \left(1-\nu^{-1}k \right)^2 - \nu^{-2}k\right\}, 
\]
for all $k \ge 0$. 
Since $|Z_t^0| \sim \mathrm{Poi}(\lambda_t)$ with $\lambda_t := \lambda(1-e^{-t})$ by Lemma \ref{l: Xia_sum}, we have
\begin{align}
&\sum_{k \ge -1}  \left| P(|Z_t^0|=k) -2P(|Z_t^0|=k+1)
+ P(|Z_t^0|=k+2)\right| + P(|Z_t^0|=0)\nonumber\\
&= \sum_{k \ge 0} \left| P(|Z_t^0|=k-2)
-2P(|Z_t^0|=k-1) + P(|Z_t^0|=k)\right| \nonumber\\
&= \sum_{k \ge 0} P(|Z_t^0|=k) \left|\left(1-\lambda_t^{-1}|Z_t^0| \right)^2
- \lambda_t^{-2}|Z_t^0|\right| \nonumber\\
&\le \mathbb{E}\left[\left(1-\lambda_t^{-1}|Z_t^0| \right)^2\right] 
+\mathbb{E}\left[  \lambda_t^{-2}|Z_t^0|\right] = \frac{2}{\lambda_t}\,.
\label{p: Delta_2_d2_sumPoi}
\end{align}
It follows from (\ref{p: Delta_2_d2_3E}), (\ref{p: Delta_2_d2_sumPoi}) and (\ref{p: Delta_1_d2_bd1b}) that
\begin{align*}
&\left|\mathbb{E}\left[ h(Z_t^0 + \eta + \delta_z + \delta_w)  - h(Z_t^0 + \eta + \delta_z)- h(Z_t^0 + \eta + \delta_w) + h(Z_t^0 + \eta)\right] \right|\\
& \le s_2(h) \left\{ \frac{1}{\lambda_t} + \frac{3(1-e^{-\lambda_t})}{\lambda_t}\right\} \le \frac{4s_2(h)}{\lambda_t}\,.
\end{align*}
A more direct bound is given by $ 4 \sup_\xi |h(\xi)| = 2s_2(h)$. With these two estimates, we find, for any $\eta \in \mathcal{N}$, and for $\tau$ chosen 
such that $e^{-\tau} = 1-2\lambda^{-1}$,
\begin{align*}
& \int_0^\infty \frac{e^{-2t}}{s_2(h)} 
\left|\mathbb{E}\left[ h(Z_t^0 + \eta + \delta_z + \delta_w)  
- h(Z_t^0 + \eta + \delta_z)
- h(Z_t^0 + \eta + \delta_w) \right. \right.\\
&\left. \left. \phantom{blaaaaaaaaaa}+ h(Z_t^0 + \eta)\right] \right|dt\\
& = \int_0^\tau 2e^{-2t}dt + \frac{4}{\lambda}\int_\tau^\infty \frac{e^{-2t}}{1-e^{-t}}\,dt
= 1- e^{-2\tau} - \frac{4}{\lambda}\,e^{-\tau} - \frac{4}{\lambda} \log(1- e^{-\tau})\\
&= 1- \left( 1- \frac{2}{\lambda}\right)^2 - \frac{4}{\lambda}\left(1- \frac{2}{\lambda} \right) - \frac{4}{\lambda}\log\left(\frac{2}{\lambda}\right)
= \frac{4}{\lambda^2} + \frac{4}{\lambda} \log\left( \frac{\lambda}{2}\right).
\end{align*}
Therefore,
\begin{align*}
\Delta_2 \upgamma  &\le s_2(h) \left\{ \frac{4}{\lambda^2} + \frac{4}{\lambda} \log\left( \frac{\lambda}{2}\right) \right\} \sum_{\eta \in \mathcal{N}}P(D_t=\eta)
\le  s_2(h) \left\{ \frac{4}{\lambda^2} + \frac{4}{\lambda} \log\left( \frac{\lambda}{2}\right) \right\} \\
&\le \frac{2}{\lambda}\left\{ 1+ 2\log_+\left( \frac{\lambda}{2}\right) \right\}s_2(h), \quad \text{for all } \lambda \ge 2. 
\end{align*}
\end{proof}
\noindent Since the class of functions $h$ is smaller than the class of functions considered for approximation in the total variation distance, 
the smoothness estimates are better: they have the desired property of decreasing with increasing $\lambda$. 
With the above lemmas, we are in a position to prove an analogue of Theorem \ref{t: dTV_PP_generator} in the weaker $d_2$-metric.
\begin{thm}\label{t: d2_PP_generator}
Suppose there exists a fixed measure $\boldsymbol{\nu}$ on $E$ and suppose that $\Xi$ is a finite simple point process on $E$ with finite mean measure $\bl$
and Janossy densities $\{j_m\}_{m \ge 0}$. Suppose the density $\mu$ of $\bl$ 
with respect to $\boldsymbol{\nu}$ is given by (\ref{d: density_mean_measure}). Let $\{N_z\}_{z \in E}$
be a neighbourhood structure satisfying (\ref{d: neighbourhood_structure}). 
Then,
\begin{align*}
&d_2\left(\mathcal{L}(\Xi), \mathrm{PRM}(\bl)\right) \\
\le &\left\{ 1 \wedge \frac{2}{\lambda} \left( 1 + 2 \log_+ \left( \frac{\lambda}{2}\right)\right) \right\}
\left( \int_E \mathbb{E} \Xi(N_z) \mu(z)\boldsymbol{\nu}(dz) \right. \\ 
&\left. +\,\,\mathbb{E} \left[ \int_E \Xi(N_z\setminus \{z\}) \Xi(dz)\right] \right)
  + \left\{1 \wedge \frac{1.65}{\sqrt{\lambda}} \right\}\int_E \mathbb{E} \left| g(z, \Xi^z) - \mu(z)\right| \boldsymbol{\nu}(dz),
\end{align*}
where $\lambda = \bl(E)$, and the conditional density $g(z, \Xi^z)$  at $z$ given the configuration $\Xi^z$ of $\Xi$ outside $N_z$ is defined in (\ref{d: conditional_density_Janossy}).
\end{thm}
\begin{proof}
Let $\upgamma$ be defined as in Proposition \ref{t: upgamma_welldefined} for a function $h \in \mathcal{H}$. By (\ref{d: def_d2}) and (\ref{d: Stein_eq_d2}),
it is sufficient to determine an upper bound on $|\mathbb{E}(\mathcal{A}\upgamma)(\Xi)|/s_2(h)$, where $\mathcal{A}$ is defined as in (\ref{d: generator}).
From the proof of Theorem \ref{t: dTV_PP_generator}, we have that
\begin{align*}
\left| \mathbb{E} (\mathcal{A}\upgamma)(\Xi) \right| & \le \Delta_2  \upgamma \left( \int_E \mathbb{E} \Xi(N_z) \mu(z)\boldsymbol{\nu}(dz) 
+\mathbb{E} \left[ \int_E \Xi(N_z\setminus \{z\}) \Xi(dz)\right] \right)\\
& + \Delta_1 \upgamma \int_E \mathbb{E} \left| g(z, \Xi^z) - \mu(z)\right| \boldsymbol{\nu}(dz).
\end{align*}
The estimates from Lemmas \ref{t: Delta_1_upgamma_d2} and \ref{t: Delta_2_upgamma_d2} for $\Delta_1 \upgamma$ and $\Delta_2 \upgamma$,
respectively, then immediately give the error bound for $d_2\left(\mathcal{L}(\Xi), \mathrm{PRM}(\bl)\right)$.
\end{proof}

In case we want to approximate the law of $\Xi$ by a Poisson process with intensity measure $\tilde{\bl} \neq\bl$, we 
have to add an estimate for $d_2(\mathrm{PRM}(\bl), \mathrm{PRM}(\tilde{\bl}))$  to the error given by Theorem \ref{t: d2_PP_generator}.
To determine such an error estimate we first need Lemma \ref{t: helping_bound_s2}.
\begin{lemma}\label{t: helping_bound_s2}
Let $\bl$ and $\tilde{\bl}$ be two finite measures over $E$ such that $\bl(E)= \tilde{\bl}(E) = \lambda$. Let $\upgamma: M_p(E) \to \mathbb{R}$ be defined as in Proposition
\ref{t: upgamma_welldefined}, where $h: M_p(E) \to \mathbb{R}$ is any function in $\mathcal{H}$. Then, for any $\xi \in M_p(E)$,
\begin{multline*}
\left| \int_E \left[ \upgamma(\xi + \delta_z) - \upgamma(\xi) \right] (\bl(dz) - \tilde{\bl}(dz))\right| \\
\le s_2(h) (1-e^{-\lambda}) \left( 1+ \frac{\lambda}{|\xi| + 1}\right)d_1(\bl, \tilde{\bl}). 
\end{multline*}
\end{lemma}
\begin{proof} 
For any $\xi \in M_p(E)$, define $\upgamma_\xi: E \to \mathbb{R}$ by $\upgamma_\xi(z) = \upgamma(\xi + \delta_z) - \upgamma(\xi)$. 
From the definition of $d_1$ in (\ref{d: def_d1}),
\begin{align}\label{p: helping_bound_s2_first}
\begin{split}  
\left| \int_E \left[ \upgamma(\xi + \delta_z) - \upgamma(\xi)\right] (\bl(dz)-\tilde{\bl}(dz))\right|
&= \left| \int_E \upgamma_\xi d\bl - \int_E \upgamma_\xi d\tilde{\bl} \right|\\
&\le s_1(\upgamma_\xi) \lambda \, d_1(\bl, \tilde{\bl}).
\end{split}
\end{align}
In order to determine an upper bound on $s_1(\upgamma_\xi)$, that is, 
on $|\upgamma_\xi(z) - \upgamma_\xi (w)|/$ $|d_0(z,w)|$ for any choice of $z\neq w \in E$,
let $Z=\{Z_t, t \in \mathbb{R}_+\}$ be an immigration-death process on $E$ with initial point configuration $\xi$, i.e. realised under $\mathbb{P}^\xi$. 
Let $T$ be an exponential random variable with rate $1$ and independent of $Z$. The processes defined by
$Z_{t}^z = Z_t + \delta_z I_{\{T > t\}}$ and $Z_{t}^w = Z_t + \delta_w I_{\{T > t\}}$ then have distributions $\mathbb{P}^{\xi+ \delta_z}$ and 
$\mathbb{P}^{\xi + \delta_w}$, respectively, and, for any $z \neq w \in E$, $\left| \upgamma_\xi(z) - \upgamma_\xi(w)\right|$ equals
\begin{align}
\begin{split}\label{p: helping_bound_s2}
\left| \upgamma(\xi + \delta_z) - \upgamma(\xi +\delta_w)\right| 
&= \left| \int_0^\infty \mathbb{E}^\xi \left[ \left\{ h(Z_t^z) - h(Z_t^w)\right\} I_{\{T > t\}} \right]dt \right| \\
& \le  \int_0^\infty e^{-t} \mathbb{E}^\xi \left| h(Z_t + \delta_z) - h(Z_t + \delta_w)  \right|dt\\
& \le s_2(h) \int_0^\infty e^{-t} \mathbb{E}^\xi \left[ d_1\left(Z_t + \delta_z, Z_t + \delta_w\right) \right]dt,
\end{split}
\end{align}
where we used Lipschitz continuity of $h$ in the last inequality. Now, note that $|Z_t + \delta_z| = |Z_t + \delta_w| = |Z_t|+1$, and $|Z_t|$ is 
an immigration-death process on $\mathbb{Z}_+$ with initial number of points $|\xi|$. 
By (\ref{p: helpful_d1_bound}) we thus have
$\mathbb{E}^\xi[ d_1\left(Z_t + \delta_z, Z_t + \delta_w\right)] $
$=$  $d_0(z,w)\mathbb{E}^\xi \left[ (|Z_t|+1)^{-1}\right]$ for any $t \in \mathbb{R}_+$.
With Lemma \ref{t: helping_bound_immdeath}, (\ref{p: helping_bound_s2}) then gives 
\[
\left| \upgamma_\xi(z) - \upgamma_\xi(w)\right| 
\le s_2(h) d_0(z,w) (1-e^{-\lambda}) \left( \frac{1}{\lambda} + \frac{1}{|\xi| + 1}\right),
\]
for any $z \neq w \in E$, and therefore,
\[
s_1(\upgamma_\xi) \le s_2(h)(1-e^{-\lambda}) \left( \frac{1}{\lambda} + \frac{1}{|\xi| + 1}\right). 
\]
Use of this bound for $s_1(\upgamma_\xi)$ in (\ref{p: helping_bound_s2_first}) completes the proof. 
\end{proof}
\noindent With the above lemma, we are in shape to determine an error estimate for the $d_2$-distance between two Poisson processes with different mean measures $\bl$ 
and $\tilde{\bl}$.
\begin{prop}\label{t: d2_two_PRM}
Let $\bl$ and $\tilde{\bl}$ be two finite measures over $E$ such that $\bl(E)= \tilde{\bl}(E) = \lambda$. Then
\[
d_2\left(\mathrm{PRM}(\bl), \mathrm{PRM}(\tilde{\bl})\right) \le (1-e^{-\lambda})(2-e^{-\lambda}) d_1(\bl, \tilde{\bl}). 
\]
\end{prop}
\begin{proof} 
Let $\Xi:= \Xi_{\tilde{\bl}} \sim \mathrm{PRM}(\tilde{\bl})$ and let $\Xi_{\bl} \sim \mathrm{PRM}(\bl)$. 
Let $\mathcal{A}$ be the generator of an immigration-death process with immigration intensity $\bl$, unit per-capita death rate, and equilibrium distribution
$\mathcal{L}(\Xi_{\bl})$. By (\ref{d: Stein_eq_d2}), $|\mathbb{E}h(\Xi) - \mathrm{PRM}(\bl)(h)|$ $= |\mathrm{PRM}(\tilde{\bl})(h) - \mathrm{PRM}(\bl)(h)|$ equals
$|\mathbb{E}(\mathcal{A} \upgamma)(\Xi)|$. From the proof of Proposition \ref{t: dTV_two_PRM} we know that
\[
\mathbb{E}(\mathcal{A}\upgamma)(\Xi) = \mathbb{E}\int_E \left[ \upgamma(\Xi + \delta_z) - \upgamma(\Xi) \right] (\bl(dz)-\tilde{\bl}(dz)), 
\]
and thus 
\begin{align*}
\frac{|\mathbb{E}(\mathcal{A}\upgamma)(\Xi)|}{s_2(h)} 
&\le \frac{1}{s_2(h)}\, \mathbb{E} \left|\int_E \left[ \upgamma(\Xi + \delta_z) - \upgamma(\Xi) \right] (\bl(dz)-\tilde{\bl}(dz)) \right| \\
& \le (1-e^{-\lambda})\left(1+ \lambda \mathbb{E} \left[ (|\Xi|+1)^{-1}\right]\right) d_1(\bl, \tilde{\bl}),
\end{align*}
where we used Lemma \ref{t: helping_bound_s2} for the second inequality. Finally, since $|\Xi| \sim \mathrm{Poi}(\lambda)$, we have
\[
\mathbb{E} \left[ (|\Xi|+1)^{-1}\right] 
= \sum_{k \ge 0} \frac{ P(|\Xi|=k) }{k+1} 
= e^{-\lambda}\sum_{k \ge 0} \frac{\lambda^k}{(k+1)!} 
= \frac{e^{-\lambda} }{\lambda} \sum_{k \ge 1} \frac{\lambda^k}{k!} = 
\frac{1-e^{-\lambda}}{\lambda},
\]
and thus 
\[
1+ \lambda \mathbb{E} \left[ (|\Xi|+1)^{-1}\right] = 2-e^{-\lambda}.
\]

\end{proof}

\chapter{Poisson and Poisson process approximation for univariate extremes}\label{Chap: Univariate_extremes}
\setcounter{thm}{0}
The tools that we established in Chapter \ref{Chap: Stein-Chen} by way of the Stein-Chen method are now applied to problems from extreme value theory, where we  
restrict ourselves, for simplicity, to samples of \iid univariate random variables. Section \ref{Sec: univ_max} relates  
extreme points to exceedances of thresholds. 
Since the number of extreme points follows a binomial distribution, the Stein-Chen method for Poisson approximation 
from Section \ref{Sec: PoiAppr_dTV} may be used in order to determine error estimates for the approximation by a Poisson distribution. 
We thereby establish bounds on the errors in the
Kolmogorov distance involved in the approximation of the law of the maximum value by a so-called extreme value distribution. 
In particular, we delineate the different 
steps, as well as the respective error estimates arising from them, that are needed for the approximation. We present our results for the cases of
random variables that follow exponential, Pareto, uniform, normal, Cauchy or geometric distributions. 
In Section \ref{Sec: Poi_proc_approc_MPPE} we generalise by introducing marked point processes of exceedances. Using results 
from Chapter \ref{Sec: Poi_proc_approx} and \ref{Sec: Improved_rates}, we determine and discuss bounds on the errors in the total variation distance 
(or the $d_2$-distance if need be) 
for processes whose marks follow any of the distributions that we already treated in Section \ref{Sec: univ_max}. 
\section{Poisson approximation for the number of extreme points and maxima
of random variables}
\sectionmark{Poisson approximation for the number of extreme points and maxima}
\label{Sec: univ_max}
The first question is of course: what is an ``extreme point''? It is an atypical value taken by a random variable. For a one-dimensional
random variable $X$ with state space $E \subseteq \mathbb{R}$ it is a value that exceeds a certain threshold, either towards 
the right or towards the left of the state space.  Suppose we have random variables $X_1, \ldots, X_n$ that are \iid copies of $X$, 
and denote by $F$ and $\overline{F}$ the distribution and survival functions of $X$, respectively. In this section, we consider upper tail extremes, i.e. we 
call ``extreme value'' or ``extreme point'' (suggesting the language of point processes) a value in $(u_n, x_F] \cap E$ (or $[u_n,x_F] \cap E$), where 
$x_F=\sup\{x \in \mathbb{R}: \, F(x)<1\}$ is the right endpoint of $F$ and $u_n$ denotes a 
threshold that varies with the chosen sample size $n$ (note that if $x_F=\infty$, then $(u_n,x_F] \cap E = (u_n,\infty)$). The number of extreme points is then given 
by 
\[
\sum_{i=1}^n I_{\{X_i > u_n\}} \sim \mathrm{Bin}(n, \overline{F}(u_n)).
\]
For a threshold $u_n$ increasing with $n$, the probability $\overline{F}(u_n)$ of exceeding the threshold decreases towards $0$. If $n\overline{F}(u_n) \to \lambda >0$
as $n \to \infty$, the law of the number of extreme points converges to a Poisson distribution with mean $\lambda$. This implies that the number 
of points exceeding the threshold $u_n$ is approximately distributed as $\mathrm{Poi}(n\overline{F}(u_n))$.
The Stein-Chen method for Poisson approximation from Section \ref{Sec: PoiAppr_dTV} provides us with the tools needed to 
investigate the sharpness of this approximation for each integer $n \ge 1$. 
Instead of counting points in $(u_n, x_F] \cap E$, we can count them in a more general set $A$
that we suppose to be a measurable subset of $E$ containing extreme values. 
Theorem \ref{t: Barbour_Hall_1984} then
gives the following result for the error in total variation that arises when approximating the law of $\sum_{i=1}^n I_{\{X_i \in A\}}$ 
by a Poisson distribution:
\begin{thm}\label{t: Poster_univ}
For each integer $n \ge 1$, let $X, X_1, \ldots, X_n$ be \iid univariate random variables with state space $E \subseteq \mathbb{R}$.
For a fixed set $A\in \mathcal{E}:= \mathcal{B}(E)$, let $W_A := \sum_{i=1}^n I_{\{X_i \in A\}}$ denote the random number of points in $A$.
Then,
\[
d_{TV}(\mathcal{L}(W_A), \mathrm{Poi}(nP(X \in A)) \le P(X \in A). 
\]
\end{thm}
\begin{proof}
We apply Theorem \ref{t: Barbour_Hall_1984} with $I_i := I_{\{X_i \in A\}}$ and $W:=W_A$. Then $p_i \equiv P(X \in A)$, 
$\lambda=\mathbb{E}W_A = nP(X \in A)$, and the upper bound in (\ref{t: bound_Barbour_Hall_1984}) equals $P(X \in A)$.
\end{proof}
\noindent For the case $A=A_n=(u_n,x_F] \cap E$, Theorem \ref{t: Poster_univ} amounts to 
\begin{equation}\label{t: Bin_to_Poi_extremes_dTV}
d_{TV}\left(\mathrm{Bin}(n, \overline{F}(u_n)), \mathrm{Poi}(n \overline{F}(u_n))\right)
 \le \overline{F}(u_n).  
\end{equation}
This result is immediately applicable to all kinds of distributions $F$, and gives
error bounds vanishing with $n \to \infty$ for suitably chosen thresholds $u_n$.
It can prominently be used to study the quality of asymptotic results given by classical extreme value theory, which
establishes limit laws for maxima of \iid random variables. 
The number of extreme points can be related to the maximum  $X_{(n)} := \max_{1 \le i \le n} X_i$ of the random variables $X_1, \ldots, X_n$ by considering that
\[
\left\{ \sum_{i=1}^n I_{\{X_i > u_n\}}=0 \right\} = \left\{ X_{(n)} \le u_n \right\}.
\]
Using the Poisson approximation to the binomial
it is thus clearly possible to determine an approximation to the law of the maximum,
and (\ref{t: Bin_to_Poi_extremes_dTV}) in particular gives the error of this approximation:
\begin{equation}\label{t: error_max}
\left| P\left(W_A = 0\right) - \mathrm{Poi}(\mathbb{E}W_A)\{0\}\right| = \left| P\left(X_{(n)} \le u_n \right) - e^{-n\overline{F}(u_n)}\right|  \le \overline{F}(u_n).
\end{equation}
Indeed, underlying classical extreme value theory is the 
following well-known limit result:
\begin{thm}\label{t: Lim_Poiapprox_Max} (Poisson approximation for maxima of \iid rv's) Let $X_1,$ $\ldots,$ $ X_n$ be \iid random variables with maximum $X_{(n)}$. 
For given $\tau \in [0, \infty]$ and a sequence
$(u_n)_{n \ge 1}$ of real numbers, the following are equivalent:
\begin{align}
  n\overline{F}(u_n) &\to \tau, \quad \textnormal{ as } n \to \infty,\\
 P\left(X_{(n)} \le u_n \right) &\to e^{-\tau},\quad \textnormal{ as } n \to \infty.
\end{align}
\end{thm}
\begin{proof} See, for instance, Proposition 3.1.1 in \cite{Embrechts_et_al:1997}.\end{proof}
\noindent More interesting than the approximation by $e^{-\tau}$ for a fixed value $\tau$, or by $e^{-n\overline{F}(u_n)}$ which varies with 
the sample size $n$, would be the approximation by
a non-degenerate distribution function that no longer depends on $n$. Such a distribution function may be found by
subjecting the maximum to a \textit{normalisation}, more precisely here, to a suitable affine transformation  $u_n = a_n x + b_n$ for $x \in \mathbb{R}$, 
$a_n, b_n \in \mathbb{R}$ with $a_n >0$. If we can indeed find non-degenerate limit distributions for maxima, then what are these?
This question is answered by one of the most fundamental
results of classical extreme value theory, the Fisher-Tippett theorem: 
\begin{thm}(Fisher-Tippett)\label{t: Fisher-Tippett}
Let $X_1, \ldots, X_n$ be a sequence of \iid random variables with maximum $X_{(n)}$. If there exist norming constants $a_n >0$, $b_n \in \mathbb{R}$ and some non-degenerate distribution
function $H$ such that 
\begin{equation}\label{d: MDA}
\frac{X_{(n)}-b_n}{a_n} \stackrel{d}{\longrightarrow} H,
\end{equation}
then $H$ belongs to the type of one of the three following distribution functions:
\begin{align*}
&\textnormal{Fr\'echet:} &\Phi_{\alpha}(x) &= \left\{\begin{array}{ll} 0, & x \le 0\\ e^{-x^{-\alpha}}, & x>0 \end{array}\right. &\alpha >0.\\
&\textnormal{Weibull:} &\Psi_{\alpha}(x) &= \left\{\begin{array}{ll} e^{-(-x)^\alpha}, & x < 0\\ 1, & x\ge 0 \end{array} \right.  & \alpha >0. \\
&\textnormal{Gumbel:} &\Lambda(x)&= e^{-e^{-x}}, x \in \mathbb{R}.   
\end{align*}
\end{thm}
\begin{proof} See, for instance, Proposition 0.3 in \cite{Resnick:1987}.\end{proof} 
\noindent The Fr\'echet, Weibull and Gumbel distributions are called \textit{extreme value distributions}. If (\ref{d: MDA}) holds, we say
that $F$ is in the \textit{maximum domain of attraction} of $H$, which we denote by $F \in \mathrm{MDA}(H)$.
We summarise the normalisations and extremal limit results for a selection of well-known (continuous) distribution functions in the following proposition:
\begin{prop}\label{t: cont_limit_maxima}
For any integer $n \ge 1$, let $X_1, \ldots, X_n$ be \iid random variables with cumulative distribution function $F$ and maximum $X_{(n)}$.
\begin{enumerate}
\item[(a)] (Exponential distribution) Let 
\[
F(y) = \left\{ \begin{array}{ll} 1-e^{-\lambda y},& y \ge 0, \\ 0 & y <0, \end{array} \right.
\]
with rate parameter $\lambda >0$. Then, for all $x \in \mathbb{R}$,
\[
P\left( X_{(n)} \le \frac{x+\log n}{\lambda}\right) \longrightarrow e^{-e^{-x}} = \Lambda(x), \quad \textnormal{ as } n \to \infty. 
\]
\item[(b)] (Pareto distribution) Let 
\begin{align}\label{d: Pareto}
\begin{split}
F(y) = \left\{ 
\begin{array}{ll}
1-\left( \frac{\phi}{y}\right)^\alpha, & y \ge \phi, \\
0, & y < \phi,
\end{array}
\right. 
\end{split}
\end{align}
where  $\alpha, \phi >0$ denote the shape and scale parameters, respectively. Then, for all $x>0$, 
\[
P\left( X_{(n)} \le \phi n^{\frac{1}{\alpha}}x\right) \longrightarrow e^{-x^{-\alpha}}= \Phi_\alpha(x),  \quad \textnormal{ as }n \to \infty. 
\]
\item[(c)] (Uniform distribution) Let 
\[F(y) = \left\{ \begin{array}{lll}
0, &  y < a,\\
\frac{y-a}{b-a}, & a \le y < b,\\
1, & y \ge b,       
       \end{array}
\right.
\]
where $a,b \in \mathbb{R}$ with $a < b$. Then, for all $x < 0$,
\[
P\left( X_{(n)} \le \frac{(b-a)x}{n} + b\right) \longrightarrow e^x = \Psi_1(x), \quad \textnormal{ as }n \to \infty. 
\]
\item[(d)] (Standard normal distribution) Let $\Phi(y):= F(y) = \int_{-\infty}^y \varphi(t)dt$ and $ \displaystyle \varphi(y)= 
 \frac{e^{-y^2/2}}{\sqrt{2\pi}} $ for $y \in \mathbb{R}$. 
Then, for all $x \in \mathbb{R}$, 
\[
P \left( X_{(n)} \le \frac{x}{\sqrt{2 \log n}} + \sqrt{2 \log n} - \frac{\log \log n + \log 4\pi}{2\sqrt{2\log n}}\right)
 \longrightarrow e^{-e^{-x}} = \Lambda(x), 
\]
as $n \to \infty$.
\item[(e)] (Standard Cauchy distribution) Let 
\[
F(y) = \arctan(y)/\pi +0.5 = \int_{-\infty}^y 1/\pi(1+t^2)dt 
\]
for $y \in \mathbb{R}$. Then, for any 
$x >0$, 
\[
P\left( X_{(n)} \le \frac{nx}{\pi}\right) \longrightarrow e^{-x^{-1}}= \Phi_1(x),  \quad \textnormal{ as }n \to \infty. 
\]
\end{enumerate}
\end{prop}
\begin{proof} For each case we use $y:=y_n := a_nx + b_n$ and
\begin{align}
\begin{split} 
\label{p: max_Fn}
P\left(X_{(n)} \le a_n x + b_n \right) 
&= P\left(X_1 \le a_n x + b_n, \ldots, X_n \le a_n x + b_n\right)\\
&= F^n(a_n x + b_n). 
\end{split}
\end{align}
(a) Fix any $x \in \mathbb{R}$. With $a_n = \lambda^{-1}$ and $b_n = \lambda^{-1} \log n$, we have, for any integer $n > e^{-x}$,
\[
F^n\left(a_nx+b_n\right) = \left( 1-\frac{e^{-x}}{n}\right)^n \longrightarrow e^{-e^{-x}}=\Lambda(x), \quad \text{as } n \to \infty.
\]
(b) Fix any $x >0$. With $a_n = \phi n^{1/\alpha}$ and $b_n \equiv 0$, we have, for any integer $n > x^{-\alpha}$,
\[
F^n(a_nx+b_n) = \left( 1-\frac{x^{-\alpha}}{n}\right)^n \to e^{-x^{-\alpha}} = \Phi_\alpha(x), \quad \text{as } n \to \infty.
\]
(c) Fix any $x <0$. With $a_n = (b-a)/n$ and $b_n\equiv b$, we have, for any integer $n > x$, 
\[
F^n(a_nx+b_n) = \left(1+\frac{x}{n}\right)^n \longrightarrow e^x = \Psi_1(x), \quad \text{as } n \to \infty.
\]
(d) For any fixed $x \in \mathbb{R}$, 
we have $y=y_n = y_n(x) = a_nx+b_n \in \mathbb{R}$,
where 
\[
a_n = \frac{1}{\sqrt{2 \log n}} \to 0 \quad \textnormal{ and } \quad b_n = \sqrt{2 \log n} - \frac{\log \log n + \log 4\pi}{2\sqrt{2 \log n}} \to \infty, 
\]
as $n \to \infty$. It follows that $y_n(x) \to \infty$ and $\overline{\Phi}(y_n) \to 0$ as $n \to \infty$. Hence,
\[
\Phi^n(y_n) = \left( 1- \overline{\Phi}(y_n) \right)^n = e^{n\log(1-\overline{\Phi}(y_n))} \sim e^{-n\overline{\Phi}(y_n)} \sim e^{-\frac{n\varphi(y_n)}{y_n}}, 
\]
as $y_n, n \to \infty$, where the first asymptotic equality is due to $\log(1-z) \sim -z$ as $z \to 0$, and the second is the Mills ratio, i.e.
$\overline{\Phi}(y_n) \sim y_n^{-1} \varphi(y_n)$ as $y_n \to \infty$. Furthermore, for $x$ fixed,
\begin{align}\label{p: normal_normalised_density}
\begin{split}
&\frac{n\varphi(y_n)}{y_n} \\
=&\,\,e^{-x} \cdot \frac{ \exp\left\{
-\frac{1}{8 \log n} \left[2x^2 - 2x(\log \log n + \log 4\pi) + (\log \log n + \log 4\pi)^2 \right]\right\}}{1+\frac{x-(\log \log n + \log 4 \pi)/2}{2\log n}}
\end{split}
\end{align}
tends to $e^{-x}$ as $n \to \infty$, and thus 
\[
\Phi^n(a_n x + b_n) \to e^{-e^{-x}} = \Lambda(x), \textnormal{ as } n \to \infty. 
\]
(e) By l'H\^opital's rule, 
\[
\lim_{y \to \infty} \frac{\overline{F}(y)}{1/\pi y} = \lim_{y \to \infty} \frac{y^2}{1+y^2} = 1. 
\]
With $a_n = n\pi^{-1}$ and $b_n \equiv 0$, we then find
\[
F^n(a_nx+b_n) = \left( 1-\overline{F}\left(\frac{nx}{\pi}\right)\right)^n \sim \left( 1-\frac{x^{-1}}{n}\right)^n \to e^{-x^{-1}} = \Phi_1(x) 
\]
for all $x>0$.
\end{proof}
\noindent The question we now ask is whether these limit results actually give good approximations for the laws of the maxima.
Some results on convergence rates are given in Chapters 2.4 in \cite{Resnick:1987} and \cite{Leadbetter_et_al:1983}, respectively, and, for maxima of normals, in 
\cite{Hall:1979}. We establish precise rates of convergence to extreme value distributions by using (\ref{t: error_max}) along with suitable
normalisations. In Section \ref{s: maxima_cont} we achieve this for each of the results from Proposition \ref{t: cont_limit_maxima}. 
Later, Section \ref{s: Problem_discrete} discusses issues that may arise for 
distributions with discontinuities in the tail, and Section \ref{s: max_discrete} offers a way to partially remedy these issues. 
\subsection{Maxima of continuous random variables}\label{s: maxima_cont}
As demonstrated in Propositions \ref{t: Max_exp}, \ref{t: Max_pareto} and \ref{t: Max_uniform} below, it is relatively straightforward to establish rates of convergence for the maximum law of
exponential, Pareto, and uniform random variables, using Theorem \ref{t: Poster_univ} and a suitable normalisation. For each of these cases the 
uniform error bound is of order $\log (n)/n \to 0$, as $n \to \infty$, thus providing a sharp approximation. The proofs for these three distributions are analogous, whereas the cases of standard normal and Cauchy 
random variables, treated in Propositions \ref{t: Max_normal} and \ref{t: Max_Cauchy}, respectively, are somewhat more involved. By choosing a non-linear normalisation for the uniform distribution, it is possible to expand 
the range of possible limiting distributions from the Weibull with parameter $1$ to Weibull distributions with any parameter $\alpha>0$. The choice $\alpha = 1$ restores 
the result from Proposition \ref{t: cont_limit_maxima}. 
\begin{prop}\label{t: Max_exp}(Exponential distribution)
For each integer $n \ge 1$, let $X_1, \ldots, X_n$ be \iid exponential random variables with parameter $\lambda > 0$. Then,
for all $x \in \mathbb{R}$,
\[
\left| P\left( X_{(n)} \le \frac{x + \log n}{\lambda} \right) -e^{-e^{-x}}\right| \le \frac{\log n}{n} + \frac{1}{n} = O\left( \frac{\log n}{n} \right).
\] 
\end{prop}
\begin{proof}
We apply Theorem \ref{t: Poster_univ} with $A=[y, \infty)$ for any choice of $y \ge 0$ to find
\[
\left| P\left( X_{(n)} \le y \right) - e^{-ne^{-\lambda y}} \right| \le e^{-\lambda y}. 
\]
Plug in $y= \lambda^{-1}(x + \log n)$, possible for $x \ge -\log n$, since then $y \ge 0$, giving
\[
\left| P\left( X_{(n)} \le \frac{x+\log n}{\lambda}\right) - e^{-e^{-x}} \right| \le \frac{e^{-x}}{n}. 
\]
In order to find a uniform bound for all $x \in \mathbb{R}$, choose $x_0 :=x_{0n}:= - \log \log n$. On the one hand, for all $x \ge x_0$, the error estimate 
$\exp{(-x)}/n$ is smaller than $\exp{(-x_0)}/n = \log (n)/n$. On the other hand, for all $ x \le x_0$,
\[
P\left( X_{(n)} \le \frac{x+\log n}{\lambda}\right) \le P\left( X_{(n)} \le \frac{x_0+\log n}{\lambda}\right) 
\le \frac{e^{-x_0}}{n} + e^{-e^{-x_0}} ,
\]
since distribution functions are non-decreasing. This implies that 
\begin{equation}\label{p: exp_dK_bound}
\left| P\left( X_{(n)} \le \frac{x+\log n}{\lambda}\right) - e^{-e^{-x}} \right| \le \frac{e^{-x_0}}{n} + e^{-e^{-x_0}} = \frac{\log n}{n} + \frac{1}{n},
\end{equation}
for all $x \le x_0$. See Figure \ref{f: Max_exponential_dTV} for a sketch of the situation.
We may use the upper bound in (\ref{p: exp_dK_bound}) for all $x \in \mathbb{R}$.
\begin{figure}[!ht]
\begin{center}{\footnotesize \input{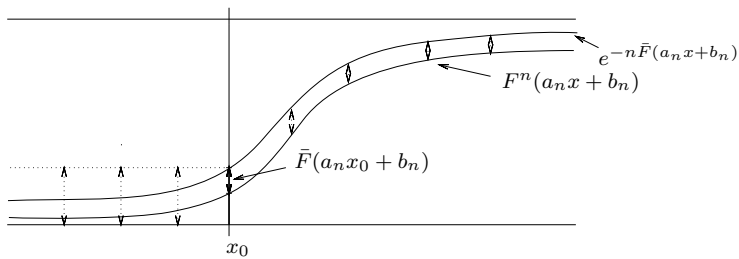}}\end{center}
\caption{At all values $x \le x_0$, the error cannot exceed the sum of the bound $\bar{F}(a_nx_0+b_n)$ on the difference
between the two distributions functions at $x_0$ and the height $\exp\{-n\bar{F}(a_nx_0+b_n)\}$ of the
approximating distribution.}
\label{f: Max_exponential_dTV}
\end{figure}
\end{proof}
\begin{prop}\label{t: Max_pareto}(Pareto distribution)
For each integer $n \ge 1$, let $X_1,$ $\ldots,$ $X_n$ be \iid Pareto random variables with shape parameter $\alpha > 0$ and scale parameter $\phi >0$. Then, 
for all $x >0$,
\[
\left| P\left( X_{(n)} \le \phi n^{1/\alpha} x \right) - e^{-x^{-\alpha}}\right| \le \frac{\log n}{n} + \frac{1}{n} = O\left( \frac{\log n}{n} \right). 
\]
\end{prop}
\begin{proof}
We apply Theorem \ref{t: Poster_univ} with $A=[y, \infty)$ for any choice of $y \ge \phi$:
\[
\left| P\left( X_{(n)} \le y \right) - e^{-n\left(\frac{\phi}{y}\right)^\alpha}\right| \le \left(\frac{\phi}{y} \right)^\alpha.
\]
Plug in $y = \phi n^{1/\alpha}  x$, possible for $x \ge n^{-1/\alpha}$, since then $y \ge \phi$, giving
\[
\left| P\left( X_{(n)} \le  \phi n^{1/\alpha} x \right) - e^{-x^{-\alpha}}\right| \le \frac{x^{-\alpha}}{n} . 
\]
In order to determine a uniform error bound, we first choose $x_0:=x_{0n}:= (\log n)^{-1/\alpha}$. Then, for all $x\ge x_0$, we have 
$x^{-\alpha}/n \le x_0^{-\alpha}/n = \log (n)/n$, whereas, for all $x \le x_0$, 
\[
\left| P\left( X_{(n)} \le  \phi n^{1/\alpha} x \right) - e^{-x^{-\alpha}}\right|  \le  \frac{x_0^{-\alpha}}{n}  + e^{-x_0^{-\alpha}} = \frac{\log n}{n} + \frac{1}{n}.
\]
This bound holds for all $x >0$.
\end{proof}
\begin{prop} (Uniform distribution)\label{t: Max_uniform}
For each integer $n \ge 1$, let $X_1,$ $\ldots,$ $X_n$ be \iid uniform random variables with parameters $a,b \in \mathbb{R}$, $a < b$, and let $\alpha > 0$. Then, for all $x <0$,
\[
\left| P\left( X_{(n)} \le \frac{-(-x)^\alpha(b-a)}{n} + b\right) - e^{-(-x)^\alpha} \right| \le \frac{\log n}{n} + \frac{1}{n} = O\left( \frac{\log n}{n} \right).
\]
\end{prop}
\begin{proof}
The proof is analogous to the proofs of Propositions \ref{t: Max_exp} and \ref{t: Max_pareto}, 
with the normalisation
$y = -(-x)^\alpha(b-a)/n + b$. For all $y \in [a,b)$ or equivalently, for all $x \in [-n^{1/\alpha},0)$, we have
\[
\left| P\left( X_{(n)} \le \frac{-(-x)^\alpha(b-a)}{n} + b\right) - e^{-(-x)^\alpha}\right| \le \frac{(-x)^\alpha}{n}. 
\]
With the choice $x_0 := x_{0n} := -(\log n)^{1/\alpha}$, we find the upper error bound 
\[
\frac{(-x_0)^\alpha}{n} + e^{-(-x_0)^\alpha} = \frac{\log n}{n} + \frac{1}{n},
\]
which we may use for all $x <0$.
\end{proof}
\noindent For the three examples above, the approximation by an extreme value distribution was effected in two steps: in the first step, we used Theorem \ref{t: Poster_univ} to approximate 
the maximum law by a Poisson probability mass function at $0$, and in the second step, a suitable normalisation to transform this Poisson 
into the required extreme value distribution. For each of these examples the total error estimate is of order $\log (n)/n$, thus providing a sharp approximation as 
$n$ increases. Moreover, the total error estimate is of the same order as the error estimate for the first step, 
which means that the principal part of the error in the approximation by an extreme value distribution arises from basic Poisson approximation, 
whereas the normalisation is negligeable. 

For the maximum of \iid standard normals, the situation is somewhat different. 
As remarked already by \cite{Fisher/Tippett:1928}, and later by \cite{Hall:1979} and \cite{Leadbetter_et_al:1983}, convergence of the maximum law of normals 
to the Gumbel distribution is extremely slow. 
\cite{Hall:1979} showed that with the choice of norming constants (\ref{t: an_bn_N}), the rate of convergence is not 
better than $(\log\log n)^2/\log n$. Moreover, he showed that if the 
norming constants $a_n$ and $b_n$ were chosen as solutions to 
\[
\frac{n\varphi(b_n)}{b_n} = 1, \quad a_n=b_n^{-1}, 
\]
then
\[
\frac{C_1}{\log n} \le \sup_{x \in \mathbb{R}} \left| P\left( X_{(n)} \le a_nx +b_n \right) - e^{-e^{-x}}\right| \le \frac{C_2}{\log n}\, 
\]
for constants $C_1, C_2 >0$, and that the rate of convergence cannot be improved by choosing different norming constants.  

Proposition \ref{t: Max_normal} below delineates the different steps needed for the approximation by the Gumbel distribution and gives error estimates for each step. 
Though the first step, which is basic Poisson approximation, gives an error that is only of
order $\log (n)/n$, the subsequent steps needed for the approximation by a Gumbel distribution give bigger error estimates. 
The next step after Poisson approximation uses the \textit{Mills ratio}:
\[
\overline{\Phi}(y) \sim \frac{\varphi(y)}{y}, \textnormal{ as }y \to \infty. 
\] 
As shown below we estimate the error arising from this step, i.e. the error of the approximation of $\exp\{- n\overline{\Phi}(y)\}$ by $\exp\{- n\varphi(y)/y\}$, 
to be of order $1/\log n$, and thereby substantially bigger than the one from Poisson approximation. 
In a last step, the normalisation needed to transform $\exp\{- n\varphi(y)/y\}$ into the required Gumbel distribution gives rise
to an even bigger error, which we estimate to be of size $(\log \log n)^2/\log n$.
\begin{prop}\label{t: Max_normal} (Normal distribution)
For each integer $n \ge 2$, let $X_1,$ $\ldots,$ $X_n$ be \iid standard normal random variables with distribution function $\Phi$ and probability density function $\varphi$.
Then, 
\\
(a) (Basic Poisson approximation) For all $y \in \mathbb{R}$,
\begin{equation}\label{t: normal_basicapprox}
\left| P\left( X_{(n)} \le y \right) - e^{-n\overline{\Phi}(y)}\right| \le \frac{\log n}{n} + \frac{1}{n} =: \delta_{\mathrm{PoiAppr}}= O\left(\frac{\log n}{n}\right).
\end{equation}
(b) (Approximation using the Mills ratio) For each integer $n \ge 21$ and for all $y \in \mathbb{R}$,
\begin{align}\label{t: normal_Mill_error}
&\left| P\left( X_{(n)} \le y \right) - e^{-\frac{n\varphi(y)}{y}}\right|  \nonumber \\
&\phantom{buhh} \le  \delta_{\mathrm{PoiAppr}} + \frac{1}{2\log n} +  e^{-0.1 \sqrt{\log n}} =: \delta_{\mathrm{Mills}}
= O \left(\frac{1}{\log n} \right).
\end{align}
(c) (Approximation by a standard Gumbel distribution) For each integer $n \ge 21$ and for all $x\in \mathbb{R}$,
\begin{align*}
\left| P\left( X_{(n)} \le a_nx +b_n \right) - e^{-e^{-x}} \right| 
&\le  \delta_{\mathrm{Mills}} + \frac{69(\log \log n + \log 4\pi)^2}{\log n}\\
& = O \left( \frac{\log^2\log n}{\log n}\right),
\end{align*}
where 
\begin{equation}\label{t: an_bn_N}
a_n = \frac{1}{\sqrt{2 \log n}}, \quad b_n = \sqrt{2\log n} - \frac{\log \log n + \log 4\pi}{2\sqrt{2\log n}}. 
\end{equation}
\end{prop}
\begin{proof}
(a) We apply Theorem \ref{t: Poster_univ} with $A=[y, \infty)$ for any choice of $y\in \mathbb{R}$:
\begin{equation}\label{p: normal_basicapprox}
\left| P(X_{(n)} \le y) - e^{-n\overline{\Phi}(y)}\right| \le \overline{\Phi}(y).
\end{equation}
For a uniform error bound, choose $y_0 :=y_{0n}:= \Phi^{-1}(1-\log (n) / n)$. Then, for all $y \ge y_0$, we have 
$\overline{\Phi}(y) \le \overline{\Phi}(y_0)=\log (n) / n $, whereas for $y \le y_0$, 
\[
\left| P(X_{(n)} \le y) - e^{-n\overline{\Phi}(y)}\right| \le \overline{\Phi}(y_0) + e^{-n\overline{\Phi}(y_0)} = \frac{\log n}{n} + \frac{1}{n}, 
\]
and we may use this bound for all $y \in \mathbb{R}$.\\
(b) 
By adding and subtracting $\exp\{-n\varphi(y)/y\}$ into (\ref{t: normal_basicapprox}),
we find, for all $y \in \mathbb{R}$,
\[
\left| P\left( X_{(n)} \le y \right) - e^{-\frac{n\varphi(y)}{y}}\right| \le  \frac{\log n}{n} + \frac{1}{n} +  
\left|e^{-n\overline{\Phi}(y)} - e^{-\frac{n\varphi(y)}{y}} \right|, 
\]
and we need to determine a uniform bound for 
$\left|\exp\{-n\overline{\Phi}(y)\} - \exp\left\{-\frac{n\varphi(y)}{y}\right\}\right|$.
Suppose first that $y >0$. With the two consecutive changes of variables $z:=t-y$ and $w:=yz$, we obtain
\begin{align*}
\overline{\Phi}(y) 
&= \frac{1}{\sqrt{2\pi}} \int_{y}^{\infty} e^{-t^2/2} dt 
= \frac{1}{\sqrt{2\pi}} \int_{0}^{\infty} e^{-(z+y)^2/2} dz\\
&= \frac{1}{\sqrt{2\pi}} e^{-y^2/2}\int_0^{\infty} e^{-z^2/2} \cdot e^{-zy} dz
= \frac{\varphi(y)}{y} \int_0^{\infty} e^{-w^2/2y^2} \cdot e^{-w} dw.
\end{align*}
With $1-x \le e^{-x} \le 1$ for $x= w^2/2y^2$, $\int_0^\infty e^{-w} dw = 1$ and $\int_0^\infty w^2e^{-w}=2$, we get
\[
1-\frac{1}{y^2} \le \int_0^{\infty} e^{-w^2/2y^2} \cdot e^{-w} dw \le 1,
\]
and thereby 
\begin{equation}\label{p: normal_millsratio}
\frac{\varphi(y)}{y} - \frac{\varphi(y)}{y^3} \le \overline{\Phi}(y) \le \frac{\varphi(y)}{y}, \quad \textnormal{ for all } y >0.
\end{equation}
It then follows that for all $y \ge 2$,
\begin{align}\label{p: norm_error_b}
\left|e^{-n\overline{\Phi}(y)} - e^{-\frac{n\varphi(y)}{y}} \right| 
&= e^{-n\overline{\Phi}(y)} - e^{-\frac{n\varphi(y)}{y}} 
= e^{-n\overline{\Phi}(y)}\left\{ 1-e^{-n\left( \frac{\varphi(y)}{y} - \overline{\Phi}(y)\right)}\right\} \nonumber \\
&\le n e^{-n\overline{\Phi}(y)} \left( \frac{\varphi(y)}{y} - \overline{\Phi}(y) \right)
\le \frac{n\varphi(y)}{y^3}\, e^{-\frac{n\varphi(y)}{y}\left( 1- \frac{1}{y^2}\right)} \nonumber\\
&\le \frac{n\varphi(y)}{y^3}\,  e^{-\frac{3n\varphi(y)}{4y}}
\le \frac{4}{3ey^2},
\end{align}
where we used $1-e^{-z} \le z$, (\ref{p: normal_millsratio}) and $y\ge2$ in the first, second and third inequalities, respectively, as well
as $z \exp\{-(3/4)z\} \le 4/(3e)$ for all $z \in \mathbb{R}$, for the last inequality. 
In order to determine a uniform bound on the error between $\exp\{-n\overline{\Phi}(y)\}$ and $\exp\{-n\varphi(y)/y\}$, 
choose $y_1:=y_{1n}:= \sqrt{2\log n} - \log \log (n) /\sqrt{2 \log n}$. 
Note that $y_{1n} \ge 2$ for all $n \ge 21$.
Since the error bound in (\ref{p: norm_error_b}) decreases for increasing $y$, 
we have 
\begin{align}
\frac{4}{3ey^2} &\le \frac{4}{3ey_1^2}= \frac{2}{3e \log n} \left[ 1- \frac{\log \log n}{\log n} + \frac{\log^2 \log n}{4 \log^2 n} \right]^{-1}\nonumber \\
&\le \frac{2}{3e \log n}\left[ 1- \frac{\log \log n}{\log n} \right]^{-1} \le \frac{4}{3e\log n} \le \frac{1}{2\log n} \label{p: norm_error_b_help0}
\end{align}
for all $y \ge y_1$, where we used that $\log\log(n)/\log n \le 1/2$ for all $n \ge 1$. On the other hand, for all $y \le y_1$,
\begin{equation}\label{p: norm_error_b_help1}
\left|e^{-n\overline{\Phi}(y)} - e^{-\frac{n\varphi(y)}{y}} \right| \le \frac{4}{3ey_1^2} + e^{-\frac{n\varphi(y_1)}{y_1}}.
\end{equation}
This gives an upper bound that we may use for all $y \in \mathbb{R}$. We have
\begin{align}
\varphi(y_1)& =  \frac{e^{-y_1^2/2}}{\sqrt{2\pi}}=\frac{\log n}{ n} \cdot \frac{e^{-\frac{\log^2\log n}{4\log n}}}{\sqrt{2\pi}} , \nonumber \\
e^{-\frac{n\varphi(y_1)}{y_1}} &= \exp\left\{ -\sqrt{\log n} \cdot \frac{e^{-\frac{\log^2 \log n}{4\log n}}}{2\sqrt{\pi}\left(1-\frac{\log \log n}{2 \log n} \right) } \right\}
\le e^{-0.1 \sqrt{\log n}}, \label{p: norm_error_b_help2}
\end{align}
where we used $(1- (\log \log n)/(2\log n))^{-1} \ge 1$ 
and $(2\sqrt{\pi})^{-1} \exp\{-(\log^2 \log n)/$ $(4 \log n)\} $ 
$\ge (2e\sqrt{\pi})^{-1} \ge 0.1$. 
It follows from (\ref{p: norm_error_b_help0})-(\ref{p: norm_error_b_help2}) that
for each integer $n \ge 21$,
\begin{equation}\label{p: normal_bound_pos_y}
\left|e^{-n\overline{\Phi}(y)} - e^{-\frac{n\varphi(y)}{y}} \right| \le  \frac{1}{2\log n} + e^{-0.1 \sqrt{\log n}}, \quad \textnormal{ for all } 
y \in \mathbb{R}.
\end{equation}
(c) With the normalisation  
\[
y:= y(x) = a_n x + b_n = \frac{x}{\sqrt{2 \log n}} + \sqrt{2 \log n} - \frac{\log \log n + \log 4\pi}{2\sqrt{2 \log n}} 
\]
we obtain $\displaystyle \frac{n\varphi(y)}{y} = e^{-x} f_n(x) $, where
\begin{equation}
f_n(x) = \frac{\exp\left\{
-\frac{1}{8 \log n} \left[2x^2 - 2x(\log \log n + \log 4\pi) + (\log \log n + \log 4\pi)^2 \right]\right\}}{1+\frac{x-(\log \log n + \log 4 \pi)/2}{2\log n}}.
\end{equation}
By adding and subtracting $\exp\{-e^{-x}\}$ into (\ref{t: normal_Mill_error}), we find
\begin{align}\label{p: norm_Gumbel_general_bound}
\left| P\left( X_{(n)} \le a_nx +b_n \right) - e^{-e^{-x}} \right| 
&\le \delta_{\mathrm{Mills}} + \left| e^{-\frac{n\varphi(y)}{y}} - e^{-e^{-x}}\right| \nonumber \\
&\le  \delta_{\mathrm{Mills}} + e^{-\min\left\{ \frac{n\varphi(y)}{y} \, , \, e^{-x}\right\}} \cdot \left| \frac{n\varphi(y)}{y} - e^{-x}\right| \nonumber \\
&\le  \delta_{\mathrm{Mills}} + e^{-\min\left\{ e^{-x}f_n(x)\, , \, e^{-x}\right\}}  e^{-x}  \left| f_n(x) -1\right|,
\end{align}
and we need to determine a uniform bound on the new error term in (\ref{p: norm_Gumbel_general_bound}), for all $x \in $ $\mathbb{R}$. 
Note that the function $2x^2-2ax+a^2$ takes the minimal value $a^2/2$ at $x=a/2$. Thus
\[
2x^2 -2x(\log \log n + \log 4\pi) + (\log \log n + \log 4\pi)^2 
\ge \frac{1}{2}(\log \log n + \log 4\pi)^2, 
\]
which is strictly positive for $n \ge 2$, implying that
\begin{equation}\label{p: term1}
f_n(x) \le \left[1+\frac{x-(\log \log n + \log 4 \pi)/2}{2\log n} \right]^{-1} \,.
\end{equation}
Suppose first that $x \ge \log \log n$. Then
\[
e^{-\min\left\{ e^{-x}f_n(x)\, , \, e^{-x}\right\}} \le 1, \quad e^{-x} \le \frac{1}{\log n},
\]
and 
\[
f_n(x) \le \left[1+ \frac{\log \log n + \log 4\pi}{4\log n}\right]^{-1} \le 1, \quad |f_n(x) - 1| \le f_n(x) + 1 \le 2. 
\]
Thus, 
\begin{equation}\label{p: stand_norm_result_c1}
\left| P\left( X_{(n)} \le a_nx +b_n \right) - e^{-e^{-x}} \right|  \le   \delta_{\mathrm{Mills}} + \frac{2}{\log n}, \quad \textnormal{ for all } x \ge \log \log n.
\end{equation}
Now suppose that $-\log \log n \le x \le \log \log n$. Then, on the one hand
\begin{align}\label{p: dK_dTV_normal_help}
f_n(x) & \le \left[1- \frac{3 \log \log n + \log 4\pi}{4\log n} \right]^{-1} \nonumber \\
&\le 1 +  \frac{3 \log \log n + \log 4\pi}{4\log n} \cdot \left[ 1- \frac{3 \log \log n + \log 4\pi}{4\log n} \right]^{-2} \nonumber\\
& \le 1+ \frac{6 \log \log n + 2\log 4\pi}{\log n},
\end{align}
where the second inequality uses Taylor expansion about $0$, and the third bounds the squared term by the constant $8$, for all $n \ge 1$. Moreover,
$e^{-x} \cdot \exp\{-e^{-x}\}$ has the global maximum $e^{-1}$. Thus, if $f_n(x) \ge 1$, the error in (\ref{p: norm_Gumbel_general_bound}) is 
\begin{align}\label{p: stand_norm_bound_A}
\begin{split}
 \delta_{\mathrm{Mills}}+ e^{-e^{-x}} \cdot e^{-x} \cdot(f_n(x) - 1) &\le
 \delta_{\mathrm{Mills}} + e^{-1}\cdot \frac{6 \log \log n + 2\log 4\pi}{\log n}\\
& \le  \delta_{\mathrm{Mills}} + \frac{3 \log \log n + \log 4\pi}{\log n}.
\end{split}
\end{align}
On the other hand, first note that for $-\log \log n \le x \le \log \log n$, 
and for all $n \ge 3$,
\begin{align*}
\left[1+\frac{x-(\log \log n + \log 4 \pi)/2}{2\log n}\right]^{-1} 
&\ge \left[1+ \frac{\log \log n -\log 4\pi}{4\log n} \right]^{-1}\\
&\ge 1 - \frac{\log \log n - \log 4 \pi}{4\log n}\,,
\end{align*}
where
\[
 0 \le  \frac{\log \log n - \log 4 \pi}{4\log n} \le 0.9, \quad \text{ for all } n \ge 2.
\]
Moreover, note that
\begin{align*}
&2x^2 - 2x(\log \log n + \log 4\pi) + (\log \log n + \log 4\pi)^2 \\
&\le 2 \log^2 \log n + 2 \log \log n( \log \log n + \log 4\pi) + (\log \log n + \log 4\pi)^2\\
& \le 5(\log \log n +\log4\pi)^2,
\end{align*}
and that $5(\log \log n + \log 4\pi)^2/(8\log n) \le 3$ for $n \ge 12$.
Therefore, 
\begin{align}
f_n(x) &\ge \exp\left\{- \frac{5(\log \log n +\log4\pi)^2}{8 \log n} \right\}\left(  1 - \frac{\log \log n - \log 4 \pi}{4\log n}\right) \nonumber \\
& = \exp\left\{- \frac{5(\log \log n +\log4\pi)^2}{8 \log n} + \log \left(1 - \frac{\log \log n - \log 4 \pi}{4\log n}\right) \right\} \label{p: stand_norm_bound1}\\
& \ge 0.1e^{-3} \ge 0.004. \label{p: stand_norm_bound2}
\end{align}
Using (\ref{p: stand_norm_bound1}) and $1-e^{-z} \le z$ for all $z \ge 0$, we obtain 
\begin{align*}
1- f_n(x) &\le  \frac{5(\log \log n +\log4\pi)^2}{8 \log n} - \log \left(1 - \frac{\log \log n - \log 4 \pi}{4\log n}\right)\\
& \le \frac{5(\log \log n +\log4\pi)^2}{8 \log n} + \frac{5(\log \log n - \log 4 \pi)}{2\log n}\\
& \le  \frac{3(\log \log n +\log4\pi)^2}{4 \log n},
\end{align*}
where we used $-\log(1-z) \le 10z$ for $0 \le z \le 0.9 $ for the second inequality, and
$\log \log n - \log 4\pi \le (\log \log n + \log 4\pi)^2/20$ for the third inequality. 
Furthermore, due to (\ref{p: stand_norm_bound2})
$e^{-x}\cdot  \exp\{-f_n(x)e^{-x}\} \le e^{-x} \cdot \exp\{-0.004e^{-x}\} \le 250/e$, for all $x \in \mathbb{R}$.
Thus, if $f_n(x) \le 1$, the error in (\ref{p: norm_Gumbel_general_bound}) is 
\begin{align}
& \delta_{\mathrm{Mills}} +  e^{- e^{-x}f_n(x)} \cdot e^{-x} (1-f_n(x))\nonumber \\
&\le  \delta_{\mathrm{Mills}} +\frac{750}{4e} \cdot \frac{(\log \log n + \log 4\pi)^2}{\log n} \nonumber \\
& \le  \delta_{\mathrm{Mills}} +  \frac{69(\log \log n + \log 4\pi)^2}{\log n}, \quad \text{ for all }n \ge 12. \label{p: stand_norm_bound_B}
\end{align}
All in all, (\ref{p: stand_norm_bound_A}) and (\ref{p: stand_norm_bound_B}) give, for all $n \ge 12$ and $-\log \log n \le x \le \log \log n$,
\begin{equation}\label{p: stand_norm_result_c2}
\left| P\left( X_{(n)} \le a_nx +b_n \right) - e^{-e^{-x}} \right| 
\le  \delta_{\mathrm{Mills}} + \frac{69(\log \log n + \log 4\pi)^2}{\log n}\,.
\end{equation}
Lastly, suppose that $x \le -\log \log n$. Since $a_nx + b_n \le -a_n \log \log n + b_n$ and since $\exp\{-n\varphi(y)/y\}$ and $\exp\{-e^{-x}\}$ are non-decreasing, we have 
\[
\left| e^{-\frac{n\varphi(y)}{y}} - e^{-e^{-x}}\right| \le e^{-\frac{n\varphi(y)}{y}} + e^{-e^{-x}} \le e^{-\frac{n\varphi(-a_n \log \log n + b_n)}{-a_n \log \log n + b_n}} + \frac{1}{n}\,,
\]
where  $\frac{n\varphi(-a_n \log \log n + b_n)}{-a_n \log \log n + b_n} =f_n(-\log \log n) \log n $, and                                                                                                             
\[
f_n(-\log \log n) \ge \frac{\exp\left\{ - \frac{5(\log \log n + \log 4 \pi)^2}{8\log n}\right\}}{1- \frac{3\log \log n  + \log 4 \pi}{4\log n}}
\ge \exp\left\{ - \frac{5(\log \log n + \log 4 \pi)^2}{8\log n}\right\},
\]
which is bigger than $e^{-3} \ge 0.04$ for all $n \ge 12$. Thus,
\[
e^{-\frac{n\varphi(-a_n \log \log n + b_n)}{-a_n \log \log n + b_n}} \le e^{-0.04 \log n },
\]
and, for all $n \ge 12$ and $x \le -\log \log n$, we have
\begin{equation}\label{p: stand_norm_result_c3}
\left| P\left( X_{(n)} \le a_nx +b_n \right) - e^{-e^{-x}} \right| 
\le \delta_{\mathrm{Mills}} + e^{-0.04 \log n } +\frac{1}{n}\,.
\end{equation}
For a bound for all $x \in \mathbb{R}$ and for all $n \ge 12$, we choose the maximum of the bounds in (\ref{p: stand_norm_result_c1}), (\ref{p: stand_norm_result_c2}) and (\ref{p: stand_norm_result_c3}):
\[
\left| P\left( X_{(n)} \le a_nx +b_n \right) - e^{-e^{-x}} \right|  \le   \delta_{\mathrm{Mills}} + \frac{69(\log \log n + \log 4\pi)^2}{\log n}\,.
\]
\end{proof}

In order to approximate the law of the maximum of standard Cauchy random variables by a Fr\'echet distribution, we need an intermediate step (comparable to the Mills ratio for the case of maxima of normals). More precisely, we use that
\[
\overline{F}(y) \sim \frac{1}{\pi y}\,, \quad \text{ as } y \to \infty, 
\]
where $\overline{F}(y)$ denotes the survival function of the Cauchy distribution. 
Contrary to the case of maxima of normals, the normalisation here does not produce any additional error and the total error is of the same order as the error that we obtain
in the first step from basic Poisson approximation, that is, of order $\log (n)/n$. 
\begin{prop}\label{t: Max_Cauchy}(Standard Cauchy distribution)
For each integer $n \ge 1$, let $X_1,$ $ \ldots, X_n$ be \iid standard Cauchy random variables with distribution function 
$F(y) =  \arctan{(y)}/\pi + 0.5$ and density $f(y) = 1/\pi(1+y^2)$, for $y \in \mathbb{R}$. Then:\\
(a) (Basic Poisson approximation) For all $y>0$,
\[
\left| P\left( X_{(n)} \le y \right) - e^{-n\overline{F}(y)}\right| \le \frac{\log n}{n} + \frac{1.74}{n} = O\left( \frac{\log n}{n}\right). 
\]
(b) (Approximation by a standard Fr\'echet distribution) For all $x >0$,
\begin{equation}\label{t: bound_Max_Cauchy}
\left| P\left( X_{(n)} \le \frac{nx}{\pi}\right) - e^{-x^{-1}} \right|
\le \frac{\log n}{n} + \frac{\pi^2 \log^3 n}{3n^2} + \frac{1}{n} = O\left( \frac{\log n}{n}\right). 
\end{equation}
\end{prop}
\begin{proof}
We first determine an upper and lower bound on 
\[
\overline{F}(y) =  0.5-\arctan{\left( y\right)}/\pi 
\]
by setting $z=1/y$ and performing a Taylor expansion of $g(z):=\overline{F}(1/z)$ about $z=0$:
\[
g(z)=  \frac{1}{2}-\frac{1}{\pi} \arctan{\left(\frac{1}{z}\right)}
= \frac{z}{\pi} - \frac{z^3}{3\pi} + \frac{z^5}{5\pi} - \frac{z^7}{7\pi} + \ldots,
\]
Here, 
\begin{align*}
& g(0+) = \lim_{z \to 0+} f(z) = 0, \quad g'(z) = \frac{1}{\pi(1+z^2)},\\
& g''(z) = -\frac{2z}{\pi(1+z^2)^2}, \quad g'''(z) = \frac{6z^2-2}{\pi(1+z^2)^3}.
\end{align*}
Since 
\[
\frac{z}{\pi} + \frac{z^3}{3!} \min_{0 \le \xi \le z} g'''(\xi) \le g(z) \le z \max_{0 \le \xi \le z} \left| g'(\xi)\right| ,
\]
where $\max_{0 \le \xi \le z} \left| g'(\xi)\right| = 1/\pi$ and $\min_{0 \le \xi \le z} g'''(\xi) = g'''(0) =-2/\pi$, we find
\[
\frac{z}{\pi} -\frac{z^3}{3\pi} \le g(z) \le \frac{z}{\pi},
\]
and thereby
\begin{equation}\label{p: bounds}
\frac{1}{\pi y} -\frac{1}{3\pi y^3} \le \overline{F}(y) \le \frac{1}{\pi y},
\end{equation}
which holds for all $y >0$.\\
(a) We apply Theorem \ref{t: Poster_univ} with $A=[y, \infty)$ for any choice of $y >0$:
\begin{equation}\label{p: cauchy_basicapprox}
\left| P(X_{(n)} \le y) - e^{-n\overline{F}(y)}\right| \le \overline{F}(y) \le \frac{1}{\pi y},
\end{equation}
where we used (\ref{p: bounds}) for the second inequality. 
In order to determine a uniform error bound, first choose $y_0 :=y_{0n}:= n/\pi \log n$. Then, for all $y \ge y_0$, we have $1/\pi y \le 1/\pi y_0 \le \log (n)/n$,
whereas, for $0 < y \le y_0$, we may bound the error by further adding the approximating distribution function at $y_0$, i.e. $\exp\{-n\overline{F}(y_0)\}$,
to $1/\pi y_0$. Therefore,
\[
\left| P(X_{(n)} \le y) - e^{-n\overline{F}(y)}\right| \le \frac{1}{\pi y_0} + e^{-n\overline{F}(y_0)} \le
\frac{\log n}{n} + \frac{1}{n} \exp\left\{\frac{\pi^2 \log^3 n}{3n^2}\right\},
\]
for all $y>0$, where we used (\ref{p: bounds}) to estimate $\exp\{-n \overline{F}(y_0)\}$.
By noting that the function $\exp\{ \pi^2 \log^3 z / 3z^2\}$ takes on its maximum $\exp\{9\pi^2/8e^3\} \approx 1.7381$ at $z=e^{3/2}$, we obtain the uniform bound
$\log (n)/ n + 1.74/n$. \\
(b) By adding and subtracting $\exp\{-n/\pi y\}$ into (\ref{p: cauchy_basicapprox}) and noting that 
\\$\exp\{-n\overline{F}(y)\}$ $ - \exp\{-n/\pi y\} \ge 0$, we obtain
\[
\left| P\left( X_{(n)} \le y \right) - e^{-\frac{n}{\pi y}}\right| \le \frac{1}{\pi y} + e^{-n\overline{F}(y)} - e^{-\frac{n}{\pi y}},
\]
where, using (\ref{p: bounds}) for the last inequality, 
\begin{align*}
e^{-n\overline{F}(y)} - e^{-\frac{n}{\pi y}} &= e^{-n\overline{F}(y)}\left\{ 1-e^{-n\left(\frac{1}{\pi y} - \overline{F}(y) \right)}\right\}
\le ne^{-n\overline{F}(y)} \left(\frac{1}{\pi y} - \overline{F}(y) \right)\\
&\le n\left(\frac{1}{\pi y} - \overline{F}(y)\right) \le \frac{n}{3\pi y^3}.
\end{align*}
Now plug in $y=nx/\pi$. Since $y>0$, we have $x >0$, and obtain
\begin{align*}
\left| P\left( X_{(n)} \le y \right) - e^{-\frac{n}{\pi y}}\right| 
&= \left| P\left( X_{(n)} \le \frac{nx}{\pi}\right) - e^{-x^{-1}} \right|\\
& \le \frac{1}{\pi y} + \frac{n}{3\pi y^3} = \frac{1}{nx} + \frac{\pi^2 }{3n^2 x^3}. 
\end{align*}
To find a uniform error bound, choose $x_0 :=x_{0n} := 1/\log n$ (which is equivalent to the choice of $y_0$ we had in (a)). For all $x \ge x_0$, we 
have 
\[
\frac{1}{nx} + \frac{\pi^2}{3n^2x^3} \le  \frac{1}{nx_0} + \frac{\pi^2}{3n^2x_0^3} = \frac{\log n}{n} + \frac{\pi^2 \log^3 n}{3n^2},
\]
whereas for all $0<x \le x_0$,
\[
\left| P\left( X_{(n)} \le \frac{nx}{\pi}\right) - e^{-x^{-1}} \right| \le  \frac{1}{nx_0} + \frac{\pi^2}{3n^2x_0^3} + e^{-x_0^{-1}} =  \frac{\log n}{n} + \frac{\pi^2 \log^3 n}{3n^2} +\frac{1}{n}.
\]
The latter bound clearly holds for all $x>0$.
\end{proof}
\subsection{The problem for maxima of discrete random variables}\label{s: Problem_discrete}
So far, we have only studied maxima of random variables with continuous distribution function $F$. As we will see below, it turns out that for some
well-known discrete distributions functions, no non-degenerate limit law $H$ as mentioned in Theorem \ref{t: Fisher-Tippett} may be found. So what are
the conditions for the existence of non-degenerate limit laws?  
The following corollaries of Theorem \ref{t: Lim_Poiapprox_Max} address this question. We again use the notation $x_F$ for the right endpoint 
of the distribution function $F$ of the \iid random variables $X_1, \ldots,X_n$, and note that $F(x) <1$ for all $x < x_F$ and $F(x_F)=1$ for all $x \ge x_F$. Moreover, 
we use the notation $F(x-) := \lim_{h \downarrow 0}F(x-h)$, for all $x \in \mathbb{R}$. 
\begin{cor}\label{t: degenerate}
(i) $X_{(n)} \to x_F $ with probability one as $n \to \infty$. \\
(ii) Suppose that $x_F < \infty$ and $F(x_F-)< 1$. Then, for every sequence $(u_n)_{n \geq 1}$ such that 
$P[X_{(n)} \leq u_n] \to \rho$, as $n \to \infty$,
we either have $\rho=0$ or $\rho=1$.
\end{cor}
\begin{proof}
See, for example, Corollary 1.5.2 in \cite{Leadbetter_et_al:1983}.
\end{proof}
\noindent Thus, for $X_1, \ldots, X_n$ having common distribution function $F$ with finite right endpoint $x_F <\infty$ and a jump at $x_F$, i.e. $F(x_F-) <1=F(x_F)$,
it follows that if $P(X_{(n)} \le u_n) = P(X_{(n)} \le a_n x + b_n) \to \rho=H(x)$ for a sequence $u_n= a_nx+b_n$, then $H(x)=0$ or $1$ for each $x$, so that
$H$ is degenerate. 
\begin{exa}(Binomial distribution) Let $X \sim \mathrm{Bin}(m,p)$ for $m \in$ $ \mathbb{N}$ and $p \in (0,1)$. Since $P(X=m) = 1$, the right endpoint $x_F$ 
is given by $m < \infty$. 
Moreover, $F(x_F-) = F(m-1) = 1-p^m < 1$. Thus, by Corollary \ref{t: degenerate}, there exists no non-degenerate limit distribution for the maximum of binomials.  
\end{exa}
\noindent More commonly, the existence of non-degenerate limit laws is impossible due to the following corollary of Theorem \ref{t: Lim_Poiapprox_Max}, which is valid for
distribution functions $F$ with right endpoint $x_F \le \infty$.
\begin{thm} \label{t: conditionF}
Let $X_1, \ldots, X_n$ be \iid random variables with common distribution function $F$ and let $\tau \in (0, \infty)$. 
Then there exists a sequence $(u_n)_{n \geq 1}$ satisfying $n \overline{F}(u_n) \to \tau$ as $n \to \infty$ if and only if 
\begin{equation} \label{t: conditionFeq}
\lim_{x \uparrow x_F} \frac{\overline{F}(x)}{\overline{F}(x-)}=1,
\end{equation}
or equivalently, if and only if
\begin{equation}\label{t: condhazardrate}
\lim_{x \uparrow x_F} \frac{p(x)}{\overline{F}(x-)} =0,
\end{equation}
where $p(x)=F(x)-F(x-)$ and $p(x)/\overline{F}(x-)$ denotes the hazard rate.
\end{thm}
\begin{proof}
See Theorem 1.7.13 in \cite{Leadbetter_et_al:1983}.
\end{proof}
\noindent Theorem \ref{t: conditionF} basically says that if the jump heights continue to be too large, there is 
no value $u_n$ such that $F(u_n)$ is close to $1-\tau/n$ and
no non-degenerate limit distribution may be found. For discrete, integer-valued random variables $X_1, \ldots, X_n$ that are \iid copies of $X$,
with $x_F = \infty$, conditions 
(\ref{t: conditionFeq}) and (\ref{t: condhazardrate}) become 
\[
\lim_{k \to \infty} \frac{\overline{F}(k)}{\overline{F}(k-1)} = 1\quad  \textnormal{ and } \quad\lim_{k \to \infty} \frac{P(X=k)}{\overline{F}(k-1)} = 0,
\]
respectively. If either of these conditions fails, we cannot find a non-degenerate limit distribution for $X_{(n)}$. 
In the following examples we will show that this is precisely the case for the Poisson and geometric distributions.
\begin{exa}(Poisson distribution) Let $X \sim \mathrm{Poi}(\lambda)$, with $\lambda >0$. Then
\begin{align*}
\frac{P(X=k)}{\overline{F}(k-1)} &= \frac{\lambda^k}{k!} \left(\sum_{l=k}^\infty \frac{\lambda^l}{l!}\right)^{-1} \\
&=\frac{\lambda^k}{k!}\left(\frac{\lambda^k}{k!} + \sum_{l=k+1}^\infty \frac{\lambda^l}{l!} \right)^{-1}\\
&= \left( 1+ \sum_{l=k+1}^\infty\frac{k! \lambda^{l-k}}{l!}\right)^{-1}.
\end{align*}
The latter sum may be rewritten and estimated as  
\[
\sum_{s=1}^{\infty} \frac{\lambda^s}{(k+1) \cdot \ldots \cdot(k+s)} 
\le \sum_{s=1}^{\infty} \left(\frac{\lambda}{k}\right)^s = \frac{\lambda/k}{1-\lambda/k} \quad \textnormal{if $\lambda < k$,} 
\]
which tends to $0$ as $k \to \infty$, and thus $P(X=k)/\overline{F}(k-1)$ tends to $1$.
Hence, by Theorem \ref{t: conditionF}, no non-degenerate limit distribution exists for Poisson maxima, and $P[M_n \leq u_n] \to \rho$ only for $\rho = 0$ or $1$.
\end{exa}
\begin{exa}(Geometric distribution) Let $X \sim \mathrm{Geo}(p)$, with the parameter $0<p<1$ denoting the success probability and the random variables $X$ counting the number of failures in 
$0$-$1$ experiments before the first success, i.e. $P(X=k) = p(1-p)^k$, for any $ k \in \mathbb{Z}_+$. Then
\[
\frac{P(X=k)}{\overline{F}(k-1)} = \frac{p(1-p)^k}{(1-p)^k} = p \in (0,1),  
\]
which violates condition (\ref{t: condhazardrate}).  
\end{exa}
We have now illustrated that for three of the most well-known and widely used discrete distributions, we may not find a limit law $H$ as in Theorem \ref{t: Fisher-Tippett}.
However, it should not be assumed that this problem occurs for discrete distributions in general. \cite{Feidt_et_al:2010} give the following 
example of discrete distribution functions that do indeed possess a non-degenerate extreme value behaviour:
\begin{exa}
Let $X \ge 0$ be an absolutely continuous random variable with distribution function $F$, probability density function $f$ and right endpoint $x_F = \infty$,  
and suppose that its hazard rate $\lambda(y) = f(y)/\overline{F}(y) \to 0$, as $y \to \infty$. Denote by $\lceil X \rceil$ the integer-valued random variable 
with distribution function $\lceil F \rceil (y) = \lceil F \rceil (\lfloor y \rfloor) = F(\lfloor y \rfloor)$, for all $y \ge 0$, that is, let $\lceil X \rceil$
be a discretised version of $X$. Then, since $\overline{F}$ is decreasing, 
\[
0 \le \frac{P\left(\lceil X \rceil = k \right)}{ \overline{\lceil F \rceil}(k-1) } = \frac{F(k)-F(k-1)}{\overline{F}(k-1)} = 
\int_{k-1}^{k} \lambda(t)\frac{\overline{F}(t)}{\overline{F}(k-1)}dt \le \int_{k-1}^k \lambda(t) dt, 
\]
which tends to $0$ as $k \to \infty$. Thus,  $\lceil X \rceil$ satisfies condition (\ref{t: condhazardrate}). 
Moreover, note that since $\lceil F \rceil(y) \le F(y)$ and $-\log \overline{F}(y) = \int_{0}^y \lambda(t) dt$ for all $y \ge0$, we have
\[
1 \le \frac{\overline{\lceil F \rceil} (y)}{\overline{F}(y)} \le \frac{\overline{F}(\lfloor y \rfloor)}{\overline{F}(\lfloor y \rfloor +1)} = 
\exp \left( \int_{\lfloor y \rfloor}^{\lfloor y \rfloor +1} \lambda(t) dt\right),
\]
which, again by the above condition on the hazard rate, tends to $1$ as $y \to \infty$. Thus, the two distribution functions $F$ and $\lceil F \rceil$ are 
\textit{tail-equivalent} (see, e.g., Definition 3.3.3 in \cite{Embrechts_et_al:1997}). 
By Proposition 1.19 in \cite{Resnick:1987} it then follows that if $F$ is in the maximum domain of attraction of an extreme value distribution $H$, then so is $\lceil F \rceil $, and
vice versa. For instance, let $X$ be Pareto distributed with parameters $\alpha, \phi >0$ with distribution function $F$ as in (\ref{d: Pareto}). Its hazard rate $\lambda(y) = \alpha/y$, for $y \ge \Phi$, 
vanishes as $y \to \infty$, as required. From Proposition \ref{t: cont_limit_maxima} we know that $F \in \mathrm{MDA}(\Phi_\alpha)$. By the above argument, we then also have that
$\lceil F \rceil \in \mathrm{MDA}(\Phi_\alpha)$.
\end{exa}

\subsection{Maxima of discrete random variables}\label{s: max_discrete}
In the previous subsection we discussed the non-existence of non-degenerate extremal limit laws for some well-known discrete distributions, like the binomial, geometric 
and Poisson. There is, however, a way to partially remedy this. \cite{Anderson_et_al:1997} and \cite{Nadarajah/Mitov:2002} determined limit laws for precisely these distributions by allowing
the distributional parameter (or one of them) to vary with the sample size $n$ at suitable rates. This means that they actually considered triangular arrays, 
\[
\begin{array}{lll}
m=1: & X_1^{(1)} & \stackrel{\phantom{\iid}}{\sim}\mathcal{P}_{\theta_1} \\
m=2: & X_1^{(2)}, X_2^{(2)} & \stackrel{\iid}{\sim}\mathcal{P}_{\theta_2}\\
m=3: & X_1^{(3)}, X_2^{(3)}, X_3^{(3)}& \stackrel{\iid}{\sim} \mathcal{P}_{\theta_3}\\
\vdots & \vdots & \vdots\\
m=n: & X_1^{(n)}, X_2^{(n)}, \ldots, X_n^{(n)}&\stackrel{\iid}{\sim} \mathcal{P}_{\theta_n},
\end{array}
\]
where $\mathcal{P}_{\theta_m}$ denotes a distribution with parameter, or collection of parameters, $\theta_m$. To keep notations simple, we omit the superscript. 
Moreover, for any $y \in \mathbb{R}$, let $\lfloor y \rfloor$ denote the integer part of $y$,
i.e. let $\lfloor y \rfloor = \max\{k \in \mathbb{Z}:\, k \le y\}$, and similarly, let $\lceil y \rceil = \min\{k \in \mathbb{Z}:\, k \ge y\}$.
For a discrete random variable $X$, taking values in $\mathbb{Z}$, we then 
naturally have $P(X=y)$ $=P(X = \lfloor y \rfloor)$ $=P(X=\lceil y \rceil)$ for all $y \in \mathbb{Z}$, whereas, for all $y \notin \mathbb{Z}$,
$P(X=y)=0$, and 
\begin{align}
\begin{split}\label{d: discretised_cdf}
P(X \le y) & = P(X < y) = P(X \le \lfloor y \rfloor) = P(X < \lceil y \rceil),\\
P(X > y) &= P(X \ge y) = P(X > \lfloor y \rfloor) = P(X \ge \lceil y \rceil).  
\end{split}
\end{align}
Having clarified notations, we now give the extremal limit results for the Poisson, binomial and geometric distributions, the first of which was proven 
by \cite{Anderson_et_al:1997} and the latter two by \cite{Nadarajah/Mitov:2002}. Proposition \ref{t: Max_Poi_lim} first uses convergence of the Poisson distribution with 
large parameter $\lambda$
to the normal distribution and then the fact that the normal distribution is in the maximum domain of attraction of the Gumbel distribution, as shown in Proposition 
\ref{t: cont_limit_maxima}(d).
Proposition \ref{t: Max_Bin_lim} proceeds similarly for the binomial distribution. 
\begin{prop}(Poisson distribution, Anderson et al., 1997)\label{t: Max_Poi_lim}
For each integer $n \ge 1$, let $X_1, \ldots, X_n$ be \iid Poisson random variables with parameter $\lambda=\lambda_n>0$. 
If $\lambda_n$ grows with $n$ such that $(\log n)^3$ $ = o(\lambda_n)$, 
then, for all $x \in \mathbb{R}$, 
\[
P\left( X_{(n)} \le \sqrt{\lambda_n} \alpha_n x + \lambda_n + \sqrt{\lambda_n}\beta_n\right) \to e^{-e^{-x}} = \Lambda(x), 
\]
as $n \to \infty$, where 
\[
\alpha_n = \frac{1}{\sqrt{2 \log n}}, \qquad \beta_n = \sqrt{2\log n} - \frac{\log \log n + \log 4\pi}{2\sqrt{2\log n}}. 
\]
\begin{flushright} $\qed$ \end{flushright} 
\end{prop}
\begin{prop}(Binomial distribution, Nadarajah and Mitov, 2002)\label{t: Max_Bin_lim}
For each integer $n \ge 1$, let $X_1, \ldots, X_n$ be \iid binomial random variables with number of trials $N=N_n \in \{1, 2, \ldots\}$ and fixed probability of success $p \in (0,1)$. 
If $N_n \to \infty$ as $n \to \infty$ such that $(\log n)^3 = o(N_n)$, 
then, for all $x \in \mathbb{R}$, 
\[
P\left( X_{(n)} \le \sqrt{p(1-p)N_n} \alpha_n x + pN_n + \sqrt{p(1-p)N_n}\beta_n\right) \to e^{-e^{-x}} = \Lambda(x), 
\]
as $n \to \infty$, where 
\[
\alpha_n = \frac{1}{\sqrt{2 \log n}}, \qquad \beta_n = \sqrt{2\log n} - \frac{\log \log n + \log 4\pi}{2\sqrt{2\log n}}. 
\]
\begin{flushright} $\qed$ \end{flushright}
\end{prop}
\noindent Note that the geometric distribution treated by \cite{Nadarajah/Mitov:2002} as a special case of the negative binomial is the shifted geometric distribution, 
which counts the number of trials until the first success. In contrast, the following proposition treats the geometric distribution which counts the number of 
failures before the first success. The normalising constants however remain the same as those used by \cite{Nadarajah/Mitov:2002}.
\begin{prop} \label{t: Max_geo_limit_NM}
(Geometric distribution) For each integer $n \ge 1$, let $X_1, \ldots, X_n$ be \iid geometric random variables with probability of success $p=p_n \in (0,1)$, probability mass function 
$P(X_1=k)=p(1-p)^k$ and cumulative distribution function 
$F(k)= 1-(1-p)^{k+1}$, for any $k \in \mathbb{Z}_+$. 
If $p_n \to 0$ as $n \to \infty$
, 
then, for all $x \in \mathbb{R}$,
\[
P\left( X_{(n)} \le \frac{\log n + x}{p_n}\right) \to e^{-e^{-x}} = \Lambda(x), \quad \textnormal{as }n \to \infty. 
\]
\end{prop}
\begin{proof}
Fix any $x \in \mathbb{R}$.
Due to (\ref{d: discretised_cdf}) and (\ref{p: max_Fn}), we have
\begin{align*}
&P\left( X_{(n)} \le \frac{\log n + x}{p_n}\right)\\
&= P\left( X_{(n)} \le \left\lfloor \frac{\log n + x}{p_n}\right\rfloor\right)
= F^n\left( \left\lfloor \frac{\log n + x}{p_n}\right\rfloor\right)\\
& = \left\{ 1- (1-p_n)^{\left\lfloor \frac{\log n + x}{p_n} \right\rfloor + 1}\right\}^n
= \left\{ 1-e^{\left( \left\lfloor \frac{\log n + x}{p_n} \right\rfloor + 1\right)\log(1-p_n)}\right\}^n,
\end{align*}
for any integer $n \ge e^{-x}$.
Using $z-1 \le \lfloor z \rfloor \le z$, for $z \in \mathbb{R}$, and 
$-p/(1-p) \le \log(1-p) \le -p$, for all $p \in (0,1)$, 
we obtain
\begin{equation}\label{p: geo_max_lim_bounds}
\left\{ 1-e^{-(\log n + x)}\right\}^n 
\le P\left( X_{(n)} \le \frac{\log n + x}{p_n}\right) 
\le \left\{ 1-e^{-\frac{\log n+x}{1-p_n}}\cdot (1-p_n)\right\}^n. 
\end{equation}
As $n \to \infty$ and $p_n \to 0$, both sides of
(\ref{p: geo_max_lim_bounds}) tend to $\exp\{-e^{-x}\} = \Lambda(x)$.
\end{proof}
\noindent The following proposition investigates the rate of convergence of the limit result from Proposition \ref{t: Max_geo_limit_NM} and suggests two improvements.
One way to reduce the error is to approximate by a discretised version of the Gumbel distribution, the other is to use different normalising constants.
\begin{prop}\label{t: max_geo_d_K}
(Geometric distribution) For each integer $n \ge 1$, let $X_1, \ldots, X_n$ be \iid geometric random variables with success probability $p_n\in (0,1)$, 
failure probability $q_n=1-p_n$, probability mass function 
$P(X_1=k)=p_nq_n^k$ and survival function $\overline{F}(k)= q_n^{k+1}$, for any $k \in \mathbb{Z}_+$. Then:\\
(a) (Approximation by a discretised Gumbel distribution) For all $k \in \mathbb{Z}_+$ and for all $k^\star \in \mathbb{R}$ defined by $k^\star = -\log n + k \log(1/q_n)$,
\begin{equation}\label{(a)Gumbel}
\left| P\left( X_{(n)} < \frac{\log n + k^\star}{\log(1/q_n)} \right) -  e^{-e^{-k^\star}}\right| \le \frac{\log n}{q_n n} + \frac{1}{n} =:\delta_{\mathrm{PoiAppr}}.
\end{equation}
(b) (Approximation by a Gumbel distribution) For all $x \in \mathbb{R}$,
\[
\left| P\left( X_{(n)} < \frac{\log n + x}{\log(1/q_n)} \right) -  e^{-e^{-x}}\right| \le \delta_{\mathrm{PoiAppr}} +e^{-1}\log(1/q_n) =: 
\delta_{\mathrm{Cont}}.
\]
(c) (Using the normalising constants from Nadarajah and Mitov, 2002)
\[
\left| P\left( X_{(n)} < \frac{\log n + x}{1-q_n} \right) -  e^{-e^{-x}}\right| \le \delta_{\mathrm{Cont}} +\frac{1-q_n}{2q_n}\left( \log^2 n +e^{-1}\right). 
\]
\end{prop}
\noindent Note that the failure probability $q_n$ need not vary with the sample size $n$ for approximation by a discretised Gumbel distribution. The error bound is sharp for any constant
$q_n \equiv q \in (0,1)$, showing clearly that it makes more sense to approximate a discrete distribution by another discrete distribution than by a continuous one, as there is
no need to add an extra error as in (b). 

The extra error in (b), $e^{-1}\log(1/q_n)$, is the discretisation error that
arises when going from the Gumbel concentrated on the lattice of
points $k^\star$ to the continuous Gumbel distribution over 
$\mathbb{R}$. It dominates the overall error in (b) unless 
$q_n$ tends to $1$ fast enough as $n \to \infty$, that is, unless
$1-q_n = O(\log(n)/n)$, in which case the discretisation error is of 
the same order as the first error term from (a).

Part (c) shows that the choice of normalising constants, more precisely, of the scaling by $p_n$ in Proposition \ref{t: Max_geo_limit_NM},
is far from optimal. In order for the approximation in (c) to be good we require $p_n = o(1/\log^2 n)$.
Its being a stronger condition than the one for the asymptotic 
result from Proposition \ref{t: Max_geo_limit_NM} is justified by (c) also being a stronger result in the sense that it gives a uniform bound. The error in (c) 
is of the same order as the error in (a) only if $1-q_n = O( 1/(n\log n))$. 
\begin{proof} For ease of notation we omit the subscript $n$. 
(a) Let $A=[y, \infty)$ for any choice of $y \ge 0$. Then, by (\ref{d: discretised_cdf}), $P(X_1 \in A) = q^{\lceil y \rceil}$, and, setting $k:=\lceil y \rceil \in \mathbb{Z}_+$,
Theorem \ref{t: Poster_univ} gives
\begin{equation}\label{p: first_bound_geo}
\left| P\left( X_{(n)} < k \right) - e^{-nq^k} \right| \le q^k. 
\end{equation}
With $k^\star \in \mathbb{R}$ chosen such that $k=(\log n + k^\star)/\log(1/q)$, we then have
\begin{equation}\label{p: Max_Geo_error}
\left| P\left( X_{(n)} < \frac{\log n+k^\star}{\log(1/q)}\right) - e^{-e^{-k^\star}} \right| \le \frac{e^{-k^\star}}{n}. 
\end{equation}
In order to find a uniform bound for all $k \in \mathbb{Z}_+$, choose $x_0:=x_{0n}:= - \log \log n$. Then, for all $k$ such that $k^\star \ge x_0$, we have
$\exp{(-k^\star)}/n \le \exp{(-x_0)}/n = \log (n)/n$, whereas for $k$ such that $k^\star \le x_0$, we may 
bound the error in (\ref{p: Max_Geo_error}) 
by further adding the Gumbel distribution to the error at $m^\star$, where 
$m := \lfloor y_0 \rfloor := \lfloor (\log n + x_0)/\log(1/q) \rfloor$,
i.e. 
\begin{align}\label{p: bound_xk_vs.xstar}
\left| P\left( X_{(n)} < \frac{\log n+k^\star}{\log(1/q)}\right) - e^{-e^{-k^\star}} \right|  
&\le \frac{e^{-m^\star}}{n}  + e^{-e^{-m^\star}}  
= e^{-m \log(1/q)} + e^{-e^{-m^\star}}\nonumber \\
&\le \frac{e^{-x_0}}{qn} + e^{-e^{-x_0}} \le \frac{\log n}{qn} + \frac{1}{n}\,,
\end{align}
where we used $m \ge y_0 -1$ in the second inequality. See Figure \ref{f: Max_geo_dK} for a sketch.\\
\begin{figure}[!ht]
\begin{center}{\footnotesize \input{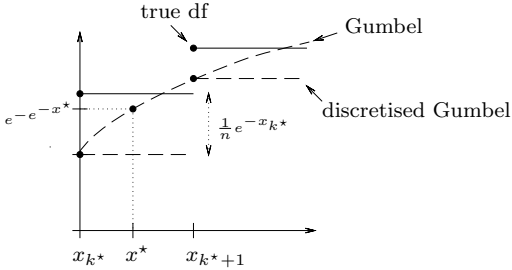}}\end{center}
\caption{At all values $k^\star \le m^\star \le x_0$, the error cannot exceed the sum of the bound $e^{-m^\star}/n$ on the difference
between the two distributions functions and the height $\exp\{-e^{-x_0}\}$ of the
Gumbel distribution.}
\label{f: Max_geo_dK}
\end{figure} 
(\ref{p: bound_xk_vs.xstar}) provides a bound for all $k^\star$,
and thus
\[
\left| P\left( X_{(n)} < \frac{\log n +k^\star}{\log(1/q)}\right) - e^{-e^{-k^\star}} \right| \le \frac{\log n}{qn} + \frac{1}{n}.
\]
(b) Let $x= \log(1/q)y-\log n$. 
By adding and subtracting $\exp\{-e^{-x}\}$ 
into (\ref{(a)Gumbel}), and noting that,
since $y \le \lceil y \rceil = k$, we have $x \le k^\star$ and
$\exp\{-e^{-k^\star}\}-\exp\{-e^{-x}\} \ge 0$. We thus obtain
\[
\left| P\left( X_{(n)} < \frac{\log n +x}{\log(1/q)} \right)
- e^{-e^{-x}} \right| \le 
\delta_{\mathrm{PoiAppr}} + e^{-e^{-k^\star}} - e^{-e^{-x}},
\]
where 
\begin{align*}
e^{-e^{-k^\star}} - e^{-e^{-x}}
&\le \int_{x}^{k^\star}e^{-t}e^{-e^{-t}}dt 
\le e^{-1}(k^\star - x)\\
&= e^{-1}\log(1/q)(\lceil y \rceil - y) 
\le e^{-1}\log(1/q).
\end{align*}
(c) 
From (a) and (b) we have
\begin{equation}\label{p: max_geo_cont}
\left| P\left( X_{(n)} < \frac{\log n +x}{ \log(1/q)}\right)
-e^{-e^{-x}}\right| \le \frac{e^{-x}}{n} + e^{-1}\log(1/q).  
\end{equation}
Choose $x' \ge -\log n$ such that 
\[
y = \frac{\log n + x}{\log(1/q)} = \frac{\log n + x'}{1-q}.
\]
By adding and subtracting $\exp\{-e^{-x'}\}$ into (\ref{p: max_geo_cont}) and observing that
$x > x'$ since $\log(1/q)$ $ > 1-q$, we then obtain
\[
\left| P\left( X_{(n)} < \frac{\log n + x'}{1-q}\right) - e^{-e^{-x'}}\right| \le \frac{e^{-x'}}{n}  + e^{-1}\log(1/q) + e^{-e^{-x}} - e^{-e^{-x'}}.
\]
For the latter error term we find
\begin{align*}
e^{-e^{-x}} - e^{-e^{-x'}} &= e^{-e^{-x}} \left[ 1-e^{-(e^{-x'} - e^{-x})}\right] \le e^{-x'}\left[ 1-e^{-(x-x')}\right] \\
& \le e^{-x'} (x-x') = e^{-x'}\left[\log (1/q) - (1-q)\right]y, 
\end{align*}
where we used $\exp\{-e^{-x}\} \le 1$ in the first inequality and $1-e^{-z} \le z$ for $z \ge 0$ in both inequalities. 
Note that use of the definition of the logarithm
and the geometric series give
\begin{align*}
&\log(1/q)-(1-q) \\
&= \sum_{j=2}^\infty \frac{(1-q)^j}{j} \le \sum_{j=2}^\infty \frac{(1-q)^j}{2} 
= \frac{1}{2}\left[ \sum_{j=0}^\infty (1-q)^j - 1 - (1-q)\right]\\
&= \frac{1}{2} \left[ \frac{1}{q} -2 + q\right]= \frac{(1-q)^2}{2q}.
\end{align*}
Then, 
\begin{align*}
e^{-e^{-x}} - e^{-e^{-x'}} &\le  \frac{(1-q)^2y}{2q}e^{-x'} 
= \frac{1-q}{2q} (\log n + x')e^{-x'}\\
&\le \frac{1-q}{2q} (e^{-x'}\log n  + e^{-1}).
\end{align*}
Thus,  
\begin{multline*}
\left| P\left( X_{(n)} < \frac{\log n + x'}{1-q}\right)
- e^{-e^{-x'}}\right|\\
\le \frac{e^{-x'}}{n}  + e^{-1}\log(1/q) + 
\frac{1-q}{2q} \left(e^{-x'}\log n  + e^{-1}\right).
\end{multline*}
For a uniform bound over all $x'$, we choose $x_0$ as before in (a) and (b), and obtain, with an analogous argument, the overall error bound
\[
\frac{\log n}{n}  + e^{-1}\log(1/q) + \frac{1-q}{2q}\left(\log^2 n + e^{-1}\right) + \frac{1}{n}.
\]  
\end{proof}

\section{Poisson process approximation for MPPE's}\label{Sec: Poi_proc_approc_MPPE}
Theorem \ref{t: Poster_univ} not only gives information on distributional approximation for the maximum 
$X_{(n)}$ of \iid random variables $X_1, \ldots, X_n$ with distribution function $F$, as in (\ref{t: error_max}), but also on approximation of the laws of each of the order statistics of this sample. 
Denote by $X_{(k)}$ the $k$-th order statistic of the sample $X_1, \ldots, X_n$, i.e. order the sample as follows: 
\[
\min_{1\le i \le n} X_i = X_{(1)} \le X_{(2)} \le \ldots \le X_{(n)} = \max_{1 \le i \le n} X_i.
\]
The number of points exceeding a threshold $u_n$ can be related to each order statistic by
\[
\left\{ \sum_{i=1}^n I_{\{X_i > u_n\}} \le k\right\} = \{X_{(n-k)} \le u_n\}, 
\]
and Theorem \ref{t: Poster_univ} gives
\[
\left| P(X_{(n-k)} \le u_n) - e^{-n\overline{F}(u_n)}\frac{(n\overline{F}(u_n))^k}{k!} \right| \le \overline{F}(u_n). 
\]

Several further generalisations can be achieved by using point processes. 
One of them incorporates a way to specify which of the $X_i$'s are the ones exceeding the threshold.
This is not immediately given by Theorem \ref{t: Poster_univ}.  
The object that we are studying, $\sum_{i=1}^n I_{\{X_i > u_n\}}$, needs to be generalised so as to additionally pinpoint the indices of the random variables exceeding $u_n$. 
This can be attained by using \textit{point processes of exceedances} (PPE's). Classically, a PPE is a point process of the form 
\begin{equation}\label{d: PPE}
N_n = \sum_{i=1}^n I_{\{X_i >u\}} \delta_{in^{-1}},
\end{equation}
that lives on the state space $E=(0,1]$. If applied to the entire state space, this point process recovers the total number of extreme points, i.e.
\[
N_n((0,1]) = \sum_{i=1}^n I_{\{X_i >u\}},
\]
but if applied to a measurable subset $B \subset (0,1]$, $N_n(B)$ gives only the random number of $X_i$'s that exceed $u_n$ and for which $in^{-1} \in B$. For instance,
suppose that the $X_i$'s describe the outcomes of $n$ identical and independent random experiments that are realised at consecutive time points $i$. The random number
of extreme points that occur after time $t >0$ is then given by $N_n((tn^{-1}, 1])$.
Now what do we know about the distribution of $N_n$?
\cite{Embrechts_et_al:1997} on p.238 states the following theorem for weak convergence of point processes of exceedances to a Poisson process:
\begin{thm}
Suppose that $(X_n)$ is a sequence of \iid random variables with common distribution function $F$. Let $(u_n)$ be threshold values such that for some 
$\tau \in (0, \infty)$, 
\[
n\overline{F}(u_n) = \mathbb{E} \sum_{i=1}^n I_{\{X_i > u_n\}} \to \tau, \textnormal{ for }n\to \infty. 
\]
Then the point processes of exceedances $N_n$, as defined in (\ref{d: PPE}), converge weakly in $M_p(E)$ to a homogeneous Poisson process $N$ on $E=(0,1]$ 
with intensity $\tau$, i.e. $N$ is $\mathrm{PRM}(\tau|\cdot|)$, where $|.|$ denotes Lebesgue measure on $E$. 
\end{thm}


Another kind of point processes studied in EVT are \textit{marked point processes} (MPP's) of the form $\sum_{i=1}^n \delta_{X_i}$, that live on the state space $E$
of the random variables $X_i$. An MPP gives a random configuration of points in space and counts the number of points in any measurable subset of the state space
that it is applied to. The MPP converges weakly in $M_p(E)$ to a Poisson process with mean measure $\bl$ if and only if its mean measure converges vaguely
to $\bl$, as $n \to \infty$ (see Chapter 3 in \cite{Resnick:1987} for more details).

We introduce yet another kind of point process, that we call \textit{marked point process of exceedances} (MPPE) and that we define as follows:
\begin{equation}\label{d: MPPEu}
\Xi_{u,n} := \sum_{i=1}^n I_{\{X_i >u\}} \delta_{X_i}.
\end{equation} 
Though the MPPE does not mark the points that exceed $u_n$ as the PPE does, it contains more information relevant to the study of extreme values than an MPP,
as it is not only a random configuration of points in space, but specifically a random configuration of points exceeding a threshold. From now on, we concentrate on
MPPE's, as they are better suited to our purposes than PPE's or MPP's. 

The state space $E \subseteq \mathbb{R}$ of $\Xi_{u,n}$ is 
the set of values $X(\Omega)$ taken by 
the \iid random variables $X_1, \ldots, X_n$.
Let $\mathcal{E}=\mathcal{B}(E)$.
For a fixed set $B \in \mathcal{E}$, $\Xi_{u,n}(B)$ gives the random number of 
$X_i$'s in $(u, \infty) \cap B$, whereas for a fixed $\omega \in \Omega$, $\Xi_{u,n}(\omega) = \sum_{ i \in\{1, \ldots,n\}:\,  x_i >u} \delta_{x_i}$ 
gives the point configuration of the realisations $X_i(\omega) = x_i > u$. 
We can generalise MPPE's by considering a set $A \in \mathcal{E}$ of an arbitrary shape instead of a 
threshold $u$: 
\begin{equation}\label{d: MPPE}
\Xi_A := \Xi_{A,n} := \sum_{i=1}^n I_{\{X_i \in A\}} \delta_{X_i}.
\end{equation}
$A$ should be chosen such that points $x_i$ lying in it can be considered to be extreme points. This generalisation makes sense mostly for higher-dimensional
points, i.e. points in $\mathbb{R}^d$, $d \ge 2$, as it gives more flexiblity to the choice of region in which it makes sense to consider points to be extreme. 
For instance, to obtain a multivariate analogue to (\ref{d: MPPEu}) with componentwise thresholds,
we can simply choose $A = (u_1, \infty)\times \dots \times(u_d, \infty)$, for $u_1, \ldots, u_d \in \mathbb{R}$. Another example would be to choose $A$ as the complement
of a disc with a `large' radius $r$ centred in some origin.  
The mean measure of $\Xi_A$ on $\mathcal{E}$ is given by
\[
\bl(\,.\,):= \bl_A(\, . \,) :=\mathbb{E}\Xi_A(\,.\,) = \sum_{i=1}^n P(X_i \in A \cap \, .\,) =  nP(X \in A \cap \,.\,).
\]
Note that $\Xi_A$ is a finite point process and that $\bl$ is a finite measure. 

As both PPE's and MPP's asymptotically, under certain conditions, behave like Poisson processes, the same can be expected of the MPPE's. 
We can indeed readily apply Theorem \ref{t: Michel} from Section \ref{s: Michel_MPP} to MPPE's to obtain a result in this vein. Our result, however, 
is stronger than a mere limit theorem: Proposition \ref{t: MyMichel} below gives an estimate
of the error of the approximation, in the total variation distance, of the law of $\Xi_{A}$ by that of a Poisson process with mean measure $\mathbb{E}\Xi_A$.
\begin{prop} \label{t: MyMichel}
For each integer $n \ge 1$, let $X_1, \ldots, X_n$ be \iid copies of an $E$-valued random variable $X$. 
For a fixed set $A \in \mathcal{E}$, 
let $\Xi_A = \sum_{i=1}^n I_{\{X_i \in A\}} \delta_{X_i}$ be the marked point
process of points in $A$ and let $W_A = \sum_{i=1}^n  I_{\{X_i \in A\}}$ denote the random number of points in $A$. Then,
\begin{equation}\label{t: MyMichel_result}
d_{TV}(\mathcal{L}(\Xi_A), \mathrm{PRM}(\mathbb{E}\Xi_A)) \le d_{TV}(\mathcal{L}(W_A), \mathrm{Poi}(\mathbb{E}W_A))
\le P(X \in A). 
\end{equation}
\end{prop}
\begin{proof} 
Let $P_A = \mathcal{L}(X|X \in A)$ and 
define an \iid random sample $X'_1, \ldots, X'_n$ with common distribution $P_A$ that is independent of the sample $X_1, \ldots, X_n$.
Then the process $\sum_{i=1}^n I_{\{X_i \in A\}}\delta_{X'_i}$ has the same distribution as the process of interest $\Xi_A$.   
Note that due to the independence of the samples $X_1, \ldots, X_n$ and $X'_1, \ldots, X'_n$, the process $\sum_{i=1}^n I_{\{X_i \in A\}}\delta_{X'_i}$
is distributed as $\sum_{j=1}^{W_A} \delta_{Z'_j}$, where the $Z'_j$ are independent, have common distribution $P_A$, and are independent of $W_A$. 
Furthermore, note that a $\mathrm{PRM}(\mathbb{E}\Xi_A)$ can be realised as $\sum_{j=1}^{W^\star} \delta_{Z'_j}$, where $W^\star \sim \mathrm{Poi}(\mathbb{E}W_A)$ is independent of
the $Z'_j$. It then follows from the proof of Theorem \ref{t: Michel} that
\begin{align*}
d_{TV}(\mathcal{L}(\Xi_A), \mathrm{PRM}(\mathbb{E}\Xi_A)) 
\le d_{TV} (\mathcal{L}(W_A), \mathrm{Poi}(\mathbb{E}W_A)),
\end{align*}
where the latter is at most $ P(X \in A)$ by Theorem \ref{t: Poster_univ}.
\end{proof}
\noindent The error of Poisson process approximation for an MPPE is thus the same as the error for Poisson approximation for the number of extreme points. 
This means that we may (more or less) recover the results from Propositions \ref{t: Max_exp} - \ref{t: Max_Cauchy} and \ref{t: max_geo_d_K}. 
However, in order to do this, we first need to subject the random variables $X_i$ to a normalisation as we did in these propositions. More precisely, define the normalised random variable
\begin{equation}\label{d: normalised_rv}
X^\star := \frac{X-b_n}{a_n},
\end{equation}
where $a_n, b_n \in \mathbb{R}$ with $a_n >0$, and let $X^\star_1, \ldots, X^\star_n$
be \iid copies of $X^\star$, with state space $E^\star = X^\star(\Omega)$. Let 
$\mathcal{E}^\star = \mathcal{B}(E^\star)$ and let $A^\star$ be the accordingly normalised version of $A^\star$, i.e. $A^\star = ((A-b_n)/a_n) \in \mathcal{E}^\star$. 
For instance, for $E= [0, \infty)$ and $A = [u, \infty)$ for $u \ge 0$, we have $E^\star = [-b_n/a_n, \infty)$ and $A^\star = [u^\star, \infty)$, 
where $u^\star = (u - b_n)/a_n \ge -b_n/a_n$. 
The distribution function of $X^\star$ is given by
\begin{equation}\label{d: Fstar}
F^\star(x)= P(X^\star \le x) = P(X \le a_nx+b_n) = F(a_nx+b_n), \quad  x \in E^\star.
\end{equation}
For any $A \in E$ and $A^\star \in E^\star$, we thus have
\begin{equation}\label{d: normalised_W_A}
P(X \in A) = P(X^\star \in A^\star) \quad \textnormal{and } \quad 
W_A = \sum_{i=1}^n I_{\{ X^\star_i \in A^\star\}} =: W^\star_{A^\star}.
\end{equation}
Similarly, we have
$
\Xi_A = \sum_{i=1}^n I_{\{ X^\star_i \in A^\star\}} \delta_{X_i}$,
which still lives on the state space $E$. We define a normalised version of this process,
\begin{equation}\label{d: normalised_MPPE}
\Xi^\star_{A^\star} := \sum_{i=1}^n I_{\{ X_i^\star \in A^\star\}} \delta_{X^\star_i},
\end{equation}
which has state space $E^\star$ and mean measure
\begin{equation}\label{d: normalised_bl}
\bl^\star(\,.\,) := \mathbb{E}\Xi^\star_{A^\star}(\,.\,) = nP(X^\star \in A^\star \cap \,.\,), 
\end{equation}
on $\mathcal{E}^\star$. Obviously, $\Xi_A \neq \Xi^\star_{A^\star}$. However, for any $R \in \mathcal{M}_p(E)$ and its normalised version $R^\star \in \mathcal{M}_p(E^\star)$, we have $P(\Xi_A \in R) = P(\Xi^\star_{A^\star} \in R^\star)$.
It follows that 
\begin{equation}\label{d: normalised_dTV}
d_{TV}(\mathcal{L}(\Xi_A), \mathrm{PRM}(\bl) )
 = d_{TV}(\mathcal{L}(\Xi^\star_{A^\star}), \mathrm{PRM}(\bl^\star)),
\end{equation}
and (\ref{t: MyMichel_result}) may equivalently be expressed as
\begin{equation}\label{t: MyMichel_normalised}
d_{TV}\left(\mathcal{L}(\Xi^\star_{A^\star}), \mathrm{PRM}(\bl^\star)\right) 
\le d_{TV}\left(\mathcal{L}(W^\star_{A^\star}), \mathrm{Poi}(nP(X^\star \in A^\star))\right) \le P(X^\star \in A^\star).
\end{equation}
We now use (\ref{t: MyMichel_normalised}) to obtain process analogues of Propositions \ref{t: Max_exp} - \ref{t: Max_Cauchy} and \ref{t: max_geo_d_K}, that is, we determine 
error bounds for Poisson process approximation of MPPE's with marks $X_1, \ldots, X_n$ following well-known distributions. Section \ref{s: MPPE_cont} discusses the case
of continuous marks and treats all distributions listed in Proposition \ref{t: cont_limit_maxima}. In Section \ref{s: MPPE_geo} we discuss an example of an MPPE 
with discrete marks. More precisely, we suppose that $X_1, \ldots, X_n$ follow the geometric distribution. We first approximate the MPPE with 
geometric marks by a Poisson process on $\mathbb{Z}_+$. Then, as processes with a continuous intensity are more practicable, we further try to approximate the MPPE by
a Poisson process with continuous intensity over $\mathbb{R}_+$. To achieve this, we use the weaker $d_2$-metric and the results from Section \ref{Sec: Improved_rates}. Throughout Section \ref{Sec: Poi_proc_approc_MPPE}, we set, for simplicity, very strong assumptions on our MPPE's, as we require \iid marks and \iid indicators. Our main efforts therefore lie, not so much in determining an error bound for basic Poisson approximation, but rather in determining error bounds between two Poisson processes, if necessary. Section \ref{s: Remarks_on_choice} gives a short discussion on what might happen for different assumptions.
\subsection{Application to MPPE's with continuous marks}\label{s: MPPE_cont}
Suppose that the distribution function $F$ of the \iid $E$-valued random variables $X$, $X_1, \ldots, X_n$ is absolutely continuous with probability density function $f(y)$, 
$y \in E$. We may then define an intensity function $\lambda(y):=nf(y)$ for all $y \in E$, and write the intensity measure $\bl$ of $\Xi_A$ as 
\[
\bl(B) = \int_{A \cap B} \lambda(y)dy, \quad \textnormal{ for any } B \in \mathcal{E}.
\]
Due to (\ref{d: Fstar}), the probability density function of $X^\star$ is given by 
\begin{equation}\label{d: density_normalised}
f^\star(x) = \frac{d}{dx} F(a_n x + b_n) = a_nf(a_nx +b_n), 
\end{equation}
and we may write the intensity measure $\bl^\star$ of $\Xi^\star_{A^\star}$ as
\begin{equation}\label{d: intensity_function_normalised}
\bl^\star(\,B^\star\,) =\int_{A^\star \cap B^\star} \lambda^\star(x)dx, \quad \textnormal{ for any } B^\star \in \mathcal{E}^\star,
\end{equation}
where $\lambda^\star(x):= nf^\star(x)$. Propositions \ref{t: MPPE_Exp_dTV} - \ref{t: MPPE_Cauchy_dTV} below determine, on the one hand, the intensity functions $\lambda^\star$ of 
MPPE's whose marks follow well-known continuous distributions, and, on the other hand, the errors in total variation arising when approximating these
MPPE's by a Poisson process with intensity function $\lambda^\star$. Knowledge of the intensity function $\lambda^\star$ 
simplifies the computation of $\bl^\star(B^\star)=nP(X^\star \in B^\star)$ for measurable sets $B^\star \subseteq E^\star$ 
of arbitrary shape, as well as of the error $P(X^\star \in A^\star) = n^{-1}\bl^\star(A^\star)$ that we obtain
from Proposition \ref{t: MyMichel}, or, equivalently, from (\ref{t: MyMichel_normalised}). 
For each of the propositions below, we choose $A^\star = [u^\star,x^\star_{F^\star}] \cap E^\star$. 
Then, the smaller the choice of the threshold $u^\star$, the bigger the error $P(X^\star \in A^\star) = P(X^\star \ge u^\star)$,
and the worse the approximation by a Poisson process. This of course exactly mirrors Poisson approximation for the number of points exceeding $u^\star$, for which 
we have the same error $P(X^\star \ge u^\star)$. Put in another way, the  distribution of the number of points exceeding $u^\star$
is binomial with success probability precisely equal to $P(X^\star \ge u^\star)$. 
The smaller the threshold, the bigger the success
probability, and the worse Poisson approximation which requires a success probability tending to zero.  
We thus aim for a high threshold $u^\star$ and consequently a small number of exceedances. For each of the propositions below we will
discuss suitable choices of $u^\star$ depending on the sample size $n$.

\begin{prop}\label{t: MPPE_Exp_dTV}
(Exponential distribution) For each integer $n \ge 1$, let $X_1, \ldots, X_n$ be \iid exponential random variables with parameter $\lambda >0$. 
Define the normalised random variables $X_i^\star = \lambda X_i - \log n$, $i = 1, \ldots, n$, taking values in $E^\star=[-\log n, \infty)$. 
Let $A^\star = [u^\star, \infty)$ for any choice of $u^\star \in E^\star$, 
and let $W^\star_{A^\star}$ and $\Xi^\star_{A^\star}$
be defined as in (\ref{d: normalised_W_A}) and (\ref{d: normalised_MPPE}), respectively. Then the mean measure of $\Xi^\star_{A^\star}$ is given by 
\[
\bl^\star(B^\star) = \int_{A^\star \cap B^\star} e^{-x} dx, \quad   \textit{ for any } B^\star \in \mathcal{E}^\star, 
\]
and
\[
d_{TV}\left(\mathcal{L}(\Xi^\star_{A^\star}), \mathrm{PRM}(\bl^\star)\right) \le d_{TV}( \mathcal{L}(W^\star_{A^\star}), \mathrm{Poi}(e^{-u^\star}))\le \frac{e^{- u^\star}}{n}. 
\]
\end{prop}
\begin{proof}
We have $f(y) = \lambda e^{-\lambda y}$ for all $y \ge 0$, $a_n=\lambda^{-1}$ and $b_n = \lambda^{-1}\log n$. By (\ref{d: density_normalised})
and (\ref{d: intensity_function_normalised}), $f^\star(x)= e^{-x}/n$ and $\lambda^\star(x)= e^{-x}$ for all $x \ge -\log n$, and 
$\bl^\star(B^\star) = \int_{A^\star \cap B^\star} \lambda^\star (x) dx$ for any $B^\star \in \mathcal{B}([-\log n, \infty))$. By
(\ref{t: MyMichel_normalised}), 
\[
d_{TV}(\mathcal{L}(\Xi^\star_{A^\star}), \mathrm{PRM}(\bl^\star)) \le P(X^\star_1 \ge u^\star) = \frac{e^{-u^\star}}{n}. 
\]
\end{proof}
\noindent The expected number of exceedances of the threshold $u^\star = u^\star_n$ is $e^{-u^\star}$, whereas the error of the approximation 
in the total variation distance is $e^{-u^\star}/n$. Thus, the lower we set the threshold $u^\star$ with respect to the sample size, the more exceedances we will expect and the bigger the error of the 
approximation by a Poisson process will be. For instance, for $u^\star_n =  -\log \log n$, the error estimate is $\log(n)/n$ and 
we expect about $\log n$ exceedances. Round $\log n$ to its nearest integer value $[ \log n ]$. The points of the point process $\Xi^\star_{A^\star}$ are then
(more or less) the realisations of the $[\log n]$ biggest normalised order statistics $X^\star_{(n- [\log n] + 1)}, \ldots, X_{(n-1)}, X_{(n)}$ of the sample.
For $u^\star_n \equiv 0$ on the other hand, we expect only $1$ threshold exceedance among $n$ random variables and the single expected point of
$\Xi^\star_{A^\star}$ is thus the realisation of the maximum $X_{(n)}$ of the sample. The
error estimate is of size $1/n$ and thereby decreases more rapidly as the sample size $n$ increases. 
\begin{prop}\label{t: MPPE_Par_dTV}
(Pareto distribution) For each integer $n \ge 1$, let $X_1,$ $\ldots,$ $X_n$ be \iid Pareto random variables with shape and scale parameters $\alpha, \phi >0$. 
Define the normalised random variables $X_i^\star = \phi^{-1} n^{-1/\alpha}X_i$, $i = 1, \ldots, n$, taking values in $E^\star=[n^{-1/\alpha}, \infty)$.
Let $A^\star = [u^\star, \infty)$ for any choice of $u^\star \in E^\star$, 
and let $W^\star_{A^\star}$ and $\Xi^\star_{A^\star}$
be defined as in (\ref{d: normalised_W_A}) and (\ref{d: normalised_MPPE}), respectively. Then the mean measure of $\Xi^\star_{A^\star}$ is given by
\[
\bl^\star(B^\star) = \int_{A^\star \cap B^\star} \alpha x^{-\alpha-1} dx, \quad   \textit{ for any } B^\star \in \mathcal{E}^\star, 
\]
and
\[
d_{TV}\left(\mathcal{L}(\Xi^\star_{A^\star}), \mathrm{PRM}(\bl^\star)\right) 
\le d_{TV}\left( \mathcal{L}(W^\star_{A^\star}), \mathrm{Poi}\left(\frac{1}{u^{\star{\alpha}}}\right)\right)
\le \frac{1}{n{u^\star}^{\alpha}}\,. 
\]
\end{prop}
\begin{proof}
We have $f(y) = \alpha \phi^\alpha y^{-\alpha -1}$ for all $y \ge \phi$, $a_n=\phi n^{1/\alpha}$ and $b_n \equiv 0$. By (\ref{d: density_normalised})
and (\ref{d: intensity_function_normalised}), $f^\star(x)= \alpha n^{-1} x^{-\alpha - 1}$ and $\lambda^\star(x)= \alpha x^{-\alpha -1}$ for all $x \ge n^{-1/\alpha}$, and 
$\bl^\star(B^\star) = \int_{A^\star \cap B^\star} \lambda^\star (x) dx$ for any $B^\star \in \mathcal{B}([n^{-1/\alpha}, \infty))$. By
(\ref{t: MyMichel_normalised}), 
\[
d_{TV}(\mathcal{L}(\Xi^\star_{A^\star}), \mathrm{PRM}(\bl^\star)) \le P(X^\star_1 \ge u^\star) = \frac{1}{n{u^\star}^{\alpha}}\,. 
\]
\end{proof}
\noindent For roughly $\log n$ expected threshold exceedances among the $n$ \iid Pareto random variables, and an error of order $\log(n)/n$, 
we need to choose $u^\star = u^\star_n = (\log n)^{-1/\alpha}$. Similarly, for only $1$ threshold exceedance (by the maximum of the Pareto variables) and a rather smaller
error $1/n$, we would have to set $u^\star_n \equiv 1$.
\begin{prop}\label{t: MPPE_U_dTV}
(Uniform distribution) For each integer $n \ge 1$, let $X_1,$ $\ldots,$ $X_n$ be \iid uniform random variables with parameters $a, b \in \mathbb{R}$ such that $a < b$. 
Define the normalised random variables $X_i^\star = -n(b-X_i)/(b-a)$, $i = 1, \ldots, n$, taking values in $E^\star=[-n, 0)$.
Let $A^\star = [u^\star, 0)$ for any choice of $u^\star \in E^\star$, 
and let $W^\star_{A^\star}$ and $\Xi^\star_{A^\star}$
be defined as in (\ref{d: normalised_W_A}) and (\ref{d: normalised_MPPE}), respectively. Then the mean measure of $\Xi^\star_{A^\star}$ is given by
\[
\bl^\star(B^\star) = \int_{A^\star \cap B^\star} 1dx, \quad   \textit{ for any } B^\star \in \mathcal{E}^\star, 
\]
and
\[
d_{TV}\left(\mathcal{L}(\Xi^\star_{A^\star}), \mathrm{PRM}(\bl^\star)\right) 
\le d_{TV} (\mathcal{L}(W^\star_{A^\star}), \mathrm{Poi}(-u^\star))
\le \frac{-u^\star}{n}\,. 
\]
\end{prop}
\begin{proof}
We have $f(y) = 1/(b-a)$ for all $y \in [a,b)$, and we use the normalisation
$y = x (b-a)/n + b$, where $x \in [-n, 0)$. Then 
\[
f^\star(x) = \frac{d}{dx} F\left(\frac{b-a}{n}\,x + b\right) = f\left(\frac{b-a}{n}\,x + b\right)\,\frac{b-a}{n} = \frac{1}{n}
\]
and $\lambda^\star(x) = 1$, for all $x \in [-n, 0)$.
Moreover, $\bl^\star(B^\star) = \int_{A^\star \cap B^\star} \lambda^\star (x) dx$ for any $B^\star \in \mathcal{B}([-n, 0))$.
By
(\ref{t: MyMichel_normalised}), 
\[
d_{TV}(\mathcal{L}(\Xi^\star_{A^\star}), \mathrm{PRM}(\bl^\star)) 
\le d_{TV} (\mathcal{L}(W^\star_{A^\star}), \mathrm{Poi}(-u^\star))
\le P(X^\star_1 \ge u^\star) = \frac{-u^\star}{n}\,. 
\]
\end{proof}
\noindent Here, e.g. $u^\star_n = -\sqrt{n}$ will lead to $\Xi^\star_{\As}$ capturing roughly the $\sqrt{n}$ biggest order 
statistics and an error estimate of size $1/\sqrt{n}$, whereas the choice $u_n^\star = - \log \log n$ will give $\log n$ 
expected exceedances and the error estimate $\log(n)/n$.

The intensity function of the MPPE with \iid normal marks is given by $\lambda^\star_n(x)=na_n \varphi(a_n x + b_n)$, 
where $a_n$, $b_n$ are the norming constants defined in (\ref{t: an_bn_N}). As seen in Proposition \ref{t: Max_normal} (c), this intensity function may be 
approximated by $e^{-x}$ as $n \to \infty$. With $\As = [u^\star, \infty)$, we then expect about $e^{-u^\star}$ threshold exceedances. 
Proposition \ref{t: MPPE_Normal_dTV} below gives estimates for the two steps involved in approximating $\mathcal{L}(\Xi^\star_{\As})$ 
by $\mathrm{PRM}(\int_{\As \cap \, .\,}e^{-x}dx)$. The (much) bigger of the two estimates is the one that 
arises from approximating 
$\mathrm{PRM}(\int_{\As \cap \, .\,}$ $\lambda^\star_n(x))$ by $\mathrm{PRM}(\int_{\As \cap \, .\,}e^{-x}dx)$.
We would like to
get a positive, but not too big, expected number of threshold exceedances. The choice $u^\star = - \log \log n$, which gives an expected number of $\log n$
threshold exceedances, is not suitable in view of error bound (b) of Proposition \ref{t: MPPE_Normal_dTV}, which becomes too big. Instead, better choose 
$u^\star = -\alpha \log \log n$ for any $\alpha \in (0,1)$. We then expect about $(\log n)^\alpha$ threshold exceedances and obtain an error
estimate of order $(\log\log n)^2/(\log n)^{1-\alpha}$. 
\begin{prop}\label{t: MPPE_Normal_dTV}
(Standard normal distribution) 
For each integer $n \ge 5$, let $X_1, \ldots,$ $ X_n$ be \iid standard normal random variables, with probability density function
$\varphi(y)$  $ = (2\pi)^{-1}$ $e^{-y^2/2}$ for all $y \in \mathbb{R}$. Define the normalised random variables $X_i^\star = a_n^{-1}(X_i-b_n)$, $i=1, \ldots, n$,
with $a_n$ and $b_n$ from (\ref{t: an_bn_N}) and with state space $E^\star = \mathbb{R}$.
Let $A^\star = [u^\star, \infty)$ for any choice of $u^\star \in [-\alpha\log \log n, 0]$, where $\alpha \in (0,1)$, and let $\Xi^\star_{A^\star}$ be defined as in  
(\ref{d: normalised_MPPE}). Then the mean measure of $\Xi^\star_{A^\star}$ is given by 
\[
\bl^\star_n(B^\star) = \int_{A^\star \cap B^\star} na_n\varphi(a_n x + b_n)dx, \quad \text{ for any } B^\star \in \mathcal{E}^\star,
\]
Moreover, define
\[
{\bl}^\star (B^\star) = \int_{A^\star \cap B^\star}e^{-x} dx, \quad  \textit{ for any } B^\star \in \mathcal{E}^\star. 
\]
Then\\
$
(a)\quad \displaystyle d_{TV}\left(\mathcal{L}(\Xi^\star_{A^\star}), \mathrm{PRM}(\bl^\star_n)\right) 
\le 
\frac{6e^{-u^\star}}{n \phantom{\frac{1}{n}}},
$ \\
$
(b) \quad  \displaystyle d_{TV}(\mathrm{PRM}(\bl^\star_n), \mathrm{PRM}({\bl}^\star)) 
\le  \frac{(3\log \log n + \log 4\pi)^2}{16\log n}\, e^{-u^\star}.
$
\end{prop}
\begin{proof}
(a) By (\ref{d: density_normalised})
and (\ref{d: intensity_function_normalised}), $\lambda^\star_n(x) = na_n\varphi(a_n x + b_n)$ for all $x \in \mathbb{R}$, 
and $\bl^\star_n(B^\star) = \int_{A^\star \cap B^\star} \lambda^\star_n(x)dx$ for any $B^\star \in \mathcal{E}^\star$. 
By
(\ref{t: MyMichel_normalised}), 
\[
d_{TV}(\mathcal{L}(\Xi^\star_{A^\star}), \mathrm{PRM}(\bl^\star_n)) 
\le P(X^\star_1 \ge u^\star) = P(X_1 \ge a_n u^\star + b_n) 
\le \frac{\varphi(a_n u^\star + b_n)}{a_n u^\star + b_n},
\]
where we used (\ref{p: normal_millsratio}) for the last inequality. 
Using (\ref{p: normal_normalised_density}) with 
$x:= u^\star $ $\in$ $[-\alpha \log \log n,$ $0]$ and (\ref{p: term1}), as well as $\alpha \log \log n \le \log \log n$, we find
\begin{align*}
\frac{\varphi(a_n u^\star + b_n)}{a_n u^\star + b_n} &\le   \frac{e^{-u^\star}}{n} \left[ 1+ \frac{u^\star -(\log \log n + \log 4\pi)/2}{2\log n}\right]^{-1}\\
&\le \frac{e^{-u^\star}}{n} \left[1- \frac{3\log \log n + \log 4\pi}{4 \log n} \right]^{-1}.
\end{align*}
By using (\ref{p: dK_dTV_normal_help}) and noting that $ (6 \log \log n + 2\log 4\pi)/\log n \le 5$ for all $n \ge 5$, we obtain the error estimate. 
(b) By Proposition \ref{t: dTV_two_PRM}, 
\begin{align*}
d_{TV}\left(\mathrm{PRM}(\bl^\star_n), \mathrm{PRM}({\bl}^\star)\right) 
&\le \int_{E^\star} |\bl^\star_n - {\bl}^\star|(dx)\\
&= \int_{u^\star}^\infty \left| na_n\varphi(a_nx+b_n) - e^{-x} \right|dx,
\end{align*}
where 
\[
na_n\varphi(a_nx+b_n) 
= e^{-x} \cdot \exp\left[ -\frac{(2x-\log \log n -\log 4\pi)^2}{16\log n}\right] \le e^{-x}.
\]
Using $1-e^{-z} \le z$ for $z \ge 0$, and $x \in [-\alpha \log \log n, 0]$, we find
\begin{align*}
0 \le  e^{-x} - na_n\varphi(a_nx+b_n) &\le \frac{(2x-\log \log n -\log 4\pi)^2}{16\log n}\,e^{-x} \\
&\le  \frac{(3\log \log n + \log 4\pi)^2}{16\log n}\,e^{-x}.
\end{align*}
It follows that
\begin{align*}
d_{TV}\left(\mathrm{PRM}(\bl^\star_n), \mathrm{PRM}({\bl}^\star)\right) 
&\le  \frac{(3\log \log n + \log 4\pi)^2}{16\log n}\,e^{-u^\star}.
\end{align*}
\end{proof}

For maxima of Cauchy-distributed random variables, we had to perform an approximation in two steps in Proposition \ref{t: Max_Cauchy}. This is also the case when approximating
an MPPE with Cauchy marks by a Poisson process. The reason for this is that the intensity function $[(\pi/n)^2 + x^2]^{-1}$ of the MPPE, 
which is the intensity function of the Poisson process we 
first approximate with, varies with the sample size $n$. Since $[(\pi/n)^2 + x^2]^{-1} \sim x^{-2}$ as $n \to \infty$, it then makes sense to further 
approximate the MPPE by a Poisson process with intensity function $x^{-2}$. 
\begin{prop}\label{t: MPPE_Cauchy_dTV}
(Standard Cauchy distribution) For each integer $n \ge 1$, let $X_1,$ $ \ldots, X_n$ be \iid standard Cauchy random variables, with probability density function
$f(y) $ $= 1/\pi(1+y^2)$ for all $y \in \mathbb{R}$. Define the normalised random variables $X_i^\star = \pi n^{-1} X_i$, $i=1, \ldots, n$, taking
values in $E^\star = \mathbb{R}$. Let $A^\star = [u^\star, \infty)$ for any choice of $u^\star >0$, and let $\Xi^\star_{A^\star}$ be defined as in  
(\ref{d: normalised_MPPE}). Then the mean measure of $\Xi^\star_{A^\star}$ is given by 
\[
\bl^\star_n(B^\star) = \int_{A^\star \cap B^\star} \frac{1}{(\pi/n)^2 + x^2}\, dx, \quad  \textit{ for any } B^\star \in \mathcal{E}^\star, 
\]
and
\[
d_{TV}\left(\mathcal{L}(\Xi^\star_{A^\star}), \mathrm{PRM}(\bl^\star_n)\right) \le \frac{1}{nu^\star}. 
\]
Moreover, define
\[
{\bl}^\star (B^\star) = \int_{A^\star \cap B^\star} x^{-2} dx, \quad  \textit{ for any } B^\star \in \mathcal{E}^\star. 
\]
Then, 
\begin{equation}\label{t: dTV_Cauchy_simple_int}
d_{TV}\left(\mathcal{L}(\Xi^\star_{A^\star}), \mathrm{PRM}({\bl}^\star)\right) \le \frac{1}{nu^\star} + 
\frac{\pi^2}{3n^2u^{\star3}}. 
\end{equation}
\end{prop}
\begin{proof}
We have $a_n = n\pi^{-1}$ and $b_n \equiv 0$. By (\ref{d: density_normalised})
and (\ref{d: intensity_function_normalised}), $f^\star(x) = n/(\pi^2 +n^2x^2)$ and $\lambda_n^\star(x) = 1/[(\pi/n)^2 + x^2]$ for all $x \in \mathbb{R}$,
and $\bl^\star_n(B^\star) = \int_{A^\star \cap B^\star} \lambda_n^\star (x) dx$ for any $B^\star \in \mathcal{B}(\mathbb{R})$. By
(\ref{t: MyMichel_normalised}), 
\begin{align*}
d_{TV}(\mathcal{L}(\Xi^\star_{A^\star}), \mathrm{PRM}(\bl_n^\star))
&\le P(X^\star_1 \ge u^\star) = P\left(X_1 \ge \frac{nu^\star}{\pi}\right)\\
& = \overline{F}\left( \frac{nu^\star}{\pi}\right)
 \le \frac{1}{nu^\star}, 
\end{align*}
where we used (\ref{p: bounds}) for the last inequality. Furthermore, note that $\bl^\star_n$ and ${\bl}^\star$ are two finite measures on
$E^\star = \mathbb{R}$, since
\[
\bl^\star_n(E^\star) = nP(X^\star_1 \ge u^\star) \le \frac{1}{u^\star} < \infty \quad \textnormal{ and } \quad 
{\bl}^\star(E^\star) = \int_{u^\star}^\infty x^{-2}dx = \frac{1}{u^\star} < \infty.  
\]
By Proposition \ref{t: dTV_two_PRM}, we then have
\begin{align*}
d_{TV}\left(\mathrm{PRM}(\bl^\star_n), \mathrm{PRM}({\bl}^\star)\right) 
&\le \int_{E^\star} |\bl^\star_n - {\bl}^\star|(dx)\\
& = 
\int_{u^\star}^\infty \left| \frac{1}{(\pi/n)^2 + x^2} - \frac{1}{x^2}\right| dx\\
& = \left( \frac{\pi}{n}\right)^2 \int_{u^\star}^\infty \frac{1}{(\pi/n)^2 x^2 + x^4}\,dx.
\end{align*}
Since $x \ge u^\star >0$ in the above integral, we obtain the upper bound
\[
\left( \frac{\pi}{n}\right)^2 \int_{u^\star}^\infty x^{-4}\,dx =  \frac{\pi^2}{3n^2u^{\star3}}.
\]
\end{proof}
\noindent The expected number of threshold exceedances for the MPPE with Cauchy marks is roughly $1/u_n^\star$.
The smaller we choose $u_n^\star$ (i.e. the closer to $0$), the bigger the expected number of threshold exceedances and the smaller the error estimate. 
Note that for all $u_n^\star \le \pi/\sqrt{3n}$, the second of the two error terms in (\ref{t: dTV_Cauchy_simple_int}) is the bigger one. 
As an example, choose $u_n^\star = 1/\sqrt{n}$. We then expect the MPPE to capture about $\sqrt{n}$ points in $[u_n^\star, \infty)$, and the error of the approximation by $\mathrm{PRM}(\int_{\As \cap \,.\,} x^{-2}dx)$ is bounded by 
\[
\frac{1}{\sqrt{n}} + \frac{\pi^2}{3\sqrt{n}} \le \frac{4.3}{\sqrt{n}}\,.
\]
\subsection{Application to MPPE's with geometric marks}\label{s: MPPE_geo}
In Proposition \ref{t: max_geo_d_K} we demonstrated that for maxima of geometric random variables, the approximation by a discretised Gumbel distribution living on
lattice points $k^\star$ gives a smaller error than the approximation by a continuous Gumbel distribution on $\mathbb{R}$. For the latter approximation to be sharp,
we need the condition that the failure probability $q_n$ depends on $n$ in such a way that $1-q_n = o(1/\log n)$ for $n \to \infty$. We encounter a similar behaviour 
when approximating an MPPE with geometric marks by a Poisson process. Proposition \ref{t: Geo_MPPE_dTV} below gives the error in total variation of the 
approximation by a Poisson process with mean measure living on the lattice $E^\star$ of normalised points $k^\star$. On the other hand, Proposition \ref{t: Geo_MPPE_d2} 
determines the error of the approximation by a Poisson process with an easy-to-use continuous mean measure, and uses the $d_2$-metric to achieve this.
\begin{prop}\label{t: Geo_MPPE_dTV}
(Geometric distribution) For each integer $n \ge 1$, let $X_1, \ldots, X_n$ be \iid geometric random variables 
with failure probability $q\in (0,1)$ and $P(X_1 \ge y)= q^{\lceil y \rceil}$, for any $y \ge 0$. 
Define the normalised random variables $X_i^\star = \log(1/q) X_i - \log n$, $i = 1, \ldots, n$, taking values in 
$E^\star= \log(1/q)\mathbb{Z}_+-\log n$. 
Let $A^\star = [u^\star, \infty)$ for any choice of $u^\star \in [-\log n, \infty)$, 
and let $\Xi^\star_{A^\star}$
be defined as in (\ref{d: normalised_MPPE}). Then the mean measure of $\Xi^\star_{A^\star}$ is given by 
\begin{equation}\label{t: mean_meas_normalised_MPPE}
\bpis(B^\star) = \sum_{k^\star \in A^\star \cap E^\star \cap B^\star}(1-q)e^{-k^\star}, \quad   \textit{ for any } B^\star \in \mathcal{B}([-\log n, \infty)), 
\end{equation}
and
\[
d_{TV}\left(\mathcal{L}(\Xi^\star_{A^\star}), \mathrm{PRM}(\bpis)\right) \le \frac{e^{- u^\star}}{n}\,. 
\] 
\end{prop}
\begin{proof}
For all $k \in \mathbb{Z}_+$, we use the normalisation $k=(k^\star + \log n)/\log(1/q)$, where 
$k^\star$ $ \in E^\star$ $= \{-\log n,$ $ \log(1/q) - \log n,$ $ 2 \log(1/q) -\log n, \ldots\}$. We then have
\begin{equation}\label{p: point_prob_geo}
P(X_1 =k) = (1-q)q^k = (1-q)\, \frac{e^{-k^\star}}{n} = P(X^\star_1 = k^\star), 
\end{equation}
and, for any $B^\star \in \mathcal{B}([-\log n, \infty))$, 
\begin{align}\label{p: mean_meas_geo_discrete}
\begin{split}
\bpis(B^\star) &= nP(X^\star \in A^\star \cap B^\star)\\
&= \sum_{k \in A^\star \cap E^\star \cap B^\star} nP(X_1^\star = k^\star) \\
&=\sum_{k \in A^\star \cap E^\star \cap B^\star} (1-q)e^{-k^\star}.
\end{split}
\end{align}
Using (\ref{t: MyMichel_normalised}) and (\ref{d: discretised_cdf}), we obtain 
\begin{align*}
d_{TV}(\mathcal{L}(\Xi^\star_{A^\star}), \mathrm{PRM}(\bpis)) 
&\le P(X^\star_1 \ge u^\star)\\
&= P\left(X_1 \ge  \frac{u^\star + \log n}{\log(1/q)}  \right) \\
&= q^{\left \lceil \frac{u^\star + \log n}{\log(1/q)} \right \rceil} 
\le \frac{e^{-u^\star}}{n}\,.
\end{align*}

\end{proof}
\noindent 
The upper error bound that we obtain here is exactly the same as the error bound that we determined in Proposition \ref{t: MPPE_Exp_dTV} for an MPPE with exponential marks, 
which makes sense as
the exponential distribution is the continuous analogue of the geometric distribution. To see this, set $\lambda = \log(1/q)$, 
and let $Z \sim \mathrm{Exp}(\lambda)$ and 
$\tilde{Z} \sim \mathrm{Geo}(1-e^{-\lambda})$. Then $P(Z \ge \lceil z \rceil) = e^{-\lambda \lceil z \rceil} = P(\tilde{Z} \ge z)$, for all $z \ge 0$.

The following proposition now uses the $d_2$-metric to approximate the MPPE with geometric marks by a Poisson process with continuous intensity,
as the total variation metric is too strong to achieve this. The continuous intensity measure we aim for is the same as that of the MPPE with 
exponential marks. The result is achieved in two steps: we first use (\ref{t: d2_smaller_dTV}) to estimate the error in the $d_2$-distance of the approximation 
by a Poisson process with mean measure given by (\ref{t: mean_meas_normalised_MPPE}), and then compare this Poisson process by another one with the desired continuous
mean measure, again in the $d_2$-distance, by making use of Proposition \ref{t: d2_two_PRM}. 
We assume here that
$d_0$ is the Euclidean distance on $\mathbb{R}$ bounded by $1$, i.e. $d_0(z_1, z_2) = \min(|z_1-z_2|, 1)$ for any $z_1, z_2 \in \mathbb{R}$, 
and define the $d_1$- and $d_2$-distances as in (\ref{d: def_d1}) and (\ref{d: def_d2}), respectively,
in Section \ref{Sec: Improved_rates}.

\begin{prop}\label{t: Geo_MPPE_d2}
(Geometric distribution) 
For each integer $n \ge 1$, let $X_i$, $X_i^\star$, $i=1, \ldots, n$, and $E^\star$ be defined as in Proposition
\ref{t: Geo_MPPE_dTV}. Let $A^\star = [u^\star, \infty)$ for any choice of $u^\star \in E^\star$, let $\Xi^\star_{A^\star}$ be defined as in (\ref{d: normalised_MPPE}),
with mean measure $\bpis$ as in (\ref{t: mean_meas_normalised_MPPE}),
and define the continuous measure
\[
\bls(B^\star) = \int_{A^\star \cap B^\star} e^{-x} dx, \quad   \textit{ for any } B^\star \in \mathcal{B}([-\log n, \infty)).
\]
Then 
\[
d_2\left(\mathcal{L}(\Xi^\star_{A^\star}), \mathrm{PRM}(\bls)\right) 
\le \frac{e^{-u^\star}}{n} +  2\left\{\log\left(1/q\right)\wedge 1\right\}. 
\]
\end{prop}
\begin{proof}
We have
\[
d_2\left(\mathcal{L}(\Xi^\star_{A^\star}), \mathrm{PRM}(\bls)\right) 
\le d_2\left(\mathcal{L}(\Xi^\star_{A^\star}), \mathrm{PRM}(\bpis)\right) + d_2\left(\mathrm{PRM}(\bpis), \mathrm{PRM}(\bls)\right),
\]
where, by (\ref{t: d2_smaller_dTV}) and Proposition \ref{t: Geo_MPPE_dTV},
\[
 d_2\left(\mathcal{L}(\Xi^\star_{A^\star}), \mathrm{PRM}(\bpis\right) \le  d_{TV}\left(\mathcal{L}(\Xi^\star_{A^\star}), \mathrm{PRM}(\bpis)\right)
\le \frac{e^{-u^\star}}{n}\,. 
\]
It thus remains to determine an estimate of 
$d_2\left(\mathrm{PRM}(\bpis), \mathrm{PRM}(\bls)\right)$. Since 
\[
\bls (A^\star) = \int_{u^\star}^\infty e^{-x} dx = e^{-u^\star} = \bpis(A^\star), 
\]
Proposition \ref{t: d2_two_PRM} gives
\begin{align}\label{p: d2_to_d1_geo}
\begin{split}
d_2\left(\mathrm{PRM}(\bpis), \mathrm{PRM}(\bls)\right) 
&\le \left(1-e^{-e^{-u^\star}}\right)\left(2-e^{-e^{-u^\star}}\right)d_1(\bpis, \bls)\\ 
&\le 2d_1(\bpis, \bls). 
\end{split}
\end{align}
By definition (\ref{d: def_d1}) of the $d_1$-distance,
\begin{equation}\label{p: d_1_geo}
d_1(\bpis, \bls) = e^{u^\star}\, \sup_{\kappa \in \mathcal{K}}\, \frac{1}{s_1(\kappa)} \, 
\left| \int_{-\log n}^\infty \kappa(x) \bpis(dx) - \int_{-\log n}^\infty \kappa(x) \bls(dx)\right|. 
\end{equation}
We may write the two integrals in the above expression as a sum of integrals over the ``normalised unit intervals'' 
$[k^\star, (k+1)^\star) = [k^\star, k^\star + \log(1/q))$, for all $k^\star \in E^\star \cap [u^\star, \infty)$. The modulus then
equals
\begin{equation}\label{p: mod_geo_d1}
\left| \sum_{k^\star \ge u^\star} \left\{ \int_{k^\star}^{k^\star + \log(1/q)} \kappa(x) \bpis(dx) 
- \int_{k^\star}^{k^\star + \log(1/q)} \kappa(x) \bls(dx)
\right\} \right| .
\end{equation}
Since $\bpis$ is concentrated on the lattice points $k^\star \in E^\star \cap [u^\star, \infty)$, we have
\[
\int_{k^\star}^{k^\star + \log(1/q)} \kappa(x) \bpis(dx) 
= \kappa(k^\star) \bpis(\{k^\star\}) = \kappa(k^\star) (1-q)e^{-k^\star}.
\]
Note that we obtain the same result by computing
\[
 \int_{k^\star}^{k^\star + \log(1/q)} \kappa(k^\star) \bls(dx) 
= \kappa(k^\star)  \int_{k^\star}^{k^\star + \log(1/q)} e^{-x} dx 
= \kappa(k^\star) (1-q)e^{-k^\star}.
\]
We may thus express (\ref{p: mod_geo_d1}) as follows:
\begin{multline*}
\left| \sum_{k^\star \ge u^\star}  \int_{k^\star}^{k^\star + \log(1/q)} 
\left\{ \kappa(k^\star) - \kappa(x) \right\}\bls(dx) 
\right|\\
\le  \sum_{k^\star \ge u^\star} \int_{k^\star}^{k^\star + \log(1/q)}
\left|  \kappa(k^\star) - \kappa(x) \right| \bls(dx). 
\end{multline*}
From Lipschitz continuity of $\kappa$, we know that $|\kappa(k^\star) - \kappa(x)| \le s_1(\kappa) d_0(k^\star, x)$
for any $x \in [k^\star, k^\star + \log(1/q))$, where $k^\star \in E^\star \cap [u^\star, \infty)$. The maximum Euclidean distance between $k^\star$ and any point in 
$[k^\star, k^\star + \log(1/q))$ is of course given by $\log(1/q)$. Since we bound $d_0$ by $1$, we have
\[
\left|  \kappa(k^\star) - \kappa(x) \right|  \le s_1(\kappa) \left\{ \log(1/q) \wedge 1 \right\}. 
\]
For the $d_1$-distance in (\ref{p: d_1_geo}) we now find, using $\bls([u^\star, \infty)) = e^{-u^\star}$,
\begin{align*}
d_1(\bpis, \bls) 
\le e^{u^\star} \sum_{k^\star \ge u^\star} \int_{k^\star}^{k^\star + \log(1/q)} \left\{ \log(1/q) \wedge 1\right\} \bls(dx) 
 = \log(1/q) \wedge 1,
\end{align*}
which we plug into (\ref{p: d2_to_d1_geo}) to obtain an estimate for $d_2(\mathrm{PRM}(\bpis), \mathrm{PRM}(\bls)) $.
\end{proof}
The approximation of $\mathcal{L}(\Xi^\star_{\As})$ by $\mathrm{PRM}(\bls)$, whose continuous intensity function $e^{-x}$ corresponds to that of MPPE's with exponential marks, gives rise to an additional error term which depends only on the failure probability of the geometric distribution. With threshold values similar to those that might be used for MPPE's with exponential marks, the error will still become small only if we allow the failure probability $q=q_n$ to tend to $1$ as $n \to \infty$. Since $\log(1/q_n)$ is the length of the normalised unit intervals, this condition causes the lattice structure to melt into the whole real subset $[-\log n, \infty)$ as $n \to \infty$.  Note that Proposition \ref{t: Geo_MPPE_d2} does not require $q_n$ to vary at a particular rate. The reason for that is that we chose the threshold $u_n^\star$ as element of the lattice $E^\star$. If we had not done so, we would have obtained an additional error term of size $\log(1/q_n)e^{-u_n^\star}$. In this case, $q_n$ would have 
needed to vary at a fast enough rate to guarantee a small error despite the factor $e^{-u_n^\star}$, which roughly corresponds to the expected number of exceedances and should thus be $\ge 1$. We refer to Section \ref{s: MO_Geo_d2}, where we established the error estimate in full detail for MPPE's with bivariate geometric marks.  
\subsection{Remarks on the choice of the point process and its approximation by a Poisson process}\label{s: Remarks_on_choice}
Throughout Section \ref{Sec: Poi_proc_approc_MPPE}, we have first approximated the law of an MPPE, as defined in (\ref{d: MPPE}), by a Poisson process with mean measure equal to that of the MPPE. If the mean measure was easy to work with, we were done; else, we approximated further by another process with an easier-to-use mean measure. We will continue to do this for MPPE's with multivariate marks in Chapter \ref{Chap: Multivariate_extremes}. The estimate for the first step, the actual ``Poisson approximation", comes easily in both chapters. The reason for this is that we use \iid samples $X_1, \ldots,X_n$ and \iid indicators $I_{\{X_1 \in A\}}, \ldots I_{\{ X_n \in A\}}$. This allows us to apply Proposition \ref{t: MyMichel}, which reduces the problem to the approximation of a binomial by a Poisson distribution. Our main effort, in both chapters, thus lies in determining error bounds on the approximation of a Poisson process by another Poisson process. As the error given by Proposition \ref{t: MyMichel} is 
only $P(X \in A)$, the error obtained by further approximating by a different Poisson process is typically the bigger of the two. This might, however, not be the case, if we had a different basic set-up, i.e. if the point process that we consider were different to the MPPE in (\ref{d: MPPE}). The error arising from approximation by a Poisson process with equal mean measure might then be bigger, and the error from further approximation by a different Poisson process (if not made redundant entirely) might be smaller. For an example of a different basic set-up, assume that we have indicator variables $I_i$ that are dependent, but independent of \iid marks $X_i$, and let $\Gamma_i^s$, $\Gamma_i^w$, $Z_i$ and $\eta_i$ be defined as in Theorem \ref{t: dTV_bounds_Poi_localapproach}. We may then apply Theorem 10.H in \cite{Barbour_et_al.:1992} to determine a bound on the approximation, in the total variation distance, of the law of $\Xi= \sum_{i=1}^n I_i \delta_{X_i}$ by $\mathrm{PRM}(\mathbb{E}\Xi)$. A process such 
as $\Xi$ might appear, for instance, in an insurance context, when considering a claim distribution that is a mixture of typical and large claim sizes. The indices of the occurrences of the large claims may then be dependent, as there may be underlying events leading to these large claims, but the large claim sizes may still be \iid 
\chapter{Poisson process approximation for multivariate extremes}\label{Chap: Multivariate_extremes}
\setcounter{thm}{0}
The previous chapter gave a first treatment of random configurations of extreme points in space. It dealt with the one-dimensional case, where
we considered Poisson process approximation for marked point processes of exceedances (that we called MPPE's) whose marks were univariate. Chapter \ref{Chap: Multivariate_extremes} now studies multivariate extremes. 
More precisely, instead of \iid random variables as marks, we now consider random vectors 
$\mathbf{X}_1, \ldots, \mathbf{X}_n$ that are \iid copies of a $d$-dimensional random vector $\mathbf{X} = (X_1, \ldots, X_d)$, where $d \ge 1$. 
Random point configurations of multivariate extreme points can be modelled, analogously to (\ref{d: MPPE}) in Chapter \ref{Chap: Univariate_extremes},
by MPPE's of the form
\[
\Xi_A = \sum_{i=1}^n I_{\left\{ \mathbf{X}_i \in A \right\}} \delta_{\mathbf{X}_i}. 
\]
We suppose that the state space $E$ of the random vectors $\mathbf{X}, \mathbf{X}_1, \ldots, \mathbf{X}_n$ is a subset of $ \mathbb{R}^d$ and let
$\mathcal{E} = \mathcal{B}(E)$. Denote by $F$ the joint distribution of the random vectors and by $F_1, \ldots, F_d$ the marginal distribution functions of their 
components, i.e. for any $\mathbf{y}= (y_1, \ldots, y_d) \in E$, let 
\[
F(\mathbf{y})= P(\mathbf{X}_i \le \mathbf{y})  =P(\mathbf{X} \le \mathbf{y})  = P(X_1 \le y_1, \ldots, X_d \le y_d),
\] and let
\[
F_j(y_j) = P(X_{ij} \le y_j) = P(X_j \le y_j),
\] 
for all $i =1, \ldots, n$ and $j = 1, \ldots, d$. Moreover, denote by $x_{F_1}, \ldots, x_{F_d}$ the right endpoints of $F_1,\ldots, F_d$, respectively. 
We fix a set $A \in \mathcal{E}$ such that points (that is, realisations of the random vectors) 
$\mathbf{x}_i = (x_{i1}, \ldots, x_{id})$ lying in it can be 
considered to be extreme points. In contrast to the univariate case, where it is clear that the set $A$ should be of the form $[u,x_F]\cap E$ (for a certain choice of 
a threshold $u$) when studying right-tail extremes, there is more flexibility as to the choice of ``extreme region'' $A$ in the multivariate case. 
We might, for instance, set $A:= \{[u_1, x_{F_1}] \times \ldots \times [u_d, x_{F_d}]\} \cap E$ which implies that points in $A$ are extreme in \textit{all} 
components. We might also define $A$ as the complement of $(-\infty, u_1) \times \ldots \times (-\infty, u_d) $;
then $A$ not only contains jointly extreme points but also points that might have only one extreme component. A similar possibility would be to take a 
$d$-dimensional ball
of a certain radius $r >0$ centred in $\mathbf{0}$ and let $A$ be the intersection of $[0,\infty)^d$ with the complement of the ball 
(or, if looking at all kinds of extreme points, i.e. not only those in the 
right tail of the marginal distributions, just let $A$ be the complement of the ball). 
Figure \ref{f: Sets_A} illustrates these three particular choices for $A$ in the 
bivariate case $E = \mathbb{R}_+^2$. 
\begin{figure}[!ht]
\begin{center}{\footnotesize \input{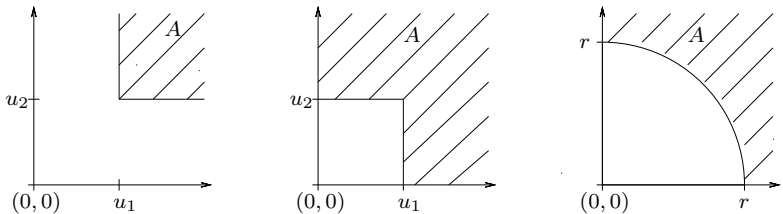}}\end{center}
\caption{We give three examples of choices of the set $A $ in the state space $E = \mathbb{R}_+^2$ (for the case $x_{F_j}=\infty$ for all $j=1, \ldots, d$). 
(Left) $A= [u_1, \infty) \times [u_2, \infty)$. 
(Middle) $A = ([0,u_1)\times [0,u_2))^C$. (Right) $A = \{(y_1,y_2) \in E:\, y_1^2 + y_2^2 \ge r^2\}$.}
\label{f: Sets_A}
\end{figure}

Analogously to the univariate case (see Proposition \ref{t: MyMichel}), we can apply Theorem \ref{t: Michel} to MPPE's with multivariate marks in order to determine 
the error in 
total variation of the approximation of the law of $\Xi_A$ by that of a Poisson process with mean measure $\mathbb{E}\Xi_A$:  
\begin{thm}\label{t: MyMichelMultivariate}
For each integer $n \ge 1$, let $\mathbf{X}_1, \ldots, \mathbf{X}_n$ be \iid copies of a $d$-di{-}mensional random vector $\mathbf{X}$ with state space 
$E \subseteq \mathbb{R}^d$, where $d \ge 1$. For a fixed set $A \in \mathcal{E}$, let
$\Xi_A = \sum_{i=1}^n I_{\{ \mathbf{X}_i \in A\}} \delta_{\mathbf{X}_i}$ be the marked point process of points in $A$ and let 
$W_A = \sum_{i=1}^n I_{\{\mathbf{X}_i \in A\}}$ denote the random number of points in $A$. Then,
\[
d_{TV}(\mathcal{L}(\Xi_A), \mathrm{PRM}(\mathbb{E}\Xi_A)) \le d_{TV}(\mathcal{L}(W_A), \mathrm{Poi}(\mathbb{E}W_A)) \le P(\mathbf{X} \in A). 
\]
\qed
\end{thm}
In a first (bivariate) example, we suppose that the components of the random vectors are standard uniformly distributed and independent of each other. More precisely, we denote the 
random vectors (or \textit{random pairs}, since $d=2$) by $\mathbf{U}_1, \ldots, \mathbf{U}_n$ and suppose that they are \iid copies of $\mathbf{U} = (U, V)$, where 
$U, V \sim \mathrm{U}(0,1)$, 
and where
\begin{equation}\label{d: independence_copula_df}
P(U \le u, V \le v) = P(U \le u)  P(V \le v)= uv.
\end{equation}
The state space of the random vectors is $E = [0,1)^2$. We fix $A=A_n=[u_n, 1) \times [v_n,1)$ for some choices of thresholds $u_n, v_n \in [0,1)$ 
and introduce the following normalisation:
\begin{align*}
&\textnormal{For } u_n, v_n \in [0,1), \textnormal{ there exist } 
s_n, t_n \in (0,n] \textnormal{ such that } \\
&\phantom{blaaaaaaaaa}u_n =1- \frac{s_n}{n} \textnormal{ and } v_n =1- \frac{t_n}{n}\,.
\end{align*}
Thus, $(u,v)$ $ \in [u_n, 1) \times [v_n,1)$ $ \subseteq [0,1)^2$ is equivalent to $(s,t)$ $ =(n(1-u),n(1-v))$ $ \in (0,s_n] \times (0,t_n] $ $\subseteq (0,n]^2$.
Note that this normalisation is equivalent to the slightly different normalisation that we introduced for the univariate case in Propositions 
\ref{t: cont_limit_maxima} (c), \ref{t: Max_uniform} and \ref{t: MPPE_U_dTV}.
Suppose, for instance, that $s_n=t_n=\log n$. 
The probability that both components of $\mathbf{U}$ are jointly extreme is 
\begin{equation}\label{p: motivational_example_prob_U_indpt}
P(\mathbf{U} \in A) = P\left(U \ge 1-\frac{s_n}{n}, V \ge 1-\frac{t_n}{n}\right) = \frac{s_nt_n}{n^2} = \left( \frac{\log n}{n}\right)^2.
\end{equation}
The probability of the occurrence of joint extremes is thus very small and Poisson approximation, by Theorem \ref{t: MyMichelMultivariate}, is very sharp. 
However, the mean of both $\mathcal{L}(W_A)$ and the approximating Poisson distribution is
$nP(\mathbf{U} \in A) =s_nt_n/n = \log^2(n)/n$, which is strictly smaller than $1$ for all $n \ge 1$, and tends to zero as $n \to \infty$. For large $n$, we therefore 
expect no joint threshold exceedances, so nothing really happens in $A$ for either $\mathcal{L}(W_A)$ or $\mathrm{Poi}(\mathbb{E}W_A)$, and Poisson approximation
has to be good.  
But clearly, the choice 
$[u_n, 1) \times [v_n, 1)$ for the set $A$ is not the most sensible one for this example.
The probability that \textit{one or both} of the components exceeds a threshold is of a higher order
than the probability in (\ref{p: motivational_example_prob_U_indpt}):
\[
P\left( \left\{U \ge 1-\frac{s_n}{n} \right\} \cup \left\{ V \ge 1-\frac{t_n}{n} \right\}\right) 
= \frac{s_n+t_n}{n} - \frac{s_nt_n}{n^2} = \frac{2 \log n}{n} - \left(\frac{\log n}{n} \right)^2,
\]
and we expect about $2 \log n$ joint threshold exceedances. 
A more suitable choice for $A$ in this example is thus $([0,1-s_n/n) \times [0,1-t_n/n))^C$. 
Of course, we can choose different $s_n$ and $t_n$.  
The choice of these values depends, 
on the one hand, on what expected number of exceedances $nP(\mathbf{U} \in A)$ we wish to consider, and,
on the other hand, on what size $P(\mathbf{U} \in A)$ of the error we judge to be sufficiently small. 
The bigger the allowed number of exceedances, the bigger the error will be
and vice versa. The approximation will get sharper the farther the set $A$ moves away from the origin $(0,0)$ as there will be less and less points in $A$. Thus for, say
$A= ([0,1-1/n)^2)^C$, we expect only about two threshold exceedances and the error is $2/n$. See Figure \ref{f: Sim_Ind_cop} for an illustration.
\begin{figure}[!ht]
\includegraphics[scale=0.65]{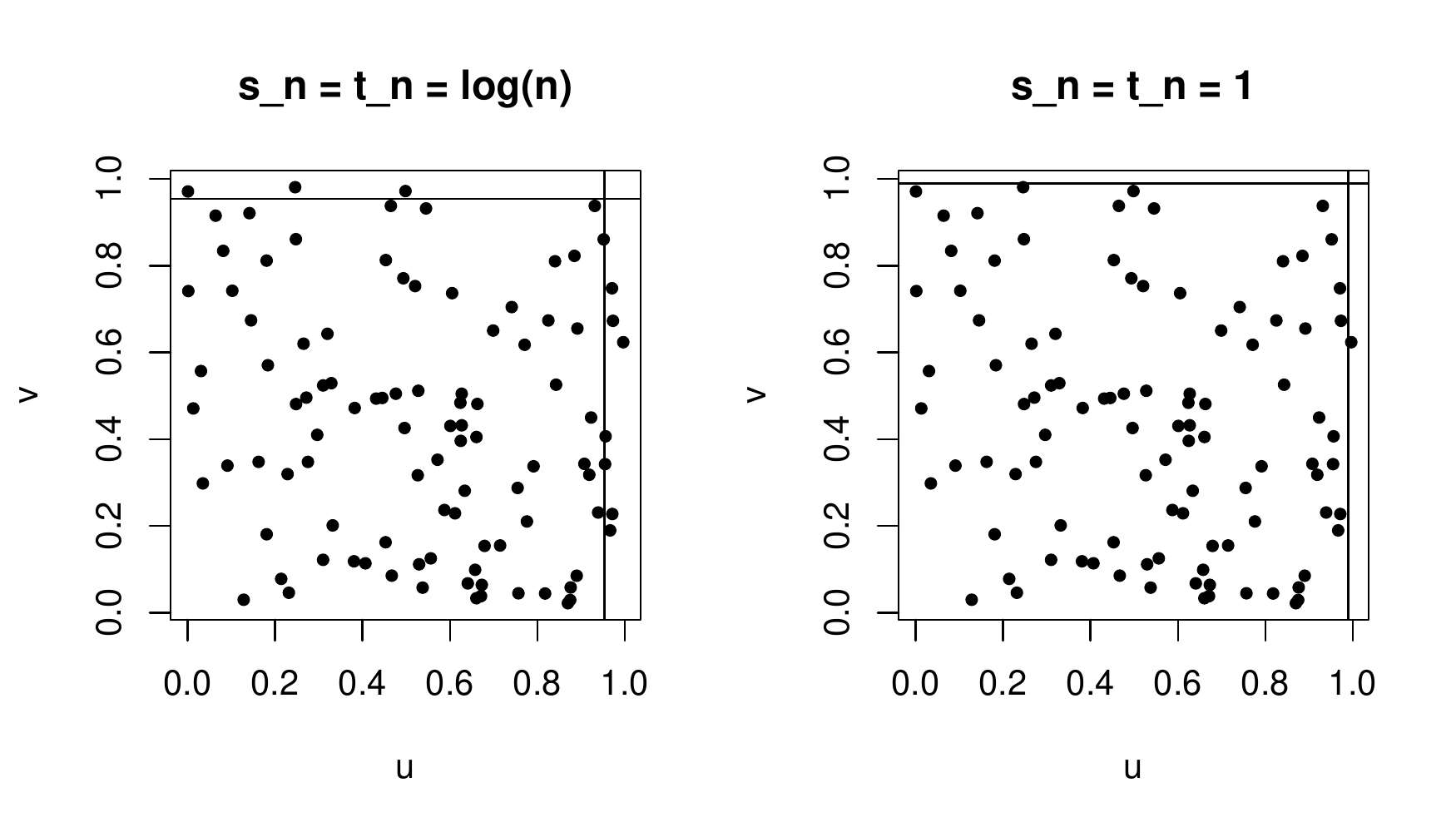}
\caption{We simulate $n=100$ points from the distribution in (\ref{d: independence_copula_df}) and compare two sets $A=([0,1-s_n/n) \times [0,1-t_n/n))^C$ 
for different choices of $s_n$ and $t_n$. 
Note that we 
observe no points in the upper-right corner $[1-s_n/n,1) \times [1-t_n/n,1)$ and that
the actual number of threshold exceedances coincides roughly with the expected number of exceedances $\mathbb{E}W_A$.
 (Left) For the choice $s_n=t_n = \log n$, 
$\mathbb{E}W_A \approx 2 \log n \approx 9.2$. (Right) For the choice $s_n=t_n = 1$, $\mathbb{E}W_A \approx 2$. }
\label{f: Sim_Ind_cop}
\end{figure}

An issue that arises with the use of random vectors as opposed to univariate random variables is thus the question of the choice of $A$. This issue is connected to the relation 
between the components of the random vectors: are the components connected in a way to exhibit dependence in the upper tail, i.e. is there some non-negligeable probability
of the components being simultaneously extreme? To address this question, we first need to define the joint distribution of the random vectors in more detail, that is, we need to specify 
the dependence structure between their margins. One way to achieve this is by using copulas. Section \ref{Sec: Copulas} gives a short introduction to copulas, as well as to 
bivariate measures of extremal dependence, the so-called \textit{coefficients of tail dependence}. Given that the
random vectors have a certain dependence structure, we may then, at least in the case $d=2$, determine from their coefficient of upper tail dependence whether they are
likely to have joint extremes and whether we should define an MPPE with a set $A$ of shape $[u_n,1) \times [v_n,1)$. We can turn this idea around. Suppose we are 
particularly interested in marked point processes with bivariate points that are extreme in both components. Then we may specifically choose copulas that 
exhibit upper tail dependence for their marks. But why should we actually
bother with distinguishing between joint and single-component extremes and not just always use a set $A$ that contains both kinds? The reason for this is that there is an
interest, for instance in finance, or in the modelling of extreme events in nature, in providing models for a ``perfect storm'' scenario, where many things go wrong at the same time.   

Another issue that arises with the use of multivariate random vectors 
$\mathbf{X}_1,$ $\ldots,$ $\mathbf{X}_n$ is that, given certain specified margins, there are 
infinitely many choices for the joint distribution
function and copula. Theorem \ref{t: MyMichelMultivariate} fortunately gives a hugely general result for the error that occurs when approximating an
process $\Xi_A$ with \iid marks $\mathbf{X}_1, \ldots, \mathbf{X}_n$ by a Poisson process with mean measure $\mathbb{E}\Xi_A$, regardless of the common joint distribution of 
the marks. However, if the joint distribution function has a complicated structure, it might be difficult to understand
in what way the error estimate $P(\mathbf{X}\in A)$
varies with the sample size $n$. To remedy this, Section \ref{Sec: PP_approx_MPPE_mult} establishes easy bounds on $P(\mathbf{X}\in A)$ for two choices of regions $A$ -- one 
where all components are extreme, and one where there might only be one extreme component -- and gives some examples. 

Also, the aim should be to approximate a certain choice of an MPPE by a ``workable'' Poisson process. That is, 
the Poisson process should have an intensity function that is easy to handle and that preferably does not depend on the sample size $n$. 
Whether the Poisson process that we approximate with is useful or not needs to be judged on a case-by-case basis and we thus 
necessarily need to restrict ourselves to examples. In cases where the intensity function is too difficult to handle, we might try to see if it behaves in a simpler way for
$n \to \infty$ and then approximate by another Poisson process with this simpler intensity, using 
Proposition \ref{t: dTV_two_PRM}. Sections \ref{Sec: Archim_tail_dep} and \ref{Sec: MO_Geo} each treat an application of this method. 
In Section \ref{Sec: Archim_tail_dep}
we first give a short introduction to the subclass of \textit{Archimedean copulas}. \cite{Charpentier/Segers:2009} list examples of bivariate Archimedean copulas according to their
asymptotic tail behaviour. Among these, we choose, as examples of possible distribution functions for the marks, those that exhibit asymptotic tail dependence, and show that MPPE's 
with such marks can be approximated by Poisson processes with practicable intensity functions. 
In Section \ref{Sec: MO_Geo} we study Poisson process approximation for MPPE's with bivariate marks that follow the \textit{Marshall-Olkin geometric distribution}, 
which is commonly thought of as 
a natural choice for a bivariate geometric distribution. For this bivariate discrete distribution we of course encounter the same problem as in the univariate case, namely that the 
total variation distance is too strong if we want to approximate by a Poisson process with a continuous intensity function. We thus use the $d_2$-distance 
as we did before in Section \ref{s: MPPE_geo} for the univariate geometric distribution. 
We contrast our results with those that we obtained in Section \ref{Sec: PP_approx_MPPE_mult} for the continuous counterpart of this distribution, the 
\textit{Marshall-Olkin exponential distribution}. 

\section{Copulas and tail dependence}\label{Sec: Copulas}
This section gives a very short introduction to copulas and coefficients of tail dependence. 
A more thorough introduction to copulas along with applications to finance are given in \cite{McNeil_et_al:2005}, a comprehensive treatment
can be found in \cite{Nelsen:2006}, whereas \cite{Genest/Neslehova:2007} treat issues that arise for copulas when using count data. 
As will be made clear by Theorem \ref{t: SklarsThm} below, copulas come in useful, on the one hand, 
when trying to understand the dependence structure between the margins of a given distribution function.
On the other hand, they are useful for building multivariate models when certain margins are given. 
\subsection{Definition and properties of copulas}
Copulas are defined as follows:
A $d$-dimensional \textit{copula} $C:\,[0,1]^d \to [0,1]$ 
is a joint distribution function with standard uniform margins. 
Let $u_1, \ldots, u_d \in [0,1]$.
Copulas are characterised by the following three properties:
\begin{enumerate}
\item[(i)] $C(u_1, \ldots, u_d)$ is increasing in each component $u_j$, $j=1, \ldots, d$. 
\item[(ii)] $C(1, \ldots, 1, u_j, 1, \ldots, 1) = u_j$, for all $j =1, \ldots, d $ and $u_j \in [0,1]$. 
\item[(iii)] For all $(a_1, \ldots, a_d)$, $(b_1, \ldots, b_d) \in [0,1]^d$ with $a_j \leq b_j$, we have 
\begin{equation}\label{d: rectangle_inequality}
\sum_{i_1=1}^{2} \ldots \sum_{i_d=1}^{2} (-1)^{i_1 + \ldots + i_d} C(u_{1i_1}, \ldots, u_{di_d}) \geq 0,
\end{equation}
where $u_{j1}= a_j$ and $u_{j2}= b_j$ for all $j = 1, \ldots, d$. 
\end{enumerate}
The first property has to be satisfied for any multivariate distribution function, whereas the second property is the requirement of standard uniform margins. 
The so-called \textit{rectangle inequality} in (\ref{d: rectangle_inequality}) makes sure that $P(a_1 \le U_1 \le b_1, \ldots , a_d \le U_d \le b_d)$ is non-negative
for a random vector $(U_1, \ldots, U_d)$ with distribution function $C$.
\begin{exa}\label{d: examples_copulas}
Let $u,v, u_1, \ldots, u_d \in [0,1]$. We list some examples of well-known copulas; many more examples of copulas can be found in \cite{Nelsen:2006}.
\begin{enumerate} 
\item[(a)] \textit{Independence copula:} $\Pi(u_1, \ldots, u_d) = \prod_{j=1}^d u_j$. We used this copula in (\ref{d: independence_copula_df}) for $d=2$. It is 
also called the \textit{product copula}. 
\item[(b)] \textit{Comonotonicity copula:} $M(u_1, \ldots, u_d) = \min_{1 \le j \le d} u_j$. 
This copula is the joint distribution function of a $d$-dimensional random vector
$(U, \ldots, U)$, where $U$ is standard uniformly distributed.
\item[(c)] \textit{Countermonotonicity copula:} $W(u, v) = \max\{u+v-1,0\}$. This copula is the joint distribution of  
$(U, 1-U)$, where $U$ is standard uniformly distributed. 
\item[(d)] \textit{Family of Gumbel(-Hougaard) copulas}: For any $\theta \in [1, \infty)$,
\[
C_\theta(u,v) = \exp\left\{ - \left[(-\log u)^\theta + (-\log v)^\theta\right]^{1/\theta} \right\}.
\]
\item[(e)] \textit{Family of Clayton copulas}: For any $\theta \in [-1,\infty)\setminus \{0\} $,
\[
C_\theta(u,v) = \left[\max\left( u^{-\theta} + v^{-\theta} -1 \,,\, 0\right) \right]^{-1/\theta}. 
\]
\item[(f)] \textit{Family of Marshall-Olkin copulas}, 
also called \textit{family of generalised Cua\-dras-Aug\'e copulas}: 
For any $\alpha, \beta \in (0,1)$,
\[
C_{\alpha, \beta} (u,v) = \min\left( u^{1-\alpha}v\, , \, uv^{1-\beta}\right) 
= \left\{ \begin{array}{ll} u^{1-\alpha} v, & u^\alpha \ge v^\beta \\ uv^{1-\beta}, & u^\alpha \le v^\beta \end{array}\right..
\]
\end{enumerate}
\end{exa}
\subsection{Fr\'echet-Hoeffding bounds}
\noindent The following theorem states that any copula may be bounded by the so-called \textit{Fr\'echet-Hoeffding bounds} (or \textit{Fr\'echet bounds}).
\begin{thm}\label{t: FH_bounds}
For every $d$-dimensional copula $C(u_1, \ldots, u_d)$, we have the bounds
\[
\max\left(\sum_{j=1}^d u_j + 1-d \, , \, 0 \right) \le  C(u_1, \ldots, u_d) \le \min_{1\le j \le d} u_{j}. 
\]
\end{thm}
\begin{proof}
Let $C$ be the joint distribution function of a $d$-dimensional random vector $\mathbf{U}=(U_1, \ldots, U_d)$ with standard uniform margins. For the lower bound,
note that, using Boole's inequality $P(\cup_j B_j) \le \sum_{j}P(B_j)$ for a countable union of events $B_1, B_2, \ldots$, we obtain 
\begin{align*}
C(u_1, \ldots, u_d) &= P \left( \cap_{1 \le j \le d} \{U_j \le u_j\}\right) = 1- P\left( \cup_{1 \le j \le d} \{ U_j > u_j\}\right)\\
&\ge 1- \sum_{j=1}^d P(U_j > u_j) = 1 - d + \sum_{j=1}^d u_j,
\end{align*}
and remember that a distribution function $C$ is always positive. We obtain the upper bound by noting that, for any $k \in \{1, \ldots, d\}$,
\[
\bigcap_{1 \le j \le d} \{U_j \le u_j\} \subset \{U_k \le u_k\}.  
\]
\end{proof}
\noindent The \textit{Fr\'echet-Hoeffding upper bound} $\min_{1\le j \le d} u_{j}$ corresponds to the comonotonicity copula from Example \ref{d: examples_copulas} (b). For $d=2$, the 
\textit{Fr\'echet-Hoeffding lower bound} is precisely the countermonotonicity copula from Example \ref{d: examples_copulas} (c). As shown in Example 5.21 in
\cite{McNeil_et_al:2005}, the $d$-dimensional Fr\'echet-Hoeffding lower bound is not a copula for $d >2$, as it does not satisfy the rectangle inequality (\ref{d: rectangle_inequality}).
In view of Sklar's Theorem below, the Fr\'echet bounds may similarly be established for the joint distribution function $F$ of any $d$-dimensional random vector $\mathbf{X}$ and are expressed
in terms of the marginal distribution functions $F_1, \ldots, F_d$ of $\mathbf{X}$:
\[
\max \left( \sum_{j=1}^d F_j(y_j) + 1 -d \, , \, 0\right) \le F(y_1, \ldots, y_d) \le \min_{1 \le j \le d} F_j(y_j). 
\]
\subsection{Sklar's Theorem}
A fundamental result is the following theorem by \cite{Sklar:1959}, which shows that
copulas can be extracted from any joint distribution function. It also shows that a copula, along with some marginal distribution functions, gives all the information 
that is necessary to define a multivariate joint distribution function.
\begin{thm}\label{t: SklarsThm} (Sklar, 1959)\\
(i) Let $F$ be a $d$-dimensional joint distribution function with marginal distribution functions $F_1, \ldots, F_d$. Then there exists a copula
$C:\,[0,1]^d \to [0,1]$ such that for all $y_1, \ldots, y_d $ $\in \overline{\mathbb{R}} = [-\infty, \infty]$,
\begin{equation}\label{t: Sklar}
F(y_1, \ldots, y_d) = C(F_1(y_1), \ldots, F_d(y_d)). 
\end{equation}
If the margins are continuous, then the copula $C$ is unique; else $C$ is uniquely determined on $\mathrm{Ran}(F_1) \times \ldots \times \mathrm{Ran}(F_d)$, where
$\mathrm{Ran}(F_j) = F_j(\overline{\mathbb{R}})$ denotes the range of $F_j$, $j=1, \ldots, d$. \\
(ii) If $C$ is a $d$-dimensional copula and $F_1, \ldots, F_d$ are univariate distribution functions, then $F$ defined by (\ref{t: Sklar}) is a joint distribution
function with margins $F_1, \ldots, F_d$. 
\end{thm}
\begin{proof}
See, for instance, \cite{Nelsen:2006} (p. 18-21). 
\end{proof}
\noindent 
If the marginal distribution functions $F_1, \ldots, F_d$
are continuous, it is easy to obtain (\ref{t: Sklar}) by using the fact that for a random variable $Y$ with continuous distribution function $G$,
the random variable $G(Y)$ is standard uniformly distributed (see, e.g., Proposition 5.2. in 
\cite{McNeil_et_al:2005}):
\begin{align*}
F(y_1, \ldots, y_d) &= P(X_1 \le y_1, \ldots, X_d \le y_d)\\
&= P(F_1(X_1) \le F_1(y_1), \ldots, F_d(X_d) \le F_d(y_d)) \\
& = P(U_1 \le F_1(y_1), \ldots, U_d \le F_d(y_d))\\
& = C(F_1(y_1), \ldots, F_d(y_d)),
\end{align*}
where the random variables $U_1, \ldots, U_d$ are standard uniformly distributed and where we let $C$ denote the joint distribution function of the random vector $(U_1, \ldots, U_d)$. 
\begin{remark} Note that there is no particular reason or justification behind transforming the marginal distributions to standard uniform distributions, and 
thereby, to using copulas. Though copulas provide a way to isolate the dependence structure of a multivariate distribution from its margins and to compare
different dependence structures, we could just as well transform the marginal distributions to any other univariate distribution. For instance, multivariate extreme value theory often 
transforms to standard Fr\'echet margins, see, e.g. Section 5.4 in \cite{Resnick:1987}. We choose copulas as one among many possibilities to describe the dependence structure in multivariate distribution functions, 
mostly out of convenience,
as they are well-established in the literature. 
\end{remark}
\begin{exa}
\cite{Gumbel:1958,Gumbel:1965} introduced the so-called \textit{Type B bivariate extreme value distribution} as a possible limiting distribution function 
of the joint distribution function of normalised component-wise maxima; it was later discussed by \cite{Kotz:2000} (p. 628) and \cite{Nelsen:2006} (p. 28) and is given by
\[
F(y_1, y_2) = \exp\left\{- \left( e^{-\theta y_1} + e^{-\theta y_2}\right)^{1/\theta} \right\},
\]
for all $y_1, y_2 \in \mathbb{R}$, where $\theta \in [1, \infty)$. 
Note that we may rewrite $F(y_1, y_2)$ in the following way in order to see that
it satisfies (\ref{t: Sklar}):
\begin{align*}
F(y_1, y_2) 
&= \exp\left\{- \left[ \left( -\log e^{-e^{-y_1}}\right)^\theta + \left( -\log e^{-e^{-y_2}}\right)^\theta \right]^{1/\theta} \right\}\\
&= C_\theta(F_1(y_1), F_2(y_2)),
\end{align*}
where $C_\theta$ is the Gumbel copula from Example \ref{d: examples_copulas} (d), and the marginal distributions are standard Gumbel, i.e. $F_1(y)=F_2(y) = \Lambda(y)$, for all 
$y \in \mathbb{R}$. 
\end{exa}
\subsection{Survival copulas}
A version of (\ref{t: Sklar}) also exists for survival functions of multivariate distribution functions. Let $\bar{F}$ be the survival
function of a random vector $\mathbf{X}= (X_1, \ldots, X_d)$ with state space $E$, i.e. let
\[
\bar{F}(\mathbf{y}) = \bar{F}(y_1, \ldots, y_d) = P(X_1 > y_1, \ldots, X_d > y_d), 
\]
for any $\mathbf{y} = (y_1, \ldots, y_d) \in E$, and let $F_j$, $\bar{F}_j=1-F_j$ 
be the marginal and survival functions of $X_j$, respectively, for each $j = 1, \ldots, d$. 
In the case of continuous margins, $(F_1(X_1), \ldots, F_d(Y_d))$ is distributed as $(U_1, \ldots, U_d) =: \mathbf{U}$, 
where the $U_j$ are standard uniform random variables. Denote by $C$ the distribution function of $\mathbf{U}$ and by $\hat{C}$ the \textit{survival copula} 
of $C$, that is, let $\hat{C}$ be the joint distribution function of $\mathbf{1-U}$. Then,
\begin{equation}\label{t: Sklar_survival}
\bar{F}(y_1, \ldots, y_d) = \hat{C}\left(\bar{F_1}(y_1), \ldots, \bar{F_d}(y_d)\right), 
\end{equation}
since 
\[
\bar{F}(y_1, \ldots, y_d) = P\left( 1-F_1(X_1) \le \bar{F_1}(y_1), \ldots, 1-F_d(X_d) \le \bar{F_d}(y_d)\right).  
\]
Note that (\ref{t: Sklar_survival}) also holds if the margins are discontinuous. Moreover, note that $\hat{C}$ is a copula and not to be mixed up 
with the survival function of a copula which is not a copula. Denote the survival function of a copula $C$ by $\bar{C}$. Then
\begin{align*}
\begin{split}
\bar{C}(u_1, \ldots, u_d) &= P(U_1 > u_1, \ldots, U_d > u_d) \\
&= P(1-U_1 \le 1-u_1, \ldots, 1-U_d \le 1-u_d)\\
& = \hat{C}(1-u_1, \ldots, 1-u_d).
\end{split}
\end{align*}
In the case $d=2$, we have the following useful relationship between a copula $C$ and its survival copula $\hat{C}$:
\[
\hat{C}(1-u, 1-v) = 1-u-v + C(u,v).  
\]
\begin{exa}\label{e: MO_exp}
\cite{Marshall/Olkin:1967a,Marshall/Olkin:1967,Marshall/Olkin:1985} 
offer three different de\-ri\-va\-tions of a specific multivariate exponential distribution. 
One of these is achieved by using a fatal shock model; see also Section~3.1.1 in \cite{Nelsen:2006}.
For simplicity, we suppose that we are in two dimensions. We consider a two-component system, for instance a two engine aircraft. The components fail after they receive a shock
(that is always fatal). 
Let $X_1$ and $X_2$ denote the
lifetimes of the first and second component, respectively, and let $F$ be the joint distribution function of $(X_1,X_2)$.
We can model the occurrence of shocks to the first, the second, and to both components, up to 
a time $t \in \mathbb{R}_+$, by three independent Poisson
processes $Z_1(t)$, $Z_2(t)$ and $Z_{12}(t)$ with parameters $\nu_1, \nu_2,  \nu_{12} >0$, respectively. For any $y_1, y_2 \in [0,\infty)$, 
the survival function of $(X_1,X_2)$ is then given by
\begin{align}\label{d: surv_fct_MO_exp}
\begin{split}
\bar{F}(y_1,y_2) &= P(X_1 > y_1, X_2 > y_2)\\ 
&= P\left[Z_1(y_1)=0, Z_2(y_2)=0, Z_{12}(\max(y_1,y_2))=0\right]\\
& = \exp\left\{ -\nu_1 y_1 -\nu_2 y_2 - \nu_{12}\max(y_1,y_2) \right\},
\end{split}
\end{align}
and we say that $(X_1,X_2)$ follows the \textit{bivariate Marshall-Olkin exponential distribution}. 
The survival function of $X_1$ is given by $\bar{F_1}(y_1):= P(X_1 >y_1) = P[Z_1(y_1)=0, Z_{12}(y_1)=0] = \exp\{-(\nu_1 + \nu_{12})y_1\}$, and
similarly, the survival 
function of $X_2$ is 
$\bar{F_2}(y_2)= \exp\{-(\nu_2 + \nu_{12})y_2\}$. In order to determine the survival copula $\hat{C}$, note first that 
$\max(y_1,y_2) $ $=$ $ y_1+y_2 - \min(y_1,y_2)$ and thus
\begin{align*}
\bar{F}(y_1,y_2) &= \exp\left\{ -(\nu_1 + \nu_{12})y_1 - (\nu_2 + \nu_{12})y_2 + \nu_{12} \min(y_1,y_2)\right\}\\
&= \bar{F_1}(y_1) \bar{F_2}(y_2) \min\left\{ \exp(\nu_{12}y_1), \exp(\nu_{12}y_2)\right\}.
\end{align*}
By setting $u:= \bar{F_1}(y_1)$, $v:= \bar{F_2}(y_2)$, $\alpha := \nu_{12}/(\nu_1 + \nu_{12}) \in (0,1)$, and 
$\beta := \nu_{12}/(\nu_2 + \nu_{12}) \in (0,1)$, 
we find that $\exp(\nu_{12}y_1) = u^{-\alpha}$, $\exp(\nu_{12}y_2) = v^{-\beta}$, and therefore
\begin{align*}
\bar{F}(y_1,y_2) &= \hat{C}\left(\bar{F_1}(y_1), \bar{F_2}(y_2)\right) 
=\hat{C}(u,v)\\
&=uv\min\left(u^{-\alpha}, v^{-\beta}\right)
= \min\left( u^{1-\alpha}v, uv^{1-\beta}\right). 
\end{align*}
The survival copula of the bivariate Marshall-Olkin exponential distribution is thus given by the Marshall-Olkin copula that we introduced in Example 
\ref{d: examples_copulas} (f). The bivariate Marshall-Olkin exponential distribution will be treated in more detail in Section \ref{Sec: MO_Geo}.
\end{exa}
\subsection{Absolutely continuous and singular components of copulas}
A $d$-dimensional copula need not be absolutely continuous as there might not be a density with respect to Lebesgue measure on $\mathbb{R}^d$; see, for instance,
Section 2.4 in \cite{Nelsen:2006} or Theorem 1.1 in \cite{Joe:1997}. The copula might thus have a singular component. More precisely (for $d=2$), 
each copula may be expressed as follows:
\[
C(u,v) = A_C(u,v) + S_C(u,v), 
\]
where we suppose that $A_C$ is absolutely continuous
with respect to two-dimen\-sional Lebesgue measure with density $a_C$. Then  
\begin{align}
A_C(u,v) = \int_0^v \int_0^u  a_C(s,t) 
\quad \textnormal{and} \quad S_C(u,v) &= C(u,v) - A_C(u,v), \label{d: A_S}
\end{align}
denote the copula's \textit{absolutely continuous} and \textit{singular components}, respectively. If $C \equiv A_C$ and $S_C \equiv 0$ on $[0,1]^2$, 
then $C$ is \textit{absolutely continuous} and $a_C(u,v) = \frac{\partial^2}{\partial v \partial u}\, C(u,v)$; 
if $C \equiv S_C$ on $[0,1]^2$, then $C$ is \textit{singular}. Among the copulas of Example \ref{d: examples_copulas}, 
the independence, Gumbel and Clayton copulas are absolutely continuous, whereas the comonotonicity and countermonotonicity copulas are singular 
(see also Figure \ref{f: Exa_cop}). The only copula among these
to have both an absolutely continuous and a singular component is the Marshall-Olkin copula:
\begin{exa}\label{e: exa_MO_cop_AC_SC}
Let $\alpha, \beta \in (0,1)$ and let $C_{\alpha, \beta}$ belong to the Marshall-Olkin family of copulas as defined in Example \ref{d: examples_copulas} (f). 
For all $(u,v) \in [0,1]^2$,
\begin{equation}\label{p: density_MO_cop}
\frac{\partial^2}{\partial u \partial v}\, C_{\alpha, \beta}(u,v) 
= \left\{ \begin{array}{ll} (1-\alpha)u^{-\alpha}, &  u^\alpha > v^\beta,\\
(1-\beta) v^{-\beta}, &  u^\alpha < v^\beta,
\end{array}\right. 
\end{equation}
and the density $a_C(u,v)$ is given by the right-hand side of (\ref{p: density_MO_cop}).
Integration of $a_C$ as in (\ref{d: A_S}) gives the absolutely continuous component of $C_{\alpha,\beta}$:
\begin{align*}
A_C(u,v) &= u^{1-\alpha} v - \frac{\alpha \beta}{\alpha + \beta - \alpha \beta} \left(v^{\beta} \right)^{\frac{\alpha + \beta - \alpha \beta}{\alpha \beta}},
\quad \textnormal{ for } u^\alpha > v^\beta,\\
\textnormal{whereas} \quad A_C(u,v) &= uv^{1-\beta} - \frac{\alpha \beta}{\alpha + \beta - \alpha \beta} \left(u^\alpha \right)^{\frac{\alpha + \beta - \alpha \beta}{\alpha \beta}}, 
\quad \textnormal{ for } u^\alpha < v^\beta.
\end{align*}
Thus, for all $(u,v) \in [0,1]^2$,
\begin{equation}\label{d: A_C_MO_copula}
A_C(u,v) = C_{\alpha, \beta}(u,v) - \frac{\alpha \beta}{\alpha + \beta - \alpha \beta} \left\{ \min\left(u^{\alpha}\, ,\, v^{\beta}\right)^{\frac{\alpha + \beta - \alpha \beta}{\alpha \beta}}\right\},
\end{equation}
and $C_{\alpha, \beta}$ has a singular component concentrated on the curve $u^{\alpha} = v^{\beta}$ in $[0,1]^2$. 
By (\ref{d: A_S}) and (\ref{d: A_C_MO_copula}), the singular component is given by
\begin{equation}\label{p: S_C_MO}
S_C(u,v) =  \frac{\alpha \beta}{\alpha + \beta - \alpha \beta} \left\{ \min\left(u^{\alpha}\, ,\,  v^{\beta}\right)^{\frac{\alpha + \beta - \alpha \beta}{\alpha \beta}}\right\}
= \int_0^{\min\left(u^\alpha \, , \, v^{\beta}\right)} t^{\frac{\alpha + \beta - 2\alpha \beta}{\alpha \beta}} dt. 
\end{equation}
For two standard uniform random variables $U$ and $V$ whose joint distribution function is given by $C_{\alpha, \beta}$, we have
\[
P\left( U^\alpha = V^\beta \right) = S_C(1,1) =   \frac{\alpha \beta}{\alpha + \beta - \alpha \beta}\,.
\]
\end{exa}
\subsection{Coefficients of tail dependence}
There exist a number of measures that can be used to quantify the dependence between the components of a random pair $(X_1,X_2)$ on $E \subseteq \mathbb{R}^2$. 
Among these are the linear correlation
between $X_1$ and $X_2$, rank correlations like Kendall's tau and Spearman's rho, as well as coefficients of tail dependence; see 
Section 5.2 in \cite{McNeil_et_al:2005} or Chapter 5 in \cite{Nelsen:2006}. Among these, we use the \textit{coefficients of tail dependence} of \cite{Joe:1993} which measure
the strength of the dependence in the tails of the distribution of a random pair. The coefficient of upper tail dependence is defined as the limiting 
probability (if it exists) of one of the components exceeding its $q$-quantile,
given that the other component exceeds its $q$-quantile, for $q \to 1$. The coefficient of lower tail dependence is defined in a similar way, with both components
now being in the lower left quadrant instead of in the upper right quadrant of $E$.  We first define the \textit{generalised inverse} $G^{\leftarrow}$ of a distribution function $G$
on $B \subseteq \mathbb{R}$ as follows: 
$G^{\leftarrow} (w) = \inf\{ y \in B:\, G(y) \ge w\}$, for any $w \in [0,1]$. Moreover, suppose that 
$X_1$ and $X_2$ are continuous random variables with distribution functions
$F_1$ and $F_2$, respectively. 
We define the \textit{coefficient of upper tail dependence} (if it exists in $[0,1]$) as follows:
\[
\lambda_u := \lambda_u(X_1, X_2) =  \lim_{q \uparrow 1} P\left( X_2 > F_2^{\leftarrow}(q) \,|\, X_1 >  F_1^{\leftarrow} (q)\right). 
\]
If $\lambda_u=0$, we say that $X_1$ and $X_2$ are \textit{asymptotically independent} in the upper tail; 
if, however, $\lambda_u \in (0,1]$, we say that they show asymptotic dependence 
in the upper tail, or \textit{extremal dependence}. Note that we can exchange 
$\{X_2 > F_2^{\leftarrow}(q)\}$ and $\{X_1 >  F_1^{\leftarrow} (q)\}$ in the above definition.
Similarly, the \textit{coefficient of lower tail dependence} is defined by
\[
\lambda_l := \lambda_l(X_1, X_2) =  \lim_{q \downarrow 0} P\left( X_2 \le F_2^{\leftarrow}(q) \,|\, X_1 \le  F_1^{\leftarrow} (q)\right),
\]
provided there exists a limit $\lambda_l \in [0,1]$. By Theorem 5.4.2 in \cite{Nelsen:2006}, for continuous $X_1$ and $X_2$, the coefficients of tail dependence depend only on the unique copula $C$ of 
the joint distribution of $(X_1,X_2)$:
\begin{align*}
\lambda_u &:= \lambda_u(C)= \lim_{q \uparrow 1} \frac{\hat{C}(1-q,1-q)}{1-q} = 2- \lim_{q \uparrow 1} \frac{1-C(q,q)}{1-q},\\
\lambda_l & := \lambda_l(C)= \lim_{q \downarrow 0} \frac{C(q,q)}{q}, 
\end{align*}
where $\hat{C}$ denotes the survival copula. Furthermore, note that we have
\begin{equation}\label{t: coefs_cop_survcop}
\lambda_u(\hat{C}) = \lambda_l(C) \quad \textnormal{ and } \quad \lambda_l(\hat{C}) = \lambda_u(C),
\end{equation}
since $\hat{\hat{C}}(1-q,1-q) = C(1-q, 1-q)$. The case of discontinuous margins is examined (for upper tail dependence) in \cite{Feidt_et_al:2010}, where Proposition 4 says that the existence of 
the coefficient of upper tail dependence is not guaranteed unless the marginal distribution functions satisfy (\ref{t: conditionFeq}). 
\begin{exa}\label{e: examples_copulas_coefs}
For the copulas in Example \ref{d: examples_copulas} (in the case $d=2$), we can easily compute the coefficients of upper and lower tail dependence:
\begin{center}
\begin{tabular}{l|cc}
Copula & $\lambda_l$ & $\lambda_u$\\
\hline 
Independence & $0$ & $0$\\
Comonotonicity & $1$ & $1$\\
Countermonotonicity & $\nexists$ & $0$\\
Gumbel & $0$ & $2-2^{1/\theta}$ \\
Clayton ($\theta >0$) & $2^{-1/\theta}$ & $0$\\
Marshall-Olkin & $0$ & $\min(\alpha, \beta)$
\end{tabular}
\end{center}
Figure \ref{f: Exa_cop} illustrates the tail behaviour of these copulas.
\begin{figure}[!ht]
\centering
\includegraphics[scale=0.65]{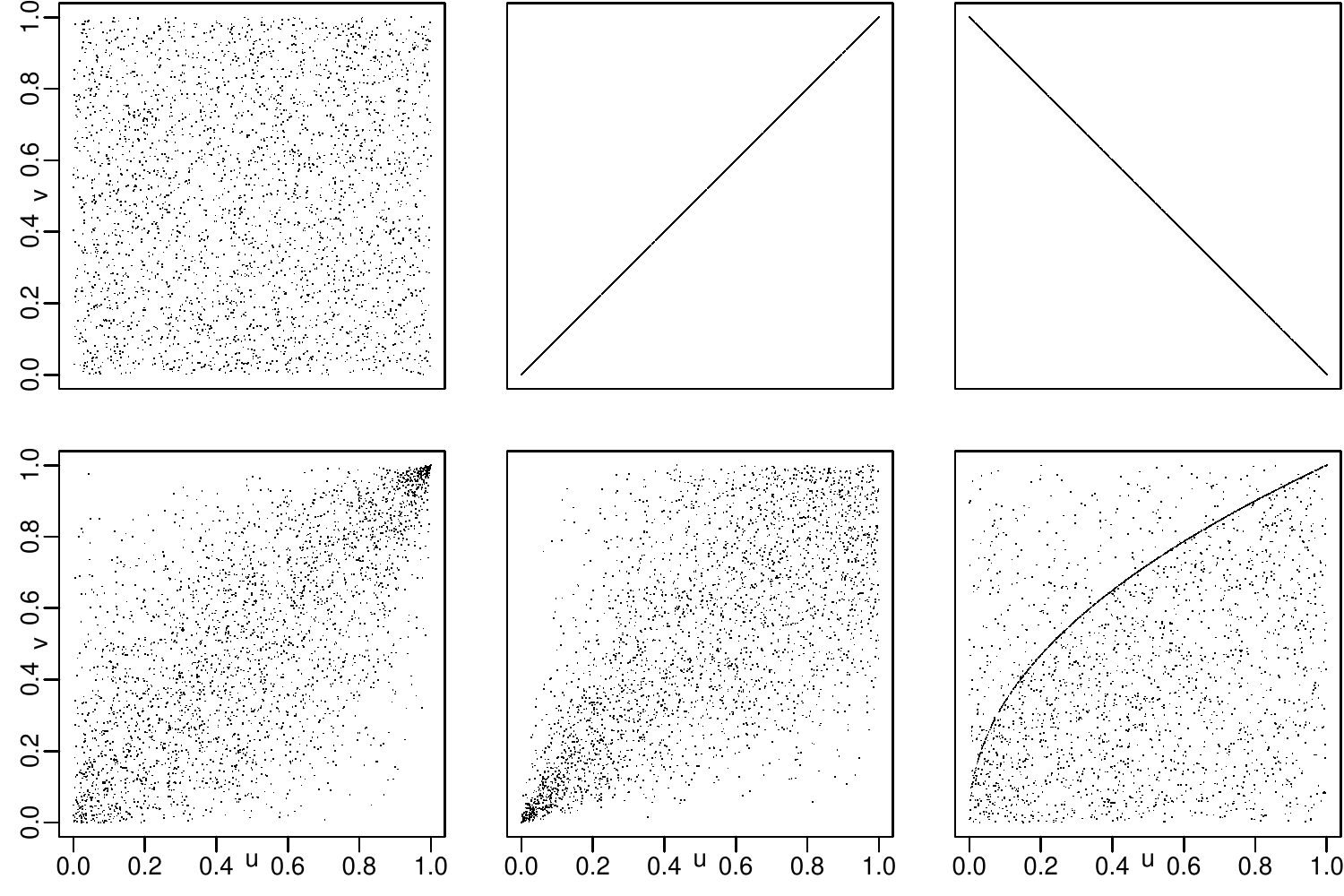}
\caption{We simulate $n=3000$ points from each of the copulas in Example \ref{d: examples_copulas}. 
(Top left) The independence copula shows no tail dependence. 
(Top middle) The comonotonicity copula has upper and lower tail dependence concentrated on the line $u=v$.
(Top right) The countermonotonicity has no tail dependence; all points are concentrated on the line $v = 1-u$. 
(Bottom left) The Gumbel copula with $\theta=2$ exhibits upper tail dependence. 
(Bottom middle) The Clayton copula with $\theta=2$ shows lower tail dependence. 
(Bottom right) The Marshall-Olkin copula
with parameters $\alpha = 0.35$ and $\beta = 0.75$ shows upper tail dependence, concentrated on the curve $u^\alpha = v^\beta$.}
\label{f: Exa_cop}
\end{figure}
\end{exa}
\section{Poisson process approximation for MPPE's with multivariate marks}\label{Sec: PP_approx_MPPE_mult}
Theorem \ref{t: MyMichelMultivariate} gives $P(\mathbf{X} \in A)$ as error estimate for the approximation, in the total variation distance, of the law of an  MPPE with 
\iid multivariate marks distributed like $\mathbf{X}$ by that of a Poisson process whose mean measure equals that of the MPPE. We now apply this result to 
two different choices of the set $A$ that denotes the region in the state space containing extreme points: in Section \ref{s: joint_thresh_exc} we define $A$
such that it contains only points having threshold exceedances in all components, whereas Section \ref{s: single_thresh_exc} allows $A$ to also contain points that exceed 
thresholds in only one component. In both cases, we establish easy bounds on the error estimate $P(\mathbf{X} \in A)$ that are valid for any choice of multivariate distribution 
function for the marks, and for any dimension $d\ge 2$. 
By way of the example of MPPE's with bivariate Marshall-Olkin exponential marks, Section \ref{s: MO_Exp_joint}, together with Example \ref{e: MO_Exp_single_comp}
from Section \ref{s: single_thresh_exc},
highlight the importance of the choice of the scaling constant of the normalisation. Depending on the scaling, the occurrence of threshold exceedances in both components can have a probability disappearing
or non-disappearing with increasing $n$. 

\subsection{Joint threshold exceedances}\label{s: joint_thresh_exc}
We suppose that the possible point configurations taken on by the MPPE's in their $d$-dimensional state space $E$ only feature points for which 
every component exceeds a threshold. As discussed in the motivational example at the beginning of Chapter \ref{Chap: Multivariate_extremes}, it only makes 
sense to define a set $A$ and a corresponding MPPE $\Xi_A$ for cases where there is actually some non-negligeable probability that the components of the marks are
jointly extreme. Theorem \ref{t: dTV_mult_df_extreme_A} gives an easy error estimate for any choice of joint distribution function, whereas Proposition 
\ref{t: dTV_copula_joint_extreme_A} is an application of Theorem \ref{t: dTV_mult_df_extreme_A} to the case of joint distribution
functions with standard uniform margins, i.e. to the case of copulas. Corollary \ref{t: dTV_mult_df_extreme_A_normalised} reformulates Theorem \ref{t: dTV_mult_df_extreme_A}
in terms of normalised random vectors, using the $^\star$ notation and multivariate analogues of 
(\ref{d: normalised_rv})-(\ref{t: MyMichel_normalised}). 
\begin{thm}\label{t: dTV_mult_df_extreme_A}
For each integer $n \ge 1$, let $\mathbf{X}_1, \ldots, \mathbf{X}_n$ be \iid copies of a $d$-di{-}mensional random vector $\mathbf{X} =(X_1, \ldots, X_d)$ with 
state space $E \subseteq \mathbb{R}^d$, joint distribution function $F$ and marginal distribution functions $F_j$ with right endpoints $x_{F_j}$, respectively, for each
$j = 1, \ldots, d$, where $d \ge 1$.  
For a fixed choice of $(u_{1n}, \ldots, u_{dn}) \in E$, define 
\[
A:=A_n:=\{(u_{1n},x_{F_1}] \times \ldots \times (u_{dn},x_{F_d}]\} \cap E.
\]
Let $\Xi_A = \sum_{i=1}^n I_{\{ \mathbf{X}_i \in A\}} \delta_{\mathbf{X}_i}$ be the marked point process of joint exceedances and let 
$W_A = \sum_{i=1}^n I_{\{\mathbf{X}_i \in A\}}$ denote the random number of points in $A$. Then,
\[
d_{TV}(\mathcal{L}(\Xi_A), \mathrm{PRM}(\mathbb{E}\Xi_A)) \le d_{TV}(\mathcal{L}(W_A), \mathrm{Poi}(\mathbb{E}W_A)) \le \min_{1 \le j \le d} \bar{F}_j(u_{jn}). 
\]
\end{thm} 
\begin{proof} By Theorem \ref{t: MyMichelMultivariate},
\begin{align*}
d_{TV}(\mathcal{L}(\Xi_A), \mathrm{PRM}(\mathbb{E}\Xi_A)) &\le d_{TV}(\mathcal{L}(W_A), \mathrm{Poi}(\mathbb{E}W_A))\\
& \le  P\left(\mathbf{X} \in A\right) = P\left( X_1 \ge u_{1n}, \ldots, X_d \ge u_{dn} \right)\\
& \le \min_{1\le j \le d} P(X_j > u_{jn}) = \min_{1 \le j \le d} \bar{F}_j(u_{jn}),
\end{align*}
where the last inequality follows from 
\[
\bigcap_{1 \le j \le d } \left\{ X_j > u_{jn}\right\} \subset \left\{ X_k > u_{kn}\right\}, \quad \textnormal{for any } k \in \{1, \ldots, d\}.  
\]
\end{proof}
\noindent We may thus bound the error of the approximation by the minimum of the marginal probabilities of threshold exceedances (or, in fact, by any of these probabilities, as 
$\min_{1 \le j \le d}$ $\bar{F}_j(u_{jn})$ $\le$ $ \bar{F}_k(u_{kn})$, for any $k \in \{1, \ldots, d\}$). The above theorem might be easier to interpret 
in applications with random variables subject to a linear transformation; we therefore restate it as follows: 
\begin{cor}\label{t: dTV_mult_df_extreme_A_normalised}
Let $\mathbf{X} =(X_1, \ldots, X_d)$ be a $d$-dimensional random vector  with 
state space $E \subseteq \mathbb{R}^d$, $d \ge 1$. For each integer $n \ge 1$, let $\mathbf{X}_1^\star, \ldots, \mathbf{X}_n^\star$ 
be \iid copies of the normalised random vector $\mathbf{X}^\star = (X_1^\star, \ldots, X_d^\star)$ with state space $E^\star$ and joint 
distribution function $F^\star$, where, for each $j=1, \ldots, d$ and constants $a_{jn}, b_{jn} \in \mathbb{R}$ with $a_{jn} >0$, 
the normalised random variable $X_j^\star = a_{jn}^{-1}(X_j - b_{jn})$ has distribution function $F_j^\star$ and right endpoint $x^\star_{F_j^\star}$. 
For a fixed choice of $(u_{1n}^\star, \ldots, u_{dn}^\star) \in E^\star$, define 
\[
A^\star = (u_{1n}^\star, x^\star_{F_1^\star}]\times \ldots \times(u_{dn}^\star, x^\star_{F_d^\star}]\cap E^\star,
\]
let $\Xi^\star_{A^\star} = \sum_{i=1}^n I_{\{\mathbf{X}_i^\star \in A^\star\}} \delta_{\mathbf{X}_i^\star}$ on $E^\star$ and 
$W^\star_{A^\star} = \sum_{i=1}^n I_{\{\mathbf{X}_i^\star \in A^\star\}}$. Then,
\[
d_{TV}\left( \mathcal{L}\left(\Xi^\star_{A^\star} \right), \mathrm{PRM}(\mathbb{E}\Xi^\star_{A^\star})\right) \le
d_{TV}\left( \mathcal{L}\left(W^\star_{A^\star} \right), \mathrm{Poi}(\mathbb{E}W^\star_{A^\star})\right)
\le \min_{1 \le j \le d} \bar{F}^\star_j(u^\star_{jn}).
\]
\qed
\end{cor}
\noindent The following proposition applies Theorem \ref{t: dTV_mult_df_extreme_A} to copulas.
\begin{prop}\label{t: dTV_copula_joint_extreme_A}
For each integer $n \ge 1$, let $\mathbf{U}_1, \ldots, \mathbf{U}_n$ be \iid copies of a $d$-dimensional random vector $\mathbf{U}$ with 
state space $E \subseteq [0,1)^d$, standard uniform margins $U_1, \ldots,U_d$, and joint distribution function $C$, where $d \ge 1$. 
For a fixed choice of $(s_{1n},$ $\ldots, s_{dn})$ $ \in$ $ (0,n]^d$, define 
\[
A:=A_n:=\left[1-\frac{s_{1n}}{n},1\right) \times\ldots \times \left[1-\frac{s_{dn}}{n},1\right).
\]
Let $\Xi_A = \sum_{i=1}^n I_{\{ \mathbf{U}_i \in A\}} \delta_{\mathbf{U}_i}$ be the marked point process of joint exceedances and let 
$W_A = \sum_{i=1}^n I_{\{\mathbf{U}_i \in A\}}$ denote the random number of points in $A$. Then,
\[
d_{TV}(\mathcal{L}(\Xi_A), \mathrm{PRM}(\mathbb{E}\Xi_A)) \le d_{TV}(\mathcal{L}(W_A), \mathrm{Poi}(\mathbb{E}W_A)) \le \min_{1 \le j \le d} \frac{s_{jn}}{n}. 
\]
\end{prop} 
\begin{proof}
By Theorem \ref{t: dTV_mult_df_extreme_A}, we have 
\[
d_{TV}(\mathcal{L}(\Xi_A), \mathrm{PRM}(\mathbb{E}\Xi_A)) 
\le d_{TV}(\mathcal{L}(W_A), \mathrm{Poi}(\mathbb{E}W_A))
 \le \min_{1 \le j \le d} \bar{F}_j(u_{jn}),
\]
where $u_{jn}= 1-s_{jn}/n$ and $\bar{F}_j(u_{jn}) = 1-u_{jn}$, for all $j=1, \ldots, d$.  
\end{proof}
\begin{remark}
Proposition \ref{t: dTV_copula_joint_extreme_A} is not a direct application of Corollary \ref{t: dTV_mult_df_extreme_A_normalised}. In terms of the notation of
Corollary \ref{t: dTV_mult_df_extreme_A_normalised}, the norming constants used for Proposition \ref{t: dTV_copula_joint_extreme_A} 
correspond to $a_{jn} = -(1/n) <0$ and $b_{jn} = 1$, for all $j= 1, \ldots, d$, and thereby do not satisfy the conditions of the corollary. 
\end{remark}
\begin{exa}\label{e: Example_copula2_joint_exc}
For $d=2$ and for any $(u,v) \in E:=[0,1]^2$, consider the family of copulas 
\[
C_\theta(u,v)= \max\left\{ 1-\left[ (1-u)^\theta + (1-v)^\theta \right]^{1/\theta}\, , \, 0\right\}, \quad \textnormal{ where }\theta \in [1, \infty).
\]
Its coefficient of tail dependence is given by $\lambda_u = 2-2^{1/\theta}$
so that the copula $C_\theta$ displays upper tail dependence if $\theta \neq 1$. Suppose that $\theta \in (1, \infty)$ and note that,
for $u$, $v$ close to $1$, $1-[(1-u)^\theta + (1-v)^\theta]^{1/\theta}$ will be positive. 
Let $(U,V), (U_1, V_1), \ldots, (U_n,V_n)$ be \iid random pairs with standard uniform margins and joint distribution function $C_\theta$. 
We choose $(s_n,t_n) \in (0,n)^2$ and define $A = [1-s_n/n,1) \times [1-t_n/n,1)$, 
and the MPPE
\[
\Xi_A  = \sum_{i=1}^n I_{\left\{U_i \ge 1-\frac{s_n}{n}, V_i \ge 1-\frac{t_n}{n}\right\}}\delta_{(U_i,V_i)}
\]
on $\mathcal{B}([0,1]^2)$. By Proposition \ref{t: dTV_copula_joint_extreme_A}, 
\[
d_{TV}(\mathcal{L}(\Xi_A), \mathrm{PRM}(\mathbb{E}\Xi_A))  \le d_{TV}(\mathcal{L}(W_A), \mathrm{Poi}(\mathbb{E}W_A)) \le
\min\left(\frac{s_n}{n},\frac{t_n}{n}\right),
\]
where $\mathbb{E}W_A = \mathbb{E}\Xi_A([0,1]^2)$ equals
\begin{align*}
 &nP\left( U \ge 1-\frac{s_n}{n}, V \ge 1- \frac{t_n}{n}\right) \\
&=  n \left\{ 1- P\left(U \le 1- \frac{s_n}{n} \right) - P\left(V \le 1- \frac{t_n}{n} \right) + C_\theta\left(  1- \frac{s_n}{n}, 1- \frac{t_n}{n}\right) \right\}\\
&= s_n +t_n -\left( s_n^\theta + t_n^\theta\right)^{1/\theta}. 
\end{align*}
We define the intensity function of the approximating Poisson process as follows:
\[
\lambda^\star(s,t) := \frac{\partial^2}{\partial s \partial t} \left\{ s+t - \left(s^\theta + t^\theta \right)^{1/\theta} \right\}
= (\theta-1) (st)^{\theta - 1} \left(s^\theta + t^\theta\right)^{\frac{1}{\theta}-2}, 
\]
for any $(s,t) \in (0,s_n]\times(0,t_n] \subseteq(0,n)^2$. Define $A^\star = (0,s_n]\times(0,t_n]$ and
$E^\star := [0,n]^2$. 
The intensity measure of the approximating Poisson process may thus be expressed in terms of the 
intensity function on $\mathcal{B}(E^\star)$, i.e. 
\[
\bl^\star(B^\star) = \int_{A^\star \cap B^\star} \lambda^\star(s,t) dsdt, \quad \textnormal{ for any } B^\star \in \mathcal{B}(E^\star).  
\]
A Poisson process with intensity measure $\bl^\star$ is an example of a Poisson process that is easy to use and thus a good choice for approximating the MPPE. The smaller
the choices for (one of) the values $s_n,t_n$, the sharper the approximation will be and the fewer joint exceedances will be expected. 
For instance, suppose that $s_n = t_n = \log n$. The error estimate
for the total variation distance is then $\log (n) /n$ and we expect $(2-2^{1/\theta})\log n$ joint threshold exceedances (again, for $\theta=1$, we expect no joint 
threshold exceedances). 
\end{exa}
\subsection{Single-component threshold exceedances}\label{s: single_thresh_exc}
In contrast to Section \ref{s: joint_thresh_exc}, we here define the set $A$ of extreme points in a way to capture all upper tail extremes of the components. 
That is, $A$ not only contains $d$-dimensional points that are extreme in all $d$ components, but also includes points that are extreme in less than $d$ or even only $1$
component. The results of this section may thus be used in general for any multivariate distribution regardless of whether it exhibits some kind of joint
upper-tail dependence or not. For that, they are also somewhat less precise: the estimate of the error in the total variation distance for approximation by a Poisson
process, given by Theorem \ref{t: dTV_mult_df_single_extreme_A}, is bounded by $d$ times the maximum marginal survival function, whereas the corresponding bound
in Section \ref{s: joint_thresh_exc} is only one time the minimum marginal survival function. Corollary \ref{t: dTV_mult_df_single_extreme_A_normalised} and 
Proposition \ref{t: dTV_copula_single_extreme_A} reformulate Theorem \ref{t: dTV_mult_df_single_extreme_A} for the cases of normalised random vectors and copulas, 
respectively. 
\begin{thm}\label{t: dTV_mult_df_single_extreme_A}
For each integer $n \ge 1$, let $\mathbf{X}_1, \ldots, \mathbf{X}_n$ be \iid copies of a $d$-di{-}mensional random vector $\mathbf{X} =(X_1, \ldots, X_d)$ with 
state space $E \subseteq \mathbb{R}^d$, joint distribution function $F$ and marginal distribution functions $F_j$ with right endpoints $x_{F_j}$, respectively, for each
$j = 1, \ldots, d$, where $d \ge 1$.  
For a fixed choice of $(u_{1n}, \ldots, u_{dn}) \in E$, define 
\[
A:=A_n:=\left( (-\infty, u_{1n}] \times \ldots \times (-\infty, u_{dn}] \right)^C. 
\]
Let $\Xi_A = \sum_{i=1}^n I_{\{ \mathbf{X}_i \in A\}} \delta_{\mathbf{X}_i}$ be the marked point process of joint exceedances and let 
$W_A = \sum_{i=1}^n I_{\{\mathbf{X}_i \in A\}}$ denote the random number of points in $A$. Then,
\[
d_{TV}(\mathcal{L}(\Xi_A), \mathrm{PRM}(\mathbb{E}\Xi_A)) \le d_{TV}(\mathcal{L}(W_A), \mathrm{Poi}(\mathbb{E}W_A)) \le \sum_{j=1}^d \bar{F}_j(u_{jn}). 
\]
\end{thm} 
\begin{proof} By Theorem \ref{t: MyMichelMultivariate}, and using Boole's inequality, we have that
\begin{align*}
d_{TV}(\mathcal{L}(\Xi_A), \mathrm{PRM}(\mathbb{E}\Xi_A)) &\le d_{TV}(\mathcal{L}(W_A), \mathrm{Poi}(\mathbb{E}W_A))\\
& \le  P\left(\mathbf{X} \in A\right) = P\left( \left\{ X_1 > u_{1n}\right\} \cup  \ldots \cup \left\{ X_d > u_{dn} \right\}\right)\\
& \le \sum_{j=1}^d P\left( X_j > u_{jn}\right) = \sum_{j=1}^d \bar{F}(u_{jn}).
\end{align*}
\end{proof}
\noindent For random vectors whose components are subject to affine transformations, Theorem \ref{t: dTV_mult_df_single_extreme_A} reads as follows:
\begin{cor}\label{t: dTV_mult_df_single_extreme_A_normalised} With the notation from Corollary \ref{t: dTV_mult_df_extreme_A_normalised} and 
$A^\star:=$ $A^\star_n:=$ $( (-\infty, u^\star_{1n}] \times $ $\ldots $ $\times (-\infty, u^\star_{dn}])^C$,
we obtain 
\[
d_{TV}\left( \mathcal{L}\left(\Xi^\star_{A^\star} \right), \mathrm{PRM}(\mathbb{E}\Xi^\star_{A^\star})\right) \le
d_{TV}\left( \mathcal{L}\left(W^\star_{A^\star} \right), \mathrm{Poi}(\mathbb{E}W^\star_{A^\star})\right)
\le  \sum_{j=1}^d \bar{F}^\star_j(u_{jn}). 
\]
\qed
\end{cor}
\noindent We apply Theorem \ref{t: dTV_mult_df_single_extreme_A} to the case where the multivariate marks of the MPPE's are distributed as copulas:
\begin{prop}\label{t: dTV_copula_single_extreme_A}
For each integer $n \ge 1$, let $\mathbf{U}_1, \ldots, \mathbf{U}_n$ be \iid copies of a $d$-di{-}mensional random vector $\mathbf{U}$ with 
state space $E \subseteq [0,1)^d$, standard uniform margins $U_1, \ldots,U_d$, and joint distribution function $C$, where $d \ge 1$. 
For a fixed choice of $\mathbf{s}_n = (s_{1n}, \ldots, s_{dn}) \in (0,n]^d$, define 
\[
A:=A_n:=\left(\left[0, 1-\frac{s_{1n}}{n}\right) \times\ldots \times \left[0, 1-\frac{s_{dn}}{n}\right)\right)^C.
\]
Let $\Xi_A = \sum_{i=1}^n I_{\{ \mathbf{U}_i \in A\}} \delta_{\mathbf{U}_i}$ be the marked point process of points in $A$ and let 
$W_A = \sum_{i=1}^n I_{\{\mathbf{U}_i \in A\}}$ denote the random number of points in $A$. Then,
\[
d_{TV}(\mathcal{L}(\Xi_A), \mathrm{PRM}(\mathbb{E}\Xi_A)) \le d_{TV}(\mathcal{L}(W_A), \mathrm{Poi}(\mathbb{E}W_A)) \le \frac{1}{n} \sum_{j=1}^d s_{jn}. 
\]
\end{prop} 
\begin{proof} By Theorem \ref{t: MyMichelMultivariate}, and using Boole's inequality, we have that
\begin{align*}
d_{TV}(\mathcal{L}(\Xi_A), \mathrm{PRM}(\mathbb{E}\Xi_A)) &\le d_{TV}(\mathcal{L}(W_A), \mathrm{Poi}(\mathbb{E}W_A))\\
& \le  P\left(\mathbf{U} \in A\right) 
\\ &= P\left( \left\{ U_1 \ge 1- \frac{s_{1n}}{n}\right\} \cup  \ldots \cup \left\{ U_d \ge 1- \frac{s_{dn}}{n} \right\}\right)\\
& \le \sum_{j=1}^d P\left( U_j \ge 1- \frac{s_{jn}}{n}\right) = \sum_{j=1}^d \frac{s_{jn}}{n}.
\end{align*}
\end{proof}
\begin{exa}\label{e: MO_Exp_single_comp} 
Consider again the example from the beginning of Chapter \ref{Chap: Multivariate_extremes},
where the common joint distribution function of \iid marks $(U,V),$ $(U_1, V_1),$ 
$\ldots,$ $(U_n,V_n)$ is given by the independence copula
$C(u,v)= uv$. For any fixed choice of $(s_n, t_n) \in (0,n]^2$, define
\begin{align*}
A &= \left(\left[0,1-\frac{s_n}{n}\right)\times \left[0,1-\frac{t_n}{n}\right)\right)^C \\
&= \left( \left[ 1- \frac{s_n}{n}, 1\right) \times [0,1)\right) 
\cup \left( [0,1) \times \left[ 1- \frac{t_n}{n}, 1\right)\right).
\end{align*}
With the normalisation $(s,t) = (n(1-u),n(1-v))$ for any $(u,v) \in [0,1)^2$, this set corresponds to 
$A^\star = ((0,s_n] \times (0,n]) \cup ((0,n] \times (0,t_n])$. 
By Proposition \ref{t: dTV_copula_single_extreme_A}, 
\[
d_{TV}(\mathcal{L}(\Xi_A), \mathrm{PRM}(\mathbb{E}\Xi_A)) \le d_{TV}(\mathcal{L}(W_A), \mathrm{Poi}(\mathbb{E}W_A)) \le \frac{s_n+t_n}{n}, 
\]
where 
\[
\mathbb{E}W_A = \mathbb{E}\Xi_A([0,1]^2) = s_n+t_n -\frac{s_nt_n}{n} \sim s_n + t_n, \quad \textnormal{ as }n \to \infty.
\]
In examples such as this it does not make sense to search for only a bivariate intensity function as in, say, Example \ref{e: Example_copula2_joint_exc},
since we essentially have two univariate problems and the bivariate contribution is negligeable. 
More precisely, with
$E^\star = [0,n]^2$, write any $B^\star \in \mathcal{B}([0,n]^2)$ as $B_s^\star \times B_t^\star$, where
$B_s^\star, B_t^\star \in \mathcal{B}([0,n])$. Then the intensity measure of $\Xi^\star_{A^\star}$ is given by
\[
\bl^\star_n(B^\star) = \int_0^{s_n} I_{B_s^\star} ds + \int_0^{t_n} I_{B_t^\star} dt + \int_0^{s_n} \int_0^{t_n} \left(- \frac{1}{n}\right) I_{B^\star}dsdt, 
\]
for any $B^\star \in \mathcal{B}([0,n]^2)$, and asymptotically behaves like 
\[
\bl^\star(B^\star) := \int_0^{s_n} I_{B_s^\star} ds + \int_0^{t_n} I_{B_t^\star} dt. 
\]
By Proposition \ref{t: dTV_two_PRM}, 
\[
d_{TV}\left( \mathrm{PRM}(\bl_n^\star),\mathrm{PRM}(\bl^\star)\right) \le \int_0^{s_n} \int_0^{t_n}\frac{1}{n}\,dsdt = \frac{s_nt_n}{n}\,.
\]
\end{exa}
\begin{exa}\label{e: MO_Exp_singlecomp}
Let $\nu_1, \nu_2, \nu_{12} >0$, let $E=[0,\infty)^2$,  
and let $\mathbf{X}=(X_1,X_2)$ follow the bivariate Marshall-Olkin exponential distribution that we introduced in Example \ref{e: MO_exp}. 
Since the survival copula $\hat{C}$ of this distribution is given by the Marshall-Olkin copula 
(see Example \ref{d: examples_copulas} (f))
which has no lower tail dependence, by (\ref{t: coefs_cop_survcop}), the Marshall-Olkin
exponential distribution has no upper tail dependence, as is also illustrated by Figure \ref{f: MO_Exp}.
\begin{figure}[!ht]
\centering
\includegraphics[scale=0.45]{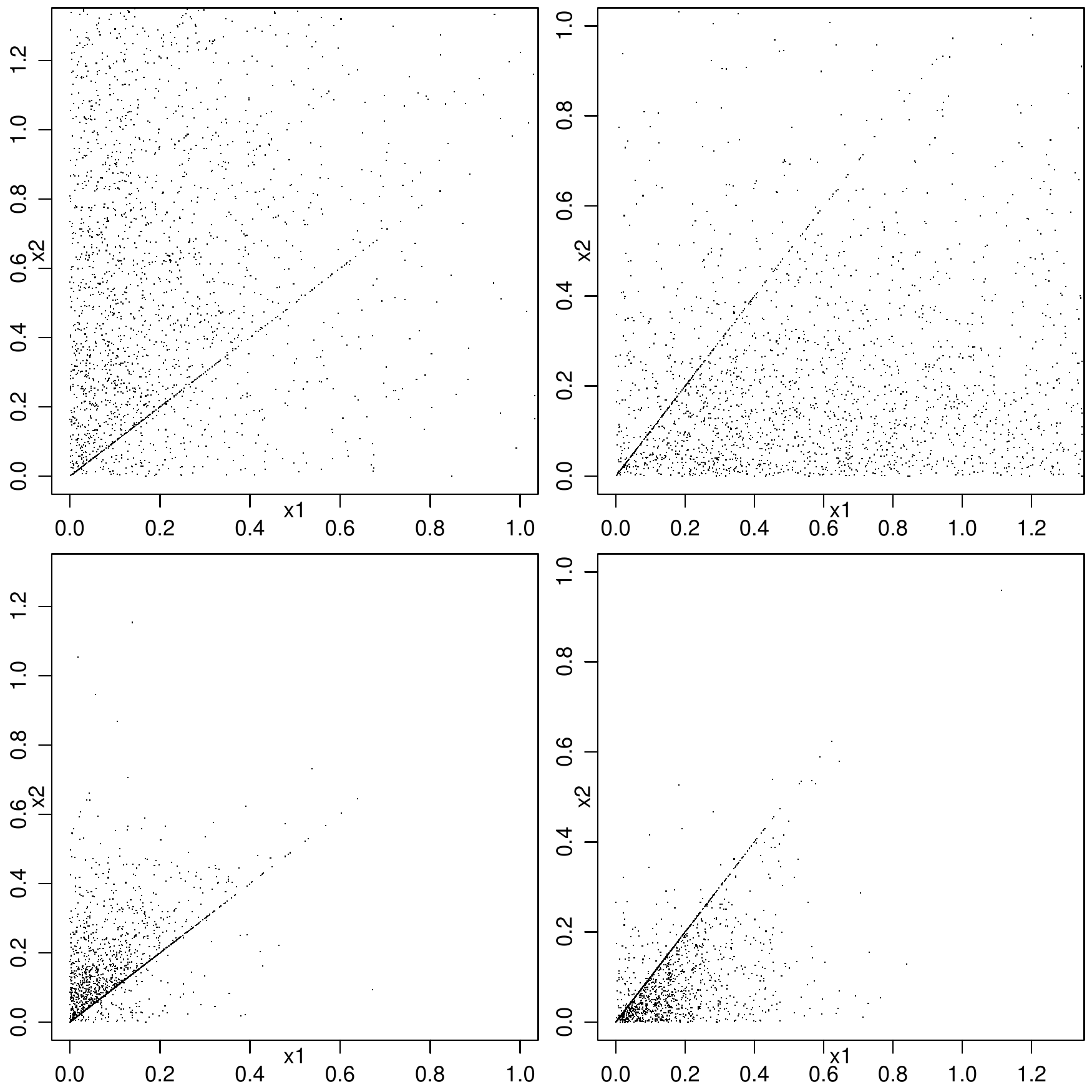}
\caption{We simulate $n=3000$ points from the Marshall-Olkin exponential distribution. The parameters $(\nu_1, \nu_2, \nu_{12})$ 
are given by $(4, 0.5, 0.5)$, $(0.5, 4, 0.5)$, $(4, 0.5, 8)$ and $(0.5, 4, 8)$, respectively (clockwise from top left).}
\label{f: MO_Exp}
\end{figure}
Define
\[
A= ([0,u_{1n}) \times [0,u_{2n}))^C = ([u_{1n}, \infty) \times [0,\infty)) \cup( [0,\infty) \times [u_{2n}, \infty) )
\]
for some choices of thresholds $u_{1n}, u_{2n} \in E$.
Let $\mathbf{X}_1^\star, \ldots, \mathbf{X}_n^\star$ be \iid copies of the normalised random variable
\begin{equation}\label{e: MO_Exp_scaling_single_comp}
\mathbf{X}^\star = (X_1^\star , X_2^\star) = \left( (\nu_1+\nu_{12}) X_1 - \log n\, ,\,  (\nu_2 + \nu_{12}) X_2 - \log n\right)
\end{equation}
with state space $E^\star = [-\log n, \infty)^2$ and marginal survival functions 
\[
\bar{F}^\star_j(x_j) = \bar{F}_j\left(\frac{x_j + \log n}{\nu_j + \nu_{12}}\right) =  \frac{e^{-x_j}}{n}\,, \quad \textnormal{ for } x_j \ge -\log n \textnormal{ and }j=1,2. 
\]
Then $A^\star$ $= ([u^\star_{1n}, \infty)$ $\times$ $ [-\log n,\infty)) \cup( [- \log n,\infty) \times [u^\star_{2n}, \infty) )$ and 
Corollary \ref{t: dTV_mult_df_single_extreme_A_normalised} gives 
\[
d_{TV}\left( \mathcal{L}\left(\Xi^\star_{A^\star} \right), \mathrm{PRM}(\mathbb{E}\Xi^\star_{A^\star})\right) \le
d_{TV}\left( \mathcal{L}\left(W^\star_{A^\star} \right), \mathrm{Poi}(\mathbb{E}W^\star_{A^\star})\right)
\le \frac{e^{-u^\star_{1n}} + e^{-u^\star_{2n}}}{n}\,.
\]
The Marshall-Olkin exponential distribution behaves in different ways above and below the diagonal $y_1= y_2$, which corresponds to
\[
x_2 = \frac{\nu_2 + \nu_{12}}{\nu_1 + \nu_{12}} \, x_1 +
 \frac{\nu_2 - \nu_1}{\nu_1 + \nu_{12}}\, \log n =: \Delta(x_1), \text{ for any }(x_1,x_2) \in E^\star. 
\]
Figure \ref{f: MO_Exp_shapes_A} shows the three possible shapes for the set $A$, depending on whether $u_{1n}> u_{2n}$, $u_{1n} = u_{2n}$, 
or $u_{1n} < u_{2n}$.
\begin{figure}
\begin{center}{\footnotesize \input{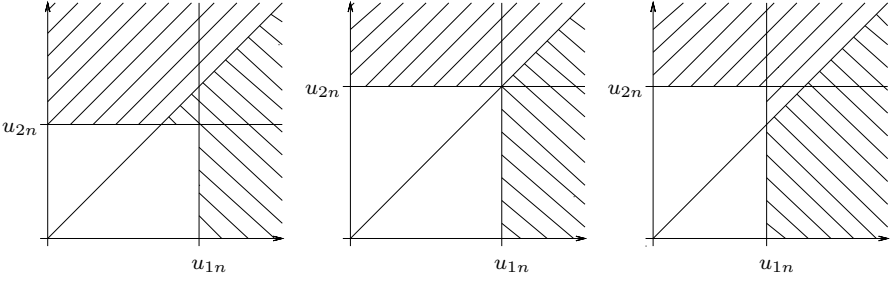}}\end{center}
\caption{Different shapes of the set $A$ from Example \ref{e: MO_Exp_singlecomp} depending on whether $u_{1n}> u_{2n}$ (left), $u_{1n} = u_{2n}$ (middle), 
or $u_{1n} < u_{2n}$ (right).}
\label{f: MO_Exp_shapes_A}
\end{figure}
With $\alpha = \nu_{12}/(\nu_1 + \nu_{12}) \in (0,1)$, $\beta= \nu_{12}/(\nu_2 + \nu_{12}) \in (0,1)$, we then have
\begin{align*}
&\mathbb{E}W^\star_{A^\star} \\
&= n \left\{ P(X_1^\star \ge u_{1n}^\star) + P(X_2^\star \ge u_{2n}^\star) - P(X_1^\star \ge u_{1n}^\star, X_2^\star \ge u_{2n}^\star)\right\}\\
&=n \bar{F}^\star_1(u_{1n}^\star) + n\bar{F}_2^\star(u_{2n}^\star) \\
&\phantom{blaaa}- n\min \left\{ {\bar{F}^\star_1(u_{1n}^\star)}^{1-\alpha}\cdot \bar{F}_2^\star(u_{2n}^\star)\,,\, \bar{F}^\star_1(u_{1n}^\star)\cdot \bar{F}_2^\star(u_{2n}^\star)^{1-\beta}\right\}\\
& = \left\{ 
\begin{array}{ll}
e^{-u_{1n}^\star} + e^{-u_{2n}^\star} - 
\left( \frac{1}{n}\right)^{1-\alpha} e^{-(1-\alpha)u_{1n}^\star-{u_{2n}^\star}} 
&\text{if } u_{2n}^\star \ge \Delta(u_{1n}^\star)  \,
\scriptstyle{(\Leftrightarrow u_{1n} \le u_{2n})} \\
e^{-u_{1n}^\star} + e^{-u_{2n}^\star} - 
\left( \frac{1}{n}\right)^{1-\beta}e^{-u_{1n}^\star-(1-\beta){u_{2n}^\star}} 
&\text{if } u_{2n}^\star \le \Delta(u_{1n}^\star) \, 
\scriptstyle{(\Leftrightarrow u_{1n} \ge u_{2n})}
\end{array}
\right.\\
& \sim e^{-u_{1n}^\star} + e^{-u_{2n}^\star}, \quad \text{as } n\to \infty.
\end{align*}
The choice of the threshold $u_{1n}^\star$ and $u_{2n}^\star$ again determines the size of the error estimate and the expected number of threshold exceedances. 
For instance, for $u_{1n}^\star = u_{2n}^\star =\log 2 -0.5\log (n)$, the error is bounded by $n^{-1/2}$ and the MPPE captures roughly the $\sqrt{n}$ points farthest away 
from the origin $(-\log n, -\log n)$ of the normalised state space. 
\end{exa}
\begin{remark}
Theorem \ref{t: dTV_mult_df_single_extreme_A} is easily adapted to the case where $A$ allows for a selection of the components to be extreme but specifically 
forbids the remaining components to be so. Let $J \subseteq \{1, \ldots, d\}$ and set
\[
A := A_n := \left( \prod_{j \in J} \left(-\infty, u_{jn}\right]\right)^C \times \left( \prod_{k \in \{1, \ldots, d\} \setminus J} \left(-\infty, u_{kn}\right]\right).
\]
Then, $d_{TV}(\mathcal{L}(\Xi_A), \mathrm{PRM}(\mathbb{E}\Xi_A))$ is bounded by
\begin{align*}
P(\mathbf{X} \in A) 
& = P\left( \bigcup_{j \in J} \left\{ X_{j} > u_{jn} \right\}, \bigcap_{k \in \{1, \ldots, d\}\setminus J} \left\{ X_k \le u_{kn}\right\} \right) \\
&\le \left( \sum_{j \in J} \bar{F}(u_{jn})\right)\cdot \left(1- \min_{k \in \{1, \ldots, d\} \setminus J} \bar{F}(u_{kn})\right) 
\\ &\le \sum_{j \in J} \bar{F}(u_{jn}) 
\le |J|\max_{j \in J}\bar{F}(u_{jn}).
\end{align*}
\end{remark}
\subsection{An example of hidden dependence for joint threshold exceedances}\label{s: MO_Exp_joint}
We consider again the bivariate Marshall-Olkin exponential distribution from Examples \ref{e: MO_exp} and \ref{e: MO_Exp_singlecomp}. 
With the normalisation used in (\ref{e: MO_Exp_scaling_single_comp}), the probability of the occurrence of joint extremes is negligeable compared against the
probability of extremes in the margins. With a different scaling however, we may zoom in on the region of joint extremes and model the behaviour of an MPPE
with Marshall-Olkin exponential marks in that region by a Poisson process with an intensity that will need to be determined. 

We use the following normalisation of the Marshall-Olkin exponentially distributed random pair $\mathbf{X} = (X_1,X_2)$:
\[
\mathbf{X}^\star = (X_1^\star, X_2^\star) = \left( \nu X_1 - \log n\, ,\,  \nu X_2 - \log n\right),
\]
where $\nu:= \nu_1 + \nu_2 + \nu_{12}$. Let $\mathbf{X}_1^\star, \ldots, \mathbf{X}_n^\star$ be \iid copies of 
$\mathbf{X}^\star$. Their state space is given by $E^\star = [-\log n, \infty)^2$, and the marginal survival functions by
\[
\bar{F}^\star_j(x_j) = \bar{F}_j\left(\frac{x_j + \log n}{\nu}\right) = \left(\frac{e^{-x_j}}{n}\right)^{\frac{\nu_j + \nu_{12}}{\nu}}, \quad \textnormal{ for } x_j \ge -\log n \textnormal{ and }j=1,2. 
\]
Let $u_{1n}^\star, u_{2n}^\star \in E^\star$ and define $A^\star = A^\star_n = (u_{1n}^\star, \infty) \times (u_{2n}^\star, \infty)$. 
The expected number of points in $A^\star$ is given by
\begin{align}\label{p: MOExp_joint_expected_number}
\begin{split}
\mathbb{E}W^\star_{A^\star} &= n \bar{F}\left(\frac{u_{1n}^\star + \log n}{\nu} , \frac{u_{2n}^\star + \log n}{\nu}\right) \\
&= \left\{ 
\begin{array}{ll}
\exp \left(-\frac{\nu_1}{\nu}\, u_{1n}^\star-\frac{\nu_2 + \nu_{12}}{\nu}\, u_{2n}^\star \right),& \text{if }u_{1n}^\star < u_{2n}^\star,\\
\exp\left( -u_{1n}^\star \right),\phantom{\frac{\nu_1}{\nu_2}}  & \text{if }u_{1n}^\star = u_{2n}^\star,\\
\exp \left(-\frac{\nu_1 + \nu_{12}}{\nu}\, u_{1n}^\star -\frac{\nu_2}{\nu}\, u_{2n}^\star \right),& \text{if }u_{1n}^\star > u_{2n}^\star. 
\end{array}
\right.
\end{split}
\end{align}
By Theorem \ref{t: MyMichelMultivariate},
\begin{align}\label{e: MO_Exp_joint_dTV_Michel}
\begin{split}
d_{TV}\left( \mathcal{L}\left(\Xi^\star_{A^\star} \right), \mathrm{PRM}(\mathbb{E}\Xi^\star_{A^\star})\right) 
&\le d_{TV}\left( \mathcal{L}\left(W^\star_{A^\star} \right), \mathrm{Poi}(\mathbb{E}W^\star_{A^\star})\right)\\
&\le P(X_1^\star \ge u_{1n}^\star, X_2^\star \ge u_{2n}^\star) = \frac{\mathbb{E}W^\star_{A^\star}}{n}\,,
\end{split}
\end{align} 
with $\mathbb{E}W^\star_{A^\star}$ from (\ref{p: MOExp_joint_expected_number}).
We could of course also use Corollary \ref{t: dTV_mult_df_extreme_A_normalised}, but it gives a worse error estimate:
\begin{align*}
d_{TV}\left( \mathcal{L}\left(\Xi^\star_{A^\star} \right), \mathrm{PRM}(\mathbb{E}\Xi^\star_{A^\star})\right) &\le
d_{TV}\left( \mathcal{L}\left(W^\star_{A^\star} \right), \mathrm{Poi}(\mathbb{E}W^\star_{A^\star})\right)\\
&\le \min\left\{ \left( \frac{e^{-u_{1n}^\star}}{n}\right)^{\frac{\nu_1 + \nu_{12}}{\nu}}\, , \, 
\left( \frac{e^{-u_{2n}^\star}}{n}\right)^{\frac{\nu_2 + \nu_{12}}{\nu}}\right\}.
\end{align*} 
\begin{exa}
Suppose that $\nu_1 \ge \nu_2$ and that
\[
u_{1n}^\star = -\frac{\nu}{\nu_1 + \nu_{12}}\, \log \log n \quad \text{and} \quad u_{2n}^\star = - \frac{\nu}{\nu_2 + \nu_{12}}\, \log \log n. 
\]
Then $u_{1n}^\star \ge u_{2n}^\star$ for all $n > 2$, and we expect $(\log n)^{1+\frac{\nu_2}{\nu_2+\nu_{12}}}$ points in $A^\star$. 
The error estimate in (\ref{e: MO_Exp_joint_dTV_Michel}) is given by 
\[
\frac{(\log n)^{1+ \frac{\nu_2}{\nu_2 + \nu_{12}}}}{n}\,.
\]
\end{exa}
\begin{exa} 
Suppose that $u_{1n}^\star = u_{2n}^\star = -\log \log n$. Then we expect $\log n$ points in $A^\star$ and the error estimate in (\ref{e: MO_Exp_joint_dTV_Michel})
is $\log(n)/n$.
\end{exa}

The Marshall-Olkin exponential distribution has both an absolutely continuous and a singular part
(see Theorem 3.1 in \cite{Marshall/Olkin:1967} or deduce it from the underlying Marshall-Olkin survival copula and Example \ref{e: exa_MO_cop_AC_SC}).
We denote the intensity functions of the approximating Poisson process in the original and in the normalised state spaces $E$ and $E^\star$
by $\lambda(y_1,y_2)$ and $\ls(x_1,x_2)$, respectively, for the absolutely continuous part (i.e. for $y_1 \neq y_2$ and $x_1 \neq x_2$), and by
$\acute{\lambda}(y)$ and ${\acute{\lambda}}^\star(x)$, respectively, for the singular part (i.e. for $y_1=y_2=y$ and $x_1=x_2=x$). 
The corresponding intensity measures are 
\[
 \bl_A(B)  =  \mathbb{E}\Xi_A(B) 
 = \int_{A \cap B} \lambda(y_1,y_2) dy_1 dy_2 
 + \int_{A \cap B \cap \{(y_1,y_2):\, y_1=y_2\}} \acute{\lambda}(y)dy, 
 \]
\begin{multline*} 
 \bl_{A^\star}^\star(B^\star)= \mathbb{E}\Xi^\star_{A^\star}(B^\star) 
= \int_{A^\star \cap B^\star} \lambda^\star(x_1,x_2)dx_1 dx_2\\
+ \int_{A^\star \cap B^\star \cap \{(x_1,x_2):\, x_1=x_2\}} {\acute{\lambda}}^\star(x)dx,
\end{multline*}
for any $B \in \mathcal{B}(E)$ and for any $B^\star \in \mathcal{B}(E^\star)$, respectively. We can use (\ref{p: density_MO_cop}) from Example \ref{e: exa_MO_cop_AC_SC}
with the transformations 
$u:= \bar{F}_1(y_1)$, $v := \bar{F}_2(y_2)$ and $\alpha = \nu_{12}/(\nu_1 + \nu_{12})$, $\beta= \nu_{12}/(\nu_2 + \nu_{12})$ in order to determine the bivariate density
function $f(y_1, y_2)$ of the Marshall-Olkin exponential distribution, and thereby also $\lambda(y_1,y_2) = nf(y_1,y_2)$:
\[
\lambda(y_1,y_2) =\left\{
\begin{array}{ll}
n \nu_1 (\nu_2 + \nu_{12}) \exp\{ -\nu_1 y_1 - (\nu_2 + \nu_{12})y_2\}, &\text{if } y_1 < y_2,\\
n \nu_2 (\nu_1 + \nu_{12}) \exp\{ - (\nu_1 + \nu_{12})y_1 - \nu_2 y_2\}, & \text{if } y_1 > y_2. \\   
\end{array}
\right.
\]
Substitution by $t:= \exp(-\nu_{12}w)$ in (\ref{p: S_C_MO}) gives 
$S_C(y,y)$ $= $ $\int_0^y \nu_{12}$  $\exp\{-\nu w\}dw$
and the intensity function on the diagonal is thus given by
\[
\acute{\lambda}(y) = n\nu_{12}\exp(-\nu y), \quad \text{for } y_1=y_2=y.
\]
With the affine transformations
\begin{align*}
&\phi:\, E^\star \to E,\quad (x_1,x_2) \mapsto \phi(x_1,x_2) = \left( \frac{x_1+ \log n}{\nu}, \frac{x_2 + \log n}{\nu}\right),\\ 
&\tau:\, [-\log n, \infty) \to [0,\infty), \quad x \mapsto \tau(x) = \frac{x + \log n}{\nu},
\end{align*}
we obtain $\lambda^\star(x_1,x_2) $ $= \nu^{-2}\lambda(\phi(x_1,x_2))$, for $x_1 \neq x_2$, and $\acute{\lambda}^\star(x)$
$= \nu^{-1}\acute{\lambda}(\tau(x))$, for $x_1 = x_2 = x$. 
Hence, for any $(x_1, x_2) \in E^\star$,
\begin{align*}
\lambda^\star(x_1,x_2) &=  
\left\{
\begin{array}{ll}
\frac{\nu_1(\nu_2 + \nu_{12})}{\nu^2} \, \exp\left( -\frac{\nu_1}{\nu}\, x_1 - \frac{\nu_2 + \nu_{12}}{\nu}\, x_2\right), & \text{if }x_1 < x_2, \\
\frac{\nu_2(\nu_1 + \nu_{12})}{\nu^2}\, \exp\left( -\frac{\nu_1 + \nu_{12}}{\nu}\, x_1 - \frac{\nu_2}{\nu}\, x_2 \right), & \text{if } x_1 > x_2,
\end{array} 
\right.\\
\acute{\lambda}^\star(x) &=\frac{\nu_{12}}{\nu}\, \exp (-x), \quad \text{if }x_1=x_2=x.\\
\end{align*}                                                                          
\section{Archimedean copulas with upper tail dependence}\label{Sec: Archim_tail_dep}
A well known class of copulas are the so-called \textit{Archimedean copulas}, discussed in detail in, e.g. Chapter 4 in \cite{Nelsen:2006}. 
A copula $C$ is called Archimedean copula with \textit{generator} $\phi$, if it can be expressed in the following way:
\[
C(u_1, \ldots, u_d) = \phi^{[-1]}(\phi(u_1) + \ldots +  \phi(u_d)), 
\]
where the function $\phi:\, [0,1] \to [0,\infty]$ is continuous, strictly decreasing, convex, and satisfies $\phi(1)=0$. 
The \textit{pseudo-inverse}
$\phi^{[-1]}:\ [0,\infty] \to [0,1]$ of $\phi$ is defined as follows:
\[
\phi^{[-1]}(x) := \left\{ 
\begin{array}{ll} 
 \phi^{-1}(x), & 0 \le x \le \phi(0), \\
 0, & \phi(0) \le x \le \infty.
\end{array} 
\right.
\]
The copula is called \textit{strict} if $\lim_{r \downarrow 0}\phi(r) =\infty$, and in this case $\phi^{[-1]} =  \phi^{-1}$.

Now let $\mathbf{U}_i = (U_i,V_i)$, $i=1, \ldots, n$ be \iid copies of a random pair $(U,V)$ 
whose joint distribution function is given by a bivariate Archimedean copula, and let $A$ denote a region
in $E=[0,1)^2$ of joint upper tail extremes. More precisely, we define, as in Proposition \ref{t: dTV_copula_joint_extreme_A},
\begin{equation}\label{d: MPPE_Arch_A}
A = A_n = \left[ 1- \frac{s_n}{n}, 1\right) \times \left[ 1- \frac{t_n}{n}, 1\right), \quad \textnormal{for any choice of } (s_n, t_n) \in (0,n]^2.
\end{equation}
Moreover, define the number $W_A$ of $(U_i,V_i)$'s in $A$, as well as the MPPE $\Xi_A$ on $E$ with marks $(U_i,V_i)$ by
\begin{equation}\label{d: MPPE_Arch}
W_A = \sum_{i=1}^n I_{\left\{ U_i \ge 1-\frac{s_n}{n}, V_i \ge 1-\frac{t_n}{n}\right\}} \quad \text{and} \quad \Xi_A = \sum_{i=1}^n I_{\left\{ U_i \ge 1-\frac{s_n}{n}, V_i \ge 1-\frac{t_n}{n}\right\}} \delta_{(U_i,V_i)},
\end{equation}
respectively.
It follows from Proposition \ref{t: dTV_copula_joint_extreme_A} that 
\begin{equation}\label{p: Arch_Poi_approx_error}
d_{TV}\left( \mathcal{L}(\Xi_A), \mathrm{PRM}(\mathbb{E}\Xi_A)\right) \le d_{TV}\left( \mathcal{L}(W_A), \mathrm{Poi}(\mathbb{E}W_A)\right)
\le \min\left( \frac{s_n}{n}, \frac{t_n}{n}\right).
\end{equation}
As usual, we can regulate the size of this error estimate by the size of the values $s_n$ and $t_n$. The smaller $\min(s_n, t_n)$ with respect to the sample size $n$,
the smaller the error of the approximation by the Poisson process. However, it is not necessarily clear how to interpret $\mathbb{E}W_A$ or $\mathbb{E}\Xi_A$, 
as the structure of the copula can be complicated. We will illustrate this by way of an example below. 
But first we define, for any $r \in [0,1]$ and for any $x \in [0,\infty]$, 
\begin{equation}\label{d: bar_generator}
\bar{\phi}(r) := \phi(1-r) \quad \textnormal{ and } \quad
\bar{\phi}^{[-1]}(x) :=  1- \phi^{[-1]}(x).
\end{equation}
The expected number of exceedances may then in general be expressed as follows:
\begin{align*}
\mathbb{E}W_A &= nP\left(U \ge 1- \frac{s_n}{n}, V \ge 1- \frac{t_n}{n} \right) \\
&= s_n + t_n - n\bar{\phi}^{[-1]}\left( \bar{\phi}\left(\frac{s_n}{n}\right) + \bar{\phi}\left( \frac{t_n}{n}\right) \right).
\end{align*}
With the normalisation $(s_n,t_n)=(n(1-u),n(1-v))$, where $(u,v) \in E$, we obtain 
\begin{equation}\label{d: Arch_Astar}
A^\star = (0,s_n] \times (0,t_n] \subseteq (0,n]^2.
\end{equation}
On any set $B^\star \in \mathcal{B}((0,n]^2)$, the intensity measure  of the MPPE is then given by 
\[
\mathbb{E}\Xi^\star_{A^\star}(B^\star) = \int_{A^\star \cap B^\star} e_n^\star(s,t) dsdt,
\]
where $e_n^\star(s,t)$ is the intensity function given by
\begin{align}\label{t: int_fct_Arch_cop}
\begin{split}
e_n^\star(s,t) 
&= \frac{\partial^2}{\partial t \partial s} nP\left( U \ge 1- \frac{s}{n} , V \ge 1- \frac{t}{n} \right)\\
&= \frac{\partial^2}{\partial t \partial s}  (-n) \bar{\phi}^{[-1]}\left( \bar{\phi}\left(\frac{s}{n}\right) + \bar{\phi}\left( \frac{t}{n}\right) \right),
\end{split}
\end{align}
for any $(s,t) \in  (0,s_n]\times(0,t_n] \subseteq (0,n]^2$. We next consider $\mathbb{E}W_A$ and $e_n^\star(s,t)$ for a specific example of an MPPE with marks distributed according to an Archimedean copula.
\begin{exa}\label{e: Arch_Gumbel}
The Gumbel copula from Example \ref{d: examples_copulas} (d)
is an Archimedean copula 
with generator $\phi(r) = (-\log r)^\theta$ and $\theta \ge 1$.
As seen in Example \ref{e: examples_copulas_coefs}, it has upper tail dependence. We assume the notation and setting introduced above. The expected number of joint exceedances is then given by
\[
\mathbb{E}W_A = -n+s_n+t_n + 
n\exp{\left\{-\left[\left(-\ln \left(1-\frac{s_n}{n}\right)\right)^\theta + \left(-\ln \left(1-\frac{t_n}{n}\right)\right)^\theta\right]^{\frac{1}{\theta}}\right\}}.
\]
Moreover, by computing the double derivative in (\ref{t: int_fct_Arch_cop}), we obtain the following intensity function:
\begin{align*}
e_n^\star(s,t) 
& = \frac{1}{n} \cdot \frac{\left(\log \left(1-\frac{s}{n}\right)\log \left(1-\frac{t}{n}\right)\right)^{\theta-1}}{\left(1-\frac{s}{n}\right)\left(1-\frac{t}{n}\right)}
\cdot  e^{-\left[\left(-\log \left(1-\frac{s}{n}\right)\right)^\theta + 
\left(-\log \left(1-\frac{t}{n}\right)\right)^\theta\right]^{1/\theta}}\\
&\phantom{=} \cdot  \left[\left(-\log \left(1-\frac{s}{n}\right)\right)^\theta + 
\left(-\log \left(1-\frac{t}{n}\right)\right)^\theta\right]^{\frac{1}{\theta}-2}\\
& \phantom{=}\cdot \left\{ \theta -1 + \left[\left(-\log \left(1-\frac{s}{n}\right)\right)^\theta + 
\left(-\log \left(1-\frac{t}{n}\right)\right)^\theta\right]^\frac{1}{\theta}\right\}. 
\end{align*}
Clearly, it is not evident how many joint threshold exceedances to expect for specific choices of $s_n$ and $t_n$. Moreover, $e_n^\star(s,t)$ is impracticable 
as intensity function of the approximating Poisson process: on the one hand its structure is too complicated to work with, and, on the other hand, it
depends on the sample size $n$. However, by subsequently using $-\log(1-w) \sim w$ and $e^{-z} \sim 1-z$ for $w$, $z \to 0$, we find that
\begin{equation}\label{p: asympt_int_meas_Gumbel}
\mathbb{E}W_A 
 \sim -n + s_n + t_n + n\exp\left\{{-\frac{\left( s_n^\theta + t_n^\theta\right)^{1/\theta}}{n}}\right\}
 \sim s_n + t_n -\left(s_n^\theta + t_n^\theta \right)^{1/\theta}, 
\end{equation}
as $n \to \infty$. A ``nicer'' approximative intensity function can thus be defined as follows
\begin{equation}\label{d: Arch_int_fct}
\lambda^\star(s,t) := \frac{\partial^2}{\partial s \partial t} \left\{ s+t - \left(s^\theta + t^\theta \right)^{1/\theta} \right\}
= (\theta-1) (st)^{\theta - 1} \left(s^\theta + t^\theta\right)^{\frac{1}{\theta}-2},
\end{equation}
for any $(s,t) \in (0,s_n]\times(0,t_n] \subseteq(0,n]^2$. Instead of approximating the law of $\Xi_A$ by a Poisson process with intensity function $e_n^\star(s,t)$, the aim 
would rather be to approximate it by a Poisson process with the simpler intensity function $\lambda^\star(s,t)$ and use Proposition \ref{t: dTV_two_PRM} to estimate 
$d_{TV}(\mathrm{PRM}(\mathbb{E}\Xi^\star_{A^\star}), \mathrm{PRM}(\bl^\star))$, where 
\begin{equation}\label{d: Arch_int_meas}
\bl^\star(B^\star) = \bl^\star_{A^\star}(B^\star)= \int_{A^\star \cap B^\star} \lambda^\star(s,t) dsdt, \quad \textnormal{ for any } B^\star \in \mathcal{B}((0,n]^2). 
\end{equation}
\end{exa}
There is an entire subclass of Archimedean copulas showing the asymptotic behaviour (\ref{p: asympt_int_meas_Gumbel}) for $\mathbb{E}W_A$. This is demonstrated in
Section \ref{s: Charpentier/Segers}, where an asymptotic result due to \cite{Charpentier/Segers:2009} provides a way to determine which Archimedean copulas exhibit upper tail dependence similar to the Gumbel copula. For any such copula, Section \ref{s: Arch_dTV_two_PRM} determines a bound on the total variation distance between $\mathrm{PRM}(\mathbb{E}\Xi^\star_{A^\star})$ and $\mathrm{PRM}(\bl^\star)$. In Section \ref{s: Arch_Examples} we apply our results to a list of examples of Archimedean copulas with upper tail dependence. Section \ref{s: Higher_dimensions} gives a discussion on ways to determine bounds if $d \ge 3$. 
\subsection{An asymptotic result}\label{s: Charpentier/Segers}
\cite{Charpentier/Segers:2009} showed that there is a subclass of Archimedean copulas all displaying the asymptotic behaviour that we found 
in (\ref{p: asympt_int_meas_Gumbel}).
They noted that the upper tail behaviour of Archimedean copulas can be determined by computing some characteristics of their generator $\phi$. 
The following theorem, formulated more generally in
\cite{Charpentier/Segers:2009} (see Theorem 4.1), gives two different possibilities for the asymptotic
behaviour, as $n \to \infty$, of 
\[
\mathbb{E}W_A =  nP\left(U \ge 1- \frac{s_n}{n}, V \ge 1- \frac{t_n}{n} \right) ,
\]
depending on the value of $\tilde{\theta}$ that we define in (\ref{t: limit_theta_Arch_cop}) below.
\begin{thm}\label{t: Charp_Segers_AD}(Charpentier and Segers, 2009)
Let $(U,V)$ be a random pair with standard uniform margins and joint distribution function $C$, which is a bivariate Archimedean copula 
with generator $\phi$. If the limit 
\begin{equation}\label{t: limit_theta_Arch_cop}
\tilde{\theta} := -\lim_{r\downarrow 0} \frac{r\phi'(u)|_{u=1-r}}{\phi(1-r)}, 
\end{equation} 
exists in $[1, \infty)$, then, for every $(s,t) \in (0,\infty)^2$,
\begin{multline}\label{t: ChapSeg_E}
\lim_{n \to \infty} nP\left(U \ge 1- \frac{s}{n}, V \ge 1- \frac{t}{n} \right)\\
= \left\{ 
\begin{array}{ll} 
0, & \textnormal{if }\tilde{\theta}=1,\\
s+t -\left( s^{\tilde{\theta}} + t^{\tilde{\theta}}\right)^{1/\tilde{\theta}}, 
& \textnormal{if }1 < \tilde{\theta} < \infty.
\end{array}
\right.   
\end{multline}
\end{thm}
\begin{remark}\label{r: Arch_RV}
\cite{Charpentier/Segers:2009} showed that (\ref{t: limit_theta_Arch_cop}) is equivalent to regular variation of the function $\bar{\phi}$ at $0$ with index $\tilde{\theta}$, i.e. to
\[
\lim_{r \downarrow 0}\frac{\bar{\phi}(rx)}{\bar{\phi}(r)} = x^{\tilde{\theta}}, \quad \text{for any } x \in (0,\infty). 
\]
This fact is used in the proof of Theorem \ref{t: Charp_Segers_AD} below.
\end{remark}
\begin{proof}
We have
\begin{align*}
&nP\left( U \ge 1- \frac{s}{n}, V \ge 1- \frac{t}{n}\right)\\
&= n \left[ P\left( U \ge 1-\frac{s}{n}\right) 
+ P\left( V \ge 1- \frac{t}{n}\right) \right.\\
&\left.\phantom{blaaaaaaaaaaaaaaaaaaaaaa}-P\left( \left\{ U \ge 1-\frac{s}{n} \right\} \cup \left\{ V \ge 1- \frac{t}{n}\right\}\right) \right]\\
&= s+t -nP\left( \left\{ U \ge 1-\frac{s}{n} \right\} \cup \left\{ V \ge 1- \frac{t}{n}\right\}\right),
\end{align*}
where
\begin{align}
&nP\left( \left\{ U \ge 1-\frac{s}{n} \right\} \cup \left\{ V \ge 1- \frac{t}{n}\right\}\right)
= n \bar{\phi}^{[-1]} \left(  \bar{\phi}\left( \frac{s}{n} \right) + \bar{\phi} \left( \frac{t}{n}\right)\right) \nonumber \\
&= \frac{1}{\bar{\phi}^{[-1]}(\bar{\phi}\left( \frac{1}{n}\right))}\cdot \bar{\phi}^{[-1]}\left( \bar{\phi}\left(\frac{1}{n}\right)\cdot \left[ \frac{\bar{\phi}\left( \frac{s}{n}\right)}{\bar{\phi}\left( \frac{1}{n}\right)} + \frac{\bar{\phi}\left( \frac{t}{n}\right)}{\bar{\phi}\left( \frac{1}{n}\right)} \right] \right).\label{p: Arch_CS}
\end{align}
By Remark \ref{r: Arch_RV}, $\bar{\phi}$ is regularly varying at $0$ with index $\tilde{\theta}$. It follows that the function $x \mapsto 1/\bar{\phi}(1/x)$ is regularly varying at infinity with index $\tilde{\theta}$. By Theorem 1.5.12 in \cite{Bingham_et_al:1987}, its inverse function, i.e. $y \mapsto 1/\bar{\phi}^{[-1]}(1/y)$ is regularly varying at infinity with index $1/\tilde{\theta}$, and the function $\bar{\phi}^{[-1]}$ is regularly varying at $0$ with index $1/\tilde{\theta}$. 
By the uniform convergence theorem (see Theorem 1.5.2 in \cite{Bingham_et_al:1987}), (\ref{p: Arch_CS}) thus converges to 
\[
(s^{\tilde{\theta}} + t^{\tilde{\theta}})^{1/{\tilde{\theta}}}, \quad \text{as } n \to \infty.
\]
\end{proof}
\noindent 
If $\tilde{\theta} = 1$, the copula displays asymptotic independence in the upper tail. For $1 < \tilde{\theta} < \infty$, it shows upper tail dependence and we recognise the limiting structure of $\mathbb{E}W_A$ in (\ref{p: asympt_int_meas_Gumbel}) that we obtained for the Gumbel example. 
Indeed, for the Gumbel copula with generator $\phi(r) = (-\log r)^\theta$, we have  
\begin{align*}
& \phi'(u)|_{u=1-r} = \frac{(-\theta)[-\log (1-r)]^{\theta-1}}{1-r},  \\
\textnormal{and} \quad & \tilde{\theta} 
= \lim_{r \downarrow 0} \frac{\theta [-\log(1-r)]^{\theta-1}}{(1-r)[-\log (1-r)]^{\theta}} 
=\lim_{r \downarrow 0} \frac{\theta r}{(1-r)[-\log(1-r)]} =\theta \in [1,\infty). 
\end{align*}
For $\theta = 1$, the Gumbel copula reduces to the independence copula, which is obviously asymptotically independent (see Example \ref{e: examples_copulas_coefs}
for the coefficients of upper tail dependence of the independence and Gumbel copulas). 
In the following, we only consider Archimedean copulas with upper tail dependence as described by Theorem \ref{t: Charp_Segers_AD} for $\tilde{\theta} \in (1, \infty)$. 

\subsection{Approximation in $d_{TV}$ by a Poisson process}\label{s: Arch_dTV_two_PRM}
The aim is now to determine an estimate of $d_{TV}(\mathrm{PRM}(\mathbb{E}\Xi_A),\mathrm{PRM}(\bl))$ 
for all Archime\-dean copulas with parameter $\theta \in (1, \infty)$
and $\tilde{\theta} = \theta$, 
where $\bl^\star$ and $\lambda^\star$ are defined as in (\ref{d: Arch_int_meas}) and (\ref{d: Arch_int_fct}), respectively. 
We define a function $\bar{\phi}_0$ and its inverse $\bar{\phi}_0^{[-1]}$ as follows: 
\begin{align}\label{d: phi0}
\begin{split}
&\bar{\phi}_0 :\, [0,1] \to [0,\infty], \quad r \mapsto \bar{\phi}_0(r) = r^\theta, \\
&\bar{\phi}_0^{[-1]}:\, [0,\infty] \to [0,1], \quad x \mapsto 
 \bar{\phi}_0^{[-1]}(x) = x^{1/\theta}. 
\end{split}
\end{align}
Now note that we may express $\lambda^\star$ in terms of $\bar{\phi}_0$ and $\bar{\phi}_0^{[-1]}$:
\begin{align*}
\lambda^\star(s,t)
&= \frac{\partial^2}{\partial s \partial t}\, (-n) \left\{ \left( \frac{s}{n} \right)^\theta + \left( \frac{t}{n}\right)^\theta\right\}^{1/\theta}\\
&= \frac{\partial^2}{\partial s \partial t}\, (-n) \bar{\phi}_0^{[-1]}\left(\bar{\phi}_0 \left( \frac{s}{n} \right) + \bar{\phi}_0\left( \frac{t}{n}\right) \right).
\end{align*}
The following theorem now gives an upper bound on the error involved in approximating a Poisson process with mean measure $\mathbb{E}\Xi^\star_{A^\star}$ by another Poisson process with the more useful mean measure $\bl^\star$. It uses Proposition \ref{t: dTV_two_PRM} and properties of the generator $\phi$. By adding the upper bound from Theorem \ref{t: MyMichel_Arch} below to the error estimate in (\ref{p: Arch_Poi_approx_error}), we obtain an upper bound for the total variation distance between $\mathcal{L}(\Xi^{\star}_{A^\star})$ and $\mathrm{PRM}(\bl^\star)$.
\begin{thm}\label{t: MyMichel_Arch}
Let $C_\theta$ be a bivariate Archimedean copula with parameter $\theta \in (1,\infty)$ such that $\tilde{\theta} = \theta$ for $\tilde{\theta}$ 
defined by (\ref{t: limit_theta_Arch_cop}). Let $\phi$ be the generator of $C_\theta$ and suppose that $\bar{\phi}(r) = w^\theta(r) = r^\theta h^\theta(r)$, 
where $w$ is twice continuously differentiable on $[0,1)$. 
Suppose that for some $\delta \in (0,1)$, we have $h(0)>0$, $h'(r) \ge 0$, and $w''(r) \ge w''(0)$ for all $r \in [0,\delta)$. 
For each integer $n \ge 1$, let $(U_1, V_1), \ldots, (U_n, V_n)$ be \iid random pairs whose common joint distribution function is 
given by $C_\theta$. Assume the setting and notation from (\ref{d: MPPE_Arch_A})-(\ref{d: Arch_int_meas}), and define 
$H(r) = \max_{0 \le \xi \le r}h'(\xi)$ and $W(r) = \max_{0 \le \xi \le r}w''(r)$. Then there exists $r_0 \in [0,\delta)$ 
such that for all $r \le r_0$, $w'(r) \le 4h(0)/3$; and for $(s_n,t_n) \in (0,n]^2$ such that 
\[
\frac{s_n}{n}\, , \frac{t_n}{n} \le \frac{3r_0}{8}\,,
\]
it follows that
\[
 d_{TV}(\mathrm{PRM}(\mathbb{E}\Xi^\star_{A^\star}), \mathrm{PRM}(\bl^\star)) \\
 \le \frac{K(s_n+t_n)^2}{n}, 
\]
where $K=K(\theta, h(0), r_0, H(r_0),W(r_0))$ is defined by
\[
K= \frac{\pi(\sqrt{2})^\theta}{2}\, 
\left[(\theta-1)\kappa +\left( \frac{4}{3}\right)^{2\theta} \frac{W(r_0)}{h(0)} \right],
\]
and
\begin{multline*}
\kappa =  \frac{H(r_0)}{h(0)}\,
\max\left\{(\theta+1)\left(1+\frac{3r_0H(r_0)}{16h(0)}\right) \left[ 1+ \frac{3r_0H(r_0)}{4h(0)} + \frac{9r_0^2H(r_0)^2}{256h(0)^2}\right]^\theta, \right.\\
\left. (2\theta-1)2^{\theta-1}\left(1+\frac{3r_0H(r_0)}{8h(0)}\right)^{\theta-1}+ 2\left( 1+ \frac{3r_0H(r_0)}{4h(0)}\right)  \right\}.
\end{multline*}
\end{thm}
\begin{proof}
By Proposition \ref{t: dTV_two_PRM},
\[
d_{TV}(\mathrm{PRM}(\mathbb{E}\Xi^\star_{A^\star}), \mathrm{PRM}(\bl^\star)) \le \int_0^{t_n} \int_0^{s_n} \left| e^\star_n(s,t)-\lambda^\star(s,t) \right| dsdt,
\]
where $e_n^\star(s,t)$ and $\lambda^\star(s,t)$ are given by (\ref{t: int_fct_Arch_cop}) and (\ref{d: Arch_int_fct}), respectively. 
Note that
\[
 \lambda^\star(nu,nv) 
 =\frac{1}{n}\,(\theta-1)(uv)^{\theta-1} \left(u^\theta + v^\theta\right)^{\frac{1}{\theta}-2}
\text{with } (u,v):=(u_n,v_n):= \left(\frac{s}{n},\frac{t}{n}\right).
\]
In order to express $e^\star_n(s,t)$ in terms of $u$ and $v$, note that for $r  \in [0,1]$, and $x=\bar{\phi}(r)$ $ = \bar{\phi}_0(w(r))$, with $\bar{\phi}_0$ from (\ref{d: phi0}),
we have 
\[
r = \bar{\phi}^{[-1]}(x) = w^{-1}(\bar{\phi}_0^{-1}(x)) .
\]
Also, note that $\partial^2/(\partial s \partial t)$ $ =$ $ n^{-2} \partial^2/(\partial u \partial v)$
by the chain rule. Thus, $e^\star_n(s,t)$ equals
\begin{multline*}
\frac{\partial^2}{\partial s \partial t} \, (-n) w^{-1}\left[\left( w^\theta \left( \frac{s}{n}\right)  + w^\theta \left( \frac{t}{n}\right)\right)^{1/\theta} \right]\\
=\left(- \frac{1}{n} \right) \frac{\partial^2}{\partial u \partial v} \,w^{-1} \left[ \left(w^\theta(u) + w^\theta(v)\right)^{1/\theta}\right],
\end{multline*}
where
\begin{align*}
&\frac{\partial^2}{\partial u \partial v} \, w^{-1} \left[ \left(w^\theta(u) + w^\theta(v)\right)^{1/\theta}\right] \\
&= 
\frac{
\left\{ w(u)w(v)\right\}^{\theta-1} w'(u)w'(v)}  
{ w'\left( w^{-1}\left[ \left( w^\theta(u) + w^\theta(v)\right)^{1/\theta}\right]\right)}\,
\left( w^\theta(u) + w^\theta(v)\right)^{\frac{1}{\theta}-2}\\
& \phantom{==}\cdot \left\{ - \frac{w''\left( w^{-1}\left[ \left( w^\theta(u) + w^\theta(v)\right)^{1/\theta}\right]\right)}{w'^{2}\left( w^{-1}\left[ \left( w^\theta(u) + w^\theta(v)\right)^{1/\theta}\right]\right)} 
\,  \left( w(u)^\theta + w(v)^\theta\right)^{1/\theta} + (1-\theta)
\right\}.
\end{align*}
We may thus bound the integrand $|e^\star_n(s,t) - \lambda^\star(s,t)| = |e^\star_n(nu,nv) - \lambda^\star(nu,nv)|$ by 
\begin{align}\label{p: two_error_terms}
\begin{split}
&\frac{1}{n}\, \left|\frac{\left\{ w(u)w(v)\right\}^{\theta-1} \cdot w'(u)w'(v)\cdot
w''\left( w^{-1}\left[ \left( w^\theta(u) + w^\theta(v)\right)^{1/\theta}\right]\right)}
{\left( w^\theta(u) + w^\theta(v)\right)^{2-\frac{2}{\theta}} \cdot w'^{3}\left( w^{-1}\left[ \left( w^\theta(u) + w^\theta(v)\right)^{1/\theta}\right]\right)} 
\right| \\
& + \frac{\theta-1}{n}\, \left| 
\frac{\left\{ w(u)w(v)\right\}^{\theta-1} \cdot w'(u)w'(v)\cdot\left(w^\theta(u) + w^\theta(v)\right)^{\frac{1}{\theta}-2}}{ w'\left( w^{-1}\left[ \left( w^\theta(u) + w^\theta(v)\right)^{1/\theta}\right]\right)}
\right.\\
&\left. \phantom{blaaaaaaaaaaaaaaaaaaaaaaaaaaaa\frac{w(t)^{\frac{1}{t}}}{w^{1}}a}
-(uv)^{\theta-1}\left( u^\theta + v^\theta \right)^{\frac{1}{\theta}-2}
\right|.
\end{split}
\end{align}
Note that since $w(r)=rh(r)$, we have
$w'(r) = h(r)+ rh'(r)$, and $w''(r) = 2h'(r) + rh''(r)$. 
Furthermore, $w'(0)=h(0)>0$, and, since $h'(r) \ge 0$, we have $w''(0) = 2h'(0) \ge 0$ and 
\begin{equation}\label{p: Arch_1Der_w_lower}
w'(r) \ge h(0) + rh'(r) \ge h(0),  \quad\text{for all } r \in [0,\delta).
\end{equation}
By continuity of $w'$ at $0$, there has to exist $r_0 \in [0,\delta)$ such that 
\begin{equation}\label{p: Arch_1Der_w_upper}
w'(r) \le \frac{4h(0)}{3}, \quad \text{for all } r \le r_0.
\end{equation}
It follows that 
\begin{equation}\label{p: Arch_h_upper}
h(r) \le \frac{4h(0)}{3}, \quad \text{for all } r \le r_0.
\end{equation}
We now determine bounds on the inverse of $w$ close to $0$.
Since $rh(0) \le w(r) \le (4/3)rh(0)$ for all $r \le r_0$,
it follows that for all $x \le r_0h(0)$,
\begin{equation}\label{p: Arch_Inverse_w}
\frac{3x}{4h(0)} \le w^{-1}(x) \le \frac{x}{h(0)}\,.
\end{equation}
See Figure \ref{f: Inverse_w} for an illustration. 
\begin{figure}
\begin{center}{\footnotesize \input{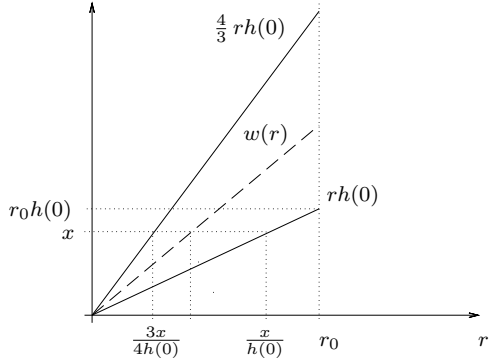}}\end{center}
\caption{
If $w(r)$ is wedged between two straight lines for all $r  \le r_0$, then its inverse must be wedged between the inverses of these lines, i.e. $\frac{3x}{4h(0)} \le w^{-1}(x) \le \frac{x}{h(0)}\,$, for all $x \le r_0h(0)$.}
\label{f: Inverse_w}
\end{figure}
Note that with the well known property 
\[ 2^{\frac{1}{p}-\frac{1}{q}}||(s,t)||_q \le ||(s,t)||_p \le ||(s,t)||_q \] 
of the $p$-norm 
$||(s,t)||_p$ $ := (s^p + t^p)^{1/p}$ that holds if $p > q >0$, we obtain the following inequalities:
\begin{equation}\label{p: Arch_p_norm}
\left( s^\theta + t^\theta\right)^{\frac{1}{\theta}} \le s+t, \quad
s^{\theta+ 1} + t^{\theta+1} \le (s+t)^{\theta+1}, \quad
s^\theta + t^\theta \ge 2^{1-\theta} (s+t)^\theta.
\end{equation}
For $(s,t)\in (0,s_n] \times (0,t_n] \subseteq (0,n]^2$ such that
\[
u=\frac{s}{n} \le \frac{s_n}{n} \le \frac{3r_0}{8} \quad  \text{and} \quad  v=\frac{t}{n} \le \frac{t_n}{n} \le \frac{3r_0}{8}\,,
\]
we thus have, using (\ref{p: Arch_p_norm}) and (\ref{p: Arch_h_upper}),
\[
\left( w^\theta\left( \frac{s}{n}\right) + w^\theta\left( \frac{t}{n}\right)\right)^{\frac{1}{\theta}}
\le w\left( \frac{s}{n}\right) + w \left(\frac{t}{n}\right) \le \frac{4h(0)(s+t)}{3n} \le r_0h(0).
\]
Therefore, by (\ref{p: Arch_Inverse_w}), 
\begin{multline}\label{p: Arch_Inverse_w_welldef}
0 \le \frac{3}{4h(0)}\,\left( w^\theta\left( \frac{s}{n}\right) + w^\theta\left( \frac{t}{n}\right)\right)^{\frac{1}{\theta}}\\
 \le 
w^{-1}\left[ \left( w^\theta\left( \frac{s}{n}\right) + w^\theta\left( \frac{t}{n}\right)\right)^{\frac{1}{\theta}}\right] 
\le \frac{1}{h(0)}\, \left( w^\theta\left( \frac{s}{n}\right) + w^\theta\left( \frac{t}{n}\right)\right)^{\frac{1}{\theta}} \le r_0.
 \end{multline}
We proceed by determining bounds on the first of the two error terms in (\ref{p: two_error_terms}).
By (\ref{p: Arch_h_upper}), and since $h(r) \ge h(0)$ (due to $h'(r) \ge 0$ for all $r \in [0,\delta$)), 
\begin{equation}\label{p: Arch_prod_of_two_w}
\left( \frac{h(0)}{n}\right)^{2\theta-2}(st)^{\theta-1}
 \le \left\{ w\left( \frac{s}{n}\right) w\left( \frac{t}{n}\right)\right\}^{\theta-1}
 \le \left( \frac{4}{3}\right)^{2\theta -2}\left( \frac{h(0)}{n}\right)^{2\theta-2}(st)^{\theta-1}.
\end{equation}
Moreover, since $\frac{2}{\theta}-2 <0$, we obtain
\begin{multline}\label{p: Arch_long_term_for1stbound}
\left( \frac{4}{3}\right)^{2-2\theta}\left( \frac{h(0)}{n}\right)^{2-2\theta}\left( s^\theta + t^\theta\right)^{\frac{2}{\theta}-2}\\
\le \left( w^\theta\left( \frac{s}{n}\right) + w^\theta\left( \frac{t}{n}\right)\right)^{\frac{2}{\theta}-2}
\le \left( \frac{h(0)}{n}\right)^{2-2\theta}\left( s^\theta + t^\theta\right)^{\frac{2}{\theta}-2}.
\end{multline}
By (\ref{p: Arch_1Der_w_lower}) and (\ref{p: Arch_1Der_w_upper}),
\begin{equation}\label{p: Arch_prod_of_two_Der_w}
h(0)^2 \le w'\left( \frac{s}{n}\right) w'\left( \frac{t}{n}\right) \le \left( \frac{4}{3}\right)^{2} h(0)^2,
\end{equation}
and also, due to (\ref{p: Arch_Inverse_w_welldef}),
\[
h(0) 
\le  w'\left( w^{-1}\left[ \left( w^\theta\left( \frac{s}{n}\right) + w^\theta\left( \frac{t}{n}\right)\right)^{\frac{1}{\theta}}\right] \right)
\le \frac{4}{3}\, h(0).
\]
It follows that
\begin{equation}\label{p: Arch_Inv_thrice_Der_w}
\left( \frac{3}{4h(0)}\right)^3
\le \frac{1}{w'^3\left( w^{-1}\left[ \left( w^\theta\left( \frac{s}{n}\right) + w^\theta\left( \frac{t}{n}\right)\right)^{\frac{1}{\theta}}\right] \right)}
\le \frac{1}{h(0)^3}\,.
\end{equation}
Furthermore, by (\ref{p: Arch_Inverse_w_welldef}), and since $w''(r) \ge w''(0) = 2h'(0)$,
\begin{equation}\label{p: Arch_2Der_w}
0 \le 2h'(0)
\le w'' \left(w^{-1}\left[ \left( w^\theta\left( \frac{s}{n}\right) + w^\theta\left( \frac{t}{n}\right)\right)^{\frac{1}{\theta}}\right]  \right)
\le \max_{0 \le \xi \le r_0}w''(\xi) = W(r_0).
\end{equation}
Inequalities (\ref{p: Arch_prod_of_two_w})-(\ref{p: Arch_2Der_w}) yield 
the following upper bound on the first of the two error terms in (\ref{p: two_error_terms}):
\begin{align}\label{p: Arch_first_of_error_bounds}
\begin{split}
 &\frac{1}{n} \cdot\frac{\left\{ w\left( \frac{s}{n}\right)w\left( \frac{t}{n}\right)\right\}^{\theta-1}  w'\left( \frac{s}{n}\right)w'\left( \frac{t}{n}\right)
w''\left( w^{-1}\left[ \left( w^\theta\left( \frac{s}{n}\right) + w^\theta\left( \frac{t}{n}\right)\right)^{1/\theta}\right]\right)}
{\left( w^\theta\left( \frac{s}{n}\right) + w^\theta\left( \frac{t}{n}\right)\right)^{2-\frac{2}{\theta}} w'^{3}\left( w^{-1}\left[ \left( w^\theta\left( \frac{s}{n}\right) + w^\theta\left( \frac{t}{n}\right)\right)^{1/\theta}\right]\right)} 
\, (\ge 0)\\
& \le \frac{1}{n} \left( \frac{4}{3}\right)^{2\theta} \frac{W(r_0)}{h(0)}\,(st)^{\theta-1}\left( s^\theta + t^\theta\right)^{\frac{2}{\theta}-2}
\\ & \le \frac{s+t}{n} \left( \frac{4}{3}\right)^{2\theta} \frac{W(r_0)}{h(0)}\,(st)^{\theta-1}\left( s^\theta + t^\theta\right)^{\frac{1}{\theta}-2},
\end{split}
\end{align}
where we used (\ref{p: Arch_p_norm}) for the last inequality.
We next determine bounds on the second of the two error terms in (\ref{p: two_error_terms}).
Note that for all $r \in [0,\delta)$, 
\begin{equation}\label{p: Arch_new_bounds_h}
h(0) \le h(r) \le h(0) + r\max_{0 \le \xi \le r}h'(r) = h(0)+ rH(r),
\end{equation}
and thereby,
\begin{equation}\label{p: Arch_new_bounds_1Der_w}
h(0) \le h(0)+ rh'(r) \le w'(r) \le h(0) + rH(r) + rh'(r) \le h(0) + 2rH(r).
\end{equation}
By (\ref{p: Arch_Inverse_w_welldef}) and (\ref{p: Arch_new_bounds_1Der_w}),
\begin{align*}
h(0) 
&\le 
 w'\left( w^{-1}\left[ \left( w^\theta\left( \frac{s}{n}\right) + w^\theta\left( \frac{t}{n}\right)\right)^{\frac{1}{\theta}}\right] \right)\\
&\le h(0) + 2w^{-1}\left[ \left( w^\theta\left( \frac{s}{n}\right) + w^\theta\left( \frac{t}{n}\right)\right)^{\frac{1}{\theta}}\right] \\
&  \phantom{blaaaaaaaaaaaaaaaaaa}\cdot H\left(w^{-1}\left[ \left( w^\theta\left( \frac{s}{n}\right) + w^\theta\left( \frac{t}{n}\right)\right)^{\frac{1}{\theta}}\right] \right),
\end{align*}
and by (\ref{p: Arch_Inverse_w_welldef}), (\ref{p: Arch_new_bounds_h}) and (\ref{p: Arch_p_norm}),
\begin{align*}
w^{-1}\left[ \left( w^\theta\left( \frac{s}{n}\right) + w^\theta\left( \frac{t}{n}\right)\right)^{\frac{1}{\theta}}\right] 
&\le \frac{1}{h(0)} \left( w^\theta\left( \frac{s}{n}\right) + w^\theta\left( \frac{t}{n}\right)\right)^{\frac{1}{\theta}} \\
&\le \frac{\frac{s}{n}\,h\left( \frac{s}{n}\right) +  \frac{t}{n}\,h\left( \frac{t}{n}\right)}{h(0)}\\
&\le \frac{s+t}{n} + \frac{s^2H \left(\frac{s}{n}\right) + t^2 H \left(\frac{t}{n}\right) }{n^2h(0)}\,.
\end{align*}
Furthermore, note that since
\[
w^{-1}\left[ \left( w^\theta\left( \frac{s}{n}\right) + w^\theta\left( \frac{t}{n}\right)\right)^{\frac{1}{\theta}}\right]  \le r_0 \quad \text{and} \quad
\frac{s}{n}\, , \frac{t}{n} \le \frac{3r_0}{8}\,,
\]
we have 
\begin{equation}\label{p: Arch_upperbounds_H}
H\left(w^{-1}\left[ \left( w^\theta\left( \frac{s}{n}\right) + w^\theta\left( \frac{t}{n}\right)\right)^{\frac{1}{\theta}}\right]  \right)  \le H(r_0) 
\text{ and } H\left( \frac{s}{n}\right), H\left( \frac{t}{n}\right)  \le H(r_0).
\end{equation}
It follows that
\begin{multline*}
h(0) 
\le  w'\left( w^{-1}\left[ \left( w^\theta\left( \frac{s}{n}\right) + w^\theta\left( \frac{t}{n}\right)\right)^{\frac{1}{\theta}}\right] \right)\\
\le h(0) + 2H(r_0)\, \frac{s+t}{n} + \frac{2H^2(r_0)}{h(0)}\, \left( \frac{s+t}{n}\right)^2,
\end{multline*}
and thereby
\begin{multline}\label{p: Arch_Inv_1Der_w}
\frac{1}{h(0)}\left\{1 - \frac{2H(r_0)(s+t)}{h(0) n} - \frac{2H^2(r_0)}{h(0)^2}\,\left(\frac{s+t}{n} \right)^2 \right\}\\
\le \frac{1}{ w'\left( w^{-1}\left[ \left( w^\theta\left( \frac{s}{n}\right) + w^\theta\left( \frac{t}{n}\right)\right)^{\frac{1}{\theta}}\right] \right)}
\le \frac{1}{h(0)}\,,
\end{multline}
where we used for the lower bound that $(b+az)^{-1} \ge b^{-1} - azb^{-2}$ for $b >0$, $a \ge 0$ and $z <1$. Similarly to (\ref{p: Arch_long_term_for1stbound}), we have 
\begin{equation}\label{p: Arch_long_term_for2ndbound_upper}
\left( w^\theta\left( \frac{s}{n}\right) + w^\theta\left( \frac{t}{n}\right)\right)^{\frac{1}{\theta}-2} \le \left( \frac{h(0)}{n} \right)^{1-2\theta}\left( s^\theta + t^\theta\right)^{\frac{1}{\theta}-2}
\end{equation}
For a lower bound, we successively use (\ref{p: Arch_new_bounds_h}), (\ref{p: Arch_upperbounds_H}), the inequalities
$(1+z)^\theta \le 1 + \theta z(1+z)^{\theta-1}$ and $(1+z)^{1/\theta -2} \ge 1-(2 - 1/\theta)z$, for $0\le z \le 1$, as well as $s/n, t/n \le 3r_0/8$ and (\ref{p: Arch_p_norm}). 
That is,
\begin{align}\label{p: Arch_long_term_for2ndbound_lower}
\begin{split}
&\left( w^\theta\left( \frac{s}{n}\right) + w^\theta\left( \frac{t}{n}\right)\right)^{\frac{1}{\theta}-2} \\
&\ge\left( \frac{h(0)}{n}\right)^{1-2\theta}
\left( s^\theta + t^\theta \right)^{\frac{1}{\theta}-2}
 \left\{\frac{s^\theta\left[ 1 + \frac{sH(r_0)}{nh(0)}\right]^\theta + t^\theta\left[ 1 + \frac{tH(r_0)}{nh(0)}\right]^\theta }{s^\theta + t^\theta}\right\}^{\frac{1}{\theta}-2} \\
&\ge \left( \frac{h(0)}{n}\right)^{1-2\theta}
\left( s^\theta + t^\theta \right)^{\frac{1}{\theta}-2} 
\left\{1+ \frac{\theta H(r_0) \left(s^{\theta+1} + t^{\theta +1}\right)}{nh(0)(s^\theta + t^\theta)}
\right. \\ & \left. \phantom{blaaaaaaaaaaaaaaaaaaaaaaaaaaaaaaaaaaaaaaa}
\cdot\left[1+\frac{3r_0H(r_0)}{8h(0)}\right]^{\theta-1}  \right\}^{\frac{1}{\theta}-2}\\
&\ge \left( \frac{h(0)}{n}\right)^{1-2\theta}
\left( s^\theta + t^\theta \right)^{\frac{1}{\theta}-2}
\left\{ 1 - \frac{(2\theta-1)2^{\theta-1}H(r_0)(s+t)}{nh(0)}\right.\\
&\left.\phantom{blaaaaaaaaaaaaaaaaaaaaaaaaaaaaaaaaaaaaaaa}\cdot\left[1+\frac{3r_0H(r_0)}{8h(0)}\right]^{\theta-1}\right\}.
\end{split}
\end{align}
Similarly, 
\begin{multline}\label{p: Arch_new_bounds_prod}
\left( \frac{ h(0)}{n}\right)^{2\theta-2} (st)^{\theta-1} 
\le \left\{ w\left( \frac{s}{n}\right) w\left( \frac{t}{n}\right)\right\}^{\theta-1}\\
\le \left( \frac{ h(0)}{n}\right)^{2\theta-2} (st)^{\theta-1}
\left[ 1 + \frac{sH(r_0)}{nh(0)}\right]^{\theta-1}\left[ 1 + \frac{tH(r_0)}{nh(0)}\right]^{\theta-1}.
\end{multline}
By (\ref{p: Arch_new_bounds_1Der_w}) and (\ref{p: Arch_upperbounds_H}),
\begin{equation}\label{p: Arch_new_bounds_prod_1Der_w}
h(0)^2 \le w'\left( \frac{s}{n}\right) w'\left( \frac{t}{n}\right) 
\le h(0)^2 \left[ 1 + \frac{2sH(r_0)}{nh(0)}\right]\left[ 1 + \frac{2tH(r_0)}{nh(0)}\right].
\end{equation}
Multiplication of (\ref{p: Arch_new_bounds_prod}) and (\ref{p: Arch_new_bounds_prod_1Der_w})
gives 
\begin{align}\label{p: Arch_helping0}
\begin{split}
&\left( \frac{1}{n}\right)^{2\theta-2} h(0)^{2\theta}(st)^{\theta-1} \\
&\le   \left\{ w\left( \frac{s}{n}\right) w\left( \frac{t}{n}\right)\right\}^{\theta-1}w'\left( \frac{s}{n}\right) w'\left( \frac{t}{n}\right) \\
&\le \left( \frac{1}{n}\right)^{2\theta-2} h(0)^{2\theta}(st)^{\theta-1}
\left[ 1 + \frac{sH(r_0)}{nh(0)}\right]^{\theta+1}\left[ 1 + \frac{tH(r_0)}{nh(0)}\right]^{\theta+1}\\
& \le \left( \frac{1}{n}\right)^{2\theta-2} h(0)^{2\theta}(st)^{\theta-1}
\left[ 1 + \frac{H(r_0)(s+t)}{nh(0)} + \frac{H(r_0)^2(s+t)^2}{4n^2h(0)^2} \right]^{\theta+1},
\end{split}
\end{align}
where we used $(1+2z)\le (1+z)^2$ for the second inequality, and $st \le (s+t)^2/4$ (due to the inequality of arithmetic and geometric means) for the third. 
For the part in brackets, we use $(s+t)/n \le 3r_0/4$ and the inequality $(1+z)^{\theta+1} \le 1+(\theta +1)z(1+z)^\theta$. We obtain
\begin{align}\label{p: Arch_helping1}
&\left[ 1 + \frac{H(r_0)(s+t)}{nh(0)} + \frac{H(r_0)^2(s+t)^2}{4n^2h(0)^2} \right]^{\theta+1} \nonumber\\
&\le \left[ 1 + \frac{H(r_0)(s+t)}{nh(0)}\left( 1+ \frac{3r_0H(r_0)}{16h(0)}\right) \right]^{\theta+1}\\
&\le  1+ \frac{(\theta+1)H(r_0)(s+t)}{nh(0)}\left(1+\frac{3r_0H(r_0)}{16h(0)}\right)
\left[ 1+ \frac{3r_0H(r_0)}{4h(0)} + \frac{9r_0^2H(r_0)^2}{256h(0)^2}\right]^\theta .  \nonumber
\end{align}
By (\ref{p: Arch_Inv_1Der_w})-(\ref{p: Arch_helping1}), we obtain, on the one hand,
\begin{align*}
&\frac{\left\{ w\left( \frac{s}{n}\right)w\left( \frac{t}{n}\right)\right\}^{\theta-1} 
\cdot w'\left( \frac{s}{n}\right)w'\left( \frac{t}{n}\right)
\cdot\left(w^\theta\left( \frac{s}{n}\right) + w^\theta\left( \frac{t}{n}\right)\right)^{\frac{1}{\theta}-2}}{ w'\left( w^{-1}\left[ \left( w^\theta\left( \frac{s}{n}\right) + w^\theta\left( \frac{t}{n}\right)\right)^{1/\theta}\right]\right)} \\
&\phantom{blaaaaaaaaaaaaaaaaaaaaaaaaaaaaaaaaaaaa}-n(st)^{\theta-1}\left( s^\theta + t^\theta \right)^{\frac{1}{\theta}-2} \\
\le & \,\,(st)^{\theta-1}\left( s^\theta + t^\theta \right)^{\frac{1}{\theta}-2}  (s+t)\\ 
& \phantom{blaa}
\cdot \frac{(\theta+1)H(r_0)}{h(0)}\left(1+\frac{3r_0H(r_0)}{16h(0)}\right) \left[ 1+ \frac{3r_0H(r_0)}{4h(0)} + \frac{9r_0^2H(r_0)^2}{256h(0)^2}\right]^\theta,\\
\end{align*}
and, on the other hand, again using $(s+t)/n \le 3r_0/4$ for the lower bound in (\ref{p: Arch_Inv_1Der_w}),
\begin{align*}
&n(st)^{\theta-1}\left( s^\theta + t^\theta \right)^{\frac{1}{\theta}-2}\\
&\phantom{blaaaaaa}-
\frac{\left\{ w\left( \frac{s}{n}\right)w\left( \frac{t}{n}\right)\right\}^{\theta-1} \cdot w'\left( \frac{s}{n}\right)w'\left( \frac{t}{n}\right)\cdot\left(w^\theta\left( \frac{s}{n}\right) + w^\theta\left( \frac{t}{n}\right)\right)^{\frac{1}{\theta}-2}}{ w'\left( w^{-1}\left[ \left( w^\theta\left( \frac{s}{n}\right) + w^\theta\left( \frac{t}{n}\right)\right)^{1/\theta}\right]\right)}\\
\le& \,\,  n(st)^{\theta-1}\left( s^\theta + t^\theta \right)^{\frac{1}{\theta}-2}\\
&\cdot \left\{1-\left[ 1 - \frac{(2\theta-1)2^{\theta-1}H(r_0)(s+t)}{nh(0)}\left(1+\frac{3r_0H(r_0)}{8h(0)}\right)^{\theta-1}\right] \right.\\
&\left. \phantom{blaaaaaaaaaaaaaaaaaaaa,a}\cdot\left[1 - \frac{2H(r_0)(s+t)}{h(0) n}\left( 1+\frac{3r_0H(r_0)}{4h(0)}\right)\right] \right\}\\
\le& \,\,  (st)^{\theta-1}\left( s^\theta + t^\theta \right)^{\frac{1}{\theta}-2}(s+t)\\
&\phantom{blaa} \cdot\frac{H(r_0)}{h(0)}  \left[(2\theta-1)2^{\theta-1}\left(1+\frac{3r_0H(r_0)}{8h(0)}\right)^{\theta-1}+ 2\left( 1+ \frac{3r_0H(r_0)}{4h(0)}\right) \right].
\end{align*}
Hence, the second of the two error terms in (\ref{p: two_error_terms}) is bounded by
\begin{equation}\label{p: Arch_second_of_error_bounds}
(\theta-1)(st)^{\theta-1}\left( s^\theta + t^\theta \right)^{\frac{1}{\theta}-2}\,\frac{\kappa(s+t)}{n}\,,
\end{equation}
where 
\begin{multline*}
\kappa =  \frac{H(r_0)}{h(0)}\,
\max\left\{(\theta+1)\left(1+\frac{3r_0H(r_0)}{16h(0)}\right) \left[ 1+ \frac{3r_0H(r_0)}{4h(0)} + \frac{9r_0^2H(r_0)^2}{256h(0)^2}\right]^\theta, \right.\\
\left. (2\theta-1)2^{\theta-1}\left(1+\frac{3r_0H(r_0)}{8h(0)}\right)^{\theta-1}+ 2\left( 1+ \frac{3r_0H(r_0)}{4h(0)}\right)  \right\}.
\end{multline*}
Thus, by (\ref{p: Arch_first_of_error_bounds}) and (\ref{p: Arch_second_of_error_bounds}),
\begin{align*}
&\int_0^{t_n} \int_0^{s_n} \left| e^\star_n(s,t)-\lambda^\star(s,t) \right| dsdt \\
&\le  \int_0^{t_n} \int_0^{s_n} \lambda^\star(s,t) \cdot\frac{s+t}{n}\cdot \left[\kappa +\left( \frac{4}{3}\right)^{2\theta} \frac{W(r_0)}{h(0)(\theta-1)} \right]\\
& \le \bl^\star(A^\star) \cdot \frac{s_n+t_n}{n}\cdot\left[\kappa +\left( \frac{4}{3}\right)^{2\theta} \frac{W(r_0)}{h(0)(\theta-1)} \right],
\end{align*}
where we used the mean value theorem for integration for the second inequality. 
Using polar coordinates $s=\rho\cos \varphi > 0$ and $t=\rho\sin \varphi > 0$, with $\rho= \sqrt{s^2 + t^2}$ and $\varphi \in (0,\pi/2)$, 
we can rewrite $\boldsymbol{\lambda}^\star(A^\star)$ as follows:
\begin{align*}
& \int_0^{s_n} \int_0^{t_n}(\theta-1) (st)^{\theta - 1} \left(s^\theta + t^\theta\right)^{\frac{1}{\theta}-2}dtds \\
& \le \int_0^{\sqrt{s_n^2 + t_n^2}} \int_0^{\frac{\pi}{2}} (\theta - 1)(\cos \varphi \sin \varphi)^{\theta -1} 
\left( \cos^\theta \varphi + \sin^\theta \varphi \right)^{\frac{1}{\theta}-2} d\varphi d\rho\\
& = (\theta -1)\sqrt{s_n^2 + t_n^2} \int_0^{\frac{\pi}{2}} (\cos \varphi \sin \varphi)^{\theta -1} 
\left( \cos^\theta \varphi + \sin^\theta \varphi \right)^{\frac{1}{\theta}-2} d\varphi\\
& \le \frac{\pi}{2}\, (\theta -1)\sqrt{s_n^2 + t_n^2} \cdot
\max_{\varphi \in \left(0, \frac{\pi}{2}\right)} \left|(\cos \varphi \sin \varphi)^{\theta -1} 
\left( \cos^\theta \varphi + \sin^\theta \varphi \right)^{\frac{1}{\theta}-2}  \right|.
\end{align*}
Note that for all $\varphi \in \left(0,\pi/4\right]$, we have $(\sqrt{2})^{-1} \le \cos \varphi < 1$ and $0 < \sin \varphi \le (\sqrt{2})^{-1}$. Therefore, 
$\cos^\theta \varphi + \sin^\theta \varphi > (\sqrt{2})^{-\theta}$ and $\cos \varphi \sin \varphi < (\sqrt{2})^{-1}$. 
Analogously, we get the same inequalities for all $\varphi \in \left[\pi/4, \pi/2 \right)$. Thus, we obtain
\[
\left( \cos^\theta \varphi + \sin^\theta \varphi \right)^{\frac{1}{\theta}-2} < (\sqrt{2})^{2\theta-1}
\quad \textnormal{ and } \quad (\cos \varphi \sin \varphi)^{\theta-1} < (\sqrt{2})^{1-\theta}
\]
for all $\varphi \in (0, \pi/2)$, and therefore
\begin{equation}\label{p: bound_lambda(A)}
\boldsymbol{\lambda}^\star(A^\star) \le \frac{\pi}{2}\, (\theta -1)(\sqrt{2})^\theta\sqrt{s_n^2 + t_n^2} 
\le \frac{\pi}{2}\, (\theta -1)(\sqrt{2})^\theta (s_n + t_n).
\end{equation}
Hence,
\begin{align*}
& \int_0^{t_n} \int_0^{s_n} \left| e^\star_n(s,t)-\lambda^\star(s,t) \right| dsdt \\
 &\le \frac{(s_n+t_n)^2}{n} \cdot \frac{\pi(\sqrt{2})^\theta}{2} \cdot 
\left[(\theta-1)\kappa +\left( \frac{4}{3}\right)^{2\theta} \frac{W(r_0)}{h(0)} \right].
\end{align*}
\end{proof}
\noindent For any choice of Archimedean copula satisfying (\ref{t: limit_theta_Arch_cop}), we thus first need to determine $h(0)$ and $w'$ by way of the generator $\phi$, in order to find $r_0$ such that for all $r \le r_0$, it holds that $w'(r)\le 4h(0)/3$. We can then choose $n$, $s_n$ and $t_n$ such that 
\begin{equation}\label{t: condition_Arch}
\frac{s_n}{n},\, \frac{t_n}{n} \le \frac{3r_0}{8}\,.
\end{equation}
Suppose, for instance, that this is satisfied for $s_n = t_n = \sqrt{\log n}/2$ for some big enough integer $n$. For an MPPE whose bivariate marks are distributed according to the chosen copula, the error of the approximation in total variation by a Poisson process with intensity measure $\bl^\star$ is then, by (\ref{p: Arch_Poi_approx_error}) and Theorem \ref{t: MyMichel_Arch},  bounded by 
\begin{equation}\label{t: Arch_result_logn}
\frac{\sqrt{\log n}}{2n} + \frac{K \log n}{n}\,,
\end{equation}
and $n$ needs to be big enough to offset the effect of the multiplication by the constant $K$. The expected number of exceedances $\mathbb{E}W^\star_{A^\star}$ of the MPPE is then approximately 
\begin{equation}\label{t: Arch_expected_logn}
\bl^\star(A^\star) = s_n+t_n - \left( s_n^\theta + t_n^\theta\right)^{\frac{1}{\theta}}
= \left( 1- 2^{\frac{1}{\theta}-1}\right) \log n.
\end{equation}
The next section gives some examples.
\subsection{Examples}\label{s: Arch_Examples}
We apply Theorem \ref{t: MyMichel_Arch} to several examples. 
\cite{Charpentier/Segers:2009} showed that (\ref{t: limit_theta_Arch_cop}) and thereby (\ref{t: ChapSeg_E}) are satisfied for the families 
of Archimedean copulas listed in Table \ref{table: examples_Arch_cop}, with $\tilde{\theta} = \theta \in [1, \infty)$. For each of these, we assume that $\theta >1$ (as the case $\theta =1$ 
gives independence in the upper tail). We denote the various families of copulas by the numbers assigned to them in \cite{Nelsen:2006} and 
\cite{Charpentier/Segers:2009}. 

Note that $\bar{\phi}$ of family (2) is exactly $\bar{\phi}_0$ from (\ref{d: phi0}). The intensity measure of the corresponding MPPE thus equals the intensity measure $\bl^\star$ of the Poisson process that we would approximate by in Theorem \ref{t: MyMichel_Arch}. For family (2), (\ref{p: Arch_Poi_approx_error}) is thus sufficient, and there is no need to apply Theorem \ref{t: MyMichel_Arch}.
Note moreover that family (4) is the family of Gumbel copulas from Example \ref{e: Arch_Gumbel}.
\begin{center}
\begin{table}[!ht]
\begin{tabular}{l|l|l}
Nr. & $C_\theta(u,v)$ & $\bar{\phi}(r) = \phi(1-r)$ \\ \hline \hline
{} & {} & {} \\
(2) & $\max\left\{ 1-\left[ (1-u)^\theta + (1-v)^\theta \right]^{\frac{1}{\theta}}\, , \, 0\right\}$ & $r^\theta$  \\ 
{} & {} & {} \\
(4) & $\exp\left\{ -\left[ (-\log u)^\theta + (-\log v)^\theta \right]^{\frac{1}{\theta}}\right\}$ & $[-\log (1-r)]^\theta$ \\
{} & {} & {} \\
(6) & $1-\left[ (1-u)^\theta + (1-v)^\theta -(1-u)^\theta(1-v)^\theta \right]^{\frac{1}{\theta}}$ & $-\log \left(1-r^\theta\right)$ \\
{} & {} & {} \\
(12) & $\left\{ 1+ \left[ (u^{-1} - 1)^\theta + (v^{-1} - 1)^\theta\right]^{\frac{1}{\theta}} \right\}^{-1}$ & $\left( \frac{r}{1-r}\right)^{\theta}$ \\
{} & {} & {} \\
(14) & $\left\{ 1+ \left[ (u^{-\frac{1}{\theta}}-1)^\theta + (v^{-\frac{1}{\theta}} - 1)^\theta \right]^{\frac{1}{\theta}} \right\}^{-\theta}$ & $\left[(1-r)^{-\frac{1}{\theta}} - 1\right]^\theta$\\
{} & {} & {} \\
(15) & $ \left\{ \max\left(1- \left[ \left( 1- u^{\frac{1}{\theta}}\right)^\theta + \left(1-v^{\frac{1}{\theta}} \right)^{\theta}\right]^{\frac{1}{\theta}} \,, \, 0\right)\right\}^\theta$ 
& $\left[1-(1-r)^{\frac{1}{\theta}}\right]^\theta$\\
{} & {} & {} \\
(21) & $1-(1-\{\max( [1-(1-u)^\theta]^{\frac{1}{\theta}}$       
 & $1-\left( 1-r^\theta\right)^{\frac{1}{\theta}}$\\
{} & \phantom{blablabla}$+ [1-(1-v)^\theta]^{\frac{1}{\theta} } -1\, ,\, 0)\}^\theta)^{\frac{1}{\theta}} $ & {} 
\end{tabular}
\caption{We list families of Archimedean copulas with parameter $\theta \in [1, \infty)$, for which the limit $\tilde{\theta}$
in (\ref{t: limit_theta_Arch_cop}) exists in $[1, \infty)$ and equals $\theta$. These copulas exhibit upper tail dependence as determined by 
Theorem \ref{t: Charp_Segers_AD} (unless $\tilde{\theta} = \theta=1$ in which case we have asymptotic independence). 
}
\label{table: examples_Arch_cop}
\end{table}
\end{center}

For each of the examples from Table \ref{table: examples_Arch_cop} (with $\theta > 1$), it is possible to show that the function 
$w(r) = \bar{\phi}^{\frac{1}{\theta}}(r)$ is twice continuously differentiable on $[0,1)$. Consider, e.g., family (4), for which
$w(r) = -\log(1-r)$:
\[
w'(r) = \frac{1}{1-r}\,,\quad w'(0)=1, \quad w''(r) = \frac{1}{(1-r)^2}\,, \quad w''(0)=1.  
\]
For each example, we now give the function 
$h(r) = \bar{\phi}^{\frac{1}{\theta}}(r)/r$ and
indicate the first few terms of its series expansion. We thereby determine the value $h(0)$. 
\begin{enumerate}
\item[(2):] $h(r) \equiv 1$. We have $h(0)= 1$.
\item[(4):] Using series expansion of the logarithm, we have 
\[
 h(r) = -\frac{\log(1-r)}{r} = \frac{r + \frac{r^2}{2} + \frac{r^3}{3} + \ldots}{r} 
= 1 + \frac{r}{2} + \frac{r^2}{3} +\ldots,
\]
and thus $h(0) = 1$.
\item[(6):] Series expansion of the logarithm gives 
\[
h(r) = \frac{\left[ -\log\left(1-r^\theta\right)\right]^{1/\theta}}{r}
= \left(\sum_{j=1}^\infty \frac{r^{(j-1)\theta}}{j} \right)^{\frac{1}{\theta}}
= \left(1+\sum_{j=2}^\infty \frac{r^{(j-1)\theta}}{j} \right)^{\frac{1}{\theta}}.
\]
Taylor expansion about $0$ gives
\[
(1+z)^{\frac{1}{\theta}} =1+ \frac{1}{\theta}\, z + \frac{1}{\theta}\left(\frac{1}{\theta }-1 \right) \frac{z^2}{2}+\ldots,
\]
and with $\displaystyle z =  \sum_{j=2}^\infty \frac{r^{(j-1)\theta}}{j}\,$, we obtain
\begin{align*}
h(r) &= 1 + \frac{1}{2\theta}\, r^\theta + \frac{1}{8\theta}\left( \frac{1}{\theta} + \frac{5}{3}\right) r^{2\theta}+\frac{1}{6\theta}\left( \frac{1}{\theta} + \frac{1}{2}\right)r^{3\theta} + \ldots,\\
h'(r) &= \frac{1}{2}\, r^{\theta-1} + \frac{1}{4}\left( \frac{1}{\theta} + \frac{5}{3}\right)r^{2\theta-1} + \frac{1}{2}\left(\frac{1}{\theta} + \frac{1}{2}\right)r^{3\theta-1}+ \ldots,\\
h''(r) &= \frac{\theta-1}{2}\, r^{\theta-2} + \frac{2 \theta -1}{4}\left( \frac{1}{\theta} + \frac{5}{3} \right) r^{2\theta-2} + \frac{3\theta-1}{2}\left( \frac{1}{\theta} + \frac{1}{2}\right)r^{3\theta-2} + \ldots
\end{align*}
Thus $h(0)=1$. 
Note that $h$ is not twice differentiable at $0$ if $\theta <2$; however, Theorem \ref{t: MyMichel_Arch} uses only $w''(r)= h'(r)+rh''(r)$, which does exist at $r=0$ for all $\theta>1$. Family (21) below behaves similarly. 

\item[(12):] The function
\[
h(r) = \frac{1}{1-r} = \sum_{j=0}^\infty r^j = 1 + r + r^2 + \ldots
\]
is the geometric series. We have $h(0)= 1$.
\item[(14):] Taylor expansion about $0$ gives
\[
(1-r)^{-\frac{1}{\theta}} = 1 + \frac{1}{\theta}\, r + \frac{1}{2\theta}\left( \frac{1}{\theta}+1\right)r^2 + \ldots,
\]
and therefore,
\[
h(r)=\frac{(1-r)^{-1/\theta}-1}{r} = \frac{1}{\theta} + \frac{1}{2\theta} \left( \frac{1}{\theta} +1\right)r + \ldots
\]
We thus obtain $\displaystyle h(0) = \frac{1}{\theta}$\,.
\item[(15):] With Taylor expansion about $0$, we  find
\begin{equation}\label{p: Arch_Taylor_15}
(1-r)^{\frac{1}{\theta}} = 1 - \frac{1}{\theta}\, r - \frac{1}{2\theta}\left( 1 - \frac{1}{\theta}\right) r^2 - \frac{1}{6\theta}\left( 1 -\frac{1}{\theta}\right)\left( 2-\frac{1}{\theta}\right)r^3 -\ldots,
\end{equation}
and therefore,
\[
h(r) = \frac{1-(1-r)^{1/\theta}}{r} = \frac{1}{\theta} + \frac{1}{2\theta}\left(1 - \frac{1}{\theta} \right)r + \ldots,
\]
and $\displaystyle h(0)= \frac{1}{\theta}$\,.
\item[(21):] 
By (\ref{p: Arch_Taylor_15}), we find
\begin{align*}
h(r) &= \frac{\left[ 1-\left(1-r^\theta\right)^{1/\theta}\right]^{1/\theta}}{r}\\
&= \left\{\frac{1}{\theta} + \frac{1}{2\theta}\left( 1 - \frac{1}{\theta}\right)r^\theta
+ \frac{1}{6\theta}\left( 1-\frac{1}{\theta}\right)\left( 2 - \frac{1}{\theta}\right)r^{2\theta} + \ldots \right\}^{\frac{1}{\theta}},
\end{align*}
and Taylor expansion about $0$ gives
\[
\left( \frac{1}{\theta} +z\right)^{\frac{1}{\theta}} = 
\left( \frac{1}{\theta}\right)^{\frac{1}{\theta}} 
+ \frac{1}{\theta}\left( \frac{1}{\theta}\right)^{\frac{1}{\theta}-1}z
- \frac{1}{2\theta}\left( 1-\frac{1}{\theta}\right) \left( \frac{1}{\theta}\right)^{\frac{1}{\theta}-2}z^2 
+ \ldots
\]
Then,
\begin{align*}
h(r) = \left( \frac{1}{\theta}\right)^{\frac{1}{\theta}} 
+ \frac{1}{2\theta^2}\left( \frac{1}{\theta}\right)^{\frac{1}{\theta}-1}\left( 1-\frac{1}{\theta}\right)r^\theta + \ldots \quad \text{and} \quad h(0)=\left( \frac{1}{\theta}\right)^{\frac{1}{\theta}}. 
\end{align*}
\end{enumerate}
\begin{table}[b]
\begin{center}
\begin{tabular}{r|r|r|r|r|r}
Nr. 	& $h(0)$			   		& $r_0$	    & $H(r_0)$	 & $W(r_0)$	 & $K$ \\ \hline \hline
(4)	& $1$   					& $0.250$   & $0.731$      & $1.778$       & 16.2 \\ 
(6)   & $1$      					& $0.851$   & $2.531$      & $21.027$     &186.0\\ 
(12) & $1$       				& $0.133$   & $1.331$      & $3.080$       &28.4\\
(14) & $2/3$	                		& $0.158$   & $0.754$      & $1.761$       &24.3\\
(15) & $2/3$      				& $0.578$   & $0.229$      & $0.703$       &9.0\\
(21) & $(2/3)^{\frac{2}{3}}$	& $0.738$   & $0.240$      & $1.053$       &10.8
\end{tabular}
\caption{For $\theta=1.5$, we compute the values $h(0)$, $r_0$, $H(r_0)$, $W(r_0)$ and the constant $K$ for each of the examples of families of copulas from Table \ref{table: examples_Arch_cop}.}
\label{table: theta1.5}
\end{center}
\end{table}
We have thus shown that $h(0)>0$ 
for each of the examples. Note that it is furthermore possible to show that
$h'(r) \ge 0$ and $w''(r) \ge w''(0)$ for all $r$ close to $0$ (or even for all $r \in [0,1)$); see also Figures \ref{f: Der_h} and \ref{f: Der_w}.  
We next need to compute the values
\[
r_0, \quad H(r_0)= \max_{0 \le \xi \le r_0} h'(\xi), \quad W(r_0)=\max_{0 \le \xi \le r_0} w''(\xi),
\]
in order to determine the constant $K$. Note that for families (6), (14), (15) and (21), the functions $h$ and $w$ depend 
on the parameter $\theta$ and the value of $\theta$ should thus be specified. For families (4) and (12), this is not the 
case. However, we still need to specify $\theta$ for all families of copulas, as the constant $K$ depends on $\theta$. 
By way of example, we compute $r_0$, $H(r_0)$, $W(r_0)$ and finally $K=K(\theta, h(0), r_0, H(r_0), W(r_0))$ for each 
family of copulas for the parameter values $\theta=1.5$ and $\theta=3$, respectively. The results are summarised in 
Tables \ref{table: theta1.5} and \ref{table: theta3}, respectively. Figures \ref{f: Der_h} and \ref{f: Der_w} 
illustrate that for $\theta=1.5$, the functions $h'$ and $w''$ are non-decreasing, and that therefore 
$H(r_0) = h'(r_0)$ and $W(r_0)=w''(r_0)$. The functions $h'$ and $w''$ behave analogously for $\theta=3$.
\begin{center}
\begin{figure}[b!]
\captionsetup[subfigure]{labelformat=empty}
\centering
\subfloat[(4)]{\includegraphics[scale=0.3]{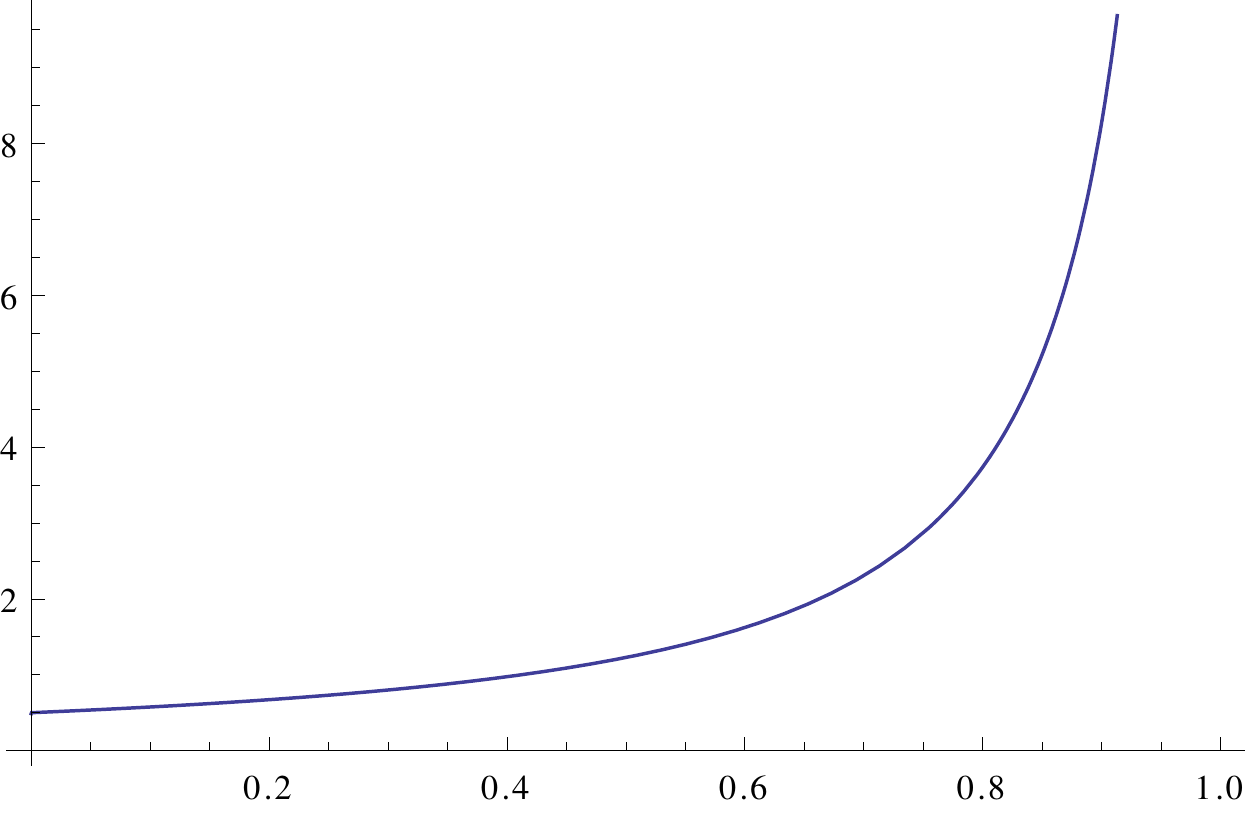}}\qquad
\subfloat[(6)]{\includegraphics[scale=0.3]{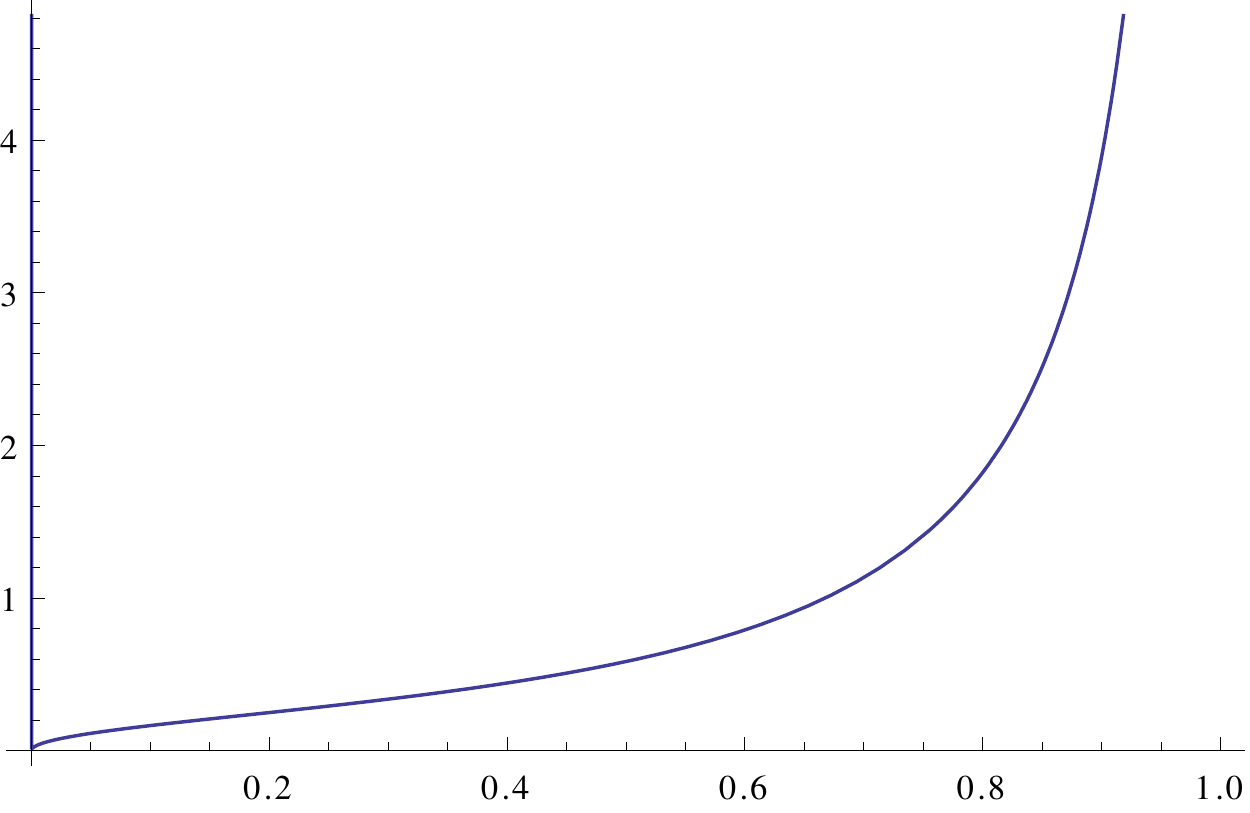}}\qquad
\subfloat[(12)]{\includegraphics[scale=0.3]{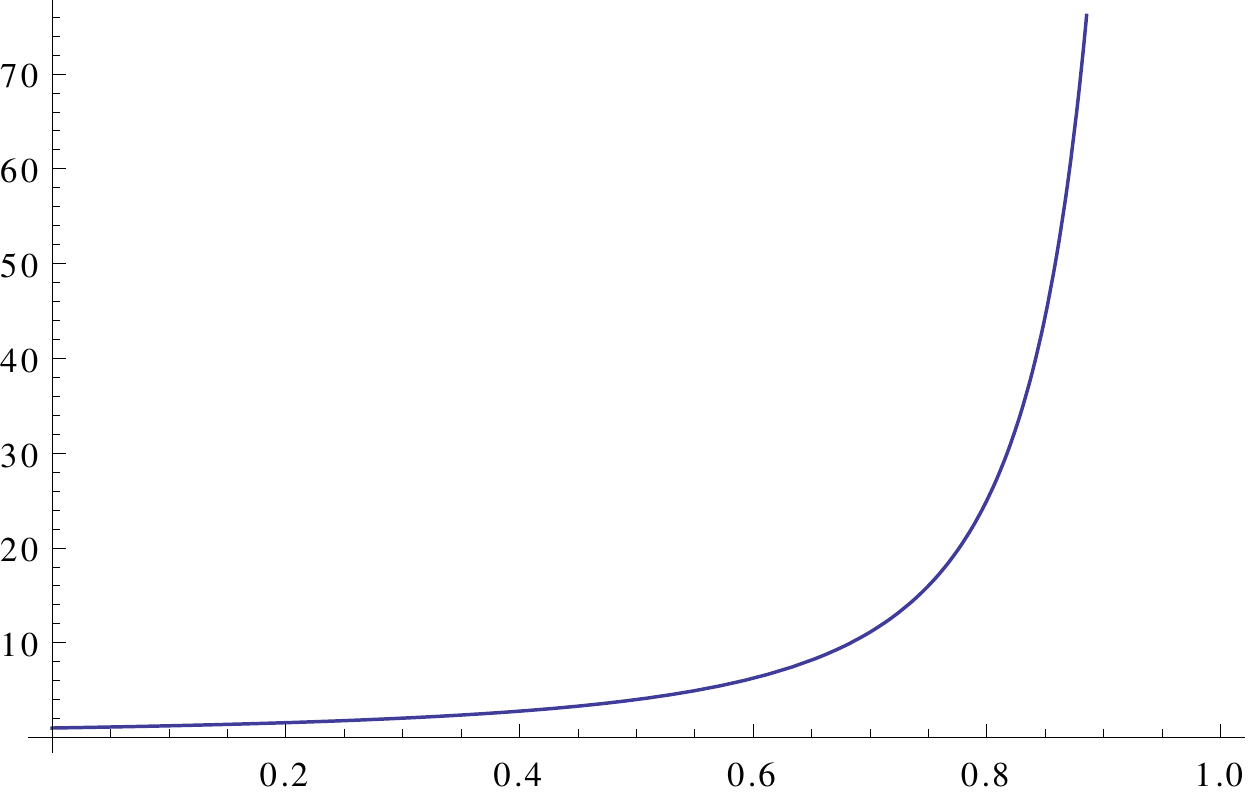}}\qquad
\subfloat[(14)]{\includegraphics[scale=0.3]{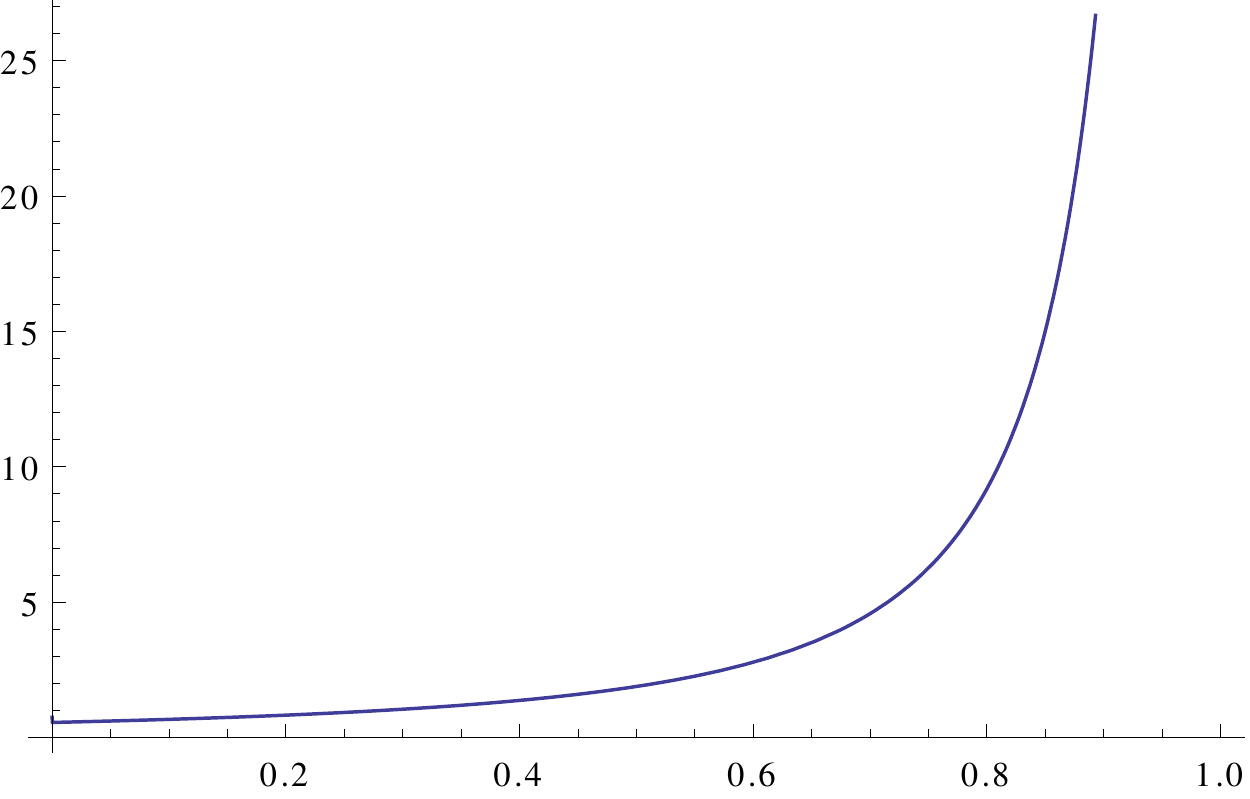}}\qquad
\subfloat[(15)]{\includegraphics[scale=0.3]{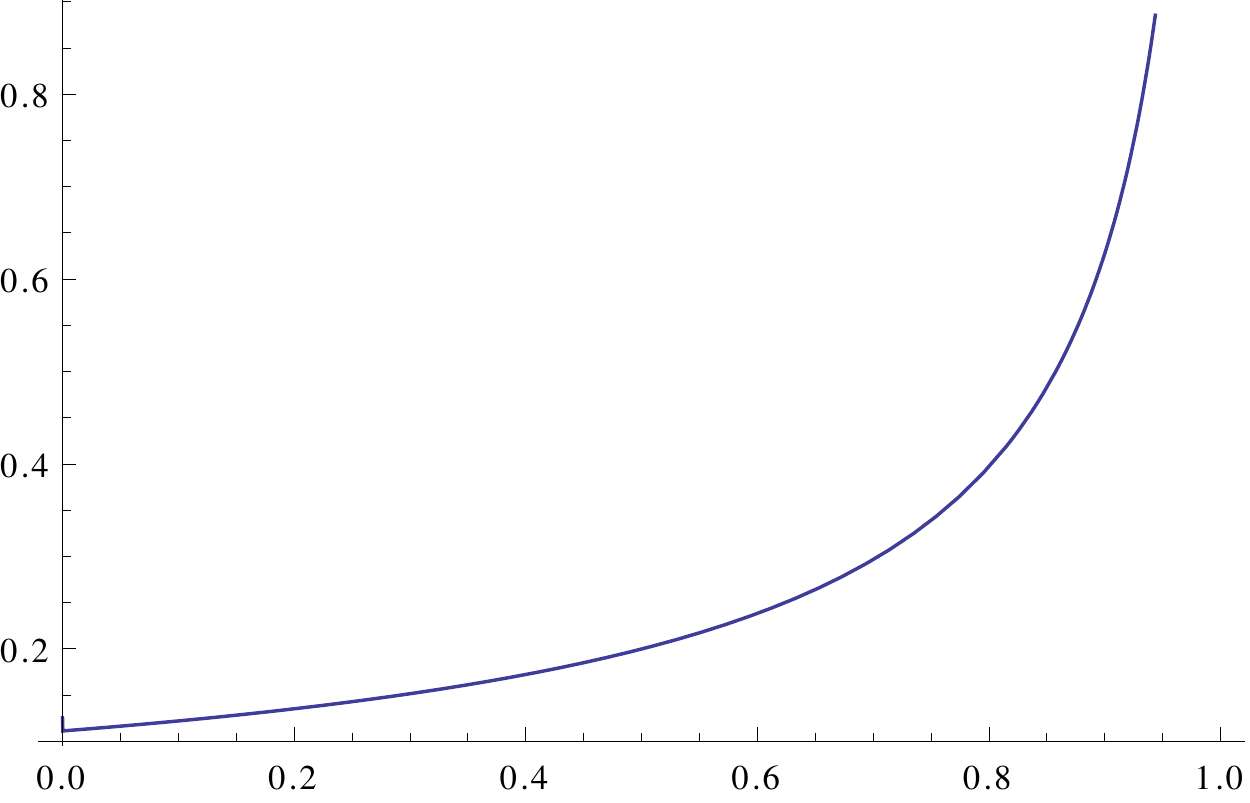}}\qquad
\subfloat[(21)]{\includegraphics[scale=0.3]{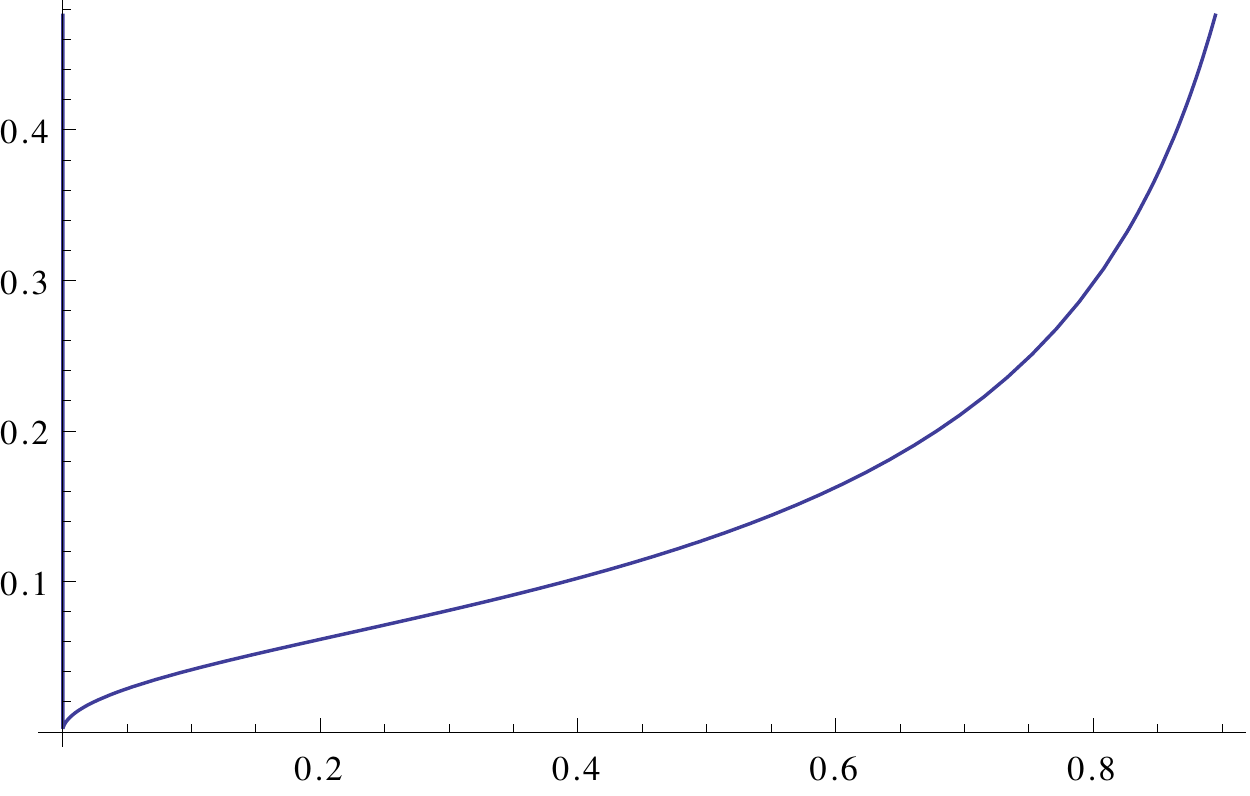}}
\caption{Plots of $h'$ on $[0,1]$ for each of the families of copulas from Table \ref{table: examples_Arch_cop}, for $\theta=1.5$.}
\label{f: Der_h}
\end{figure}
\end{center}
\begin{remark}
Note that $h(0)$ does not depend on the value of $\theta$. We thus do not include it again in Table 
\ref{table: theta3} for $\theta=3$. Also note that in Tables \ref{table: theta1.5} and \ref{table: theta3}, $r_0$ was rounded down to three decimal places (any $r$ smaller than $r_0$ would indeed do), but that we used six decimal places to compute $H(r_0)$ and $W(r_0)$. In Tables  \ref{table: theta1.5} and \ref{table: theta3}, the results for $H(r_0)$ and $W(r_0)$ are rounded to three decimal places, and the results for $K$ to one. As can be seen when 
considering the formula for $K$ in Theorem \ref{t: MyMichel_Arch} or when
comparing the two tables, the value of $K$ increases as $\theta$ increases. 
\end{remark}

The results from the tables can be interpreted by considerations like the following: for instance, if we choose the Gumbel copula, i.e. copula (4), with $\theta=1.5$, and if we choose $s_n=t_n =\sqrt{\log n}/2$, the sample size $n$ needs to be big enough to allow for 
\[
\frac{\sqrt{\log n}}{2n} \le \frac{3r_0}{8}= \frac{3}{32} \approx 0.0938.
\]
This is satisfied for each integer $n \ge 8$. As determined in (\ref{t: Arch_expected_logn}), the expected number of exceedances is approximately
\[
\left( 1-2^{-\frac{1}{3}}\right)\sqrt{\log n} \approx 0.2 \sqrt{\log n}, 
\]
and the overall error bound from (\ref{t: Arch_result_logn}) with $K=16.2$ is smaller than $1$ only for $n \ge 70$. For example, for $n=100$, we only expect $0.2 \sqrt{\log 100} \approx 0.43$, i.e. less than one joint exceedance of the thresholds. The sample size thus has to be very big in order to expect only as much as one threshold exceedance and to get a small error. 
\begin{table}[h!]
\begin{center}
\begin{tabular}{r|r|r|r|r}
Nr. 		& $r_0$	    & $H(r_0)$	 & $W(r_0)$	 & $K$ \\ \hline \hline
(4)		& $0.250$   & $0.731$      & $1.778$       & 207.2 \\ 
(6)  		& $0.701$   & $0.375$      & $2.078$       &1401.1\\ 
(12)    	& $0.133$   & $1.331$      & $3.080$       &372.4\\
(14)	        & $0.194$   & $0.291$      & $0.736$       &313.9\\
(15)    	& $0.350$   & $1.773$      & $0.457$       &107.3\\
(21) 	 	& $0.774$   & $0.238$      & $1.479$       &126.1
\end{tabular}
\caption{For $\theta=3$, we compute the values $h(0)$, $r_0$, $H(r_0)$, $W(r_0)$ and the constant $K$ for each of the 
families of copulas from Table \ref{table: examples_Arch_cop}.} 
\label{table: theta3}
\end{center}
\end{table}
\begin{center}
\begin{figure}[!ht]
\captionsetup[subfigure]{labelformat=empty}
\centering
\subfloat[(4)]{\includegraphics[scale=0.45]{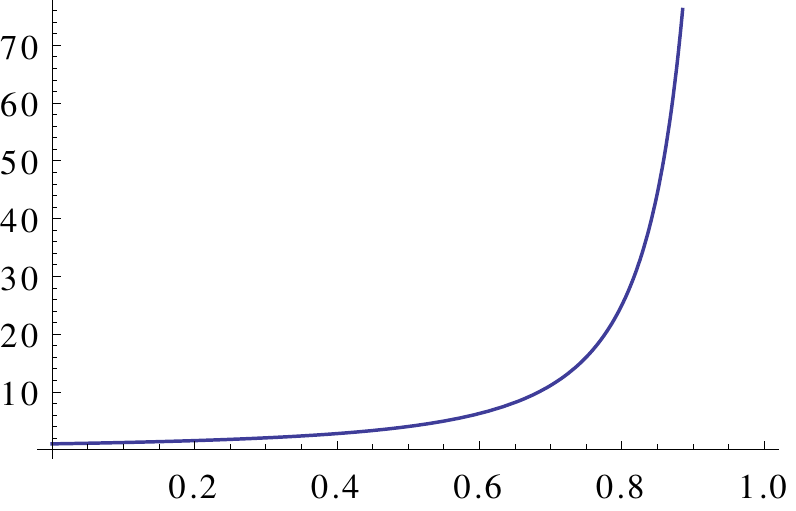}}\qquad
\subfloat[(6)]{\includegraphics[scale=0.45]{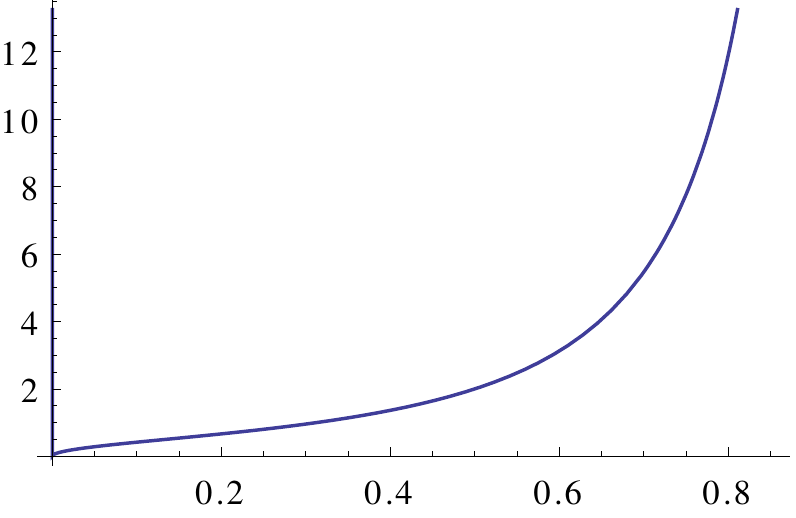}}\qquad
\subfloat[(12)]{\includegraphics[scale=0.45]{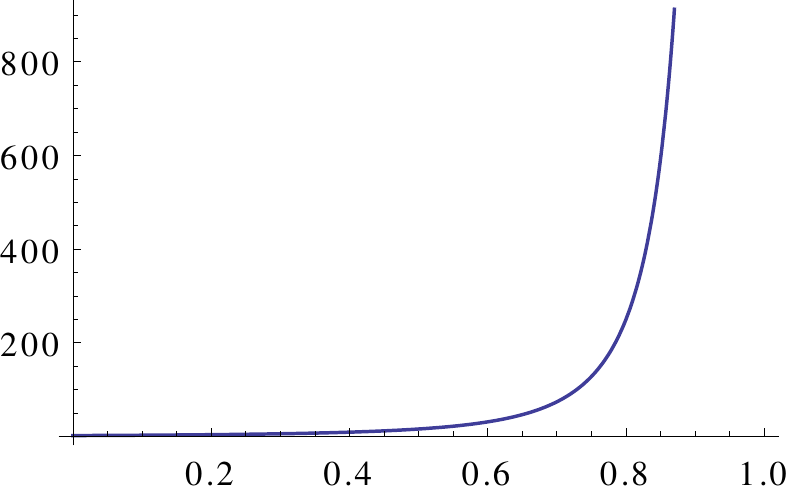}}\qquad
\subfloat[(14)]{\includegraphics[scale=0.45]{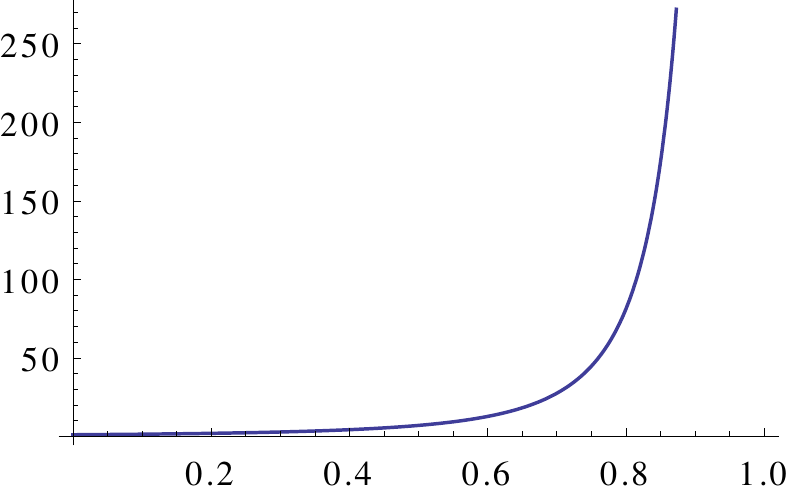}}\qquad
\subfloat[(15)]{\includegraphics[scale=0.45]{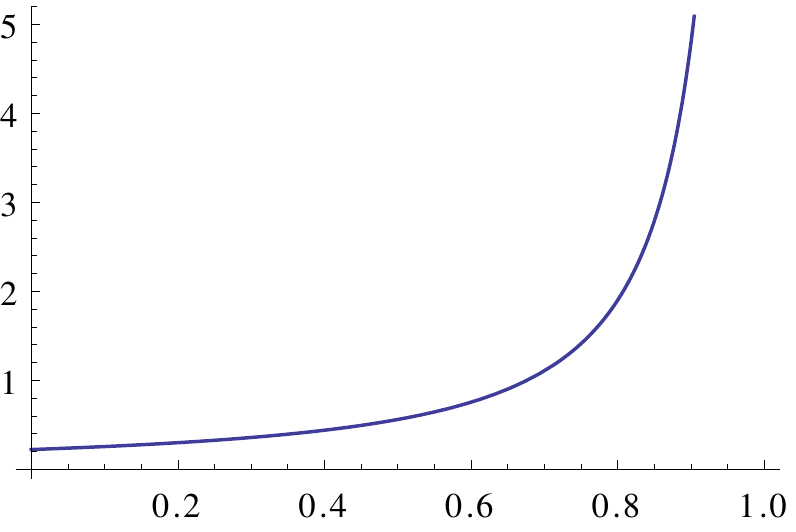}}\qquad
\subfloat[(21)]{\includegraphics[scale=0.45]{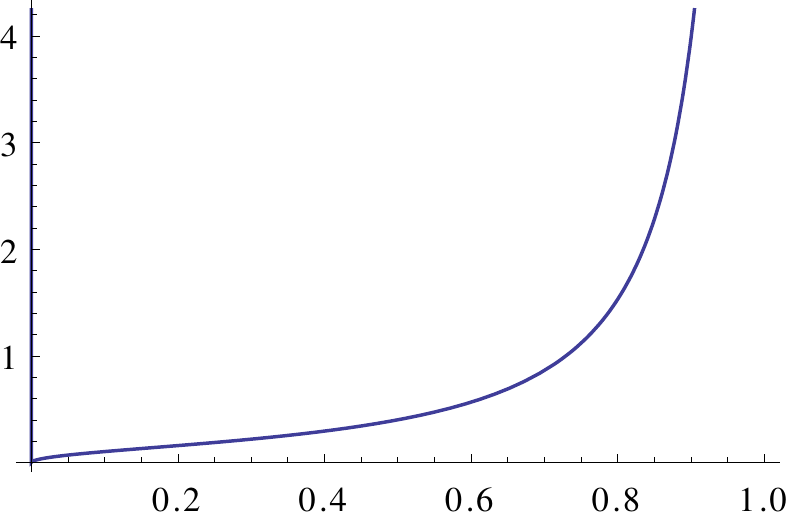}}
\caption{Plots of $w''$ on $[0,1]$ for each of the families of copulas from Table \ref{table: examples_Arch_cop}, for $\theta=1.5$.}
\label{f: Der_w}
\end{figure}
\end{center}
Some values for $K$ in the tables might be unnecessarily high. These might be 
reduced by choosing a smaller $r_0$ than the one indicated. For instance, for 
copula (6) with $\theta=3$, Table \ref{table: theta3} indicates $r_0=0.701$. We 
might instead choose, say, $r_0=0.1$, as the inequality $w'(r) \le 4h(0)/3$ will 
then still be satisfied. Then, $H(r_0) = 0.005$, $W(r_0) = 0.02$ and $K=1.5$ 
(instead of $K=1401.1$ as before). Condition (\ref{t: condition_Arch}) is 
satisfied for $s_n$ and $t_n$ chosen as above for all integers $n \ge 24$. 
For $n=24$, 
the size of the error bound is $\approx 0.24$ and the expected number of 
joint threshold exceedances is $\approx 0.63$, showing again that an even
bigger sample size is needed to expect at least one joint threshold exceedance.
\subsection{Higher dimensions}\label{s: Higher_dimensions}
Sections \ref{s: Charpentier/Segers}-\ref{s: Arch_Examples} treat the case $d=2$. 
Theorem \ref{t: Charp_Segers_AD} is, however, also valid for $d \ge 3$; see 
Theorem 4.1 in \cite{Charpentier/Segers:2009}. This section sketches two ways 
to determine bounds on the error between $\mathcal{L}(\Xi^\star_{A^\star})$ and 
$\mathrm{PRM}(\bls_{\As})$ for processes on the $d$-dimensional (normalised) space $E^\star= (0,n]^d$. 

Suppose that $\mathbf{U},\mathbf{U}_1, \ldots, \mathbf{U}_n$ are \iid$d$-dimensional random vectors 
distributed according to an Archimedean copula $C(u_1, \ldots, u_d) = \phi^{[-1]}(\phi(u_1)+ \ldots + \phi(u_d))$ on $E=[0,1)^d$, 
satisfying (\ref{t: limit_theta_Arch_cop}), for some fixed parameter value $\theta >1$. 
Let $\mathbf{V}, \mathbf{V}_1, \ldots, \mathbf{V}_n$ be \iid$d$-dimensional random vectors 
distributed according to a $d$-variate copula from family (2) (see Table \ref{table: examples_Arch_cop}) 
with parameter $\theta$, that we denote by $C_0$ $(:= C_{0,\theta})$:
\begin{align*}
 C_0(u_1, \ldots, u_d) 
&=\phi_0^{[-1]}(\phi_0(u_1) + \ldots + \phi_0(u_d)) \\
&=\max \left\{ 1- \left[(1-u_1)^\theta + \ldots + (1-u_d)^\theta \right]^{\frac{1}{\theta}}\, ,\,0 \right\}.
\end{align*}
Suppose, for simplicity, that $s_{1n}=s_{2n}=\ldots = s_{dn} = s_{n} \in (0,n]^d$, and define 
\[
A = A_n = \left[1- \frac{s_n}{n},1 \right)^d \quad \text {and} \quad A^\star = A^\star_n= (0,s_n]^d.
\] 
Let $\Xi^\star_{A^\star} = \sum_{i=1}^n I_{\{\mathbf{U}^\star_i \in \As \}}\delta_{\mathbf{U}^\star_i}$ and $W^\star_{\As} =\sum_{i=1}^n I_{\{\mathbf{U}^\star_i \in \As \}}$, where the $\mathbf{U}_i^\star$'s are \iid copies of $\mathbf{U}^\star = (n(1-U_1), \ldots, n(1-U_d)) = n(\mathbf{1} - \mathbf{U})$. Furthermore, let $\mathbf{V}^\star = n(\mathbf{1} - \mathbf{V})$. As we chose \iid $\mathbf{U}_i$'s, a bound on the error $d_{TV}(\mathcal{L}(\Xi^\star_{\As}), \mathrm{PRM}(\mathbb{E}\Xi^\star_{\As}))$ is easy to determine: Proposition \ref{t: dTV_copula_joint_extreme_A} gives the upper bound 
$s_n/n$. However, as for $d=2$, we prefer to approximate further by $\mathrm{PRM}(\bl^\star_{\As})$, whose intensity function is given by
\[
\ls(s_1, \ldots,s_d) =\left( \prod_{j=1}^d
 \left[ 1-(j-1)\theta\right]\right) (s_1 \cdot \ldots \cdot s_d)^{\theta-1} \left( s_1^\theta + \ldots + s_d^\theta \right)^{\frac{1}{\theta}-d},
\] 
for all $(s_1, \ldots, s_d) \in E^\star = (0,n]^d $. 
  
One possibility is to proceed analogously to the bivariate case treated in 
Theorem \ref{t: MyMichel_Arch}, where we determined an upper bound on 
$d_{TV}(\mathrm{PRM}(\mathbb{E}\Xi^\star_{\As}),$ $\mathrm{PRM}(\bls_{\As}))$ 
by straightforward comparison of their intensity functions, which led to 
rather involved computations. Computations would of course become even harder 
for $d \ge 3$ and we would need to introduce more and more assumptions on the 
functions $w$ and $h$. For instance, $w$ would need to be $d$-times continuously 
differentiable on $[0,1)$. 

Another possibility is to look for a cruder bound that does not require these assumptions. We can use the weaker $d_2$-distance from Section \ref{Sec: Improved_rates}.
To achieve this note first that we can express $\As$ as a union of $d$-rectangles of side lengths, say, $s_n/m$, for some integer $m \ge 1$. Define a $d$-rectangle as follows: 
\[
R^\star_{\mathbf{k}} := \left(k_1 \frac{s_n}{m},(k_1 +1) \frac{s_n}{m}\right] \times \ldots \times \left(k_d \frac{s_n}{m}, (k_d+1)\frac{s_n}{m} \right],
\]
for all $\mathbf{k}=(k_1, \ldots,k_d) \in \As$  with $k_j \in \{ 0,1, \ldots, m-1\} $, for each $1 \le j \le d$. The set $\As$ is then a union of $m^d$ $d$-rectangles.  Then, define $P_{\mathbf{k}} = \mathcal{L}(\mathbf{U}^\star|\mathbf{U}^\star \in R^\star_{\mathbf{k}})$ and 
$P'_{\mathbf{k}} = \mathcal{L}(\mathbf{V}^\star|\mathbf{V}^\star \in R^\star_{\mathbf{k}})$ for each $\mathbf{k}$,  and note that by proceeding similarly as in the proof of Proposition \ref{t: MyMichel}, 
we can show that
\begin{align*}
&\mathrm{PRM}(\mathbb{E}\Xi^\star_{R^\star_\mathbf{k}}) 
= \mathcal{L} \left( \sum_{j=1}^{N_{\mathbf{k}}} \delta_{\mathbf{Z}_{j, \mathbf{k}}}\right),\\
&\qquad \qquad \qquad\text{where } \mathbf{Z}_{j, \mathbf{k}} \stackrel{\iid}{\sim} P_{\mathbf{k}}, \text{ independent of }N_{\mathbf{k}} \sim \mathrm{Poi}(\mathbb{E}W^\star_{R^\star_{\mathbf{k}}});\\
&\mathrm{PRM}(\bl^\star_{R^\star_{\mathbf{k}}}) 
=  \mathcal{L} \left( \sum_{j=1}^{L_{\mathbf{k}}} \delta_{\mathbf{Z}'_{j,\mathbf{k}}}\right),\\
&\qquad \qquad \qquad \text{where } \mathbf{Z}'_{j,\mathbf{k}} \stackrel{\iid}{\sim} P'_{\mathbf{k}}, \text{ independent of }L_{\mathbf{k}}  \sim \mathrm{Poi}(\bl^\star(R^\star_{\mathbf{k}})),
\end{align*}
where 
$
\mathbb{E}W^\star_{R^\star_{\mathbf{k}}}
= nP(\mathbf{U}^\star \in R^\star_{\mathbf{k}})\text{ and } \bl^\star(R^\star_{\mathbf{k}})= nP(\mathbf{V}^\star \in R^\star_{\mathbf{k}}) 
$.
Now construct, for each $\mathbf{k}$, an additional Poisson process with intensity measure $\boldsymbol{\nu}^\star_{R^\star_{\mathbf{k}}}$ as follows:
\[
\mathrm{PRM}(\boldsymbol{\nu}^\star_{R^\star_{\mathbf{k}}}) 
= \mathcal{L}\left( \sum_{j=1}^{N_{\mathbf{k}}}\delta_{\mathbf{Z}'_{j,\mathbf{k}}}\right) , \quad 
\text{where } \mathbf{Z}'_{j,\mathbf{k}} \stackrel{\iid}{\sim} P'_{\mathbf{k}}, \text{ independent of } N_{\mathbf{k}}.
\]
The corresponding Poisson processes on $\As$ can be realised in an analogous way, and, since Poisson processes constructed on disjoint sets are independent, we have that 
\[
\mathrm{PRM}(\mathbb{E}\Xi^\star_{\As}) = \sum_{\mathbf{k}}\sum_{j=1}^{N_{\mathbf{k}}} \delta_{\mathbf{Z}_{j,\mathbf{k}}},
\] 
for all $\mathbf{k}=(k_1, \ldots,k_d) \in \As$  with $k_j \in \{ 0,1, \ldots, m-1\} $ for each $1 \le j \le d$,
and similarly for $\mathrm{PRM}(\bl^\star_{\As})$ and $\mathrm{PRM}(\boldsymbol{\nu}^\star_{\As})$. 
The error that we want to estimate can then be split up into two parts: 
\begin{multline*}
d_2\left(\mathrm{PRM}(\mathbb{E}\Xi^\star_{\As}), \mathrm{PRM}(\bl^\star_{\As})\right)\\
\le d_2\left( \mathrm{PRM}(\mathbb{E}\Xi^\star_{\As}),\mathrm{PRM}(\boldsymbol{\nu}^\star_{\As}) \right) 
+ d_2 \left(\mathrm{PRM}(\boldsymbol{\nu}^\star_{\As})  ,\mathrm{PRM}(\bl^\star_{\As})\right).
\end{multline*}
Since $\mathrm{PRM}(\mathbb{E}\Xi^\star_{\As})$ and $\mathrm{PRM}(\boldsymbol{\nu}^\star_{\As})$ have the same number $nP(\mathbf{U}^\star \in \As)$ of expected points in $\As$, Proposition \ref{t: d2_two_PRM} gives the upper bound
$2d_1(\mathbb{E}\Xi^\star_{\As}, \boldsymbol{\nu}^\star_{\As})$ for the first of the above two summands. With (\ref{d: closest_matching}), the $d_1$-distance may be bounded by the maximum $d_0$-distance that points of the two Poisson processes may be apart, which is the length of the space diagonal of the $d$-rectangle in the $d_0$-distance. 
The $d_0$-distance can be chosen in a way to give a good estimate. Choose, e.g., the Euclidean distance bounded by $1$:
\[
d_0(\mathbf{x}, \mathbf{y}) = \min\left( \sqrt{\sum_{j=1}^d (x_j - y_j)^2}\, , \, 1\right),
\]
for any $\mathbf{x}=(x_1, \ldots, x_d)$ and $\mathbf{y}= (y_1, \ldots, y_d)$ $\in \mathbb{R}$.
The $d_0$-distance between two diagonally opposite corner points of $R^\star_{\mathbf{k}}$ is then
$\sqrt{d}\,s_n/m$ 
and we have the following estimate for the error caused by smearing out points over the rectangles: 
\[
d_2\left( \mathrm{PRM}(\mathbb{E}\Xi^\star_{\As}),\mathrm{PRM}(\boldsymbol{\nu}^\star_{\As}) \right)  \le \frac{2\sqrt{d}\,s_n}{m}.
\]
The second error term, i.e. $d_2(\mathrm{PRM}(\boldsymbol{\nu}^\star_{\As}), \mathrm{PRM}(\bl^\star_{\As}))$, may be estimated as follows:
\begin{align}
\begin{split}
& \sum_{\mathbf{k}} d_{TV}(\mathrm{PRM}(\boldsymbol{\nu}^\star_{R^\star_{\mathbf{k}}}),\mathrm{PRM}(\bl^\star_{R^\star_{\mathbf{k}}}) )\\
& \le \sum_{\mathbf{k}} d_{TV}(\mathrm{Poi}(nP(\mathbf{U}^\star \in R^\star_{\mathbf{k}})), \mathrm{Poi}(nP(\mathbf{V}^\star \in R^\star_{\mathbf{k}}))) \\
& \le \sum_{\mathbf{k}} n|P(\mathbf{U}^\star \in R^\star_{\mathbf{k}}) - P(\mathbf{V}^\star \in R^\star_{\mathbf{k}})|, \label{p: diff_two_d2_Arch_higher}
\end{split}
\end{align}
where we use (\ref{t: d2_smaller_dTV}), as well as an argument similar to Michel's argument in the proof of Theorem \ref{t: Michel} and the fact that $d_{TV}(\mathrm{Poi}(\mu), \mathrm{Poi}(\mu')) \le |\mu - \mu'|$. Note that, with 
\[
P(\mathbf{U}^\star \in R^\star_{\mathbf{k}}) = \prod_{j=1}^d \mathbb{E}[I_{\{ U^\star_j \le  (k_j +1)s_n/m\}} - I_{\{ U_j^\star \le k_j s_n/m\}}] 
\]
for each $\mathbf{k}$ (and analogously for $P(\mathbf{V}^\star \in R^\star_{\mathbf{k}})$), each of the two probabilities in (\ref{p: diff_two_d2_Arch_higher}) may be expressed as a sum of $2^d$ $d$-dimensional copulas, and 
\begin{align}\label{p: dooof}
&|P(\mathbf{U}^\star \in R^\star_{\mathbf{k}}) - P(\mathbf{V}^\star \in R^\star_{\mathbf{k}})| \nonumber \\
 & \le 2^d \max_{(u_1, \ldots, u_d) \in A} |C(u_1, \ldots, u_d) - C_0(u_1, \ldots, u_d)|\\
& = 2^d \max_{(s_1, \ldots, s_d) \in A^\star} 
\left| w^{-1}\left[ \left(w^\theta \left( \frac{s_1}{n} \right)+ \ldots w^\theta \left( \frac{s_d}{n} \right) \right)^{\frac{1}{\theta}}\right] - 
\frac{1}{n}\left( s_1^\theta + \ldots + s_d^\theta\right)^{\frac{1}{\theta}}\right|.\nonumber
\end{align}
As in Theorem \ref{t: MyMichel_Arch}, assume that $w(r)=rh(r)$ and that $h$ has a positive derivative for all $r \in [0,\delta)$, for some $\delta >0$ close to $0$. 
We then have $h(0) \le h(r) \le h(0) + r \max_{0 \le \xi \le r} h'(r)$ for all $r \in [0,\delta)$. Also, for $x$ small enough,
we can show, with arguments similar to (\ref{p: Arch_Inverse_w}), that there is a constant $c>0$ such that 
\begin{equation}\label{p: goofy}
\frac{x}{h(0)} - cx^2 \le w^{-1}(x) \le \frac{x}{h(0)}. 
\end{equation}
Supposing that
$s_n/n$ is small enough for (\ref{p: goofy}) to be satisfied for $x :=$ $(w^\theta(s_1/n)$ 
$+ \ldots $ $+ w^\theta(s_d/n))^{1/\theta}$, where $s_1, \ldots, s_d \le s_n$, 
it is then possible to show that 
\[
w^{-1}\left[ \left(w^\theta \left( \frac{s_1}{n} \right)+ \ldots w^\theta \left( \frac{s_d}{n} \right) \right)^{\frac{1}{\theta}}\right]
= \frac{1}{n}\left( s_1^\theta + \ldots + s_d^\theta\right)^{\frac{1}{\theta}} + O \left( \left(\frac{s_n}{n} \right)^2\right).
\]
Then there exists a constant $\alpha>0$ such that (\ref{p: dooof}) is smaller than $2^d\alpha(s_n/n)^2$, and we obtain the following estimate for the second error term:
\[
d_2(\mathrm{PRM}(\boldsymbol{\nu}^\star_{\As}), \mathrm{PRM}(\bl^\star_{\As})) \le  \frac{(2m)^d \alpha s_n^2}{n}\,. 
\]
The bound of the total error $d_2\left(\mathrm{PRM}(\mathbb{E}\Xi^\star_{\As}), \mathrm{PRM}(\bl^\star_{\As })\right)$ is thus composed of a term of order $s_n/m$ and another term of order
$m^ds_n^2/n$. Now choose, for instance, $m = (n/s_n)^{1/(d+1)}$. Then both terms are of order $(s_n^{d+2}/n)^{1/(d+1)}$ and the total error is small only
if $s_n \ll n^{1/(d+2)}$, i.e. if the threshold value $s_n$ is smaller than a small power of $n$. This result is reminiscent of the result obtained in Theorem
\ref{t: MyMichel_Arch}, which requires $s_n \ll n^{1/2}$. 

\section{MPPE's with bivariate Marshall-Olkin geometric marks}\label{Sec: MO_Geo}
We consider MPPE's with bivariate marks that follow a certain bivariate geometric distribution, the
Marshall-Olkin geometric distribution. We can readily approximate the law of this process by that of a Poisson process with the same mean measure by way of Theorem \ref{t: MyMichelMultivariate}. However, as the marks have geometric, and thereby discrete margins, the mean measure will live on a lattice and be rather tedious to work with in practial applications. We would therefore prefer to approximate by a further Poisson process with a continuous mean measure. As the total variation distance is too strong for this kind of approximation, we use the weaker $d_2$-distance instead, which is not as sensitive towards small changes in the positions of the points of the point processes. As for MPPE's with univariate geometric marks, which we studied in Section \ref{s: MPPE_geo}, the error that arises when going from a process on a lattice to a process with continuous intensity will only be small if the parameters of the distribution of the marks vary with the sample size $n$ at a suitable rate. 

Section \ref{s: MO_Geo} introduces the bivariate Marshall-Olkin geometric distribution and relates it to its continuous counterpart, the bivariate Marshall-Olkin exponential distribution. 
Section \ref{s: MO_Geo_dTV} determines an error estimate in the total variation distance for the approximation of the law of the MPPE by that of a Poisson process with equal mean measure.
In Section \ref{s: MO_Geo_cont_int_fct}, we construct a continuous intensity function by spreading out the point probabilities of the Marshall-Olkin distribution over the entire space. As this intensity function depends on $n$, 
Section \ref{s: MO_Geo_cont_nicer_int_fct} makes some assumptions on the parameters of the Marshall-Olkin geometric distribution. These allow us to find another continuous intensity function that is asymptotically equal to the one that we constructed previously, but no longer varies with $n$.
Section \ref{s: MO_Geo_d2} establishes error estimates in the $d_2$-distance for the approximation by the Poisson process whose intensity function we constructed in Section \ref{s: MO_Geo_cont_int_fct}, whereas
Section \ref{s: MO_Geo_nicer} gives error bounds, both in $d_{TV}$ and in $d_2$, for further approximating by a Poisson process
with the intensity we found in Section \ref{s: MO_Geo_cont_nicer_int_fct}.
In Section \ref{s: MO_Geo_final_bound_d2}, we summarise the results by adding up the $d_2$-error bounds arising from each step, thus giving the total error bound for the approximation of the MPPE by the final Poisson process.
\subsection{The bivariate Marshall-Olkin geometric distribution}\label{s: MO_Geo}
The bivariate Marshall-Olkin geometric distribution arises as a natural generalisation of the geometric distribution to two dimensions.
It was first introduced by \cite{Hawkes:1972} and later studied by \cite{Marshall/Olkin:1985} as the discrete counterpart to their bivariate exponential 
distribution, first derived by them in \cite{Marshall/Olkin:1967} using shock models. Limit distributions for maxima of i.i.d. Marshall-Olkin geometric
random pairs were established in \cite{Mitov/Nadarajah:2005} and \cite{Feidt_et_al:2010}. 

Underlying the Marshall-Olkin geometric distribution are Bernoulli trials. Suppose $S$ and $T$ are two Bernoulli random variables 
with joint probability mass function $P(S=i,T=j)=p_{ij}$, for all $i,j=0,1$, and let $S_1, S_2, \dots$ and $T_1, T_2, \dots$ be i.i.d. copies of $S$ 
and $T$, respectively. Let $X_1$ and $X_2$ denote the numbers of $0$'s before the first $1$ in the sequences $S_1, S_2, \dots$ and $T_1, T_2,\dots$, respectively. 
Obviously, $X_1$ and $X_2$ follow geometric distributions with failure probabilities
$q_1 := P(S=0) = p_{00} + p_{01}$ and $q_2:= P(T=0) = p_{00}+p_{10}$, respectively. 
Their joint probability mass function is given by
\begin{equation}\label{d: MO_pointprob}
 P(X_1=k, X_2=l)= 
 \left\{ \begin{array}{ll}
 p_{00}^k q_2^{l-k} (1-p_{00}/ q_2 - q_2 + p_{00} ) & \textrm{ for } k < l , \\
 p_{00}^k ( 1-q_1 - q_2 + p_{00})& \textrm{ for } k=l ,  \\
 p_{00}^l q_1^{k-l} (1-q_1 - p_{00}/q_1 + p_{00} ) & \textrm{ for } k>l ,
 \end{array} \right.
\end{equation}
for any $k,l \in \mathbb{Z}_+$. The distribution of $\mathbf{X}= (X_1,X_2)$ thus depends on three parameters: the two marginal failure probabilities $q_1$ and
$q_2$, as well as $p_{00}=P(S=0,T=0)$, the probability of joint failure.
We assume that $p_{00} \ge q_1 q_2$. 
We have
\begin{equation}\label{d: MO_survival}
P(X_1 \ge k, X_2 \ge l)
 = \left\{ \begin{array}{ll}
 p_{00}^k q_2^{l-k} & \textrm{ for } k < l , \\
 p_{00}^k & \textrm{ for } k=l , \\
 p_{00}^l q_1^{k-l} & \textrm{ for } k>l.
 \end{array} \right. 
 \end{equation}
The survival copula $\hat{C}$ is given by a Marshall-Olkin copula $C_{\alpha, \beta}$ as defined in Example \ref{d: examples_copulas} (f). 
To show this, we may proceed as in Example \ref{e: MO_exp} for the Marshall-Olkin exponential distribution. 
That is, rewrite (\ref{d: MO_survival})
as
\[
\left( \frac{p_{00}}{q_2}\right)^k \left( \frac{p_{00}}{q_1}\right)^l \left( \frac{q_1 q_2}{p_{00}}\right)^{\max (k,l)} 
= q_1^k q_2^l \min \left\{ \left(\frac{p_{00}}{q_1q_2}\right)^k, \left(\frac{p_{00}}{q_1q_2}\right)^l\right \}, 
\]
using $\max (k,l) = k+l - \min(k,l)$ and $p_{00} \ge q_1 q_2$. With $u = P(X_1 \ge k)= q_1^k$, $v= P(X_2 \ge l) = q_2^l$, and
\[
\alpha =\frac{\log (p_{00}/q_1q_2)}{\log (1/q_1)}\,, \qquad \beta = \frac{\log (p_{00}/q_1q_2)}{\log(1/q_2)}\,, 
\]
we have $(p_{00}/q_1q_2)^k = u^{-\alpha}$ and $(p_{00}/q_1q_2)^l = v^{-\beta}$, and $\hat{C}(u,v) = C_{\alpha,\beta}(u,v)$, for all $(u,v) \in (0,1)^2$, with parameters
$\alpha, \beta \in [0,1]$ since $p_{00} \ge q_1q_2$ and $q_1,q_2 \ge p_{00}$. For $\alpha,\beta \in (0,1)$, the copulas in this family have full support, i.e. $[0,1]^2$. 
Note that if $p_{00} = q_1q_2$, the Marshall-Olkin geometric distribution corresponds to a bivariate distribution with independent geometric margins. 

We can relate the Marshall-Olkin geometric distribution to its continuous counterpart, the Marshall-Olkin exponential distribution, by noting that 
(\ref{d: MO_survival}) is equal to $P(\tilde{X}_1 \ge k, \tilde{X}_2 \ge l)$ for all $(k,l) \in \mathbb{Z}_+^2$ and for $(\tilde{X}_1, \tilde{X}_2)$ distributed according
to (\ref{d: surv_fct_MO_exp}) with parameters $\nu_1, \nu_2, \nu_{12}>0$, $\nu:= \nu_1 + \nu_2 + \nu_{12}$, if we set 
\begin{align*}
q_1 := e^{-(\nu_1 + \nu_{12})}, \quad q_2:= e^{-(\nu_2 + \nu_{12})}, \quad p_{00} := e^{-\nu}.
\end{align*}
The condition $p_{00} \ge q_1 q_2$ corresponds to $\nu_{12} \ge 0$. For $\nu_{12}=0$, $\tilde{X}_1$ and $\tilde{X}_2$ are independent. 

\subsection{Approximation in $d_{TV}$ by a Poisson process on a lattice}\label{s: MO_Geo_dTV}
For any integer $n \ge 1$, let $\mathbf{X}_1, \ldots, \mathbf{X}_n$ be \iid copies of the random pair 
$\mathbf{X} = (X_1,X_2)$, which follows the Marshall-Olkin geometric distribution 
from Section \ref{s: MO_Geo} and takes values in $\mathbb{Z}_+^2 \subset [0,\infty)^2$. Let $A \in \mathcal{B}([0,\infty)^2)$. We consider the MPPE
$
\Xi_A = \sum_{i=1}^n I_{\{\textbf{X}_i\in A\}} \delta_{\mathbf{X}_i},  
$
which lives on the lattice $\mathbb{Z}_+^2$. 
The following normalisation is the Marshall-Olkin geometric counterpart to the normalisation used in Section \ref{s: MO_Exp_joint}
for studying joint threshold exceedances of Marshall-Olkin exponential marks:
\begin{equation}\label{d: MO_Geo_normalisation_joint}
(k^\star, l^\star) = \left(k \log (1/p_{00}) - \log n \,,\, l \log (1/p_{00}) - \log n\right), \quad \text{for any } (k,l) \in \mathbb{Z}_+.  
\end{equation}
Under this normalisation, $\Xi_A$ corresponds to 
\begin{equation}\label{d: MO_Geo_MPPE}
\Xi^\star_{A^\star} =\sum_{i=1}^n I_{\left\{\textbf{X}^\star_i \in A^\star\right\}} \delta_{\mathbf{X}_i^\star},
\end{equation}
which lives on the lattice $E^\star$ of normalised points $(k^\star, l^\star)$. Note that $E^\star \subset [-\log n, \infty)^2$. Furthermore, denote by
\begin{equation}\label{d: MO_Geo_mean}
W^\star_{A^\star}= \sum_{i=1}^n I_{\left\{\textbf{X}^\star_i \in A^\star\right\}} 
\end{equation}
the random number of normalised points in $A^\star$. 
For the particular choice $A=A_n = [u_{n},\infty)^2$ for some threshold $u_{n} \in [0,\infty)$,
we obtain $A^\star= A^\star_n = [u_n^\star, \infty)^2$ with $u_{n}^\star = u_n \log(1/p_{00}) - \log n$, and $\Xi^\star_{A^\star}$
captures joint threshold exceedances of the components of the normalised random pairs $\mathbf{X}^\star_1, \ldots, \mathbf{X}^\star_n$. 

The following proposition gives straightforward error estimates for the approximation of the law of $\Xi^\star_{A^\star}$ by that of a Poisson process
with mean measure $\mathbb{E}\Xi^\star_{A^\star}$, both for general sets $A^\star$, and for the particular choice $A^\star = [u_n^\star, \infty)^2$.
\begin{prop}\label{t: MO_Geo_dTV_lattice}
Suppose $\textbf{X}=(X_1,X_2)$ follows the Marshall-Olkin geometric distribution with parameters $q_1, q_2, p_{00} \in (0,1)$.   
For each integer $n \ge 1$, let $\mathbf{X}^\star_1, \ldots, \mathbf{X}^\star_n$ be \iid copies of the normalised random pair 
$\mathbf{X}^\star = (X^\star_1, X^\star_2)$ with state space $E^\star$, where $X_j^\star = \log(1/p_{00})X_j - \log n$, for $j=1,2$. 
Let $A^\star \in \mathcal{B}([0,\infty)^2)$ and let $\Xi^\star_{A^\star}$ and $W^\star_{A^\star}$ be defined as in (\ref{d: MO_Geo_MPPE}) and 
(\ref{d: MO_Geo_mean}), respectively. 
Then the mean measure of $\Xi^\star_{A^\star}$ is given by
\[
\bpis(B^\star) := \bpis_{\As}(B^\star) := \mathbb{E}\Xi^\star_{A^\star} (B^\star)= \sum_{(k^\star,l^\star) \in A^\star \cap E^\star \cap B^\star} nP(X_1^\star = k^\star, X_2^\star = l^\star), 
\]
for any $B^\star \in \mathcal{B}([-\log n, \infty)^2)$, where, for any $(k^\star, l^\star) \in E^\star$,
\begin{align}\label{d: MO_Geo_mass_fct_normalised}
&P\left(X_1^\star = k^\star, X_2^\star = l^\star\right) \nonumber \\
&= \left\{
\begin{array}{ll}
\frac{1}{n}\, (1-\frac{p_{00}}{q_2} - q_2 + p_{00})\,e^{-\frac{\log(p_{00}/q_2)}{\log p_{00}} \, k^\star} e^{- \frac{\log q_2}{\log p_{00}}\, l^\star}  & \textrm{ for } k^\star < l^\star , \\
\frac{1}{n}\,( 1-q_1 - q_2 + p_{00})\, e^{-k^\star} \phantom{e^{-\frac{\log(p_{00}/q_2)}{\log p_{00}} \, k^\star}}& \textrm{ for } k^\star=l^\star ,  \\
\frac{1}{n}\,(1-q_1 - \frac{p_{00}}{q_1} + p_{00} )\, e^{-\frac{\log q_1}{\log p_{00}}\,k^\star} e^{-\frac{\log(p_{00}/q_1)}{\log p_{00}}\,l^\star}& \textrm{ for } k^\star>l^\star ,
\end{array}
\right.
\end{align}
and 
$d_{TV}\left( \mathcal{L}\left(\Xi^\star_{A^\star} \right), \mathrm{PRM}(\bpis)\right)  \le P(\mathbf{X}^\star \in A^\star)$.
With $A^\star = A^\star_n = [u_n^\star, \infty)^2$ for any choice of $u_n^\star \ge -\log n$, we obtain
\begin{equation}\label{t: MO_Geo_dTV_basic_Poiapprox_appl}
d_{TV}\left( \mathcal{L}\left(\Xi^\star_{A^\star} \right), \mathrm{PRM}(\bpis)\right) 
\le \frac{e^{-u^\star_n}}{n}\,.
\end{equation}
\end{prop}
\begin{proof}
With (\ref{d: MO_pointprob}) and 
\[
P\left(X_1^\star = k^\star, X_2^\star = l^\star\right) = P\left( X_1 = \frac{k^\star + \log n}{\log(1/p_{00})}, X_2 = \frac{l^\star + \log n }{\log(1/p_{00})}\right), 
\]
we obtain (\ref{d: MO_Geo_mass_fct_normalised}) for the joint probability mass function of $\mathbf{X}^\star$.
For any set $B^\star \in \mathcal{B}([-\log n, \infty)^2)$, 
the mean measure of $\Xi^\star_{A^\star}$ applied to $B^\star$ is then given by
\[
nP(\mathbf{X}^\star \in A^\star \cap B^\star) =\sum_{(k^\star,l^\star) \in A^\star \cap E^\star \cap B^\star} nP(X_1^\star = k^\star, X_2^\star = l^\star).
\]
By Theorem \ref{t: MyMichelMultivariate},
\[
d_{TV}\left( \mathcal{L}\left(\Xi^\star_{A^\star} \right)\,, \mathrm{PRM}(\bpis)\right) 
 \le P(\mathbf{X}^\star \in A^\star),
\]
where, using (\ref{d: discretised_cdf}) and (\ref{d: MO_survival}), we find
\begin{align*}
P\left( \mathbf{X}^\star \in A^\star\right) 
&= P\left( X_1^\star \ge u_n^\star , X_2^\star \ge u_{n}^\star\right) \\
&= P \left( X_1 \ge \frac{u^\star_n + \log n}{\log(1/p_{00})}, X_2 \ge \frac{u^\star_n + \log n}{\log(1/p_{00})\, l^\star} \right)  \\
&= P \left( X_1 \ge \left \lceil \frac{u^\star_n + \log n}{\log(1/p_{00})} \right \rceil, 
X_2 \ge \left \lceil \frac{u^\star_n + \log n}{\log(1/p_{00})} \right \rceil \right)\\
&= p_{00}^{\left \lceil \frac{u^\star_n + \log n}{\log(1/p_{00})} \right \rceil} 
\le \frac{e^{-u^\star_n}}{n}\, .
\end{align*}
\end{proof}
\begin{remark}
For $A^\star = [u_{1n}^\star, \infty) \times [u_{2n}^\star, \infty)$ with $u_{1n}^\star \neq u_{2n}^\star \in [-\log n, \infty)$, we can 
proceed as in Section \ref{s: MO_Exp_joint} for the Marshall-Olkin exponential distribution, i.e. use Theorem \ref{t: MyMichelMultivariate},
in order to determine an estimate for $d_{TV}(\mathcal{L}(\Xi^\star_{A^\star}),$ $ \mathrm{PRM}(\bpis))$. 
\end{remark}
The error bound in (\ref{t: MO_Geo_dTV_basic_Poiapprox_appl}) is exactly the same as the one that we found in (\ref{e: MO_Exp_joint_dTV_Michel}) 
for analogous MPPE's with Marshall-Olkin exponential instead of geometric marks. The difference is of course that the mean measure $\bpis$ of the MPPE with Marshall-Olkin
geometric marks lives only on points $(k^\star, l^\star) \in A^\star \cap E^\star$ instead of on the whole of $A^\star \cap [-\log n, \infty)^2$.
\subsection{Construction of a ``continuous'' intensity function}\label{s: MO_Geo_cont_int_fct} 
Proposition \ref{t: MO_Geo_dTV_lattice} gives an error bound for the approximation of the MPPE $\Xi^\star_{A^\star}$ by a Poisson process whose
mean measure $\mathbb{E}\Xi^\star_{A^\star}$ lives on the lattice of normalised points $(k^\star, l^\star)$, i.e. on
\begin{align*}
E^\star 
&= \left\{ (k^\star, l^\star):\, k^\star = k \log(1/p_{00}) - \log n, 
l^\star = l \log(1/p_{00}) - \log n, \right.\\
&\left.\phantom{blaaaaaaaaaaaaaaaaaaaaaaaaaaaaaaaaaaaaaaaaaaaa}\text{for all }(k,l) \in \mathbb{Z}_+^2\right\} \\
&= \left(\log(1/p_{00})\mathbb{Z}_+  - \log n\right)^2 \subset [-\log n, \infty)^2.
\end{align*}
We would however prefer to approximate the law of the MPPE by that of a Poisson process with an easier-to-use and more flexible \textit{continuous} 
intensity measure $\bl^\star = \bl^\star_{A^\star}$ living on $A^\star \cap [-\log n, \infty)^2$. 

As discussed in Section \ref{s: MO_Geo}, the survival copula of the Marshall-Olkin geometric distribution is a Marshall-Olkin copula, and thereby consists 
of both an absolutely continuous part and a singular part on the curve $u^\alpha = v^\beta$ (which corresponds to the diagonal in $[-\log n, \infty)^2$). 
The ``continuous'' intensity measure $\bl^\star$ will have to mirror this behaviour, i.e. it 
will have to be of the form 
\begin{equation}\label{d: MO_Geo_cont_int_meas}
\bl^\star(B^\star) = \int_{A^\star \cap B^\star} \lambda^\star(s,t)dsdt + \int_{A^\star \cap B^\star \cap \{(s,t):\, s=t\}} \acute{\lambda}^\star(s)ds, 
\end{equation}
for any $B^\star \in \mathcal{B}([-\log n, \infty)^2)$, for ``continuous'' intensity functions $\lambda^\star$ and $\acute{\lambda}^\star$ that, 
if integrated over the entire space, will give $n$, i.e. that will ensure that
\[
\int_{-\log n}^\infty \int_{-\log n}^\infty \lambda^\star(s,t) dsdt + \int_{-\log n}^\infty \acute{\lambda}^\star(s)ds = nP\left(\mathbf{X}^\star \in [-\log n, \infty)^2 \right) = n.
\]
\begin{remark}
Note that for simplicity of language we here (and later on) somewhat abuse terminology when speaking of a ``continuous'' intensity function $\ls$ or a ``continuous'' intensity 
measure $\bls$. The bivariate intensity function $\ls$ is not continuous, but piecewise continuous, having a jump along the diagonal. 
The measure $\bls$ is continuous only in the sense that it has an intensity with respect to 
Lebesgue measure ($2$-dimensional on the off-diagonal and $1$-dimensional on the diagonal) and not with respect to a point measure. 
\end{remark}
The idea is to spread the point mass sitting on each of the off-diagonal lattice points $(k^\star, l^\star) \in E^\star$, $k^\star \neq l^\star$, uniformly
over each of their corresponding coordinate rectangles (or rather, coordinate squares)
\[
R^\star_{k^\star,l^\star} 
= \left[k^\star, k^\star + \log \left(\frac{1}{p_{00}}\right)\right) \times \left[l^\star, l^\star + \log \left(\frac{1}{p_{00}}\right)\right), \quad k^\star \neq l^\star,
\]
and to also spread the point probabilities of the diagonal points $(k^\star,k^\star)$ over the diagonal line $s=t$, where $s,t \ge - \log n$. 
We achieve this in the following three steps.

\textbf{Step 1.} Consider only the off-diagonal lattice points. We of course have 
\[
P\left(\mathbf{X}^\star \in R^\star_{k^\star,l^\star}  \right)= P\left(X_1^\star = k^\star, X_2^\star = l^\star\right),
\]
which is given by (\ref{d: MO_Geo_mass_fct_normalised}), and we may express the mean $nP(\mathbf{X}^\star \in A^\star)$ as 
\begin{multline}\label{p: spread_prob_1}
\sum_{ (k^\star,l^\star) \in A^\star, k^\star \neq l^\star } n
\int \int_{\Rstar} \frac{P(X_1^\star=k^\star, X_2^\star=l^\star)}{\log^2 (1/p_{00})}\, dsdt \\
+ \sum_{(k^\star, k^\star) \in A^\star}n P\left(X_1^\star=k^\star, X_2^\star=k^\star\right),
\end{multline}
where $\log^2 (1/p_{00})$ is the surface area of $\Rstar$. We have not actually changed anything yet as the integrand 
$P(X_1^\star=k^\star, X_2^\star=l^\star)/\log^2 (1/p_{00})$ is constant
with respect to the integrating variables $s$ and $t$, and 
\[
 \int \int_{\Rstar} \frac{P(X_1^\star=k^\star, X_2^\star =l^\star)}{\log^2 (1/p_{00})} \,dsdt = P(X_1^\star=k^\star, X_2^\star=l^\star).
\]
\begin{figure}
\begin{center}{\footnotesize \input{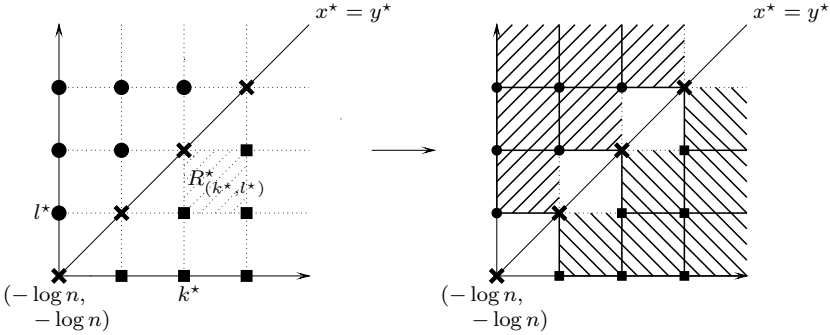}}\end{center}
\caption{Spread the point masses on the off-diagonal points over the corresponding coordinate rectangles.}
\label{f: spread_prob_1}
\end{figure}
As we aim to find a continuous intensity function over the entire space $[-\log n,$ $\infty)^2$,
we exchange $k^\star$ and $l^\star$ in the expression of the point probability $P(X^\star_1=k^\star, X_2^\star=l^\star)$ from (\ref{d: MO_Geo_mass_fct_normalised})
by $s$ and $t$, respectively. 
E.g., suppose that $k^\star < l^\star$. Then we replace the integral in (\ref{p: spread_prob_1}) by
\begin{equation}\label{p: spread_prob_2}
\int \int_{\Rstar} \frac{1-p_{00} /q_2 - q_2 + p_{00}}{\log^2 (1/p_{00})} \,
e^{-\frac{\log (p_{00} /q_2)}{\log p_{00}}\,s}
e^{-\frac{\log q_2}{\log p_{00}}\,t} ds dt
\end{equation}
Evaluation of this new integral gives
\begin{equation}\label{p: spread_prob_3}
\frac{ 1- p_{00}/q_2 - q_2 + p_{00}}{\log (p_{00}/q_2) \log q_2}\,P(X^\star_1=k^\star, X_2^\star=l^\star).
\end{equation}
The switch to variable $s$ and $t$ thus results only in the multiplication of the original point probability by a factor. 
The goal, however, is to integrate a function in $s$ and $t$ over $\Rstar$ and obtain the original point probability. This may
be achieved by simply dividing the integrand in (\ref{p: spread_prob_2}) by the multiplying factor found in 
(\ref{p: spread_prob_3}). Hence, we rewrite the mean as follows
\[
\sum_{ (k^\star,l^\star) \in A^\star, k^\star \neq l^\star } 
\int \int_{\Rstar} \lambda^\star(s,t)dsdt + n\sum_{(k^\star, k^\star) \in A^\star} P\left(X_1^\star=k^\star, X_2^\star=k^\star\right),
\]
where
\begin{equation}\label{p: int_fct_above}
\lambda^\star(s,t) = 
\frac{\log (p_{00}/q_2) \log q_2}{\log^2(1/p_{00})}\, e^{-\frac{\log (p_{00}/q_2)}{\log p_{00}}\,s} e^{-\frac{\log q_2}{\log p_{00}}\,t}, 
\forall (s,t) \in \Rstar\textnormal{with }k^\star < l^\star.
\end{equation}
Analogously, we find
\begin{equation}\label{p: int_fct_below}
\lambda^\star(s,t) = 
\frac{\log (p_{00}/q_1) \log q_1}{\log^2(1/p_{00})}\, e^{-\frac{\log q_1}{\log p_{00}}\,s} e^{-\frac{\log (p_{00}/q_1)}{\log p_{00}}\,t}
, \forall (s,t) \in \Rstar\textnormal{with }k^\star > l^\star.
\end{equation}
(\ref{p: int_fct_above}) and (\ref{p: int_fct_below}) supply suitable choices for the intensity function on coordinate 
rectangles lying above and below the diagonal, respectively. Figure \ref{f: spread_prob_1} illustrates Step 1. 

\textbf{Step 2.} We expand $\lambda^\star(s,t)$ from (\ref{p: int_fct_above}) and (\ref{p: int_fct_below}) to the entire space (without the diagonal), i.e. we define
\[
\lambda^\star(s,t) := \left\{ \begin{array}{ll} 
\frac{\log (p_{00}/q_2) \log q_2}{\log^2(1/p_{00})}\, e^{-\frac{\log (p_{00}/q_2)}{\log p_{00}}\,s} e^{-\frac{\log q_2}{\log p_{00}}\,t} & \textnormal{ for } s<t,\\
\frac{\log (p_{00}/q_1) \log q_1}{\log^2(1/p_{00})} \, e^{-\frac{\log q_1}{\log p_{00}}\,s} e^{-\frac{\log (p_{00}/q_1)}{\log p_{00}}\,t} & \textnormal{ for } s>t,
\end{array}\right.  
\]
for all $(s,t) \in [-\log n, \infty)^2$; see Figure \ref{f: spread_prob_2}. However, this adds surplus mass on the diagonal rectangles $R^\star_{k^\star, k^\star}$.
\begin{figure}
\begin{center}{\footnotesize \input{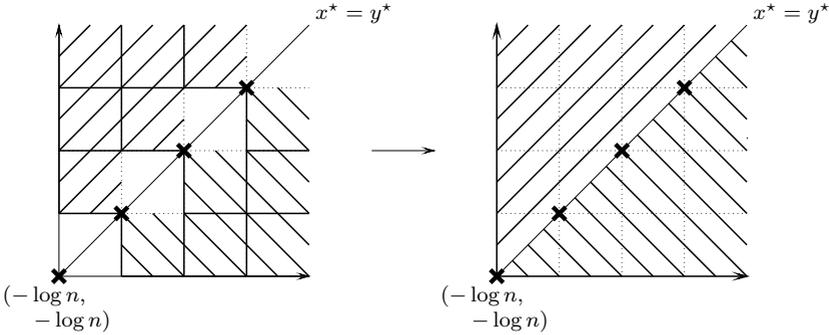}}\end{center}
\caption{Define the intensity functions determined for off-diagonal rectangles on the entire space.}
\label{f: spread_prob_2}
\end{figure}

\textbf{Step 3.} We adjust for the surplus mass on the diagonal rectangles by subtracting it from the point probabilities of the diagonal lattice points $(k^\star,k^\star)$,
and accordingly rewrite the mean as follows:
\begin{multline}\label{p: spread_prob_5}
\int \int_{A^\star} \lambda_n^\star(s,t)dsdt \\
+ n\sum_{(k^\star, k^\star) \in A^\star} \left\{ P(X_1^\star=k^\star, X_2^\star=k^\star) - \frac{1}{n}\int_{R^\star_{k^\star, k^\star}} \lambda_n^\star(s,t)dsdt\right\}.
\end{multline}
Computation of the term in curly brackets shows that the new mass that we put on the diagonal segments of each diagonal rectangle $R^\star_{k^\star,k^\star}$ is given by
\begin{equation}\label{p: spread_prob_6}
\frac{ e^{-k^\star}}{n}\,(1-p_{00})\left[ \frac{\log (1/q_1 q_2)}{\log (1/p_{00})}-1\right]. 
\end{equation}
Note that this equals 
\[
\int_{k^\star}^{k^\star + \log(1/p_{00})} \frac{e^{-s}}{n}\, \left[\frac{\log (1/q_1q_2)}{\log (1/p_{00})}-1 \right]ds,           
\]
for each $k^\star \in E^\star$, where we have parameterised the intensity function on the diagonal as projection along the $s$-axis.
We thus define:
\begin{align}\label{d: ls}
\begin{split}
\lambda^\star(s,t)&= \left \{
\begin{array}{lll} 
\frac{\log (p_{00}/q_2) \log q_2}{\log^2(1/p_{00})}\, e^{-\frac{\log (p_{00}/q_2)}{\log p_{00}}\,s} e^{-\frac{\log q_2}{\log p_{00}}\,t} &\textnormal{ for }& s < t ,  \\
\frac{\log (p_{00}/q_1) \log q_1}{\log^2(1/p_{00})}\, e^{-\frac{\log q_1}{\log p_{00}}\,s} e^{-\frac{\log (p_{00}/q_1)}{\log p_{00}}\,t} &\textnormal{ for }& s > t,   
\end{array} \right.\\
\acute{\lambda}^\star(s)&=\frac{\log (p_{00}/q_1 q_2)}{\log (1/p_{00})}\, e^{-s}\quad \textnormal{for } s = t.\\
\end{split}
\end{align}
Figure \ref{f: spread_prob_3} illustrates this last step in the construction of $\lambda^\star$. 
\begin{figure}
\begin{center}{\footnotesize \input{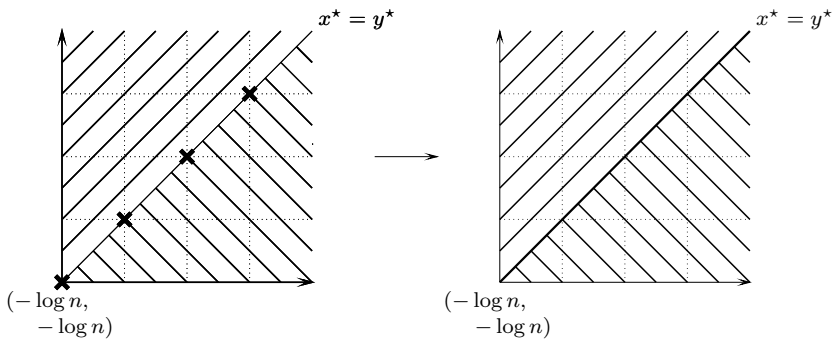}}\end{center}
\caption{The mass on the diagonal lattice points is spread over the entire diagonal.} 
\label{f: spread_prob_3}
\end{figure}

The above construction guarantees the following:
\begin{prop}\label{t: MO_Geo_same_meas_rectangles}
Let $\bl^\star$, $\lambda^\star$ and $\acute{\lambda}^\star$ be defined by (\ref{d: MO_Geo_cont_int_meas}) and (\ref{d: ls}). Then,\\
(i)$\,\, \displaystyle  \bl^\star \left(\Rstar\right) = \bpis\left(\Rstar\right)$, for any $(k^\star,l^\star) \in E^\star$,\\
(ii)$\,\, \displaystyle \int_{[-\log n, \infty)^2} \lambda^\star(s,t)dsdt + \int_{-\log n}^\infty \acute{\lambda}^\star(s)ds= n$. \\ \qed 
\end{prop}
\begin{remark}
Proceeding as in Section \ref{s: MO_Exp_joint} for the Marshall-Olkin exponential distribution, 
we may express the new intensity functions
$\lambda^\star$ and $\acute{\lambda}^\star$ in the original coordinate system by
\begin{align*}
\lambda(x,y)&=\lambda_n(x,y) = \left\{ 
\begin{array}{lll}
n \log (q_2) \log (\frac{p_{00}}{q_2})   p_{00}^x q_2^{y-x} &\textnormal{ for } x <y,\\
n \log (q_1) \log (\frac{p_{00}}{q_1}) q_1^{x-y} p_{00}^y &\textnormal{ for } x>y,
\end{array} \right.\\
\acute{\lambda}(x) &= \acute{\lambda}_n(x)= n \log \left( \frac{p_{00}}{q_1 q_2}\right) p_{00}^x \quad \textnormal{ for } x=y,\\
\end{align*} 
for any $(x,y) \in [0,\infty)^2$. We recognise a weighted and continuous version of $P(X_1 \ge k,X_2 \ge l)$ from (\ref{d: MO_survival}). 
\end{remark}
\subsection{Assumptions on the distributional parameters}\label{s: MO_Geo_cont_nicer_int_fct}
The continuous intensity measure $\bl^\star$ defined by (\ref{d: MO_Geo_cont_int_meas}) and (\ref{d: ls}) depends on the parameters $q_1$, $q_2$ and $p_{00}$ of the Marshall-Olkin geometric distribution.
Our aim is to determine a bound on the error for the approximation of the Poisson process with mean measure $\mathbb{E}\Xi^\star_{A^\star}$, living on the lattice $E^\star$,
by a Poisson process with mean measure $\bl^\star$. 
We already did something similar for the univariate geometric distribution in Proposition \ref{t: Geo_MPPE_d2}, and it turned out that the error could only become small
if the success probability $p=p_n$ vanished as $n \to \infty$. As Section \ref{s: MO_Geo_d2} will show below, the 
probability of simultaneous success, $p_{11}$, for the Marshall-Olkin geometric distribution, will similarly have to tend to $0$ as $n \to \infty$.
Since $p_{00}+p_{01}+p_{10}+p_{11}=1$, this of course influences 
the distributional parameters $p_{00}, q_1$ and $q_2$ in that it also makes them dependent on $n$.
The continuous intensity functions $\lambda^\star$ and $\acute{\lambda}^\star$ thus have the drawback that, through their dependence on the 
parameters $p_{00}$, $q_1$ and $q_2$, they are also dependent on $n$.
We thus try to find other suitable continuous intensity functions that no longer vary with the sample size.  

For simplicity, we make the assumption that $p_{10}$ and $p_{01}$ vary at the same rate as $p_{11}=p_{11n}$; more precisely, assume
$p_{10} = p_{10n}=\gamma p_{11n}$ and $p_{01} = p_{01n}=\delta p_{11n}$, where $\gamma$ and $\delta$ are strictly positive real numbers, bounded such
that $p_{10}$ and $p_{01}$ are smaller than 1. We assume that $p_{11n}$ tends to $0$ as $n \to \infty$ at a rate that will be determined later, and 
express the distributional parameters as functions of it:
\begin{equation}\label{d: conditions}
\begin{split}
q_{1n} &=  1-(1+\gamma)p_{11n},\\
q_{2n} &= 1-(1+\delta)p_{11n},\\
p_{00n} &= 1-(1+\gamma+\delta)p_{11n}.
\end{split}
\end{equation}
Plugging into (\ref{d: ls}) and using the relation $\log(1-z) \sim -z$ for $|z| < 1$ and $z \to 0$, we find that $\lambda^\star(s,t)$ and $\acute{\lambda}^\star(s)$ are, 
for $p_{11n} \to 0$ as $n \to \infty$,
asymptotically equal to 
\begin{align}\label{d: lsnew}
\begin{split}
 \lsnew(s,t)&:= \left \{
 \begin{array}{lll} 
 \frac{\gamma(1+\delta)}{(\gade)^2}\,e^{-\frac{\gamma}{\gade}\,s}e^{-\frac{1 + \delta}{\gade}\,t} &\textnormal{ for }& s < t ,  \\
 \frac{\delta (1 + \gamma)}{(\gade)^2}\,e^{-\frac{1+\gamma}{\gade}\,s} e^{-\frac{\delta}{\gade}\,t} &\textnormal{ for }& s > t,  
 \end{array} \right.\\
 \text{and } \quad \acute{\lambda}^\star_{\gamma, \delta}(s)&:=\frac{1}{\gade}\, e^{-s} \quad \textnormal{ for }  s = t,
\end{split}
\end{align}
respectively, for all $(s,t) \in [-\log n, \infty)^2$. At first glance $\lsnew$ and $\acute{\lambda}^\star_{\gamma, \delta}$ seem to be valid choices 
for continuous intensity functions independent of $n$. 
We will investigate in Section \ref{s: MO_Geo_nicer} whether a Poisson process with mean measure $\bl^\star$ on $A^\star$ may indeed be approximated by a 
Poisson process with mean measure 
\begin{equation}\label{d: blsnew}
\blsnew(B^\star) := \int \int_{A^\star \cap B^\star} \lsnew(s,t)dsdt + \int_{A^\star \cap B^\star \cap \{(s,t):\, s=t\}}\acute{\lambda}^\star_{\gamma, \delta}(s)ds,
\end{equation}
for all $B^\star \in \mathcal{B}([-\log n, \infty)^2)$. 
To do the corresponding error calculations for a fixed sample size $n$ we first need to examine in further detail the differences between 
the exponent terms in $\lambda^\star(s,t)$ and $\lsnew(s,t)$:
\begin{lemma}\label{t: Lemma_cond} 
For each integer $n \ge 1$, let $p_{11n} \in (0,1)$ and let $q_{1n}$, $q_{2n}$, $p_{00n}$ $\in (0,1)$  be defined by (\ref{d: conditions}). Then,
\begin{align*}
(i) \qquad 0 \le \frac{1+\delta}{\gade} - \frac{\log q_{2n}}{\log p_{00n}} &\le \frac{\gamma p_{11n}}{1-(\gade)p_{11n}}\,,\\
(ii) \qquad 0 \le \frac{1+\gamma}{\gade} -\frac{\log q_{1n}}{\log p_{00n}} &\le \frac{\delta p_{11n}}{1-(\gade)p_{11n}}\,.
\end{align*}
Moreover,
\begin{align*}
(iii) \qquad 0 \le \log\left(\frac{q_{2n}}{p_{00n}}\right) \le \frac{\gamma p_{11n}}{1-(\gade)p_{11n}}\,, \\
(iv) \qquad 0 \le \log\left(\frac{q_{1n}}{p_{00n}}\right) \le \frac{\delta p_{11n}}{1-(\gade)p_{11n}}\,,
\end{align*}
and 
\[
(v)\qquad \log\left(\frac{1}{p_{00n}}\right)\log\left(\frac{p_{00n}}{q_{1n}q_{2n}}\right) \le \frac{(\gade)p_{11n}}{\{1-(\gade)p_{11n}\}^2}\,.
\]
\end{lemma}
\begin{proof}
(i) For ease of notation we omit the subscript $n$. Since, for all $|z| < 1$, $-\log(1-z)/z$ is increasing and $-(1-z)\log(1-z)/z$ is decreasing, we obtain the
following lower and upper bound, respectively, for $-(\log q_2)/(\log p_{00})$, where $q_2 < p_{00}$:
\[
- \frac{1+\delta}{\gade} \le - \frac{\log q_2}{\log p_{00}} \le - \frac{(1+\delta)\cdot [1-(\gade)p_{11}]}{[1-(1+\delta)p_{11}]\cdot(\gade)}.
\]
Therefore,
\begin{align*}
0 \le \frac{1+\delta}{\gade} - \frac{\log q_2}{\log p_{00}} 
&\le \frac{1+\delta}{\gade} \left\{ 1- \frac{1-(\gade)p_{11}}{1-(1+\delta)p_{11}}\right\}\\
&= \frac{(1+\delta)\gamma p_{11}}{(\gade) [1- (1+\delta)p_{11}]}
\le \frac{\gamma p_{11}}{1-(\gade)p_{11}}\,.
\end{align*}
(iii) Moreover, since $q_2 =p_{00}+p_{10}$, we have $\log(q_2/ p_{00}) \ge 0$. Using $\log(1+ z) \le z$ for positive $z$, we obtain
\[
\log(q_2/p_{00}) = \log \left( \frac{p_{00}+p_{10}}{p_{00}}\right) \le \frac{p_{10}}{p_{00}} = \frac{\gamma p_{11}}{1-(\gade)p_{11}}\,.
\]
(ii) and (iv) can be shown analogously to (i) and (iii), respectively.\\
(v) We have
\begin{align*}
&\log\left(\frac{1}{p_{00n}}\right)\log\left(\frac{p_{00n}}{q_{1n}q_{2n}}\right) \\
&= (-\log p_{00}) \left\{ -\log(p_{00} + p_{01}) - \log(p_{00}+p_{10}) + \log p_{00}\right\}\\
& \le (-\log p_{00}) \left\{ -\log p_{00} - \log p_{00} + \log p_{00} \right\}\\
& = (-\log p_{00})^2 \le \frac{(1-p_{00})^2}{p_{00}^2} \le \frac{1-p_{00}}{p_{00}^2} \\
&= \frac{(\gade)p_{11}}{\left\{ 1- (\gade)p_{11}\right\}^2}\,. 
\end{align*}
\end{proof}
\noindent We will use Lemma \ref{t: Lemma_cond} to determine error estimates in Sections \ref{s: MO_Geo_d2} and \ref{s: MO_Geo_nicer}. 
\begin{remark}
We suppose here that $\gamma$ and $\delta$ do not vary with $n$. However, the asymptotic equivalence of (\ref{d: ls}) and (\ref{d: lsnew}),
and later results (i.e. Propositions \ref{t: MO_Geo_disc_to_cont_appl} and \ref{t: MO_Geo_cont_to_nicer_appl}, as well as Corollary \ref{t: MO_Geo_final_bound_d2}) 
also hold for the case $\gamma=\gamma_n$ and $\delta=\delta_n$. These results are thus actually stronger than we make them out to be.  
\end{remark}
\subsection{Approximation in $d_2$ by a Poisson process with continuous intensity}\label{s: MO_Geo_d2} 
We now determine the error of the approximation of the Poisson process with mean measure $\bpis$, living on lattice points $(k^\star, l^\star) \in A^\star \cap E^\star$, 
and the Poisson process with continuous mean measure $\bls$, living on $A^\star \cap [-\log n, \infty)^2$. As the total variation distance is too strong to achieve this, 
we use the weaker $d_2$-distance that we introduced in Section \ref{Sec: Improved_rates}. Theorem \ref{t: approx_disc_cont} gives a general error estimate for any set 
$\As \in \mathcal{B}([-\log n, \infty)^2)$, which Proposition \ref{t: MO_Geo_disc_to_cont_appl} in turn applies to the particular choice $A^\star = [u^\star, \infty)^2$.

Note that any not too small set $\As \in \mathcal{B}([-\log n, \infty)^2)$ contains subsets that are unions of coordinate rectangles $\Rstar$, i.e. of the form
\begin{equation}\label{d: TildeAstar}
 \bigcup_{(k^\star,l^\star) \in M^\star}  R^\star_{k^\star, l^\star} \subseteq A^\star, 
\end{equation}
where $M^\star$ is a countable subset of $E^\star$. 
Let $\tilde{A}^\star$ denote the biggest set $\subseteq A^\star$ of the form (\ref{d: TildeAstar}); see Figure \ref{f: TildeAstar} for some examples.
In order to prove Theorem \ref{t: approx_disc_cont}, we distinguish between the errors on
$\tilde{\As}$ and $\As \setminus \tilde{\As}$. Even though $\bpis(\As)=\bls(\As)$ is not necessarily satisfied, Proposition \ref{t: MO_Geo_same_meas_rectangles} ensures that at least $\bpis(\tilde{\As})= \bls(\tilde{\As})$. We may therefore use Lemma \ref{t: helping_bound_s2} to bound the error on $\tilde{\As}$ by way of the $d_1$-distance between $\bpis$ and $\bls$ on $\tilde{\As}$. The size of the $d_1$-distance depends on the choice of the $d_0$-distance. As in Section \ref{s: MPPE_geo}, where we treated MPPE's with univariate geometric marks, we choose the Euclidean distance bounded by $1$. For the remaining error, we rely on the ``small" size of $\As \setminus \tilde{\As}$ and use Lemma \ref{t: Delta_1_upgamma_d2} for an upper bound on $\Delta_1 \upgamma$, where $\upgamma$ is the solution to an appropriate Stein equation. 
\begin{figure}
\begin{center}{\footnotesize \input{TildeAstar.TpX}}\end{center}
\caption{Examples of sets $\tilde{A^\star}$.}
\label{f: TildeAstar}
\end{figure}
\begin{thm}\label{t: approx_disc_cont}
With the notations from Sections \ref{s: MO_Geo}-\ref{s: MO_Geo_cont_nicer_int_fct}, 
we obtain, for a set $A^\star \in \mathcal{B}([-\log n, \infty)^2)$,
\begin{equation}\label{t: approx_disc_cont_eq} 
d_2(\mathrm{PRM}(\bpis), \mathrm{PRM}(\bl^\star))\\
\le 2\sqrt{2} \log(1/ p_{00}) + \left(1 \wedge \frac{1.65}{\sqrt{\bl^\star(A^\star)}} \right) \bl^\star(A^\star \setminus \tilde{A}^\star),
\end{equation}
where $\tilde{A}^\star$ denotes the biggest set $\subseteq A^\star$ that is a union of coordinate rectangles, i.e. 
$\tilde{A}^\star = \cup_{(k^\star, l^\star) \in M^\star} R^\star_{k^\star,l^\star}$, where $M^\star$ is the biggest subset of 
$E^\star$ such that $\tilde{A}^\star \subseteq \As$.\\
\end{thm}
\begin{proof}
Let $\Xi_{\bpis} \sim \mathrm{PRM}(\bpis)$ and $\Xi_{\bls} \sim \mathrm{PRM}(\bls)$. Suppose that $Z= \{Z_t, t\in \mathbb{R}_+\}$ is an immigration-death process on $A^\star$ with 
immigration intensity $\bls$, unit per-capita death rate, equilibrium distribution $\mathcal{L}(\Xi_{\bls})$, and generator $\mathcal{A}$.
Furthermore, let $\mathcal{H}$ denote the set of functions $h:\, M_p(A^\star) \to \mathbb{R}$ such that (\ref{d: s_2h}) is satisfied
and let $\upgamma:\, M_p(A^\star) \to \mathbb{R}$ be defined by $\upgamma(\xi)$ $= $
$-\int_0^\infty \{\mathbb{E}^\xi h(Z_t) - \mathrm{PRM}(\bls)\}dt$, for any $\xi \in M_p(A^\star)$. By Proposition \ref{t: upgamma_welldefined}, $\upgamma$ is well-defined, 
and by (\ref{d: Stein_eq_d2}), $|\mathrm{PRM}(\bpis)(h)- \mathrm{PRM}(\bls)(h)|$ equals $|\mathbb{E}(\mathcal{A}\upgamma)(\Xi_{\bpis})|$. 
Proceeding as in the proof of Proposition \ref{t: dTV_two_PRM}, we find that
\[
\mathbb{E}(\mathcal{A}\upgamma)(\Xi_{\bpis})= \mathbb{E} \int_{A^\star} \left[ \upgamma(\Xi_{\bpis} + \delta_{\mathbf{z}}) - \upgamma(\Xi_{\bpis})\right]
\left( \bls(d\bz) - \bpis(d\bz)\right), 
\]
and thus 
\[
\frac{|\mathbb{E}(\mathcal{A}\upgamma)(\Xi_{\bpis})|}{s_2(h)} \le
\frac{1}{s_2(h)}\, \mathbb{E} \left| \int_{A^\star} \left[ \upgamma(\Xi_{\bpis} + \delta_{\bz}) - \upgamma(\Xi_{\bpis})\right]
\left( \bls(d\bz) - \bpis(d\bz)\right)\right|, 
\]
where, for any $\xi \in M_p(A^\star)$,
\begin{multline}\label{p: partitionA}
\left|\int_{\As} [\upgamma(\xi+\delta_{\bz})-\upgamma(\xi)](\bls(d\bz)-\bpis(d\bz))\right|\\ 
\le \left|\int_{\tilde{A}^\star} [\upgamma(\xi+\delta_{\bz})-\upgamma(\xi)]
(\bls(d\bz)-\bpis(d\bz)) \right|\\
+\left|\int_{\As\smallsetminus\tilde{A}^\star} 
[\upgamma(\xi+\delta_{\bz})-\upgamma(\xi)](\bls(d\bz)-\bpis(d\bz))\right|.
\end{multline}
The second summand may be bounded by 
\begin{equation} \label{p: summand2}
 \int_{\As \smallsetminus \tilde{A}^\star} |\upgamma(\xi+\delta_{\bz})-\upgamma(\xi)|\cdot 
 |\bls (d\bz)-\bpis(d\bz)| 
 \le \Delta_1\upgamma \int_{\As \smallsetminus \tilde{A}^\star}|\bls(d\bz)-\bpis(d\bz)|.
\end{equation}
Note that
\[
 \int_{\As \smallsetminus \tilde{A}^\star}|\bls(d\bz)-\bpis(d\bz)| \le \bls(\As \setminus \tilde{A}^\star), 
\]
and that Lemma \ref{t: Delta_1_upgamma_d2} gives 
\begin{equation} \label{p: deltaone}
 \Delta_1\upgamma \le s_2(h)\left( 1 \wedge \frac{1.65}{\sqrt{\bls(\As)}} \right).
\end{equation}
By Proposition \ref{t: MO_Geo_same_meas_rectangles}, $\bls(\tilde{A}^\star) = \bpis(\tilde{A}^\star)= nP(\mathbf{X}^\star \in \tilde{A}^\star) < \infty$. 
We may therefore use Lemma \ref{t: helping_bound_s2} to bound the first summand by
\begin{equation} \label{p: summand1}
s_2(h)\left( 1-e^{-\bls(\tilde{A}^\star)} \right) \left( 1+\frac{\bls(\tilde{A}^\star)}{|\xi|+1} \right) 
d_1(\bpis,\bls)|_{\tilde{A}^\star}\,, 	
\end{equation}
where $d_1(.\,,.)|_{\tilde{A}^\star}$ denotes the $d_1$-distance on $\tilde{A}^\star$ (instead of on $A^\star$). 
We have
\begin{equation} \label{p: expected}
\mathbb{E}\left(\frac{1}{|\Xi_{\bpis}|+1}\right)= \frac{1-e^{-\bpis(\As)}}{\bpis(\As)},
\end{equation}
since $|\Xi_{\bpis}|$ $ \sim$ $ \mathrm{Poi}(\bpis(\As))$.
Taking expectations in (\ref{p: partitionA}) and using (\ref{p: summand2}) - (\ref{p: expected}), we obtain
\begin{multline}\label{p: d2boundfirst}
\left |\mathbb{E}(\mathcal{A}\upgamma)(\Xi_{\bpis}) \right|/s_2(h) \\
\le  \left( 1-e^{-\bls(\tilde{A}^\star)} \right) 
\left\{ 1+ \frac{\bls(\tilde{A}^\star)}{\bpis(\As)}\left( 1-e^{-\bpis(\As)}\right)\right\}
d_1(\bpis,\bls)|_{\tilde{A}^\star} \\
+ \left( 1 \wedge \frac{1.65}{\sqrt{\bls(\As)}} \right) 
\bls(\As \setminus  \tilde{A}^\star). 
\end{multline}
We may further simplify by bounding $1-e^{-\bls(\tilde{A}^\star)}$ and $1-e^{-\bpis(\tilde{A}^\star)}$ by $1$
and noting that, since $\bpis(\As)= \bpis(\As \setminus \tilde{A}^\star) + \bls(\tilde{A}^\star) 
\ge \bls(\tilde{A}^\star)$,
we have
\begin{equation} \label{p: smaller2}
 1 + \frac{\bls(\tilde{\As})}{\bpis(\As)} \le 2.
\end{equation}
With the definitions of $\mathcal{K}$ and $s_1(\kappa)$ from Section \ref{Sec: Improved_rates}, the $d_1$-distance between $\bls$  and $\bpis$ on ${\tilde{A}}^\star$ 
is given by 
\[
 d_1(\bpis,\bls)|_{\tilde{A}^\star} = \frac{1}{\bls(\tilde{A}^\star)} \,
\sup_{\kappa \in \mathcal{K}} \frac{\left| 
\int_{\tilde{A}^\star}\kappa d\bpis-\int_{\tilde{A}^\star} \kappa d\bls \right|}{s_1(\kappa)}\,.
\]
As $\tilde{A}^\star$ is a union of coordinate rectangles $\Rstar$, the term
$\int_{\tilde{A}^\star}\kappa d\bpis-\int_{\tilde{A}^\star} \kappa d\bls $ may be expressed as
\begin{equation} \label{p: d1diff}
  \sum_{(k^\star,l^\star) \in \tilde{A}^\star} \left\{ \int_{\Rstar} \kappa(\bz)\bpis(d\bz) 
 -\int_{\Rstar}\kappa(\bz)\bls(d\bz) \right\}.
\end{equation}
Furthermore, again by Proposition \ref{t: MO_Geo_same_meas_rectangles},
\[
 \int_{\Rstar}\kappa(\bz)\bpis(d\bz)  = \kappa((k^\star,l^\star)) \bpis(\Rstar) = \kappa((k^\star,l^\star)) \bls(\Rstar).
\]
Hence, we find the following upper bound for (\ref{p: d1diff}):
\[
 \sum_{(k^\star,l^\star) \in \tilde{A}^\star}  \int_{\Rstar} \left|\kappa((k^\star, l^\star)) -\kappa(\bz)\right|\bls(d\bz),
 \]
which, by definition of the Lipschitz constant $s_1(k)$, is smaller than 
\[
  s_1(\kappa)d_0((k^\star,l^\star),\bz)\bls(\tilde{A}^\star).
\]
The biggest possible Euclidean distance between the lower left corner point $(k^\star, l^\star)$ and any other point $\bz$ in the 
rectangle $\Rstar$ is given by the length $\sqrt{2}\log(1/ p_{00})$ of its diagonal. Thus,
\begin{equation} \label{p: diaglength}
d_1(\bpis,\bls)|_{\tilde{A}^\star} \le \sqrt{2}\log (1/p_{00}).
\end{equation}
(\ref{p: smaller2}) and (\ref{p: diaglength}) give the upper bound $2\sqrt{2}\log(1/ p_{00})$ for the first 
summand of the error term in (\ref{p: d2boundfirst}). This completes the proof.
\end{proof}
\noindent Theorem \ref{t: approx_disc_cont} gives sharp results only if the probability of simultaneous failure, $p_{00} = p_{00n}$ tends to $1$ as $n \to \infty$. This makes sense since $\log(1/p_{00})$, introduced as scaling factor of the original marginal geometric random variables, provides the side lengths of the rescaled lattice squares. The condition $p_{00n} \uparrow 1$  makes the side lengths of the coordinate squares tend to $0$ and thus causes the ``disappearance" of the lattice into the whole real subset $[-\log n, \infty)^2$. The same holds for the area $\As \setminus \tilde{\As}$, thereby also causing the disappearance of the second error term as $n \to \infty$.

For sets $\As$ that are unions of coordinate rectangles, we immediately obtain the following corollary, as there is no left-over area $\As \setminus \tilde{\As}$, and by consequence no second error term. 
\begin{cor}\label{t: approx_disc_cont_rect_set}
Let $\As \in \mathcal{B}([-\log n, \infty)^2)$ be a union of coordinate rectangles, i.e. $A^\star = \cup_{(k^\star, l^\star) \in M^\star} R^\star_{k^\star,l^\star}$ 
where $M^\star \subseteq E^\star$. Then, 
\[
d_2(\mathrm{PRM}(\bpis), \mathrm{PRM}(\bl^\star)) \le 2\sqrt{2} \log(1/ p_{00}) .
\] \qed
\end{cor}
\begin{figure}
\begin{center}{\footnotesize \input{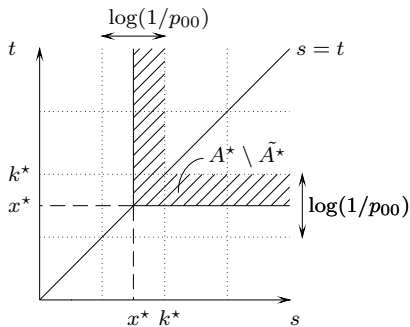}}\end{center}
\caption{The set $\As \setminus \tilde{\As}$.}
\label{f: proof_cor_d2_discr_to_cont}
\end{figure}
We now apply Theorem \ref{t: approx_disc_cont} to the case where $\As =\As_n = [u_n^\star, \infty)^2$ and express the error estimate in terms of the threshold $u_n^\star$ and the probability of simultaneous success $p_{11n}$. 
To achieve this we assume that the distributional parameters $p_{00}, q_{1}$ and $q_2$ are defined as in Section \ref{s: MO_Geo_cont_nicer_int_fct}.
\begin{prop}\label{t: MO_Geo_disc_to_cont_appl}
Let $p_{11n} \in (0,1)$ and assume that $q_{1n}$, $q_{2n}$ and $p_{00n}$ satisfy (\ref{d: conditions}). For any choice of $u_n^\star \ge -\log n$, 
define $A^\star = [u_n^\star,\infty)^2$.  
With the notations from Theorem \ref{t: approx_disc_cont},
\begin{multline*}
d_2(\mathrm{PRM}(\bpis), \mathrm{PRM}(\bl^\star))\\
\le \frac{(\gade)p_{11n}}{[1-(\gade)p_{11n}]^2}\left\{2\sqrt{2} + 3\left(e^{-u_n^\star} \wedge 1.65 e^{-u_n^\star/2} \right)\right\}.
\end{multline*}
\end{prop}
\begin{proof}
For ease of notation we omit the subscript $n$. We apply result (\ref{t: approx_disc_cont_eq}) from Theorem \ref{t: approx_disc_cont} to the special case 
$A^\star = [u^\star, \infty)^2$. Due to (\ref{d: conditions}) and $-\log(1-z) \le z/(1-z)$ for $|z| < 1$, we may bound the first of the two error terms in 
(\ref{t: approx_disc_cont_eq}) as follows:
\begin{equation}\label{p: MO_Geo_bubulala}
2\sqrt{2}\log(1/p_{00}) \le \frac{2\sqrt{2}(\gade)p_{11}}{1-(\gade)p_{11}} \le \frac{2\sqrt{2}(\gade)p_{11}}{[1-(\gade)p_{11}]^2}\,.
\end{equation}
Direct computation yields $\bls(\As)= e^{-u^\star}$. As illustrated by Figure \ref{f: proof_cor_d2_discr_to_cont}, 
$\bls(\As \setminus \tilde{\As})$ may be bounded by
\begin{multline*}
\int_{u^\star}^{\infty} \log (1/p_{00}) \sup_{s \in [u^\star,k^\star]} \ls(s,t)dt \\
+ \int_{u^\star}^{\infty} \log (1/p_{00}) \sup_{t \in [u^\star,k^\star]} \ls(s,t)ds
+ \int_{u^\star}^{k^\star}\sqrt{2} \log(1/p_{00}) \sup_{s \in [u^\star,k^\star]} \acute{\lambda}^\star(s) ds.
\end{multline*}
Note that 
\begin{align*}
\sup_{s \in [u^\star,k^\star]} \exp{\left\{-\frac{\log(p_{00}/q_2)}{\log p_{00}}\,s\right\}} &\le \exp{\left\{-\frac{\log(p_{00}/q_2)}{\log p_{00}}\,u^\star\right\}},\\
\sup_{t \in [u^\star,k^\star]} \exp{\left\{-\frac{\log(p_{00}/q_1)}{\log p_{00}}\,t\right\}} &\le \exp{\left\{-\frac{\log(p_{00}/q_1)}{\log p_{00}}\,u^\star\right\}},
\end{align*}
and $\sup_{s \in [u^\star,k^\star]} e^{-s} \le e^{-u^\star}.$ Thus, by definition (\ref{d: ls}) of $\ls(s,t)$,
\begin{multline*}
\int_{u^\star}^{\infty} \log (1/p_{00}) \sup_{s \in [u^\star,k^\star]} \ls(s,t)dt \\
\le \frac{\log(p_{00}/q_2) \log q_2}{\log(1/p_{00})}\, e^{-\frac{\log(p_{00}/q_2)}{\log p_{00}}\,u^\star} \int_{u^\star}^{\infty} e^{-\frac{\log q_2}{\log p_{00}}\,t}dt,
\end{multline*}
which equals $\log(q_2/p_{00})e^{-u^\star}$. Analogously,
\[
\int_{u^\star}^{\infty} \log (1/p_{00}) \sup_{t \in [u^\star,k^\star]} \ls(s,t)ds \le \log(q_1/p_{00})e^{-u^\star},
\]
whereas 
\[
\int_{u^\star}^{k^\star} \sqrt{2}\log(1/p_{00}) \sup_{s \in [u^\star,k^\star]} \acute{\lambda}^\star(s) ds 
 \le  2\log^2(1/p_{00})  \,\frac{\log(p_{00}/q_1q_2)}{\log (1/p_{00})} \,e^{-u^\star},
\]
since $k^\star - u^\star \le \sqrt{2}\log(1/p_{00})$. We obtain
\[
\bls(\As \setminus \tilde{\As}) 
 = e^{-u^\star} \left\{ \log \left(\frac{q_2}{p_{00}}\right) + \log \left(\frac{q_1}{p_{00}}\right) +  2\log \left( \frac{1}{p_{00}}\right)
\log\left( \frac{p_{00}}{q_1 q_2}\right)\right\}.
\]
By Lemma \ref{t: Lemma_cond} (iii)-(v), the term in curly brackets may be bounded by
\[
\frac{(\gamma + \delta)p_{11}}{1-(\gade)p_{11}} + \frac{2(\gade)p_{11}}{\left\{ 1-(\gade)p_{11}\right\}^2} \le  \frac{3(\gade)p_{11}}{\left\{ 1-(\gade)p_{11}\right\}^2}\,.
\]
An upper bound for the second error term in (\ref{t: approx_disc_cont_eq}) is thus given by
\[
\left(e^{-u_n^\star} \wedge 1.65 e^{-u_n^\star/2} \right)\frac{3(\gade)p_{11n}}{[1-(\gade)p_{11n}]^2}\,. 
\]
By adding this to the bound in (\ref{p: MO_Geo_bubulala}) we obtain the result. 
\end{proof}
\noindent The first of the error terms given by Proposition \ref{t: MO_Geo_disc_to_cont_appl}, i.e.
\[
 \frac{2\sqrt{2}(\gade)p_{11n}}{[1-(\gade)p_{11n}]^2},
\]
is a bound on the error $2\sqrt{2}\log(1/p_{00n})$ from Theorem \ref{t: approx_disc_cont}, where we used the assumption from Section \ref{s: MO_Geo_cont_nicer_int_fct} that $p_{00n} = 1-(\gade)p_{11n}$. This error term thus becomes small only if the probability of simultaneous success, $p_{11n}$, tends to $0$ as $n$ increases. The second error term, i.e. 
\[
\left(e^{-u_n^\star} \wedge 1.65 e^{-u_n^\star/2} \right)\frac{3(\gade)p_{11n}}{[1-(\gade)p_{11n}]^2}\,,
\]
is the bigger of the two, and determines the rate at which $p_{11n}$ must converge to $0$. The reason for that is that $p_{11n}$ must converge fast enough in order to offset the effect of the factor $e^{-u_n^\star}$ which we will want to be increasing with increasing $n$, since $e^{-u_n^\star} = \bls(\As_n)$ is 
the expected number of points in $\As_n$ of the approximating Poisson process, as well as more or less the expected number of threshold exceedances of the MPPE, for which we have $e^{-u_n^\star}/ p_{00n} \le \bpis(\As_n) \le e^{-u_n^\star}$. For instance, for a threshold $u_n^\star$ of size $- \log \log n$, the expected number of points in $\As$ of the two Poisson processes is $\log n$, the MPPE captures roughly the biggest $\log n$ points of its sample, and we need $p_{11n}= o(\log^{-1} n)$ for a sharp error bound. Suppose, for example, that $p_{11n} = n^{-1}$. 
Then, by (\ref{d: conditions}), the marginal probabilities of failure of $\mathbf{X}^\star_n$, $q_{1n}$ and $q_{2n}$, as well as the probability of simultaneous failure, $p_{00n}$, tend to $1$ very fast.

The mean measure $\bls$ is by definition dependent on the values of the distributional parameters. Since these need to vary with the sample size $n$ in order to obtain a small error for the approximation of $\mathrm{PRM}(\bpis)$ by $\mathrm{PRM}(\bls)$, it follows that $\bls = \bls_n$ (and of course also $\bpis = \bpis_n$). 
Though we have now achieved the goal of successfully approximating by a Poisson process with a continuous intensity, 
the conditions needed to accomplish this imply that we are not satisfied with our results yet, 
since we prefer to approximate by a Poisson process with continuous intensity that does not vary with $n$. 
As the next section will demonstrate, a suitable candidate is given by the Poisson process with intensity measure $\blsnew$ defined in (\ref{d: blsnew}). 
\subsection{Approximation in $d_2$ and $d_{TV}$ by a Poisson process independent of $n$}\label{s: MO_Geo_nicer}
We determine an error estimate for the approximation of the Poisson process with intensity measure $\bls=\bls_n$ by the Poisson process with intensity measure $\blsnew$, defined in (\ref{d: blsnew}), that does not depend on the sample size $n$. Since both intensities are continuous, there is no special need to use the $d_2$-distance. We give the error in both the total variation and the $d_2$ distances. For the error in total variation we may straightforwardly use Proposition \ref{t: dTV_two_PRM} for the approximation of two Poisson processes. For the $d_2$-error, which will be smaller than the $d_{TV}$, we may additionally use Lemma \ref{t: Delta_1_upgamma_d2} for an upper bound on $\Delta_1 \upgamma$, where $\upgamma$ is the solution of an adequate Stein equation. This bound, containing the factor $\bls(\As)^{-1/2}$ (or $\blsnew({\As})^{-1/2}$), serves in reducing the $d_2$-error. 
The error bounds given by Theorem \ref{t: approx_cont_nicer} will become small for large $n$ due to the 
pointwise convergence of the intensity functions $\ls_n(s,t)$ and $\acute{\lambda}^\star_n(s)$ to the intensity functions 
$\lsnew(s,t)$ and $\acute{\lambda}^\star_{\gamma,\delta}(s)$, respectively, as $n \to \infty$.
\begin{thm}\label{t: approx_cont_nicer}
With the notations from Sections \ref{s: MO_Geo}-\ref{s: MO_Geo_cont_nicer_int_fct}, we obtain, for any set $A^\star \in \mathcal{B}([-\log n, \infty)^2)$,
\begin{align*}
(i) \quad &d_{TV} \left(\mathrm{PRM}(\bl^\star), \mathrm{PRM}(\blsnew) \right) \le \int_{A^\star}| \bl^\star(d \mathbf{z})-\blsnew(d\mathbf{z})|,\\
(ii) \quad&d_2 \left(\mathrm{PRM}(\bl^\star), \mathrm{PRM}(\blsnew) \right) \\
&\le \left( 1 \wedge 1.65 \min \left\{ \bl^\star(A^\star)^{-1/2} \, , \, \blsnew(A^\star)^{-1/2}\right\}\right)
\int_{A^\star}| \bl^\star(d \mathbf{z})-\blsnew(d\mathbf{z})|.
\end{align*}
\end{thm}
\begin{proof}
(i) By Proposition \ref{t: MO_Geo_same_meas_rectangles} (ii), $\bls$ is finite. Moreover, $\blsnew$ is finite since integration of 
$\lsnew$ and $\acute{\lambda}^\star_{\gamma,\delta}$ over $[u^\star, \infty)^2$ 
gives
\begin{align*}
&\int_{u^\star}^\infty dt \int_{u^\star}^t ds\, \frac{\gamma(1+\delta)}{(\gade)^2}\,e^{-\frac{\gamma}{\gade}\,s}e^{-\frac{1 + \delta}{\gade}\,t}\\
&+ \int_{u^\star}^\infty ds \int_{u^\star}^s dt \, \frac{\delta (1 + \gamma)}{(\gade)^2}\,e^{-\frac{1+\gamma}{\gade}\,s} e^{-\frac{\delta}{\gade}\,t}
+ \int_{u^\star}^\infty ds \,\frac{1}{\gade}\, e^{-s} \\
=\,\,&\frac{\gamma}{\gade}\, e^{-u^\star} + \frac{\delta}{\gade}\, e^{-u^\star} + \frac{1}{\gade}\, e^{-u^\star}= e^{-u^\star},
\end{align*}
which equals $n$ for $u^\star = -\log n$. 
Proposition \ref{t: dTV_two_PRM} then immediately gives the result. \\
(ii) Using the same immigration-death process $Z$ and arguments as in the proof of Theorem \ref{t: approx_disc_cont}, we can show that for 
$\Xi^\star_{\gamma,\delta} \sim \mathrm{PRM}(\blsnew)$,
\begin{multline*}
 \mathbb{E}h(\Xi^\star_{\gamma,\delta})-\mathrm{PRM}(\bls)(h)\\ 
=\mathbb{E} \left\{  \int_{\As} [\upgamma(\Xi^\star_{\gamma,\delta} + \delta_{\bz}) - \upgamma(\Xi^\star_{\gamma,\delta})] 
 (\bls(d\bz) - \blsnew(d\bz)) \right\}.
\end{multline*}
Analogously to (\ref{p: summand2}) and (\ref{p: deltaone}), the integrand may be bounded by 
\begin{align*}
\Delta_1\upgamma \int_{\As}|\bls(d\bz)-\blsnew(d\bz)| 
\le s_2(h)\left( 1 \wedge \frac{1.65}{\sqrt{\bls(\As)}}\right) \int_{\As}|\bls(d\bz)-\blsnew(d\bz)|.
\end{align*}
Here, $1.65(\bls(\As))^{-\frac{1}{2}}$ may be replaced by 
$1.65(\blsnew(\As))^{-\frac{1}{2}}$ by going through the same arguments as before, but instead starting
with an immigration-death process over $\As$ with immigration intensity $\blsnew$, unit per-capita death rate, and equilibrium distribution 
$\mathrm{PRM}(\blsnew)$.
\end{proof}
We now again assume that the distributional parameters $p_{00n}, q_{1n}$ and $q_{2n}$ satisfy (\ref{d: conditions}) from Section \ref{s: MO_Geo_cont_nicer_int_fct} and apply Theorem \ref{t: approx_cont_nicer} to the case where $\As =\As_n = [u_n^\star, \infty)^2$.
We express the error bounds in terms of the threshold $u_n^{\star}$ and of the probability of simultaneous success $p_{11n}$.
\begin{prop}\label{t: MO_Geo_cont_to_nicer_appl}
Let $p_{11n} \in (0,1)$ and assume that $q_{1n}$, $q_{2n}$ and $p_{00n}$ satisfy (\ref{d: conditions}). For any choice of $u_n^\star \ge -\log n$, 
define $A^\star = [u_n^\star,\infty)^2$.  
With the notations from Sections \ref{s: MO_Geo}-\ref{s: MO_Geo_cont_nicer_int_fct},
\begin{align*}
(i) \,\,&d_{TV} \left(\mathrm{PRM}(\bl^\star), \mathrm{PRM}(\blsnew) \right) \le \frac{4(\gade)^2 p_{11n}}{[1-(\gade)p_{11n}]^3}\, e^{-u_n^\star}\,,\\
(ii) \,\, &d_2 \left(\mathrm{PRM}(\bl^\star), \mathrm{PRM}(\blsnew) \right)
\le \left( e^{-u_n^\star} \wedge 1.65 e^{-u_n^\star/2}\right) 
\frac{4(\gade)^2 p_{11n}}{[1-(\gade)p_{11n}]^3}\,.
\end{align*}
\end{prop}
\begin{proof}
For ease of notation we again omit the subscript $n$. 
(i) By Theorem \ref{t: approx_cont_nicer}, 
\begin{align*}
&d_{TV} \left(\mathrm{PRM}(\bl^\star), \mathrm{PRM}(\blsnew) \right) \\
\le\,\, &\int_{u^\star}^\infty  \int_{u^\star}^t  \left|\ls(s,t) -\lsnew(s,t)\right|dsdt
+ \int_{u^\star}^\infty  \int_{u^\star}^s  \left|\ls(s,t) -\lsnew(s,t)\right|dtds\\
&+ \int_{u^\star}^\infty \left| \acute{\lambda}^\star(s)-\acute{\lambda}^\star_{\gamma,\delta}(s)\right|ds.
\end{align*}
Define 
\[
h := h(p_{11}) := \frac{1+\delta}{\gade} - \frac{\log q_{2}}{\log p_{00}}\, \quad \text{and} \quad
g:= g(p_{11}) := \frac{1+\gamma}{\gade} -\frac{\log q_{1}}{\log p_{00}}\,.
\]
We first consider the case $s=t$. Note that, with definitions (\ref{d: ls}) and (\ref{d: lsnew}),
\begin{align*}
\acute{\lambda}^\star(s)= \frac{\log(p_{00}/q_1q_2)}{\log(1/p_{00})}\, e^{-s}
&= \left[\frac{\log q_1+\log q_2}{\log p_{00}} - \frac{2+ \gamma + \delta}{\gade} + \frac{1}{\gade}\right]e^{-s}\\
&= \left[\frac{1}{\gade} - h(p_{11}) - g(p_{11})\right]e^{-s},
\end{align*}
and that, since $h, g \ge 0$ by Lemma \ref{t: Lemma_cond} (i) and (ii), we thus have $\acute{\lambda}^\star(s) \le \acute{\lambda}^\star_{\gamma,\delta}(s)$.
Hence,
\[
\int_{u^\star}^\infty \left| \acute{\lambda}^\star(s)-\acute{\lambda}^\star_{\gamma,\delta}(s)\right|ds = \int_{u^\star}^\infty \left(h + g\right)e^{-s}ds 
 \le \frac{(\gamma+\delta)p_{11}}{1-(\gade)p_{11}}\, e^{-u^\star},
\]
again by Lemma \ref{t: Lemma_cond} (i) and (ii). For $s < t$, note that 
\begin{align*}
\ls(s,t)
&= \left[ \frac{\gamma}{\gade} + h\right] \left[ \frac{1+\delta}{\gade} - h\right]
e^{- \frac{\gamma}{\gade}\,s} e^{-\frac{1+\delta}{\gade}\,t}e^{h(t-s)}\\
& = \lsnew(s,t) e^{h(t-s)} + \left(\frac{1+ \delta - \gamma}{\gade} \, h - h^2\right)e^{-\frac{\gamma}{\gade}\, s} e^{-\frac{1+\delta}{\gade}\, t} e^{h(t-s)},
\end{align*}
where $\ls(s,t)$ and $ \lsnew(s,t)$ are defined by (\ref{d: ls}) and (\ref{d: lsnew}), respectively.
Thereby,
\[
\left| \ls(s,t) - \lsnew(s,t)e^{h(t-s)}\right| = \left| \frac{1+ \delta - \gamma}{\gade} \, h - h^2\right| 
e^{-\frac{\gamma}{\gade}\, s} e^{-\frac{1+\delta}{\gade}\, t} e^{h(t-s)},
\]
where $\left| \frac{1+ \delta - \gamma}{\gade} \, h - h^2\right|  \le h + h^2 $. 
Note that we have
\begin{equation}\label{p: MO_Geo_nicer_appl_0}
\left| \ls(s,t) - \lsnew(s,t)\right| \le \left| \ls(s,t) - \lsnew(s,t)e^{h(t-s)}\right| + \lsnew(s,t)\left| e^{h(t-s)} - 1\right|.
\end{equation}
We first compute the following integral:
\begin{align}\label{p: MO_Geo_d2_nicer_appl_integral}
\begin{split}
\int_{u^\star}^\infty \int_{u^\star}^t e^{-\frac{\gamma}{\gade}\, s-\frac{1+\delta}{\gade}\, t+ h(t-s)} dsdt
&= \int_{u^\star}^\infty \int_{u^\star}^t e^{-\frac{\log (p_{00}/q_2)}{\log p_{00}}\,s-\frac{\log q_2}{\log p_{00}}\,t}dsdt\\
&= \frac{\log p_{00}}{\log q_2}\, e^{-u^\star},
\end{split}
\end{align}
where, using $z \le -\log(1-z) \le \frac{z}{1-z}$ for all $|z| \le 1$, and (\ref{d: conditions}),
\begin{align}\label{MO_Geo_nicer_appl_2}
\begin{split} 
\frac{\log p_{00}}{\log q_2} &=\frac{-\log[1-(\gade)p_{11}]}{- \log[1-(1+\delta)p_{11}]} \le \frac{\gade}{(1+\delta)[1-(\gade)p_{11}]}\,\\
&\le \frac{\gade}{1+ \delta}\left\{ 1+ \frac{(\gade) p_{11}}{[1-(\gade)p_{11}]^2}\right\}.
\end{split}
\end{align}
Moreover, by Lemma \ref{t: Lemma_cond} (i), 
\[
h+ h^2 \le \frac{\gamma p_{11}}{1-(\gade)p_{11}} +  \left[\frac{\gamma p_{11}}{1-(\gade)p_{11}} \right]^2 \le \frac{2\gamma p_{11}}{[1-(\gade)p_{11}]^2},
\]
since $\gamma p_{11} = p_{10} < 1$, and therefore $(\gamma p_{11})^2 \le \gamma p_{11}$. 
Then,
\begin{align}\label{p: MO_Geo_nicer_appl_4}
\begin{split}
\int_{u^\star}^\infty \int_{u^\star}^t  \left| \ls(s,t) - \lsnew(s,t)e^{h(t-s)}\right|dsdt 
\le \frac{2 \gamma (\gade)p_{11}}{(1+\delta)[1-(\gade)p_{11}]^3}\, e^{-u^\star},
\end{split}
\end{align}
which gives a bound for the integral of the first error term in (\ref{p: MO_Geo_nicer_appl_0}). For the second error term in (\ref{p: MO_Geo_nicer_appl_0}), note first that
$|e^{h(t-s)}-1| = e^{h(t-s)} -1$, since $h \ge 0$ and $t > s$. By (\ref{p: MO_Geo_d2_nicer_appl_integral}) and (\ref{MO_Geo_nicer_appl_2}), 
and with definition (\ref{d: lsnew}) of $\lsnew(s,t)$, we obtain 
\begin{align}\label{t: MO_Geo_nicer_appl_2}
\begin{split}
\int_{u^\star}^\infty \int_{u^\star}^t  \lsnew(s,t)e^{h(t-s)} &= \frac{\gamma(1+\delta)\log p_{00}}{(\gade)^2 \log q_2}\, e^{-u^\star}\\
&\le \frac{\gamma}{ \gade }\left\{ 1+ \frac{(\gade) p_{11}}{[1-(\gade)p_{11}]^2}\right\}\, e^{-u^\star},
\end{split}
\end{align}
whereas
\begin{equation}\label{t: MO_Geo_nicer_appl_3}
\int_{u^\star}^\infty \int_{u^\star}^t \lsnew(s,t)dsdt = \frac{\gamma}{\gade}\, e^{-u^\star}.
\end{equation}
By (\ref{t: MO_Geo_nicer_appl_2}) and (\ref{t: MO_Geo_nicer_appl_3}), we may thus bound the integral of the second error term in (\ref{p: MO_Geo_nicer_appl_0})
as follows:
\begin{align}\label{p: MO_Geo_nicer_appl_5}
\begin{split}
\int_{u^\star}^\infty \int_{u^\star}^t \lsnew(s,t)\left| e^{h(t-s)} - 1\right|dsdt
&\le \frac{\gamma p_{11}}{[1-(\gade)p_{11}]^2}\,e^{-u^\star}.
\end{split}
\end{align}
Hence, for $s < t$, (\ref{p: MO_Geo_nicer_appl_0}), (\ref{p: MO_Geo_nicer_appl_4}) and (\ref{p: MO_Geo_nicer_appl_5}) give
\begin{multline*}
\int_{u^\star}^\infty \int_{u^\star}^t  \left| \ls(s,t) - \lsnew(s,t)\right|dsdt\\
\le  \frac{\gamma p_{11}}{[1-(\gade)p_{11}]^3}\, \left\{ \frac{2(\gade)}{1+ \delta} + 1 \right\}e^{-u^\star}.
\end{multline*}
By proceeding analogously for $s>t$, we obtain
\begin{multline*}
\int_{u^\star}^\infty \int_{u^\star}^s  \left| \ls(s,t) - \lsnew(s,t)\right|dtds\\
\le \frac{\delta p_{11}}{[1-(\gade)p_{11}]^3}\, \left\{ \frac{2(\gade)}{1+ \gamma} + 1 \right\}e^{-u^\star}.
\end{multline*}
The sum of the bounds for the three cases $s=t$, $s<t$ and $s>t$ yields the overall bound
\begin{align*}
&\int_{\As} \left| \bls(d\bz) - \blsnew(d\bz)\right| \\
&\le \frac{2p_{11}}{[1-(\gade)p_{11}]^3}\left\{ \frac{\gamma(\gade)}{1+\delta} + \frac{\delta(\gade)}{1+\gamma} + \gamma + \delta\right\}e^{-u^\star}\\
&\le \frac{4(\gade)^2 p_{11}}{[1-(\gade)p_{11}]^3}\, e^{-u^\star},
\end{align*}
where we used $(1+\gamma)^{-1}, (1+\delta)^{-1} < 1$, and $\gamma + \delta \le \gade \le (\gade)^2$ for the second inequality. \\
(ii) Direct computations give $\bls(\As)= e^{-u^\star} = \blsnew(\As)$. Theorem \ref{t: approx_cont_nicer} (ii), together with the bound from (i), then
immediately gives the result. 
\end{proof}
\noindent The error bounds established in Proposition \ref{t: MO_Geo_cont_to_nicer_appl} are similar to the error bound from Proposition \ref{t: MO_Geo_disc_to_cont_appl}. As before, $p_{11n}$ needs to converge to $0$ fast enough to make up for the factor $e^{-u_n^\star}$ which increases the size of the error as soon as $u_n^\star < 0$. And since $u_n^\star \ge 0$ gives $1$ or no points in $\As$, the mean number of points in $\As$ being given by $e^{-u_n^\star}$ for either process, we would certainly want the threshold $u_n^\star$ to be negative. 

The biggest difference between the $d_2$-bounds from Propositions \ref{t: MO_Geo_disc_to_cont_appl} and  \ref{t: MO_Geo_cont_to_nicer_appl} is that 
the former contains the multiplicative factor $[1-(\gade)p_{11n}]^{-2}$ and the latter the bigger factor $[1-(\gade)p_{11n}]^{-3}$. However, 
since we need $p_{11n} \to 0$ as $n \to \infty$, we will have $(\gade)p_{11n} \le 1/2$ for all $n$ large enough. 
Then $[1-(\gade)p_{11n}]^{-3} \le 2 [1-(\gade)p_{11n}]^{-2}$ so that 
both error bounds will be of the same rate. Hence, for large enough $n$, the approximation by a further Poisson process does not add an error of a bigger size 
than the one that arises from the approximation by only $\mathrm{PRM}(\bls)$.
\subsection{Final bound in the $d_2$-distance}\label{s: MO_Geo_final_bound_d2}
The following corollary summarises the results from Sections \ref{s: MO_Geo_dTV}, \ref{s: MO_Geo_d2} and \ref{s: MO_Geo_nicer}. 
It gives an estimate for the error in the $d_2$-distance of the approximation of the law of an MPPE $\Xi^\star_{A^\star}$ with \iid Marshall-Olkin geometric marks, living on a lattice of points contained in $\As \cap[-\log n, \infty)^2$, by the law of a Poisson process with a continuous intensity measure $\blsnew$ over $\As \cap[-\log n, \infty)^2$, where $\As = [u^\star, \infty)^2$ for some choice of threshold $u^\star \ge -\log n$.
\begin{cor}\label{t: MO_Geo_final_bound_d2} Let $p_{11n} \in (0,1)$ and assume that $q_{1n}$, $q_{2n}$ and $p_{00n}$ satisfy (\ref{d: conditions}). For any choice of $u_n^\star \ge -\log n$, 
define $A^\star = [u_n^\star,\infty)^2$.  
With the notations from Sections \ref{s: MO_Geo}-\ref{s: MO_Geo_cont_nicer_int_fct},
\begin{multline*}
d_2\left(\mathcal{L}(\Xi^\star_{\As}), \mathrm{PRM}(\blsnew)\right) \\
\le \frac{e^{-u^\star_n}}{n} + \frac{(\gade)^2p_{11n}}{[1-(\gade)p_{11n}]^3} \left\{ 2\sqrt{2} + 7\left(  e^{-u_n^\star} \wedge 1.65 e^{-u_n^\star/2}\right) \right\}.
\end{multline*}
\begin{proof}
We have 
\begin{align*}
&d_2\left(\mathcal{L}(\Xi^\star_{\As}), \mathrm{PRM}(\blsnew)\right) \\
&\le  d_2\left(\mathcal{L}(\Xi^\star_{\As}), \mathrm{PRM}(\bpis)\right)
+    d_2\left(\mathrm{PRM}(\bpis), \mathrm{PRM}(\bls)\right)\\
&\phantom{blaaaaaaaaaaaaaaaaaaaaaaaaaaaaaaaaaaaa}+    d_2\left(\mathrm{PRM}(\bls), \mathrm{PRM}(\blsnew)\right).
\end{align*}
By (\ref{t: d2_smaller_dTV}) and Theorem \ref{t: MO_Geo_dTV_lattice},
\[
d_2\left(\mathcal{L}(\Xi^\star_{\As}), \mathrm{PRM}(\bpis)\right) \le d_{TV}\left(\mathcal{L}(\Xi^\star_{\As}), \mathrm{PRM}(\bpis)\right) \le \frac{e^{-u_n^\star}}{n}\,.
\]
Furthermore, with the results from Propositions \ref{t: MO_Geo_disc_to_cont_appl} and \ref{t: MO_Geo_cont_to_nicer_appl}, and using $(\gade)\le(\gade)^2$ and 
$[1-(\gade)p_{11}]^{-2} \le [1-(\gade)p_{11}]^{-3}$, we obtain
\begin{multline*}
d_2\left(\mathrm{PRM}(\bpis), \mathrm{PRM}(\bls)\right) + d_2\left(\mathrm{PRM}(\bls), \mathrm{PRM}(\blsnew)\right) \\
\le \frac{(\gade)^2p_{11n}}{[1-(\gade)p_{11n}]^3} \left\{ 2\sqrt{2} + 7 \left( e^{-{u^\star_n}} \wedge 1.65 e^{-{u^\star_n}/2}\right)\right\}.
\end{multline*}
\end{proof}
\end{cor}
\noindent By far the smallest component of the error estimate from Corollary \ref{t: MO_Geo_final_bound_d2} is given by $e^{-u_n^\star}/n$, the error arising from approximating $\mathcal{L}(\Xi^\star_{A^\star})$ by $\mathrm{PRM}(\mathbb{E}\Xi^\star_{A^\star})$, which lives on the lattice $\As \cap E^\star$ just as $\Xi^\star_{\As}$. This part of the error corresponds exactly to the overall error estimate that we obtained in Section \ref{s: MO_Exp_joint} for the approximation of an MPPE $\Xi^\star_{A^\star}$ with Marshall-Olkin exponential marks by a Poisson process with mean measure $\mathbb{E}\Xi^\star_{A^\star}$. Yet the Marshall-Olkin exponential is a continuous distribution and the mean measure $\mathbb{E}(\Xi^\star_{A^\star})$ is thereby also continuous.
As for MPPE's with univariate geometric marks (see Section \ref{s: MPPE_geo}), a far bigger error emerges for the MPPE with Marshall-Olkin geometric marks when going from the Poisson process on the lattice to a Poisson process on $\As \cap [-\log n, \infty)^2$ with continuous intensity. This error can only be small if the probability of simultaneous success of the Marshall-Olkin geometric distribution, $p_{11}$, and thereby also the marginal success probabilities $1-q_{1n}$ and $1-q_{2n}$, tend to zero as $n \to \infty$ at a rate fast enough to compensate for the factor $e^{-u_n^\star}$, the (rough) number of points expected in $\As_n$ for each of the processes. For instance, for $\As_n = [-\log \log n, \infty)^2$ and $p_{11n}=1/n$, we expect $\log n$ joint threshold exceedances, and obtain
\begin{align*}
&d_2\left(\mathcal{L}(\Xi^\star_{\As}), \mathrm{PRM}(\blsnew)\right) \\
&\le \frac{\log n}{n} + \frac{(\gade)^2}{n[1-(\gade)/n]^3} \left\{ 2\sqrt{2} + 7\left(  \log n \wedge 1.65 \sqrt{\log n}\right) \right\}\\
& \le \frac{C \log n}{n},
\end{align*}
where $C$ is some constant. With the (very strong) condition $p_{11n}=1/n$, we thus obtain an error of the same size as the error that we obtain when approximating $\mathcal{L}(\Xi^\star_{\As})$ only by $\mathrm{PRM}(\mathbb{E}\Xi^\star_{A^\star})$.
\bibliographystyle{apalike}
\bibliography{Bibliography}

@book {Ambrosio_et_al:2005,
    AUTHOR = {Ambrosio, Luigi and Gigli, Nicola and Savar{\'e}, Giuseppe},
     TITLE = {Gradient flows in metric spaces and in the space of
              probability measures},
    SERIES = {Lectures in Mathematics ETH Z\"urich},
 PUBLISHER = {Birkh\"auser Verlag},
   ADDRESS = {Basel},
      YEAR = {2005},
     PAGES = {viii+333}}

@article{Anderson_et_al:1997,
    AUTHOR = {Anderson, Clive W. and Coles, Stuart G. and H\"usler, J\"urg},
    TITLE = {Maxima of {P}oisson-like variables and related triangular arrays},
    FJOURNAL = {The Annals of Applied Probability},
    JOURNAL = {Ann. Appl. Probab.},
    VOLUME = {7},
    YEAR = {1997},
    PAGES = {953--971}}

@article {Apostol:2000,
    AUTHOR = {Apostol, Tom M.},
     TITLE = {Calculating higher derivatives of inverses},
   JOURNAL = {Amer. Math. Monthly},
  FJOURNAL = {The American Mathematical Monthly},
    VOLUME = {107},
      YEAR = {2000},
    NUMBER = {8},
     PAGES = {738--741}}

@article {Barbour/Eagleson:1983,
    AUTHOR = {Barbour, A. D. and Eagleson, G. K.},
     TITLE = {Poisson approximation for some statistics based on
              exchangeable trials},
   JOURNAL = {Adv. in Appl. Probab.},
  FJOURNAL = {Advances in Applied Probability},
    VOLUME = {15},
      YEAR = {1983},
    NUMBER = {3},
     PAGES = {585--600},
      ISSN = {0001-8678},
     CODEN = {AAPBBD},
   MRCLASS = {60F05 (60G09)},
  MRNUMBER = {706618 (85c:60021)},
MRREVIEWER = {Olav Kallenberg},
       DOI = {10.2307/1426620},
       URL = {http://dx.doi.org/10.2307/1426620},
}

@article {Barbour/Hall:1984,
    AUTHOR = {Barbour, A. D. and Hall, Peter},
     TITLE = {On the rate of {P}oisson convergence},
   JOURNAL = {Math. Proc. Cambridge Philos. Soc.},
  FJOURNAL = {Mathematical Proceedings of the Cambridge Philosophical
              Society},
    VOLUME = {95},
      YEAR = {1984},
    NUMBER = {3},
     PAGES = {473--480},
      ISSN = {0305-0041},
     CODEN = {MPCPCO},
   MRCLASS = {60F05 (62E20)},
  MRNUMBER = {755837 (86m:60052)},
MRREVIEWER = {G. K. Eagleson},
       DOI = {10.1017/S0305004100061806},
       URL = {http://dx.doi.org/10.1017/S0305004100061806},
}

@article {Barbour:1988,
    AUTHOR = {Barbour, A. D.},
     TITLE = {Stein's method and {P}oisson process convergence},
      NOTE = {A celebration of applied probability},
   JOURNAL = {J. Appl. Probab.},
  FJOURNAL = {Journal of Applied Probability},
      YEAR = {1988},
    NUMBER = {Special Vol. 25A},
     PAGES = {175--184},
      ISSN = {0021-9002},
     CODEN = {JPRBAM},
   MRCLASS = {60F17 (60G55)},
  MRNUMBER = {974580 (90a:60064)},
MRREVIEWER = {Jan Grandell},
}

@article {Barbour/Brown:1992,
    AUTHOR = {Barbour, A. D. and Brown, T. C.},
     TITLE = {Stein's method and point process approximation},
   JOURNAL = {Stochastic Process. Appl.},
  FJOURNAL = {Stochastic Processes and their Applications},
    VOLUME = {43},
      YEAR = {1992},
    NUMBER = {1},
     PAGES = {9--31}}

@book {Barbour_et_al.:1992,
    AUTHOR = {Barbour, A. D. and Holst, Lars and Janson, Svante},
     TITLE = {Poisson approximation},  
 PUBLISHER = {The Clarendon Press Oxford University Press},
      YEAR = {1992}}

@incollection {Barbour:1997,
    AUTHOR = {Barbour, A. D.},
 BOOKTITLE = {Encyclopedia of statistical sciences, Update Volume 1},
     TITLE = {Stein's method},
     PAGES = {513--521},
 PUBLISHER = {Wiley},
   ADDRESS = {New York},
      YEAR = {1997}}

@book {Beirlant_et_al:2004,
    AUTHOR = {Beirlant, Jan and Goegebeur, Yuri and Teugels, Jozef L. and
              Segers, Johan},
     TITLE = {Statistics of {E}xtremes},
 PUBLISHER = {Wiley},
   ADDRESS = {Chichester},
      YEAR = {2004}}

@book {Bingham_et_al:1987,
    AUTHOR = {Bingham, N. H. and Goldie, C. M. and Teugels, J. L.},
     TITLE = {Regular variation},
    SERIES = {Encyclopedia of Mathematics and its Applications},
    VOLUME = {27},
 PUBLISHER = {Cambridge University Press},
   ADDRESS = {Cambridge},
      YEAR = {1987},
     PAGES = {xx+491},
      ISBN = {0-521-30787-2},
   MRCLASS = {26A12 (11K65 11N60 30-02 40E05 60-02 60Fxx)},
  MRNUMBER = {898871 (88i:26004)},
MRREVIEWER = {R. A. Maller}}

@article{Campbell/Tsokos:1973,
    AUTHOR = {Campbell, Janet W. and Tsokos, Chris P.},
    TITLE = {The asymptotic distribution of maxima in bivariate samples},
    JOURNAL = {J. Amer. Statist. Assoc.},
    VOLUME = {68},
    YEAR = {1973},
    PAGES = {734--739}}

@article{Charpentier/Segers:2009,
    AUTHOR = {Charpentier, Arthur and Segers, Johan},
    TITLE = {Tails of multivariate Archimedean copulas},
    JOURNAL = {Journal of Multivariate Analysis},
    VOLUME = {100},
    YEAR = {2009},
    PAGES = {1521--1537}}


@article {Chen:1975a,
    AUTHOR = {Chen, Louis H. Y.},
     TITLE = {An approximation theorem for sums of certain randomly selected indicators},
   JOURNAL = {Z. Wahrscheinlichkeitstheorie und Verw. Gebiete},
    VOLUME = {33},
      YEAR = {1975},
    NUMBER = {1},
     PAGES = {69--74}}

@article {Chen:1975b,
    AUTHOR = {Chen, Louis H. Y.},
     TITLE = {Poisson approximation for dependent trials},
   JOURNAL = {Ann. Probability},
    VOLUME = {3},
      YEAR = {1975},
    NUMBER = {3},
     PAGES = {534--545}}

@incollection {Chen:1998,
    AUTHOR = {Chen, Louis H. Y.},
     TITLE = {Stein's method: some perspectives with applications},
 BOOKTITLE = {Probability Towards 2000},
    EDITOR = {Accardi, L. and Heyde, C. C.},
      YEAR = {1998},
    SERIES = {Lecture Notes in Statistics},
    VOLUME = {128},
     PAGES = {97--122},
 PUBLISHER = {Springer},
   ADDRESS = {New York}}

@article{Coles/Pauli:2001,
    AUTHOR = {Coles, Stuart G. and Pauli, Francesco},
    TITLE = {Extremal limit laws for a class of bivariate {P}oisson vectors},
    FJOURNAL = {Statistics \& Probability Letters},
    JOURNAL = {Statist. Probab. Lett.},
    VOLUME = {54},
    YEAR = {2001},
    PAGES = {373--379}}

@book {Daley/VereJones:2008,
    AUTHOR = {Daley, D. J. and Vere-Jones, D.},
     TITLE = {An Introduction to the Theory of Point Processes, Vol. II:
General Theory and Structure},
    SERIES = {Probability and its Applications},
 PUBLISHER = {Springer},
   EDITION = {2nd},
   ADDRESS = {New York},
      YEAR = {2008}}

@book {Daley/VereJones:2003,
    AUTHOR = {Daley, D. J. and Vere-Jones, D.},
     TITLE = {An Introduction to the Theory of Point Processes, Vol. I:
Elementary Theory and Methods},
    SERIES = {Probability and its Applications},
 PUBLISHER = {Springer},
   ADDRESS = {New York},
   EDITION = {2nd},
      YEAR = {2003}}

@article{Deheuvels:1978,
    AUTHOR = {Deheuvels, Paul},
    FJOURNAL = {Publications de l'Institut de statistique des Universit\'es de Paris},
    JOURNAL = {Publ.\ Inst.\ Statist.\ Univ.\ Paris},
    TITLE = {Ca\-rac\-t\'e\-ri\-sa\-tion com-pl\`e\-te des lois ex\-tr\^e\-mes mul\-ti\-va\-ri\'ees
    et de la con\-ver\-gen\-ce aux ty\-pes ex\-tr\^e\-mes},
    VOLUME = {22},
    YEAR = {1978},
    PAGES = {1--36}}

@book{Embrechts_et_al:1997,
    AUTHOR = {Embrechts, Paul and Kl\"uppelberg, Claudia and Mikosch, Thomas},
    TITLE = {Modelling {E}xtremal {E}vents for {I}nsurance and {F}inance},
    PUBLISHER = {Springer},
    ADDRESS = {New York},
    YEAR = {1997}}

@incollection {Erhardsson:2005,
    AUTHOR = {Erhardsson, Torkel},
     TITLE = {Stein's method for {P}oisson and compound {P}oisson approximation},
     PAGES = {61--113},
 BOOKTITLE = {An Introduction to {S}tein's method},
    SERIES = {Lecture Notes Series, Institute for Mathematical Sciences,
              National University of Singapore},
    VOLUME = {4},
    EDITOR = {Barbour, A. D. and Chen, Louis H. Y.},
 PUBLISHER = {Singapore University Press},
   ADDRESS = {Singapore},
      YEAR = {2005}}

@article {Feidt_et_al:2010,
   author = {Feidt, A. and Genest, Chr. and Ne{\v{s}}lehov\'a, J.},
   title = {Asymptotics of joint maxima for discontinuous random variables},
   journal = {Extremes},
   publisher = {Springer U.S.},
   issn = {1386-1999},
   keyword = {Mathematics and Statistics},
   pages = {35-53},
   volume = {13},
   issue = {1},
   year = {2010}}

@article{Fisher/Tippett:1928,
    AUTHOR = {Fisher, R. A. and Tippett, L. H. C.},
    FJOURNAL = {Proceedings of the Cambridge Philosophical Society},
    JOURNAL = {Proceedings of the Cambridge Philosophical Society},
    TITLE = {Limiting forms of the frequency distributions of the largest or smallest member of a sample},
    VOLUME = {24},
    YEAR = {1928},
    PAGES = {180--190}}

@incollection{Fougeres:2004,
    AUTHOR = {Foug\`eres, Anne-Laure},
    BOOKTITLE = {Extreme Values in Finance, Telecommunications, and the Environment},
    EDITOR = {Finkenst\"adt, B. and Rootz\'en, H.},
    PUBLISHER = {Chapman \& Hall},
    ADDRESS = {London},
    SERIES = {Monographs on Statistics and Applied Probability 99},
    TITLE = {Multivariate Extremes},
    YEAR = {2004}}

@book{Galambos:1987,
    ADDRESS = {Melbourne, FL},
    AUTHOR = {Galambos, Janos},
    PUBLISHER = {Robert E. Krieger Publishing Co. Inc.},
    TITLE = {The Asymptotic Theory of Extreme Order Statistics, {\rm Second Edition}},
    YEAR = {1987}}

@article{Genest/MacKay:1986,
    AUTHOR = {Genest, Christian and MacKay, R. Jock},
    TITLE = {Copules archim\'ediennes et familles de lois bidimensionnelles dont les marges sont donn\'ees},
    FJOURNAL = {Canadian Journal of Statistics},
    VOLUME = {14},
    NUMBER = {2},
    YEAR = {1986},
    PAGES = {145--159}}

@article{Genest/Neslehova:2007,
    AUTHOR = {Genest, Christian and Ne{\v{s}}lehov\'a, Johanna},
    TITLE = {A primer on copulas for count data},
    JOURNAL = {Astin Bull.},
    VOLUME = {37},
    YEAR = {2007},
    PAGES = {475--515}}

@article {Gumbel:1958,
    AUTHOR = {Gumbel, Emile J.},
     TITLE = {Distributions \`a plusieurs variables dont les marges sont
              donn\'ees},
   JOURNAL = {C. R. Acad. Sci. Paris},
    VOLUME = {246},
      YEAR = {1958},
     PAGES = {2717--2719}}

@article {Gumbel:1965,
    AUTHOR = {Gumbel, E. J.},
     TITLE = {Two systems of bivariate extremal distributions. },
   JOURNAL = {Bull. Inst. Internat. Statist},
    VOLUME = {41},
      YEAR = {1965},
     PAGES = {749--763}}

@article {Hall:1979,
    AUTHOR = {Hall, Peter},
     TITLE = {On the rate of convergence of normal extremes},
   JOURNAL = {J. Appl. Probab.},
  FJOURNAL = {Journal of Applied Probability},
    VOLUME = {16},
      YEAR = {1979},
    NUMBER = {2},
     PAGES = {433--439},
      ISSN = {0021-9002},
     CODEN = {JPRBAM},
   MRCLASS = {60F05 (62G30)},
  MRNUMBER = {531778 (80d:60025)},
MRREVIEWER = {N. R. Mohan},
}

@article{Hawkes:1972,
     jstor_articletype = {research-article},
     title = {A Bivariate Exponential Distribution with Applications to Reliability},
     author = {Hawkes, Alan G.},
     journal = {Journal of the Royal Statistical Society. Series B (Methodological)},
     jstor_issuetitle = {},
     volume = {34},
     number = {1},
     jstor_formatteddate = {1972},
     pages = {pp. 129-131},
     url = {http://www.jstor.org/stable/2985057},
     ISSN = {00359246},
     abstract = {In a recent paper Downton introduced a bivariate exponential distribution in a reliability context. This note proposes a more general distribution which may have some advantages in practice.},
     language = {English},
     year = {1972},
     publisher = {Blackwell Publishing for the Royal Statistical Society},
     copyright = {Copyright © 1972 Royal Statistical Society},
    }

@article {Hsing:1989,
    AUTHOR = {Hsing, Tailen},
     TITLE = {Extreme value theory for multivariate stationary sequences},
   JOURNAL = {J. Multivariate Anal.},
  FJOURNAL = {Journal of Multivariate Analysis},
    VOLUME = {29},
      YEAR = {1989},
     PAGES = {274--291}}

@article{Joe:1993,
    AUTHOR = {Joe, Harry},
    TITLE = {Multivariate dependence measures and data analysis},
    FJOURNAL = {Computational Statistics \& Data Analysis},
    JOURNAL = {Comput. Statist. Data Anal.},
    VOLUME = {16},
    YEAR = {1993},
    PAGES = {279--297}}

@book {Joe:1997,
    AUTHOR = {Joe, Harry},
     TITLE = {Multivariate models and dependence concepts},
    SERIES = {Monographs on Statistics and Applied Probability},
    VOLUME = {73},
 PUBLISHER = {Chapman \& Hall},
   ADDRESS = {London},
      YEAR = {1997},
     PAGES = {xviii+399}}

@book {Kallenberg:1983,
    AUTHOR = {Kallenberg, Olav},
     TITLE = {Random measures},
   EDITION = {4th},
 PUBLISHER = {Akademie-Verlag},
   ADDRESS = {Berlin},
      YEAR = {1986},
     PAGES = {187},
      ISBN = {0-12-394960-2},
   MRCLASS = {60G57},
  MRNUMBER = {854102 (87k:60137)}}

@article {Kerstan:1964,
    AUTHOR = {Kerstan, Johannes},
     TITLE = {Verallgemeinerung eines {S}atzes von {P}rochorow und {L}e
              {C}am},
   JOURNAL = {Zeitschrift f\"ur Wahrscheinlichkeitstheorie und Verwandte Gebiete},
    VOLUME = {2},
      YEAR = {1964},
     PAGES = {173--179},
   MRCLASS = {60.30},
  MRNUMBER = {0165555 (29 \#2835)},
MRREVIEWER = {J. E. Cigler}}

@book{Kotz:2000,
    AUTHOR = {Kotz, S. and Balakrishnan, N. and Johnson, N. L.},
    TITLE = {Continuous Multivariate Distributions, {\rm Second Edition}},
    PUBLISHER = {Wiley},
    ADDRESS = {New York},
    YEAR = {2000}}

@book{Leadbetter_et_al:1983,
    AUTHOR = {Leadbetter, M. Ross and Lindgren, G. and Rootz\'en, Holger},
    TITLE = {Extremes and Related Properties of Random Sequences and Processes},
    PUBLISHER = {Springer},
    ADDRESS = {New York},
    YEAR = {1983}}

@article {LeCam:1960,
    AUTHOR = {Le Cam, Lucien},
     TITLE = {An approximation theorem for the {P}oisson binomial
              distribution},
   JOURNAL = {Pacific J. Math.},
  FJOURNAL = {Pacific Journal of Mathematics},
    VOLUME = {10},
      YEAR = {1960},
     PAGES = {1181--1197},
      ISSN = {0030-8730},
   MRCLASS = {62.15 (60.30)},
  MRNUMBER = {0142174 (25 \#5567)},
MRREVIEWER = {F. L. Spitzer},
}

@misc{Lindner/Szimayer:2004,
    AUTHOR = {Lindner, A. and Szimayer, A.},
    ADDRESS = {www-m4.mathematik.tu-muenchen.de/m4/pers/lindner/},
    TITLE = {A limit theorem for copulas},
    HOWPUBLISHED = {Preprint, Technische Universit\"at M\"unchen},
    YEAR = 2004}}

@book {Lindvall:2002,
    AUTHOR = {Lindvall, Torgny},
     TITLE = {Lectures on the coupling method},
      NOTE = {Corrected reprint of the 1992 original},
 PUBLISHER = {Dover Publications Inc.},
   ADDRESS = {Mineola, NY},
      YEAR = {2002},
     PAGES = {xiv+257}}

@incollection{Marshall:1996,
    AUTHOR = {Marshall, Albert W.},
    TITLE = {Copulas, marginals, and joint distributions},
    BOOKTITLE = {Distributions with Fixed Marginals and Related Topics
    (Seattle, WA, 1993)},
    PUBLISHER = {Inst. Math. Statist.},
    ADDRESS = {Hayward, CA},
    SERIES = {IMS Lecture Notes Monogr. Ser.},
    VOLUME = {28},
    YEAR = {1996},
    PAGES = {213--222}}

@article {Marshall/Olkin:1967a,
    AUTHOR = {Marshall, Albert W. and Olkin, Ingram},
     TITLE = {A generalized bivariate exponential distribution},
   JOURNAL = {J. Appl. Probability},
  FJOURNAL = {Journal of Applied Probability},
    VOLUME = {4},
      YEAR = {1967},
     PAGES = {291--302},
      ISSN = {0021-9002},
   MRCLASS = {60.20},
  MRNUMBER = {0214111 (35 \#4962)},
MRREVIEWER = {H. Teicher},
}

@article{Marshall/Olkin:1967,
     jstor_articletype = {research-article},
     title = {A Multivariate Exponential Distribution},
     author = {Marshall, Albert W. and Olkin, Ingram},
     journal = {Journal of the American Statistical Association},
     jstor_issuetitle = {},
     volume = {62},
     number = {317},
     jstor_formatteddate = {Mar., 1967},
     pages = {pp. 30-44},
     year = {1967},
     publisher = {American Statistical Association},
     copyright = {Copyright © 1967 American Statistical Association},
    }

@article{Marshall/Olkin:1985,
    AUTHOR = {Marshall, Albert W. and Olkin, Ingram},
    TITLE = {A family of bivariate distributions generated by the bivariate {B}ernoulli distribution},
    FJOURNAL = {Journal of the American Statistical Association},
    JOURNAL = {J. Amer. Statist. Assoc.},
    VOLUME = {80},
    YEAR = {1985},
    PAGES = {332--338}}

@book{McNeil_et_al:2005,
    Address = {Princeton, NJ},
    Author = {McNeil, Alexander J. and Frey, R\"udiger and Embrechts, Paul},
    Publisher = {Princeton University Press},
    Title = {Quantitative {R}isk {M}anagement: {C}oncepts, {T}echniques, and {T}ools},
    Year = {2005}}

@article{Michel:1988,
    AUTHOR = {Michel, R.},
    TITLE = {An improved error bound for the compound Poisson approximation of a nearly homogeneous portfolio},
    JOURNAL = {Astin Bull.},
    FJOURNAL = {ASTIN Bulletin},
    VOLUME = {17},
    YEAR = {1988},
    PAGES = {165--169}}

@article{Mitov_et_al:2003,
    AUTHOR = {Mitov, Kosto V. and Pakes, Anthony G. and Yanev, George P.},
    TITLE = {Extremes of geometric variables with applications to branching processes},
    FJOURNAL = {Statistics \& Probability Letters},
    JOURNAL = {Statist. Probab. Lett.},
    VOLUME = {65},
    YEAR = {2003},
    PAGES = {379--388}}

@article{Mitov/Nadarajah:2005,
    AUTHOR = {Mitov, Kosto and Nadarajah, Saralees},
    TITLE = {Limit distributions for the bivariate geometric maxima},
    FJOURNAL = {Extremes. Statistical Theory and Applications in Science, Engineering and Economics},
    JOURNAL = {Extremes},
    VOLUME = {8},
    YEAR = {2005},
    PAGES = {357--370}}

@article{Nadarajah/Mitov:2002,
    AUTHOR = {Nadarajah, Saralees and Mitov, Kosto},
    TITLE = {Asymptotics of maxima of discrete random variables},
    FJOURNAL = {Extremes. Statistical Theory and Applications in Science, Engineering and Economics},
    JOURNAL = {Extremes},
    VOLUME = {5},
    YEAR = {2002},
    PAGES = {287--294},}

@article{Nadarajah/Mitov:2004,
    AUTHOR = {Nadarajah, S. and Mitov, K.},
    TITLE = {Extremal limit laws for discrete random variables},
    FJOURNAL = {Journal of Mathematical Sciences (New York)},
    JOURNAL = {J. Math. Sci. (N. Y.)},
    VOLUME = {122},
    YEAR = {2004},
    PAGES = {3404--3415}}

@book{Nelsen:2006,
    AUTHOR = {Nelsen, R. B.},
    TITLE = {An {I}ntroduction to {C}opulas, {\rm Second Edition}},
    PUBLISHER = {Springer},
    ADDRESS = {New York},
    YEAR = {2006}}

@inproceedings{Pickands:1981,
    AUTHOR = {Pickands, III, James},
    BOOKTITLE = {Proceedings of the 43rd Session of the International Statistical Institute, Vol.\ 2 (Buenos Aires, 1981)},
    FJOURNAL = {Bulletin de l'Institut International de Statistique},
    JOURNAL = {Bull. Inst. Internat. Statist.},
    NOTE = {With a discussion},
    NUMBER = {2},
    TITLE = {Multivariate extreme value distributions},
    VOLUME = {49},
    YEAR = {1981},
    PAGES = {859--878, 894--902}}

@article{Preston:1975,
    AUTHOR = {Preston, Chris},
     TITLE = {Spatial birth-and-death processes},
   JOURNAL = {Bull. ISI},
    VOLUME = {46},
      YEAR = {1975},
    NUMBER = {2},
     PAGES = {371--391}}

@book{Resnick:1987,
    AUTHOR = {Resnick, Sidney I.},
    TITLE = {Extreme Values, Regular Variation, and Point Processes},
    PUBLISHER = {Springer},
    ADDRESS = {New York},
    YEAR = {1987}}

@article {Serfling:1975,
    AUTHOR = {Serfling, R. J.},
     TITLE = {A general {P}oisson approximation theorem},
   JOURNAL = {Annals of Probability},
    VOLUME = {3},
      YEAR = {1975},
    NUMBER = {4},
     PAGES = {726--731},
   MRCLASS = {60F05},
  MRNUMBER = {0380946 (52 \#1843)},
MRREVIEWER = {M. Iosifescu}}

@article {Sklar:1959,
    AUTHOR = {Sklar, M.},
     TITLE = {Fonctions de r\'epartition \`a {$n$} dimensions et leurs
              marges},
   JOURNAL = {Publ. Inst. Statist. Univ. Paris},
    VOLUME = {8},
      YEAR = {1959},
     PAGES = {229--231},
   MRCLASS = {60.20},
  MRNUMBER = {0125600 (23 \#A2899)},
MRREVIEWER = {M. Lo{\`e}ve}}

@inproceedings {Stein:1972,
    AUTHOR = {Stein, Charles},
     TITLE = {A bound for the error in the normal approximation to the
              distribution of a sum of dependent random variables},
 BOOKTITLE = {Proceedings of the {S}ixth {B}erkeley {S}ymposium on
              {M}athematical {S}tatistics and {P}robability ({U}niv.
              {C}alifornia, {B}erkeley, {C}alif., 1970/1971), {V}ol. {II}:
              {P}robability theory},
     PAGES = {583--602},
 PUBLISHER = {Univ. California Press},
   ADDRESS = {Berkeley, Calif.},
      YEAR = {1972},
   MRCLASS = {60F05},
  MRNUMBER = {0402873 (53 \#6687)},
MRREVIEWER = {A. K. AleSkeviciene},
}
		
@book {Stein:1986,
    AUTHOR = {Stein, Charles},
     TITLE = {Approximate computation of expectations},
    SERIES = {Institute of Mathematical Statistics Lecture Notes---Monograph
              Series, 7},
 PUBLISHER = {Institute of Mathematical Statistics},
   ADDRESS = {Hayward, CA},
      YEAR = {1986},
     PAGES = {iv+164},
      ISBN = {0-940600-08-0},
   MRCLASS = {60F05 (60C05 62E20)},
  MRNUMBER = {882007 (88j:60055)},
MRREVIEWER = {M. B. Woodroofe},
}

@book {Thorisson:2000,
    AUTHOR = {Thorisson, Hermann},
     TITLE = {Coupling, stationarity, and regeneration},
    SERIES = {Probability and its Applications},
 PUBLISHER = {Springer, New York},
      YEAR = {2000},
     PAGES = {xiv+517}}       }

@incollection {Xia:2005,
    AUTHOR = {Xia, Aihua},
     TITLE = {Stein's method and {P}oisson process approximation},
     PAGES = {115--181},
 BOOKTITLE = {An Introduction to {S}tein's method},
    SERIES = {Lecture Notes Series, Institute for Mathematical Sciences,
              National University of Singapore},
    VOLUME = {4},
    EDITOR = {Barbour, A. D. and Chen, Louis H. Y.},
 PUBLISHER = {Singapore University Press},
   ADDRESS = {Singapore},
      YEAR = {2005}}


\begin{thebibliography}{}

\bibitem[Ambrosio et~al., 2005]{Ambrosio_et_al:2005}
Ambrosio, L., Gigli, N., and Savar{\'e}, G. (2005).
\newblock {\em Gradient flows in metric spaces and in the space of probability
  measures}.
\newblock Lectures in Mathematics ETH Z\"urich. Birkh\"auser Verlag, Basel.

\bibitem[Anderson et~al., 1997]{Anderson_et_al:1997}
Anderson, C.~W., Coles, S.~G., and H\"usler, J. (1997).
\newblock Maxima of {P}oisson-like variables and related triangular arrays.
\newblock {\em Ann. Appl. Probab.}, 7:953--971.

\bibitem[Barbour, 1988]{Barbour:1988}
Barbour, A.~D. (1988).
\newblock Stein's method and {P}oisson process convergence.
\newblock {\em J. Appl. Probab.}, (Special Vol. 25A):175--184.
\newblock A celebration of applied probability.

\bibitem[Barbour and Brown, 1992]{Barbour/Brown:1992}
Barbour, A.~D. and Brown, T.~C. (1992).
\newblock Stein's method and point process approximation.
\newblock {\em Stochastic Process. Appl.}, 43(1):9--31.

\bibitem[Barbour and Eagleson, 1983]{Barbour/Eagleson:1983}
Barbour, A.~D. and Eagleson, G.~K. (1983).
\newblock Poisson approximation for some statistics based on exchangeable
  trials.
\newblock {\em Adv. in Appl. Probab.}, 15(3):585--600.

\bibitem[Barbour and Hall, 1984]{Barbour/Hall:1984}
Barbour, A.~D. and Hall, P. (1984).
\newblock On the rate of {P}oisson convergence.
\newblock {\em Math. Proc. Cambridge Philos. Soc.}, 95(3):473--480.

\bibitem[Barbour et~al., 1992]{Barbour_et_al.:1992}
Barbour, A.~D., Holst, L., and Janson, S. (1992).
\newblock {\em Poisson approximation}.
\newblock The Clarendon Press Oxford University Press.

\bibitem[Bingham et~al., 1987]{Bingham_et_al:1987}
Bingham, N.~H., Goldie, C.~M., and Teugels, J.~L. (1987).
\newblock {\em Regular variation}, volume~27 of {\em Encyclopedia of
  Mathematics and its Applications}.
\newblock Cambridge University Press, Cambridge.

\bibitem[Charpentier and Segers, 2009]{Charpentier/Segers:2009}
Charpentier, A. and Segers, J. (2009).
\newblock Tails of multivariate archimedean copulas.
\newblock {\em Journal of Multivariate Analysis}, 100:1521--1537.

\bibitem[Chen, 1975a]{Chen:1975a}
Chen, L. H.~Y. (1975a).
\newblock An approximation theorem for sums of certain randomly selected
  indicators.
\newblock {\em Z. Wahrscheinlichkeitstheorie und Verw. Gebiete}, 33(1):69--74.

\bibitem[Chen, 1975b]{Chen:1975b}
Chen, L. H.~Y. (1975b).
\newblock Poisson approximation for dependent trials.
\newblock {\em Ann. Probability}, 3(3):534--545.

\bibitem[Chen, 1998]{Chen:1998}
Chen, L. H.~Y. (1998).
\newblock Stein's method: some perspectives with applications.
\newblock In Accardi, L. and Heyde, C.~C., editors, {\em Probability Towards
  2000}, volume 128 of {\em Lecture Notes in Statistics}, pages 97--122.
  Springer, New York.

\bibitem[Daley and Vere-Jones, 2003]{Daley/VereJones:2003}
Daley, D.~J. and Vere-Jones, D. (2003).
\newblock {\em An Introduction to the Theory of Point Processes, Vol. I:
  Elementary Theory and Methods}.
\newblock Probability and its Applications. Springer, New York, 2nd edition.

\bibitem[Daley and Vere-Jones, 2008]{Daley/VereJones:2008}
Daley, D.~J. and Vere-Jones, D. (2008).
\newblock {\em An Introduction to the Theory of Point Processes, Vol. II:
  General Theory and Structure}.
\newblock Probability and its Applications. Springer, New York, 2nd edition.

\bibitem[Embrechts et~al., 1997]{Embrechts_et_al:1997}
Embrechts, P., Kl\"uppelberg, C., and Mikosch, T. (1997).
\newblock {\em Modelling {E}xtremal {E}vents for {I}nsurance and {F}inance}.
\newblock Springer, New York.

\bibitem[Erhardsson, 2005]{Erhardsson:2005}
Erhardsson, T. (2005).
\newblock Stein's method for {P}oisson and compound {P}oisson approximation.
\newblock In Barbour, A.~D. and Chen, L. H.~Y., editors, {\em An Introduction
  to {S}tein's method}, volume~4 of {\em Lecture Notes Series, Institute for
  Mathematical Sciences, National University of Singapore}, pages 61--113.
  Singapore University Press, Singapore.

\bibitem[Feidt et~al., 2010]{Feidt_et_al:2010}
Feidt, A., Genest, C., and Ne{\v{s}}lehov\'a, J. (2010).
\newblock Asymptotics of joint maxima for discontinuous random variables.
\newblock {\em Extremes}, 13:35--53.

\bibitem[Fisher and Tippett, 1928]{Fisher/Tippett:1928}
Fisher, R.~A. and Tippett, L. H.~C. (1928).
\newblock Limiting forms of the frequency distributions of the largest or
  smallest member of a sample.
\newblock {\em Proceedings of the Cambridge Philosophical Society},
  24:180--190.

\bibitem[Genest and Ne{\v{s}}lehov\'a, 2007]{Genest/Neslehova:2007}
Genest, C. and Ne{\v{s}}lehov\'a, J. (2007).
\newblock A primer on copulas for count data.
\newblock {\em Astin Bull.}, 37:475--515.

\bibitem[Gumbel, 1958]{Gumbel:1958}
Gumbel, E.~J. (1958).
\newblock Distributions \`a plusieurs variables dont les marges sont donn\'ees.
\newblock {\em C. R. Acad. Sci. Paris}, 246:2717--2719.

\bibitem[Gumbel, 1965]{Gumbel:1965}
Gumbel, E.~J. (1965).
\newblock Two systems of bivariate extremal distributions.
\newblock {\em Bull. Inst. Internat. Statist}, 41:749--763.

\bibitem[Hall, 1979]{Hall:1979}
Hall, P. (1979).
\newblock On the rate of convergence of normal extremes.
\newblock {\em J. Appl. Probab.}, 16(2):433--439.

\bibitem[Hawkes, 1972]{Hawkes:1972}
Hawkes, A.~G. (1972).
\newblock A bivariate exponential distribution with applications to
  reliability.
\newblock {\em Journal of the Royal Statistical Society. Series B
  (Methodological)}, 34(1):pp. 129--131.

\bibitem[Joe, 1993]{Joe:1993}
Joe, H. (1993).
\newblock Multivariate dependence measures and data analysis.
\newblock {\em Comput. Statist. Data Anal.}, 16:279--297.

\bibitem[Joe, 1997]{Joe:1997}
Joe, H. (1997).
\newblock {\em Multivariate models and dependence concepts}, volume~73 of {\em
  Monographs on Statistics and Applied Probability}.
\newblock Chapman \& Hall, London.

\bibitem[Kallenberg, 1986]{Kallenberg:1983}
Kallenberg, O. (1986).
\newblock {\em Random measures}.
\newblock Akademie-Verlag, Berlin, 4th edition.

\bibitem[Kerstan, 1964]{Kerstan:1964}
Kerstan, J. (1964).
\newblock Verallgemeinerung eines {S}atzes von {P}rochorow und {L}e {C}am.
\newblock {\em Zeitschrift f\"ur Wahrscheinlichkeitstheorie und Verwandte
  Gebiete}, 2:173--179.

\bibitem[Kotz et~al., 2000]{Kotz:2000}
Kotz, S., Balakrishnan, N., and Johnson, N.~L. (2000).
\newblock {\em Continuous Multivariate Distributions, {\rm Second Edition}}.
\newblock Wiley, New York.

\bibitem[Le~Cam, 1960]{LeCam:1960}
Le~Cam, L. (1960).
\newblock An approximation theorem for the {P}oisson binomial distribution.
\newblock {\em Pacific J. Math.}, 10:1181--1197.

\bibitem[Leadbetter et~al., 1983]{Leadbetter_et_al:1983}
Leadbetter, M.~R., Lindgren, G., and Rootz\'en, H. (1983).
\newblock {\em Extremes and Related Properties of Random Sequences and
  Processes}.
\newblock Springer, New York.

\bibitem[Lindvall, 2002]{Lindvall:2002}
Lindvall, T. (2002).
\newblock {\em Lectures on the coupling method}.
\newblock Dover Publications Inc., Mineola, NY.
\newblock Corrected reprint of the 1992 original.

\bibitem[Marshall and Olkin, 1967a]{Marshall/Olkin:1967a}
Marshall, A.~W. and Olkin, I. (1967a).
\newblock A generalized bivariate exponential distribution.
\newblock {\em J. Appl. Probability}, 4:291--302.

\bibitem[Marshall and Olkin, 1967b]{Marshall/Olkin:1967}
Marshall, A.~W. and Olkin, I. (1967b).
\newblock A multivariate exponential distribution.
\newblock {\em Journal of the American Statistical Association}, 62(317):pp.
  30--44.

\bibitem[Marshall and Olkin, 1985]{Marshall/Olkin:1985}
Marshall, A.~W. and Olkin, I. (1985).
\newblock A family of bivariate distributions generated by the bivariate
  {B}ernoulli distribution.
\newblock {\em J. Amer. Statist. Assoc.}, 80:332--338.

\bibitem[McNeil et~al., 2005]{McNeil_et_al:2005}
McNeil, A.~J., Frey, R., and Embrechts, P. (2005).
\newblock {\em Quantitative {R}isk {M}anagement: {C}oncepts, {T}echniques, and
  {T}ools}.
\newblock Princeton University Press, Princeton, NJ.

\bibitem[Michel, 1988]{Michel:1988}
Michel, R. (1988).
\newblock An improved error bound for the compound poisson approximation of a
  nearly homogeneous portfolio.
\newblock {\em Astin Bull.}, 17:165--169.

\bibitem[Mitov and Nadarajah, 2005]{Mitov/Nadarajah:2005}
Mitov, K. and Nadarajah, S. (2005).
\newblock Limit distributions for the bivariate geometric maxima.
\newblock {\em Extremes}, 8:357--370.

\bibitem[Nadarajah and Mitov, 2002]{Nadarajah/Mitov:2002}
Nadarajah, S. and Mitov, K. (2002).
\newblock Asymptotics of maxima of discrete random variables.
\newblock {\em Extremes}, 5:287--294.

\bibitem[Nelsen, 2006]{Nelsen:2006}
Nelsen, R.~B. (2006).
\newblock {\em An {I}ntroduction to {C}opulas, {\rm Second Edition}}.
\newblock Springer, New York.

\bibitem[Preston, 1975]{Preston:1975}
Preston, C. (1975).
\newblock Spatial birth-and-death processes.
\newblock {\em Bull. ISI}, 46(2):371--391.

\bibitem[Resnick, 1987]{Resnick:1987}
Resnick, S.~I. (1987).
\newblock {\em Extreme Values, Regular Variation, and Point Processes}.
\newblock Springer, New York.

\bibitem[Serfling, 1975]{Serfling:1975}
Serfling, R.~J. (1975).
\newblock A general {P}oisson approximation theorem.
\newblock {\em Annals of Probability}, 3(4):726--731.

\bibitem[Sklar, 1959]{Sklar:1959}
Sklar, M. (1959).
\newblock Fonctions de r\'epartition \`a {$n$} dimensions et leurs marges.
\newblock {\em Publ. Inst. Statist. Univ. Paris}, 8:229--231.

\bibitem[Stein, 1972]{Stein:1972}
Stein, C. (1972).
\newblock A bound for the error in the normal approximation to the distribution
  of a sum of dependent random variables.
\newblock In {\em Proceedings of the {S}ixth {B}erkeley {S}ymposium on
  {M}athematical {S}tatistics and {P}robability ({U}niv. {C}alifornia,
  {B}erkeley, {C}alif., 1970/1971), {V}ol. {II}: {P}robability theory}, pages
  583--602, Berkeley, Calif. Univ. California Press.

\bibitem[Stein, 1986]{Stein:1986}
Stein, C. (1986).
\newblock {\em Approximate computation of expectations}.
\newblock Institute of Mathematical Statistics Lecture Notes---Monograph
  Series, 7. Institute of Mathematical Statistics, Hayward, CA.

\bibitem[Thorisson, 2000]{Thorisson:2000}
Thorisson, H. (2000).
\newblock {\em Coupling, stationarity, and regeneration}.
\newblock Probability and its Applications. Springer, New York.

\bibitem[Xia, 2005]{Xia:2005}
Xia, A. (2005).
\newblock Stein's method and {P}oisson process approximation.
\newblock In Barbour, A.~D. and Chen, L. H.~Y., editors, {\em An Introduction
  to {S}tein's method}, volume~4 of {\em Lecture Notes Series, Institute for
  Mathematical Sciences, National University of Singapore}, pages 115--181.
  Singapore University Press, Singapore.

\end{thebibliography}
\end{document}